\documentclass[11pt,a4paper,twoside]{amsproc}
\usepackage[top=1in, bottom=1.25in, marginparwidth=1in, inner=1in, outer=1in]{geometry}

\usepackage{lmodern,microtype}
\usepackage[T2A,T1]{fontenc} %
\usepackage[utf8]{inputenc}
\usepackage[USenglish]{babel}
\usepackage{booktabs}%
\usepackage{pinlabel}

\usepackage{graphicx}
\usepackage{tikz,tikz-cd}
\usetikzlibrary{calc,cd,decorations.pathreplacing,decorations.markings, decorations.pathmorphing, decorations.shapes,shapes,decorations,arrows,arrows.meta,bending}
\tikzcdset{arrow style=tikz, diagrams={>=latex}}
\tikzset{>=latex}
\tikzstyle{tight}=[font=\scriptsize, inner sep=1pt, outer sep=1pt]
\tikzstyle{label}=[inner sep=2pt, outer sep=1pt,fill=white, rounded corners]
\usepackage{float}
\usepackage[labelfont=bf]{caption}%
\DeclareCaptionFormat{myformat}{\fontsize{9}{10}\selectfont#1#2#3}
\usepackage[labelfont=up,margin=5pt]{subcaption}
\usepackage{etoolbox}
\usepackage{tabularx}
\AtBeginEnvironment{tabularx}{\footnotesize}
\usepackage{picins}

\usepackage{enumitem}
\usepackage{setspace} 
\usepackage[hidelinks]{hyperref}
\hypersetup{%
	pdfencoding=auto, %
	bookmarksdepth=3%
}
\usepackage[nameinlink,capitalise]{cleveref}
\let\cref\Cref
\usepackage{bookmark}
\usepackage{xurl}%

\usepackage[only,llbracket,rrbracket]{stmaryrd}
\usepackage{amsmath}
\usepackage{nicefrac}
\usepackage{amsfonts}
\usepackage{amssymb}
\usepackage{amsthm}
\usepackage{mathtools}
\usepackage{mathrsfs}
\mathtoolsset{mathic=true}
\def\co{\colon\thinspace\relax}%

\newtheorem{theorem}{Theorem}[section]
\newtheorem*{theorem*}{Theorem}
\newtheorem{lemma}[theorem]{Lemma}

\newtheorem{corollary}[theorem]{Corollary}
\newtheorem*{corollary*}{Corollary}
\newtheorem{proposition}[theorem]{Proposition}
\theoremstyle{definition}
\newtheorem{definition}[theorem]{Definition}
\newtheorem{notation}[theorem]{Notation}
\newtheorem{assumptions}[theorem]{Assumptions}
\newtheorem{example}[theorem]{Example}

\newtheorem{remark}[theorem]{Remark}
\newtheorem{warn}[theorem]{Warning}

\makeatletter
\newcommand*{\math@version@bold}{bold}
\DeclareMathOperator\DD{%
	\textrm{%
		\usefont{T2A}{cmr}{\ifx\math@version\math@version@bold bx\else m\fi}{n}%
		\CYRD
	}%
} 
\makeatother

\renewcommand{\geq}{\geqslant}
\renewcommand{\leq}{\leqslant}

\newcommand{\Z}{\mathbb{Z}}
\newcommand{\Q}{\mathbb{Q}} 
\newcommand{\R}{\mathbb{R}}
\newcommand{\CoeffRing}{\mathbf{R}}
\newcommand{\CoeffSRing}{\mathbf{S}}
\newcommand{\CoeffField}{\mathbf{k}}
\newcommand{\field}{\CoeffField}%
\newcommand{\CoeffFieldTwo}{\mathbb{F}}

\renewcommand{\C}{\mathcal{C}}
\newcommand{\F}{\mathcal{F}}

\newcommand{\halfbullet}[1][bottom]{%
	\mathbin{%
		\mathchoice
		{\halfbulletaux{#1}{0.5ex}}%
		{\halfbulletaux{#1}{0.4ex}}%
		{\halfbulletaux{#1}{0.3ex}}%
		{\halfbulletaux{#1}{0.25ex}}%
	}%
}
\newcommand{\halfbulletaux}[2]{%
	\vc{\begin{tikzpicture}
			\path (0,0) circle (#2*1.45);
			\path[draw, fill=white] (0,0) circle (#2);
			\ifthenelse{\equal{#1}{left}}{%
				\clip (-#2,-#2) rectangle (0,#2);
			}{%
				\ifthenelse{\equal{#1}{right}}{%
					\clip (0,-#2) rectangle (#2,#2);
				}{%
					\ifthenelse{\equal{#1}{top}}{%
						\clip (-#2,0) rectangle (#2,#2);
					}{%
						\ifthenelse{\equal{#1}{bottom}}{%
							\clip (-#2,-#2) rectangle (#2,0);
						}{}}}}%
			\fill (0,0) circle (#2);
			\draw (0,0) circle (#2);
	\end{tikzpicture}}%
}

\DeclareMathOperator{\varG}{\mathtt{G}} %

\newcommand{\A}{\mathcal{A}}
\newcommand{\B}{\mathcal{B}}
\newcommand{\D}{\mathcal{D}}
\newcommand{\Acone}{\A_{\operatorname{cone}}}
\newcommand{\Bcone}{\Binf_{\operatorname{cone}}}
\newcommand{\Ccone}{\C_{\operatorname{cone}}}

\newcommand{\Rd}{R^\partial}
\newcommand{\Binf}{\mathcal{B}^\infty}
\newcommand{\BinfU}{\mathcal{B}^\infty_U}

\newcommand{\Ad}{\operatorname{\mathcal{A}}^\partial}
\newcommand{\Aminus}{\operatorname{\mathcal{A}}^-}
\newcommand{\SaddleBC}{S_{\circ}}
\newcommand{\SaddleCB}{S_{\bullet}}
\newcommand{\hol}{{\circ}}
\newcommand{\sol}{{\bullet}}
\newcommand{\DotcobB}{D_{\bullet}}
\newcommand{\DotcobC}{D_{\circ}}

\newcommand{\W}{\mathcal{W}}

\DeclareMathOperator{\MF}{MF}
\DeclareMathOperator{\I}{\mathcal{I}} 
 
\DeclareMathOperator{\Mod}{Mod}

\DeclareMathOperator{\GL}{\mathit{GL}}
\DeclareMathOperator{\PSL}{\mathit{PSL}}

\DeclareMathOperator{\eval}{eval}
\newcommand{\boldnowarningf}{\text{\itshape\textbf{f}}\,}

\DeclareMathOperator{\Mor}{Mor}
\DeclareMathOperator{\End}{End}

\DeclareMathOperator{\Tw}{Tw}

\DeclareMathOperator{\Cobb}{\operatorname{Cob}_{\bullet}}
\DeclareMathOperator{\Cobl}{\operatorname{Cob}_{/{\mathit{l}}}}

\DeclareMathOperator{\gr}{g}

\newcommand{\ZmodTwo}{\Z/2}

\newcommand{\Diag}{\mathcal{D}} %
\DeclareMathOperator{\HFK}{\widehat{HFK}}

\newcommand{\typeAD}[3]{\!\prescript{}{#1}{#2}^{#3}}
\newcommand{\CAB}{\typeAD{A}{\C}{B}}
\newcommand{\CuAB}{\typeAD{A}{\underline{\C}}{B}}
\newcommand{\ddCAB}{\partial_{\!\CAB}}

\newcommand{\BraidB}{\sigma_{\bullet}}
\newcommand{\BraidW}{\sigma_{\circ}}
\newcommand{\TwistB}{\tau_{\bullet}}
\newcommand{\TwistW}{\tau_{\circ}}
\newcommand{\TwistL}{\tau_{\halfbullet}}

\newcommand{\BB}[1]{\prescript{}{\B}{#1}^{\B}}

\newcommand{\HF}{\operatorname{HF}}

\newcommand{\conn}[1]{\operatorname{x}(#1)}
\newcommand{\conni}{\operatorname{x}_i}
\newcommand{\connx}{\operatorname{x}_1}
\newcommand{\conny}{\operatorname{x}_3}
\newcommand{\connz}{\operatorname{x}_2}
\newcommand{\conns}{\operatorname{x}(\slope)}

\DeclareMathOperator{\HFT}{HFT}
\newcommand{\CFTminus}{\operatorname{CFT}^-}
\DeclareMathOperator{\CFTd}{CFT^\partial}

\newcommand{\KhT}[1]{\llbracket #1 \rrbracket} %
\newcommand{\KhTs}[1]{{\llbracket #1 \rrbracket}^s} %
\DeclareMathOperator{\Kh}{Kh}%
\DeclareMathOperator{\BN}{BN}%
\newcommand{\Khr}{\widetilde{\Kh}}%
\newcommand{\BNr}{\widetilde{\BN}}%

\newcommand{\aKh}{\mathbf{a}}
\newcommand{\rKh}{\mathbf{r}}
\newcommand{\sKh}{\mathbf{s}}

\newcommand{\slope}{\nicefrac{p}{q}}

\newcommand{\mutx}{\operatorname{mut}_1} %
\newcommand{\muty}{\operatorname{mut}_3} %
\newcommand{\mutz}{\operatorname{mut}_2} %
\newcommand{\muti}{\operatorname{mut}_i} %

\newcommand{\ThreePuncturedDisk}{D^2_3}
\newcommand{\FourPuncturedSphere}{S^2_4}

\newcommand{\TwoTorus}{\mathbb{T}^2}
\newcommand{\TwoTorusKh}{\TwoTorus_{4}}
\newcommand{\PlanarCover}{\R^2\smallsetminus\Z^2}

\newcommand{\PuncturedPlane}{\mathbb{R}^2\smallsetminus \Z^2}

\newcommand{\QPI}{\operatorname{\mathbb{Q}P}^1}

\newcommand{\nodes}{\operatorname{nodes}}
\newcommand{\arcs}{\operatorname{arcs}}

\newcommand{\EndsAll}{\mathtt{E}_*}
\newcommand{\EndsThree}{\mathtt{E}}
\newcommand{\Endi  }{\mathtt{e}_1}
\newcommand{\Endii }{\mathtt{e}_2}
\newcommand{\Endiii}{\mathtt{e}_3}

\DeclareFontFamily{U}{mathb}{\hyphenchar\font45}
\DeclareFontShape{U}{mathb}{m}{n}{%
	<-6> mathb5
	<6-7> mathb6
	<7-8> mathb7
	<8-9> mathb8
	<9-10> mathb9
	<10-12> mathb10
	<12-> mathb12 }{}
\DeclareSymbolFont{mathb}{U}{mathb}{m}{n}
\DeclareMathSymbol{\arcT}{\mathbin}{mathb}{"0D}
\DeclareMathSymbol{\arcD}{\mathbin}{mathb}{"05}
\newcommand{\DiagD}{\Diag_\arcD}
\newcommand{\DiagT}{\Diag_\arcT}
\newcommand{\lab}{\ell}
\newcommand{\KK}{\mathbf{K}}

\newcommand{\Cone}[1]{\mathrm{Cone}\left(#1\right)}

\newcommand{\iW}{\iota_\arcD}
\newcommand{\iB}{\iota_\arcT}
\newcommand{\sW}{s_\arcD}
\newcommand{\sB}{s_\arcT}
\newcommand{\dW}{d_\arcD}
\newcommand{\dB}{d_\arcT}

\newcommand{\Ib}{\iota_\bullet}
\newcommand{\Sw}{S_\circ}
\newcommand{\Sb}{S_\bullet}
\newcommand{\Dw}{D_\circ}
\newcommand{\Db}{D_\bullet}

\newcommand{\pairofpants}{%
	\vc{$
		\begin{tikzpicture}[scale=0.07]
			\draw[rotate around={60:(-30:2)}] (-30:2) ellipse (1 and 0.4);
			\draw[rotate around={-60:(210:2)}] (210:2) ellipse (1 and 0.4);
			\draw (90:2) ellipse (1 and 0.4);
			\draw[-] ($(210:2)+(120:1)$) .. controls ($(210:2)+(120:1)+(30:1)$) and (-1,1) .. (-1,2);
			\draw[-] ($(-30:2)+(60:1)$) .. controls ($(-30:2)+(60:1)+(150:1)$) and (1,1) .. (1,2);
			\draw[-] ($(-30:2)+(60:-1)$) .. controls ($(-30:2)+(60:-1)+(150:1)$) and ($(210:2)+(120:-1)+(30:1)$) .. ($(210:2)+(120:-1)$);
		\end{tikzpicture}
		$}%
}
\newcommand{\HMS}{\text{\upshape\textsc{hms}}\big(\pairofpants\big)}
\newcommand{\DW}[2]{\(#1\)\,\colorbox{lightgray}{\(#2\)}}
\newcommand{\DWo}[1]{\colorbox{lightgray}{#1}}

\newcommand{\vc}[1]{\vcenter{\hbox{#1}}}%
\newcommand{\mypic}[2]{%
  \newcommand{#2}{%
    \vc{%
      \includegraphics[page=#1]%
      {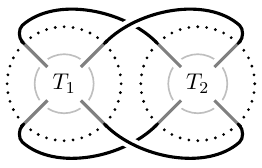}%
    }%
  }%
}%

\mypic{1}{\tanglepairingI}
\mypic{2}{\tanglepairingII}
\mypic{3}{\ConnZ}
\mypic{4}{\Li}
\mypic{5}{\Lo}
\mypic{6}{\Ni}
\mypic{7}{\No}
\mypic{6}{\ConnX}
\mypic{7}{\ConnY}
\mypic{10}{\Nil}
\mypic{11}{\Nol}
\mypic{12}{\CrossingL}
\mypic{13}{\CrossingR}
\mypic{20}{\LoDotB}
\mypic{21}{\NiDotL}
\mypic{22}{\NiDotR}
\mypic{23}{\NoDotT}
\mypic{24}{\NoDotB}
\mypic{26}{\Circle}
\mypic{27}{\WindingNonSpecialA}
\mypic{28}{\WindingNonSpecialB}
\mypic{29}{\WindingNonSpecialC}
\mypic{30}{\WindingNonSpecialD}
\mypic{31}{\WindingNonSpecialE}
\mypic{32}{\LittleSpecialWrapping}
\mypic{33}{\MuchSpecialWrapping}
\mypic{34}{\GeographyForSpecials}
\mypic{35}{\tube}
\mypic{36}{\plane}
\mypic{37}{\planedot}
\mypic{38}{\planedotdot}
\mypic{39}{\planedotstar}
\mypic{40}{\SpherePic}
\mypic{41}{\Spheredot}
\mypic{42}{\Spheredotdot}
\mypic{43}{\Spheredotdotdot}
\mypic{44}{\DiscT}
\mypic{45}{\DiscB}
\mypic{46}{\DiscR}
\mypic{47}{\DiscL}
\mypic{48}{\DiscTdot}
\mypic{49}{\DiscBdot}
\mypic{50}{\DiscRdot}
\mypic{51}{\DiscLdot}
\mypic{52}{\GeographyCovering}
\mypic{53}{\MutationTangle}
\mypic{54}{\MutationTangleX}
\mypic{55}{\MutationTangleY}
\mypic{56}{\MutationTangleZ}
\mypic{57}{\tanglesum}
\mypic{58}{\braidaction}
\mypic{59}{\braidactionB}
\mypic{60}{\braidactionW}

\begin{document}
\title{On mutation invariance in Khovanov homology}

\newcommand{\myemail}[1]{\href{mailto:#1}{#1}}
\author{Artem Kotelskiy}
\address{Mathematics Department \\ Stony Brook University}
\email{\myemail{artofkot@gmail.com}}
\urladdr{\url{https://artofkot.xyz/}}

\author{Liam Watson}
\address{Department of Mathematics \\ University of British Columbia}
\email{\myemail{liam@math.ubc.ca}}
\urladdr{\url{https://personal.math.ubc.ca/~liam/}}

\author{Claudius Zibrowius}
\address{Fakultät für Mathematik \\ Ruhr-Universität Bochum}
\email{\myemail{claudius.zibrowius@rub.de}}
\urladdr{\url{https://cbz20.raspberryip.com/}}

\begin{abstract}
We show that reduced Khovanov homology over any field is invariant under component-preserving Conway mutation. 
Our proof relies on strong geography restrictions for a certain Khovanov multicurve invariant associated with Conway tangles that we introduced in previous work \cite{KWZ}. 
Applying ideas from homological mirror symmetry, we give a full classification of the components of this invariant.
\end{abstract}
\maketitle

\let\emptyset\undefined
\let\coloneqq\undefined
\let\eqqcolon\undefined

{
\renewcommand{\thetheorem}{\Alph{theorem}}
\addtocounter{theorem}{1}

\section{Introduction}\label{sec:intro}

The purpose of this paper is twofold. 
First, we establish invariance of reduced Khovanov homology with field coefficients under component-preserving Conway mutation. 
Second, we develop structural results for certain multicurve invariants associated with Conway tangles. 
An earlier, unpublished version of this latter work establishes analogous results over the field~\(\CoeffFieldTwo\) of two elements \cite{KWZ-linear}. %
The present paper formulates these results over general fields by incorporating sign refinements, which is an essential part of the mutation result.

The multicurve results of this paper have already found a number of applications.
Over~\(\CoeffFieldTwo\), 
they form a key input to the proof of an equivariant version 
of the Cosmetic Surgery Conjecture
and allow for an elementary reproof 
of the Cosmetic Crossing Conjecture 
for split links \cite{KLMWZ}, 
originally due to Wang \cite{Wang}.
Furthermore, 
they provide the foundation of the thin gluing criteria established in~\cite{KWZ-thin}
and play an important role in the definition of the concordance invariant \(\vartheta\) from Khovanov homology and formulas for the Rasmussen invariant of Whitehead doubles and other satellite knots \cite{LZ};
and they explain phenomena observed by the second author in the study of the Khovanov homology of strongly invertible knots \cite{Watson2017,KWZ-strong}. 
With the extension to arbitrary fields established here, 
these applications generalize accordingly. 

\subsection{Khovanov homology and Conway mutation}
Conway mutation is a basic operation on knots and links that is defined as follows:
Given a knot or link \(L\) in \(S^3\), consider a sphere that intersects \(L\) transversely in four points, a so-called Conway sphere. 
This sphere decomposes \(S^3\) into two three-dimensional balls and \(L\) into two Conway tangles within those balls. 
Conway mutation is the process of regluing these tangles using one of the three hyper-elliptic involutions of the Conway sphere; see \cref{sec:mutation} for details. 
The new knot or link obtained is called a Conway mutant of~\(L\). 
Examples are shown in \cref{fig:example}. 
A Conway mutation is said to be component-preserving if any two open tangle strands in the Conway tangles lie on the same component of \(L\) if and only if they lie on the same component of the mutant.
Note that there are two distinct cases to consider: 
If all open tangle components lie on the same component of the link \(L\), 
as in the example in \cref{fig:example:mutant}, 
any mutation is component-preserving. 
If the two open components of each tangle lie on different components of \(L\),
only mutation along one of the axes is component-preserving.
This is illustrated in \cref{fig:example:mutant:Wehrli}.  

\begin{figure}[t]
	\centering
	\begin{subfigure}{0.5\textwidth}
		\centering
		\includegraphics[scale=0.9]{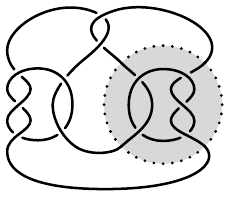}
		\includegraphics[scale=0.9]{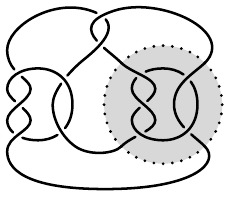}
		\caption{}\label{fig:example:mutant}
	\end{subfigure}%
	\begin{subfigure}{0.5\textwidth}
		\centering
		\includegraphics[scale=0.9]{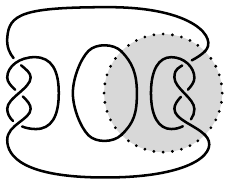}
		\includegraphics[scale=0.9]{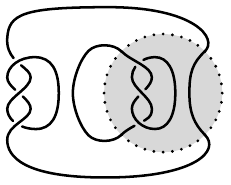}
		\caption{}\label{fig:example:mutant:Wehrli}
	\end{subfigure}
	\caption{%
		Two pairs of mutant links: 
	    Figure (a) shows the Kinoshita-Terasaka knot %
	    and the Conway knot, %
	    whose Khovanov homologies agree. 
	    One knot is obtained from the other by rotating 
	    the highlighted Conway tangle in the plane by \(\pi\).
	    In contrast, the links shown in (b) 
	    have different Khovanov homology with rational coefficients. 
	    Here, mutation does not preserve components. 
	}\label{fig:example}
\end{figure}

Interestingly, Conway mutation preserves many knot and link invariants, including the Seifert matrix up to S-equivalence \cite{Cooper}, the volume of hyperbolic knots \cite{mutation_volume}, the Alexander polynomial, the Jones polynomial, and, more generally, the HOMFLY-PT polynomial \cite{mutation_HOMFLY}. 
Counterexamples include the Seifert genus; for instance, the Kinoshita-Terasaka knot (\cref{fig:example:mutant}) has Seifert genus 2, while it is equal to 3 for the Conway knot, its mutant. 
Consequently, knot Floer homology \(\HFK\), which categorifies the Alexander polynomial, is not mutation invariant, since it detects the Seifert genus \cite{OS-taut}. 
However, the total rank of knot Floer homology is mutation invariant, as is a certain singly graded version of \(\HFK\) \cite{pqSym}.  

For Khovanov homology, a categorification of the Jones polynomial, the situation is also subtle, but in a different way. 
In 2003, Wehrli found a pair of mutant links (namely the one in \cref{fig:example:mutant:Wehrli}) whose Khovanov homologies with rational coefficients have different dimensions \cite{WehrliCounterexample}. 
In 2005, Bar-Natan outlined a strategy for proving that Khovanov homology was invariant under component-preserving Conway mutation  \cite{BN_mutation}, but there were gaps in the argument. 
These were subsequently filled in by Wehrli for Khovanov homology with coefficients in the field \(\CoeffFieldTwo\) of two elements \cite{wehrli2010mutation}, and Bloom gave an independent proof that generalizes to odd Khovanov homology \cite{Bloom}. 
Yet another proof of mutation invariance of Khovanov homology with \(\CoeffFieldTwo\) coefficients was given by Saltz \cite{Saltz}. 
This perspective generalizes to a class of strong Khovanov-Floer theories that includes Szabó homology and singular instanton homology. 

These proofs all make a crucial appeal to the tangle {\it in situ} as part of a link diagram. 
In contrast, our proof establishes symmetries of the Khovanov tangle invariants, 
which may be seen as a categorification of the symmetry seen at the level of the Kaufman bracket. In earlier work, we gave a fourth proof of mutation invariance with \(\CoeffFieldTwo\) coefficients \cite{KWZ}. This proof goes back to Bar-Natan's original proposal for the proof of mutation invariance \cite{BN_mutation} and is extremely short. 
The key ingredient is a certain choice of basis for the morphism spaces in Bar-Natan's cobordism category; these basis elements have the property that Conway mutation simply acts by multiplication by \(\pm1\). 
By classifying Bar-Natan's tangle invariants using multicurves, the authors then used this observation to prove mutation invariance of Rasmussen invariants over fields of any characteristic \cite{KWZ}. Combining these ideas with our new geography restrictions for multicurves discussed below, we now show:

\begin{theorem}\label{thm:mutation_invariance:links}
	For any field \(\CoeffField\) and any two pointed links \(L\) and \(L'\) that are related by component-preserving Conway mutation, the reduced Khovanov homologies of \(L\) and \(L'\) with coefficients in \(\CoeffField\) are bigraded isomorphic vector spaces.
\end{theorem}

Here, we require the link components carrying the basepoint to intersect the Conway sphere.
Since Conway mutation on a knot is always component-preserving, we obtain:

\begin{corollary*}
	Over any field \(\CoeffField\), the reduced Khovanov homologies of any two Conway mutant knots are bigraded isomorphic vector spaces.\qed
\end{corollary*}

\subsection{On the geography of multicurves} 
A Conway tangle is a proper embedding of two intervals and a finite (possibly empty) set of circles into the three-dimensional ball $B^3$.
We consider such embeddings up to ambient isotopy fixing the boundary pointwise. 
Given a Conway tangle $T$ there are two associated objects in the Fukaya category of the four-punctured sphere $\partial B^3\smallsetminus\partial T$, namely, the bigraded multicurves
\[
\HFT(T)
\quad
\text{and}
\quad
\Khr(T)
\] 
The first invariant was defined by the third author using Heegaard Floer theory \cite{pqMod}; it should be understood as a relative version of knot Floer homology \(\HFK\) %
\cite{OSHFK,Jake}.
Similarly, the second invariant should be understood as a relative version of reduced Khovanov homology; it was defined by the authors as a geometric interpretation of Bar-Natan's Khovanov homology for tangles \cite{KWZ}. 
Remarkably, these two multicurve invariants share many formal properties. 
For instance, both satisfy gluing theorems that recover the corresponding link homology theory in terms of Lagrangian Floer theory \cite{KWZ,pqMod} and
both invariants detect rational tangles or, more generally, they detect if a tangle is split \cite{LMZ,KLMWZ}. 
As a key ingredient in the proof of \cref{thm:mutation_invariance:links}, the main purpose of this paper is to extend the similarity between \(\HFT(T)\) and \(\Khr(T)\) to the fundamental structure of the invariants themselves.

Setting aside technicalities such as local systems and gradings, we now briefly explain the main structural result about \(\Khr(T)\) established in this paper.
By definition, both \(\HFT(T)\) and \(\Khr(T)\) take the form of collections of immersed curves on the four-punctured sphere 
\(S^2 \smallsetminus 4\text{pt} =\partial B^3\smallsetminus\partial T\), 
considered up to regular isotopy. 
In each case, the immersed curves in question are rather highly structured. To describe this structure, we first pass to a planar cover 
that factors through the toroidal two-fold cover: %
\[
\left(\mathbb R^2 \smallsetminus \mathbb Z^2\right)  
\to 
\left(\mathbb T^2 \smallsetminus 4\text{pt}\right) 
\to 
\left(S^2 \smallsetminus 4\text{pt}\right) 
\]
Our interest is in those immersed curves whose lift to this cover are homotopic to straight lines; in general terms, the lines of interest fall into two classes, which we call rational and special. 
This dichotomy will be discussed in detail (see \cref{def:RationalVsSpecialKht}), but the important point is that the difference between the two classes amounts to how the lines interact with lifts of the tangle ends in the planar cover. In \cite{pqSym}, the third author solved the geography question for components of \(\HFT(T)\):

\begin{theorem*}\label{thm:geography:HFT:intro}
Every component of \(\HFT(T)\) is either rational or special.
\end{theorem*}

\(\HFT(T)\) is a geometric interpretation of an algebraic invariant, namely, a curved type~D structure \(\CFTd(T)\) over an algebra~\(\Ad\). The proof of this theorem is based on the existence of an extension of \(\CFTd(T)\) to a curved type D structure \(\CFTminus(T)\) over an algebra~\(\Aminus\) that comes with an epimorphism \(\Aminus\rightarrow\Ad\). This additional structure and, ultimately, the geography result that it leads to play a key role in establishing $\delta$-graded mutation invariance in link Floer homology. To some degree, this extension remains internal to tangle Floer homology; its existence appeals to the Heegaard diagram present in the definition of the invariant. For context, an earlier appeal to extensions of this form appears in work of the second author with Hanselman and Rasmussen \cite{HRW-prop}.

Similarly, \(\Khr(T)\) is a geometric interpretation of a type~D structure \(\DD_1(T)^{\B}\) over an algebra~\(\B\). This type~D structure recasts Bar-Natan's tangle variant of Khovanov homology \cite{BarNatanKhT}. 
Given this starting point it is considerably less clear where the requisite additional structure comes from. Nevertheless, in this paper, we show:

\begin{theorem}\label{thm:geography:Khr:intro}
	Every component of \(\Khr(T)\) is either rational or special.
\end{theorem}

\begin{figure}[t]
	\centering
	\tikzstyle{mypath}=[decorate,decoration={snake,post length=3pt,segment length=3pt,amplitude=1pt,post length=5pt,pre length=5pt}]
	\begin{tikzpicture}[xscale=3]
		\node (cor) at (0,0) [rounded corners,draw] {cube of resolutions};
		\node (mf) at (0,2) {\(M(\Diag_T)\)};
		\node (BinfU) at (2,2) {\(\DD(\Diag_T)^{\BinfU}\)};%
		\node (B) at (2,0) {\(\DD(T)^{\B}\)};
		\footnotesize
		\draw[mypath,->] (cor) -- (mf);
		\draw[mypath,->] (cor) -- (B);
		\draw[<->] (mf) -- node (x)[above]{\(\HMS\)} (BinfU);
		\draw[->] (BinfU) -- node (x)[right]{\(U=0\)} (B);
	\end{tikzpicture}
	\caption{%
		Outline of the proof of the extension property 
		appealing to the homological mirror symmetry \(\protect\HMS\) 
		for the pair of pants. 
	}\label{fig:overview_extension}
\end{figure}

A more detailed version of this statement is given in \cref{thm:geography:Khr} once the necessary notation is in place.
Its proof, 
outlined in \cref{fig:overview_extension}, 
is also based on an extension of the algebraic invariant: 
We extend \(\DD_1(T)^{\B}\) 
to a type~D structure \(\DD_1(\Diag_T)^{\BinfU}\) 
over an algebra~\(\BinfU\).
However, the algebra \(\BinfU\) is an \(A_\infty\)-algebra 
and the extension \(\DD_1(\Diag_T)^{\BinfU}\) 
reaches outside of Bar-Natan's framework. 
Namely, we leverage the matrix factorizations framework~\cite{KR_mf_I,KR_mf_II} 
to define an $A_\infty$-enhancement of Bar-Natan's cobordism category $\Cobb$, 
and in the case of Conway tangles, 
we describe this $A_\infty$-enhancement explicitly. 
The latter description depends 
on a particular quasi-isomorphism of algebras 
provided by the homological mirror symmetry 
of the three-punctured sphere~\cite{AAEKO,LekPol,Orlov}.
This quasi-isomorphism allows us to define the extension \(\DD_1(\Diag_T)^{\BinfU}\)
in terms of a matrix factorization \(M(\Diag_T)\)
associated with a tangle diagram \(\Diag_T\) of \(T\) and its cube of resolutions. 

\subsection*{Notation}
Throughout this paper, we fix a commutative unital ring \(\CoeffRing\). 
When we restrict coefficients to fields, 
we will write \(\CoeffField\) for \(\CoeffRing\).
The field of two elements is denoted by \(\CoeffFieldTwo\).

\subsection*{Organization}
\Cref{sec:review:Kh}
reviews the construction of immersed curve invariants
in Khovanov theory for tangles, 
following 
\cite{KWZ}, 
and provides the framework for the precise formulation
of \cref{thm:geography:Khr:intro} 
in \cref{sec:new:Kh} 
(see \cref{thm:geography:Khr}). 
There we also introduce two further essential new results: 
naturality of the invariants 
under the mapping class group action over arbitrary fields 
(\cref{thm:twisting:general-coeff}) 
and connectivity detection 
(\cref{prop:connectivity-detected-by-rationals-of-odd-length}; see also \cref{thm:Khr:connectivity-detection}). 
Deferring their proofs to later sections, 
we apply these tools 
in \cref{sec:mutation}
to establish the central result of this paper: 
component-preserving mutation invariance 
of reduced Khovanov homology over any field 
(\cref{thm:mutation_invariance:links}).

\cref{sec:MCGaction,sec:connectivity} 
contain the proofs of naturality and connectivity detection, 
respectively. 
The proof of 
\cref{thm:geography:Khr:intro}, 
occupying the remainder of the paper, 
is reduced 
in \cref{sec:geography:Khr}
to a statement about the behaviour 
of multicurve invariants at the marked boundary point \(*\) 
(\cref{thm:no_wrapping_around_special}). 
Its proof combines two ingredients: 
a variation of Khovanov--Rozansky's framework of matrix factorizations, 
developed in 
\cref{sec:mfs}, 
and a deep result from homological mirror symmetry 
for the three-punctured sphere, 
which we recall and adapt to our setup in 
\cref{sec:quasi-iso}. 
\cref{sec:no_wrapping_around_special} 
then uses these two inputs, 
together with a technical lemma proved in 
\cref{sec:delooping:mf}, 
to construct the desired  \(A_\infty\)-extension 
of the tangle multicurve invariant.
}
\section{\texorpdfstring{The tangle invariant \(\Khr\)}{The tangle invariant Khr}}\label{sec:review:Kh}

\subsection{Bar-Natan's cobordism category and the quiver algebra 
	\texorpdfstring{\(\B\)}{B}}

\begin{notation}\label{def:Cob}
	Let \(\EndsAll\) be the set of the following four points 
	on the boundary of the 2-dimensional unit disk:
	\begin{align*}
		\ast & = -e^{-i\pi/4}
		&
		\Endi & = -e^{i\pi/4}
		&
		\Endii & = e^{-i\pi/4}
		&
		\Endiii & = e^{i\pi/4}
	\end{align*}
	We call \(*\) the distinguished end
	and write \(\EndsThree\) for \(\EndsAll\smallsetminus\{*\}\).
	\end{notation}

Let \(\CoeffRing[\varG]\) denote the polynomial ring in a single variable \(\varG\). 
Consider the following graded \(\CoeffRing[\varG]\)-linear category, denoted $\operatorname{Cob}$:
	\begin{enumerate}
		\item 
		The objects of the category are 
		crossingless tangle diagrams 
		in the closed unit disk 
		with four ends on \(\EndsAll\). 
		\item 
		A morphism \( T_0\rightarrow T_1 \) 
		between two objects \( T_0 \) and \( T_1 \)
		is a formal \(\CoeffRing[\varG]\)-linear combination of dotted cobordisms 
		from \( T_0 \) to \( T_1 \). 
		\end{enumerate}
In this setting, a cobordism is an orientable (possibly disconnected) surface with boundary,
		together with an identification of the boundary with
		\[
		(T_0\times\{0\})\cup (\EndsAll\times [0,1])\cup (T_1\times\{1\})
		\subset
		\partial(D^2\times [0,1]).
		\]
A dotted cobordism is a cobordism in the above sense 
		together with a finite (possibly empty) set 
		of marked points \(\bullet\) 
		on the surface. 
		We consider dotted cobordisms up to 
		boundary preserving homeomorphisms 
		that bijectively map the set of marked points to each other.
		Identity morphisms are given by product cobordisms and 
		composition of cobordisms is given by concatenation. We define the quantum grading of a dotted cobordism \( C \) by
	\[
	q(C)=\chi(C)-2-\#\{\text{marked points on }C\}
	\] 
	where \(\chi(C)\) denotes the Euler characteristic of \(C\).
	The category \( \Cobb \) is defined as a quotient of $\operatorname{Cob}$
	by imposing certain relations on morphism spaces \cite[Section 4.1]{KWZ}. 
	Here is a concise summary: 
	\begin{equation}\label{eq:cobb_relations}
		\SpherePic=0,~  
		\Spheredot=1,~ 
		\planedotdot=-\varG\cdot\planedot\,,~ 
		\tube=\DiscLdot\DiscR+\DiscL\DiscRdot +\varG\cdot \DiscL \DiscR,~
		\planedotstar =0
	\end{equation}
	The above quantum grading descends to \( \Cobb \). 
\begin{remark}\label{rem:G-vs-H}
	The variable \(\varG\) is related to the variable \(H\) used in previous works by the identity \(\varG=-H\) \cite{KWZ,KWZ-strong,KWZ-thin}. 
	It turns out that the action of \(\varG\) on the tangle invariants is more natural than the action by \(H\). 
	Specifically, \cref{thm:twisting:general-coeff} does not hold if we use \(H\) in place of \(\varG\).
\end{remark}

\begin{remark}\label{rmk:frobenius-extensions}
	There are different versions of the category \(\Cobb\)
	 corresponding to various Frobenius extensions \cite{Kh_frob}. 
	Bar-Natan's original version 
	\cite[Page~1493]{BarNatanKhT}
	corresponds to the $\mathcal F_1$ Frobenius extension 
	\cite[Section~11.2]{BarNatanKhT}
	while \(\Cobb\) corresponds to~$\mathcal F_7$.
	Furthermore, 
	\(\Cobb\) is equivalent 
	to a category \(\Cobl\) 
	whose morphisms are represented by undotted cobordisms
	\cite[Definition 4.3]{KWZ}.
	While \(\Cobb\) is often more convenient for calculations, the category \(\Cobl\) is somewhat more natural than \(\Cobb\), specifically when it comes to Conway mutation. In fact, the equivalence between \(\Cobb\) and \(\Cobl\) is crucial for understanding how Conway mutation affects Bar-Natan's tangle invariants \cite[Theorem 9.8, Observation~4.24]{KWZ}. 
\end{remark}

\begin{example}\label{ex:cobb}
	The category \(\Cobb\) has two distinguished objects, 
	the crossingless tangles without closed components \(\Li\) and \(\Lo\). 
	There are distinguished morphisms between \(\Li\) and \(\Lo\) 
	with maximal quantum grading, given by contractible, unmarked cobordisms; 
	these are called the saddle cobordisms. 
	The identity morphisms are the product cobordisms
	\(\Li\times [0,1]\) and \(\Lo\times [0,1]\).
	By the final relation in \eqref{eq:cobb_relations}, 
	a single dot on the component of a product cobordism 
	that is adjacent to the distinguished end \(\ast\)
	results in the zero cobordism. 
	However, a single dot on the other component
	yields a non-zero morphism, 
	called a dot cobordism.
\end{example}

The following result is due to Bar-Natan; see \cite[Observation~4.18]{KWZ}.

\begin{lemma}[Delooping Lemma]
	\label{lem:delooping}
	Let \(\Diag_\varnothing\in\Cobb\) 
	be a diagram obtained 
	from some \(\Diag\in\Cobb\)
	by removing a closed component. In the additive enlargement of \(\Cobb\),
	there are isomorphisms
	\[
	\Diag
	\longrightarrow
	q^{-1}\Diag_\varnothing\oplus q^{+1}\Diag_\varnothing
	\longrightarrow
	\Diag	\]
	defined by extending the following morphisms 
	by product cobordisms: 
	\[
	\begin{tikzcd}[row sep=-0.2cm, column sep=2cm]
		&
		q^{-1} \varnothing
		\arrow{dr} [near start] {\DiscRdot}%
		\\
		\Circle
		\arrow{ur} [near end] {\DiscL}%
		\arrow{dr} [swap, near end] {\DiscLdot+\varG\cdot\DiscL}%
		&\oplus&
		\Circle
		\\
		&
		q^{+1}\varnothing
		\arrow{ur} [swap, near start] {\DiscR\phantom{+\varG\cdot\DiscL}}%
	\end{tikzcd}
	\]
\end{lemma}

By iteratively applying the Delooping Lemma to closed components, 
every diagram \(\Diag\in\Cobb\) 
is isomorphic in the additive enlargement of \(\Cobb\) 
to a direct sum of (quantum grading shifted) copies of either \(\Li\) or \(\Lo\).

\begin{definition}\label{def:B}
	Let \(\mathcal{Q}\) be the quiver 
	\[
	\begin{tikzcd}[row sep=2cm, column sep=1.5cm]
		\bullet
		\arrow[leftarrow,in=145, out=-145,looseness=5]
		{rl}[description]{{}_{\bullet}D_{\bullet}}
		\arrow[leftarrow,bend left]
		{r}[description]{{}_{\bullet}S_{\circ}}
		&
		\circ
		\arrow[leftarrow,bend left]
		{l}[description]{{}_{\circ}S_{\bullet}}
		\arrow[leftarrow,in=35, out=-35,looseness=5]
		{rl}[description]{{}_{\circ}D_{\circ}}
	\end{tikzcd}
	\]
	and let \(\CoeffRing[\mathcal{Q}]\) denote the free path algebra of \(\mathcal{Q}\). 
	The algebra \(\B\) is the quotient of \(\CoeffRing[\mathcal{Q}]\) by the ideal
	\[
			({}_{\bullet}D_{\bullet} \cdot {}_{\bullet}S_{\circ},
			{}_{\bullet}S_{\circ}\cdot  {}_{\circ}D_{\circ},
			{}_{\circ}D_{\circ}\cdot  {}_{\circ}S_{\bullet},
			{}_{\circ}S_{\bullet}\cdot  {}_{\bullet}D_{\bullet})
	\]
	and \(\B\) carries a quantum grading \(q\) determined by 
	\[
	q({}_{\bullet}D_{\bullet}) = q({}_{\circ}D_{\circ}) = -2
	\qquad 
	\text{and}
	\qquad
	q({}_{\circ}S_{\bullet}) = q({}_{\bullet}S_{\circ}) = -1.
	\] 
	Setting 
	\[
	\varG
	= 
	{}_{\bullet}S_{\circ}{}_{\circ}S_{\bullet}+
	{}_{\circ}S_{\bullet}{}_{\bullet}S_{\circ}-
	{}_{\circ}D_{\circ}-
	{}_{\bullet}D_{\bullet}
	\in\B
	\]
	 gives \(\B\)
	an \(\CoeffRing[\varG]\)-module structure.
	The idempotents given by constant paths on 
	\(\bullet\) and \(\circ\) are denoted by
	\(\iota_{\bullet}\) and \(\iota_{\circ}\), respectively.
\end{definition}

\begin{notation}
	We will sometimes abuse notation by using $S$ 
	for either ${}_{\circ}S_{\bullet}$, ${}_{\bullet}S_{\circ}$, or
	${}_{\circ}S_{\bullet}+{}_{\bullet}S_{\circ}$
	and using $D$ 
	for either ${}_{\circ}D_{\circ}$, ${}_{\bullet}D_{\bullet}$, or
	${}_{\circ}D_{\circ}+{}_{\bullet}D_{\bullet}$.
	This allows shorthand notation \(\varG= S^2-D\).
	In addition, 
	a subscript $\halfbullet \in \{\circ, \bullet\}$ on the left or right of an algebra element $a$ 
	indicates that left or right multiplication by $\iota_{\halfbullet}$, respectively, 
	leaves $a$ unchanged.
	This allows shorthand notation 
	such as $S_{\circ}={}_{\bullet}S_{\circ}$ and 
	$
	S^3_{\bullet}
	=
	{}_{\circ}S_{\bullet}\cdot {}_{\bullet}S_{\circ} \cdot  {}_{\circ}S_{\bullet}
	$.
\end{notation}

The following is a consequence of \cite[Theorem~4.21]{KWZ}.
	
\begin{theorem}\label{thm:OmegaFullyFaithful}
	There is a \(\varG\)-equivariant isomorphism
	\[ \omega\co\B\rightarrow \End_{\Cobb}(\Li\oplus\Lo)\]
	uniquely determined by 
	\(
	\bullet\mapsto\Lo, 
	\circ\mapsto\Li
	\)
	and the requirement that the \(S\) are sent to saddle morphisms and the \(D\) are sent to dot morphisms (following the terminology of \cref{ex:cobb}).
\end{theorem}

\begin{definition}\label{def:standard_basis}
	The standard basis of \(\B\) is given by 
	\begin{align*}
		\iota_\bullet\B\iota_\bullet
		&=\CoeffRing
		\langle
		\iota_\bullet,
		\varG^n\! D_\bullet,
		\varG^n\!S S_\bullet\rangle_{n\geq0}
		&
		\iota_\circ\B\iota_\bullet
		&=\CoeffRing
		\langle
		\varG^n\! S_\bullet\rangle_{n\geq0}
		\\
		\iota_\circ\B\iota_\circ
		&=\CoeffRing
		\langle
		\iota_\circ,
		\varG^n\!D_\circ,
		\varG^n\!S S_\circ
		\rangle_{n\geq0}
		&
		\iota_\bullet\B\iota_\circ
		&=\CoeffRing
		\langle
		\varG^n\! S_\circ
		\rangle_{n\geq0}
	\end{align*}
	We will call elements of the standard basis of \(\B\) pure.
	We will also make use of the \(\varG\)-equivariant basis of \(\B\) given by 
	\begin{align*}
		\iota_\bullet\B\iota_\bullet
		&=\CoeffRing
		\langle
		\varG^n\! D_\bullet,
		\varG^n\! \iota_\bullet
		\rangle_{n\geq0}
		&
		\iota_\circ\B\iota_\bullet
		&=\CoeffRing
		\langle
		\varG^n\! S_\bullet
		\rangle_{n\geq0}
		\\
		\iota_\circ\B\iota_\circ
		&=\CoeffRing
		\langle
		\varG^n\! D_\circ,
		\varG^n\! \iota_\circ
		\rangle_{n\geq0}
		&
		\iota_\bullet\B\iota_\circ
		&=\CoeffRing
		\langle
		\varG^n\!S_\circ
		\rangle_{n\geq0}
	\end{align*}
\end{definition}

\subsection{\texorpdfstring{From tangles to complexes over \(\B\)}{From tangles to complexes over B}}\label{sub:complexes}

Consider the set of marked points \(\EndsAll=\{\pm e^{\pm i\pi/4}\}\) (\cref{def:Cob})
as a subset of the equatorial circle \(S^2\cap \R^2\times\{0\}\subset S^2\). 
Let \( T \) be an oriented tangle in the unit 3-ball \(D^3\) 
with four ends at \(\EndsAll\times\{0\}\subset\partial D^3\).
Choose a diagram \(\Diag\) for \( T \)
in the unit disk \(D^2\) with four ends at \(\EndsAll\)
and fix an order on the crossings of \(\Diag\).
Let \( n_+ \) be the number of all positive crossings, 
\( n_- \) the number of all negative crossings so that 
\( n=n_++n_- \) is the total number of crossings. 
Define the 0- and 1-resolution of a crossing by
\[ 
\text{(0-resolution) }
\No 
\leftarrow 
\CrossingR 
\rightarrow
\Ni
\text{ (1-resolution)}
\] 
Given vertices $v,v'\in \{0,1\}^n$, where $n$ is a non-negative integer, an edge \(v\rightarrow v'\) is a pair of vertices that differ in exactly one coordinate. This gives rise to the standard preorder on vertices of  the hypercube: write \(v\leq v'\) if and only if there exists a (possibly length 0) sequence of edges \(v\rightarrow\dots\rightarrow v'\). 

Bar-Natan defines a chain complex \( \KhT{\Diag}\) over \(\Cobb\) as follows.
Each vertex \( v=(v_1,\dots,v_n)\in\{0,1\}^n \) is assigned
the object \(\Diag(v) \in\Cobb\) 
obtained by taking 
the \( v_i \)-resolution of 
the \( i^\text{th} \) crossing 
for each \( i=1,\dots, n\). 
Each edge \(v\rightarrow v'\) (where $v_i=0$ and $v'_i=1$, say)
is assigned the morphism  
in \(\Cobb(\Diag(v),\Diag(v'))\)
that is a product cobordism outside a neighbourhood of the \( i^\text{th} \) crossing and in a neighbourhood of the \( i^\text{th} \) crossing 
agrees with the saddle cobordism, up to a sign. The sign is chosen %
so that that the number of signs 
on the boundary of every 2-dimensional face of the cube is odd. 
Finally, the objects \(\Diag(v) \in\Cobb\) 
are shifted into the following quantum and homological gradings:
\[
\KhT{\Diag}_v
= 
h^{|v|-n_-}q^{|v|+n_+-2n_-}\Diag(v)
\]
The morphism associated with an edge \(v\rightarrow v'\) 
gives rise to a morphism 
\[
d^{v}_{v'}
\co
\KhT{\Diag}_v
\rightarrow
\KhT{\Diag}_{v'}
\]
that preserves the quantum grading and raises homological grading by 1. 
(If \(v,v'\in\{0,1\}^n\) are not related by an edge \(v\rightarrow v'\) then \(d^{v}_{v'}=0\).)
The chain complex \( \KhT{\Diag} \) is then defined as 
the formal sum of the objects \(\KhT{\Diag}_v\)
over all vertices \( v\in\{0,1\}^n \) 
together with the endomorphism given by the matrix 
\((d^{v}_{v'})_{v,v'\in\{0,1\}^n}\).
The chain complex \(\KhT{\Diag} \) is
an invariant of the oriented tangle \(T\) 
up to chain homotopy \cite[BN05, Theorem 1, Section 4.2]{BarNatanKhT}; write $\KhT{T}$ for this invariant.

As a consequence of \cref{lem:delooping}, 
the chain complex \(\KhT{\Diag}\) is chain homotopic to a chain complex 
that is built entirely out of the tangles \(\Li\) and \(\Lo\). Denote this complex by \(\KhTs{\Diag}\).
The isomorphism \(\omega\) 
from \cref{thm:OmegaFullyFaithful} 
allows to view \(\KhTs{\Diag}\)
as a chain complex $\DD(\Diag)$ over $\B$. %
(An explicit description 
of the differential of $\DD(\Diag)$ %
is given in \cref{sec:no_wrapping_around_special:delooping_cobs}.)
Write \(\DD(T)\) %
for any bigraded chain complex in the same homotopy class as 
\(\DD(\Diag) \), or \(\DD(\Diag;\CoeffRing)\) when dependence on \(\CoeffRing\) needs emphasis. %
Since the homotopy type of \(\KhT{\Diag}\) is a tangle invariant, we obtain: 

\begin{theorem}
	Up to homotopy, \(\DD(T)\) is an invariant of the oriented tangle \(T\).\qed
\end{theorem}

For each \(n\geq0\)
there is a bigraded chain complex \(\DD_n(\Diag)\)
obtained as the mapping cone 
\[
\DD_n(\Diag)
=
\Cone{\varG^n\cdot 1_{\DD(\Diag)}}
=
\Big[
q^{-n}h^{-1}\DD(\Diag)
\xrightarrow{\varG^n\cdot 1_{\DD(\Diag)}}
q^{n}h^{0}\DD(\Diag)
\Big]
\]
The homotopy class of \(\DD_n(\Diag)\) is also a tangle invariant; write \(\DD_n(T)\) for \(\DD_n(\Diag)\). 

\subsection{An aside on type D structures}\label{sub:typeD} The complexes \(\DD(\Diag)\) and \(\DD_n(\Diag)\) are equivalent to type D structures over $\B$, which are more familiar in the context of Heegaard Floer theory \cite{LOT}; see \cite[Section 2]{KWZ}. Let $\I\subset\B$ be the subring of idempotents and consider the additive category of (bigraded) right $\I$-modules $\underline{\mathcal{C}}^\B$. The morphisms $\underline{\mathcal{C}}^\B(X,Y)$ for objects $X,Y\in\underline{\mathcal{C}}^\B$ are given by \(\I\)-module homomorphisms of the form \(f\co X\rightarrow Y\otimes \B\). (The identity morphism for any object \(X\) is given by \(1_X\otimes \eval_1\) where \(\eval_1\co \Z\rightarrow \B,1\mapsto 1\).) Writing $\mu_\B$ for multiplication in $\B$, composition is given by \[g\circ f
	=
	(1_Z\otimes\mu_\B)
	\circ
	(g\otimes 1_\B)
	\circ
	f\]
	for
	\(X,Y,Z\in\underline{\mathcal{C}}^\B\) and $f\in\underline{\mathcal{C}}^\B(X,Y)$, $g\in\underline{\mathcal{C}}^\B(Y,Z)$.
Then the category 
	\(\C^{\B}\)
	of type~D structures over \(\B\) is the dg category of chain complexes over
	$\underline{\mathcal{C}}^\B$; see \cite[Definition~2.4]{KWZ}. 
	Explicitly:

\begin{definition}	The objects in the category 
	\(\C^{\B}\)
	are pairs \((Y,\delta)\),
	where 
	\(Y\in\underline{\mathcal{C}}^\B\)
	and
	\(\delta\in\underline{\mathcal{C}}^\B(Y,Y)\)
	such that
	\(\delta^2=0\) and with
	grading
	\(\gr(\delta)=q^0h^{1}\). 
	For any 
	\((X,\delta_X),(Y,\delta_Y)\in\C^{\B}\),
	the differential 
	\(\partial_{\C^{\B}}\)
	on
	\(\C^{\B}(X,Y)=\underline{\mathcal{C}}^{\B}(X,Y)\)
	is defined by 
	\[
	\partial_{\C^{\B}}(f)
	=
	\delta_Y\circ f
	-
	(-1)^{|f|}
	f\circ\delta_X
	\]	satisfying the Leibniz  rule
	\[
	\partial_{\C^{\B}}(g\circ f)
	=
	\partial_{\C^{\B}}(g)\circ f
	+
	(-1)^{|g|}
	g\circ \partial_{\C^{\B}}(f)
	\]
where $|\cdot|$ in these expressions denotes the homological grading. 
\end{definition}
Type D structures are special cases of type DA bimodules, which are described in detail in \cref{sub:bordered}. Following the notation from \cite{LOT}, the superscript \(\cdot^{\B}\) can be used to highlight the algebra over which a type~D structure is defined. For instance, \(\DD(T)\) will sometimes be written as \(\DD(T)^{\B}\). 

\begin{remark}\label{rem:typeD:graphical-notation}
	Objects and morphisms in \(\underline{\C}^{\B}\) and \(\C^{\B}\) can be described graphically. 
	Suppose
	\(X,Y\in\underline{\C}^{\B}\)
	are two \(\I\)-modules freely generated by \(\{x_i\}_i\) and \(\{y_j\}_j\), respectively.  
	Then a morphism \(f\in\underline{\C}^{\B}(X,Y)\) 
	with 
	\(f(x_i)=\sum_j y_j \otimes a_{j,i}\) can be represented uniquely 
	by a collection of arrows
	\[
	\begin{tikzcd}
		x_i
		\arrow{r}{a_{j,i}}
		&
		y_j
	\end{tikzcd}
	\]
	By convention arrows labelled by the zero algebra element are omitted.
	Composition of morphisms in \(\underline{\C}^{\B}\) can then be expressed as follows:
	\[
	\Big(
	\begin{tikzcd}
		y
		\arrow{r}{b}
		&
		z
	\end{tikzcd}
	\Big)
	\circ
	\left(
	\begin{tikzcd}
		x
		\arrow{r}{a}
		&
		y
	\end{tikzcd}
	\right)
	=
	\Big(
	\begin{tikzcd}
		x
		\arrow{r}{ba}
		&
		z
	\end{tikzcd}
	\Big)
	\]
\end{remark}

	The centre of \(\B\), denoted \(Z(\B)\), acts on the morphism space \(\C^\B(X,X')\) for any \(X,X'\in\C^\B\): given \(f\in\Mor(X,X')\)
	and \(z\in Z(\B)\) define
	\[
	z\cdot f
	=
	(1_Y\otimes \mu_\B)
	\circ
	(f\otimes z)
	\]
 	In particular,  
	\(
	\partial_{\C^\B}(z\cdot f)
	=
	z\cdot\partial_{\C^\B}(f)
	\)
	and, if 
	\(g\in\C^\B(X',X'')\)
	for some
	\(X''\in\C^\B\), 
	\[
	z\cdot(g\circ f)
	=
	(z\cdot g)\circ f
	=
	g\circ(z\cdot f)
	\]

\begin{lemma}\label{lem:homotopies-of-identity-multiples}
	Let \(z,w\in Z(\B)\) and
	\(X,Y\in\C^\B\) 
	with \(X\simeq Y\). 
	Then 
	\(z\cdot1_X\simeq w\cdot1_X\)
	implies
	\(z\cdot1_Y\simeq w\cdot1_Y\). 
	\qed
\end{lemma}

\begin{lemma}\label{lem:nullhomotopy-splits-under-direct-sums}
	If 
	\(X,Y\in\C^\B\) 
	and \(z\in Z(\B)\)
	with
	\(z\cdot1_{X\oplus Y}\simeq 0\)
	then 
	\(z\cdot1_X\simeq 0\)
	and 
	\(z\cdot1_Y\simeq 0\).
\end{lemma}

\begin{proof}
	If 
	\(
	\left(
	\begin{smallmatrix}
		f & *
		\\
		* & g
	\end{smallmatrix}
	\right)
	\)
	is a nullhomotopy for \(z\cdot1_{X\oplus Y}\), 
	then so is 
	\(
	\left(
	\begin{smallmatrix}
		f & 0
		\\
		0 & g
	\end{smallmatrix}
	\right)
	\) and
	\(f\) and \(g\) 
	are nullhomotopies for 
	\(X\) and \(Y\), 
	respectively.
\end{proof}

\begin{lemma}
	As a subalgebra, \(Z(\B)\) is generated by 
	\(\Db\), \(\Dw\), and \((\Sb+\Sw)\). \qed
\end{lemma}
In particular, 
for \(z\in\{\varG,D_{\circ},D_{\bullet},S^2=(\Sb+\Sw)^2\}\), 
the type D structure endomorphisms 
\begin{equation*}%
	\varG\cdot1_X, ~ 
	S^2\cdot 1_X , ~ 
	D_{\bullet} \cdot 1_X, ~
	D_{\circ} \cdot 1_X
\end{equation*}
on any chain complex \(X\) over \(\B\)
 are defined by the formulas
\begin{align*}
	\varG\cdot1_X (x) 
	&= x \otimes \varG, 
	&
	S^2\cdot1_X (x) 
	&= x \otimes S^2, 
	\\ 
	D_{\bullet} \cdot1_X (x_{\bullet})
	&= x_{\bullet} \otimes D_{\bullet},
	~
	D_{\bullet}  \cdot 1_X (x_{\circ})
	= 0,
	&
	D_{\circ}  \cdot 1_X (x_{\circ})
	&= x_{\circ} \otimes D_{\circ},
	~
	D_{\circ}  \cdot 1_X (x_{\bullet})
	= 0,
\end{align*}
where $x_{\bullet}$ and $x_{\circ}$ denote generators in idempotents \(\iota_\bullet\) and \(\iota_\circ\), respectively.

\subsection{Curves on the four-punctured sphere}
Towards the classification of  bigraded chain complexes over the algebra \(\B\) up to chain homotopy, let \(\mathcal{X}\) be a finite, connected, oriented graph with vertices of valence at most~2. 
If every vertex has valence 2 we call \(\mathcal{X}\) cyclic and we call \(\mathcal{X}\) linear otherwise.
In other words, \(\mathcal{X}\) is an oriented graph of the form
\[
	\begin{tikzcd}
		x_0
		\arrow[leftrightarrow]{r}{e_1}
		&
		x_1
		\arrow[leftrightarrow]{r}{e_2}
		&
		\dots
		\arrow[leftrightarrow]{r}{e_\ell}
		&
		x_\ell
	\end{tikzcd}
	\quad
	\text{or}
	\quad
	\begin{tikzcd}
		x_0
		\arrow[leftrightarrow]{r}{e_1}
		&
		x_1
		\arrow[leftrightarrow]{r}{e_2}
		&
		\dots
		\arrow[leftrightarrow]{r}{e_\ell}
		&
		x_\ell
		\arrow[leftrightarrow]{r}{e_0}
		&
		x_0
		\arrow[llll,-,
		to path={ ([xshift=-1pt]\tikztostart.south)
			-- ([yshift=-6pt,xshift=-1pt]\tikztostart.south)
			-- ([yshift=-6pt,xshift=1pt]\tikztotarget.south)
			-- ([xshift=1pt]\tikztotarget.south)}]
		\arrow[llll,-,
		to path={ ([xshift=1pt]\tikztostart.south)
		-- ([yshift=-8pt,xshift=1pt]\tikztostart.south)
		-- ([yshift=-8pt,xshift=-1pt]\tikztotarget.south)
		-- ([xshift=-1pt]\tikztotarget.south)}]
	\end{tikzcd}
\]
where each edge is oriented either  to the left or to the right. %

\begin{definition}
	A formal curve $F$ of shape $\mathcal{X}$ is 
	\begin{enumerate}
		\item an object \(F(x)\in\{\bullet,\circ\}\) for every vertex \(x\) in \(\mathcal{X}\); and 
		\item for every edge \(e\co x\rightarrow y\) in \(\mathcal{X}\)
		an element 
		\(
		F(e)=\iota_{F(y)}\cdot F(e)\cdot\iota_{F(x)}
		\)
		of the standard basis of \(\B\)
		with \(F(e)\neq \iota_{\bullet},\iota_{\circ}\)
		such that
		\[
		F(e)\in\{\varG^n D_\circ,\varG ^n D_\bullet\mid n\geq0\}
		\Longleftrightarrow
		F(e')\in\{\varG^n S_\bullet,\varG^n S_\circ,\varG^n SS_\bullet,\varG ^n SS_\circ\mid n\geq0\}
		\]
		for any two distinct edges \(e\) and \(e'\) 
		sharing a common vertex.
	\end{enumerate}
	If \(\mathcal{X}\) is cyclic we call \(F\) compact. 
	If \(\mathcal{X}\) is linear we call \(F\) non-compact. 
	Two formal curves 
	$F$ %
	and 
	$F'$ %
	are equivalent if there exists an isomorphism 
	\(f\co\mathcal{X}'\rightarrow\mathcal{X}\)
	such that 
	\(Ff=F'\).	
\end{definition}

	Formal curves correspond to certain $\B$-decorated graphs in the sense of \cite{HanselmanWatson}.
	In particular \(F\) can be realized as 
	a decoration of the vertices of the graph \(\mathcal{X}\) 
	by objects \(\halfbullet_i= F(x_i)\in\{\bullet,\circ\}\)
	and its edges 
	by elements \(a_i=F(e_i)\) of the standard basis (subject to the conditions above):
	 \begin{equation}\label{eqn:curve-like_functor}
		\begin{tikzcd}
			\halfbullet_0
			\arrow[leftrightarrow]{r}{a_1}
			&
			\halfbullet_1
			\arrow[leftrightarrow]{r}{a_2}
			&
			\dots
			\arrow[leftrightarrow]{r}{a_\ell}
			&
			\halfbullet_\ell
		\end{tikzcd}
		\quad
		\text{or}
		\quad
		\begin{tikzcd}
			\halfbullet_0
			\arrow[leftrightarrow]{r}{a_1}
			&
			\halfbullet_1
			\arrow[leftrightarrow]{r}{a_2}
			&
			\dots
			\arrow[leftrightarrow]{r}{a_\ell}
			&
			\halfbullet_\ell
			\arrow[leftrightarrow]{r}{a_0}
			&
			\halfbullet_0
			\arrow[llll,-,
			to path={ ([xshift=-1pt]\tikztostart.south)
				-- ([yshift=-6pt,xshift=-1pt]\tikztostart.south)
				-- ([yshift=-6pt,xshift=1pt]\tikztotarget.south)
				-- ([xshift=1pt]\tikztotarget.south)}]
			\arrow[llll,-,
			to path={ ([xshift=1pt]\tikztostart.south)
				-- ([yshift=-8pt,xshift=1pt]\tikztostart.south)
				-- ([yshift=-8pt,xshift=-1pt]\tikztotarget.south)
				-- ([xshift=-1pt]\tikztotarget.south)}]
		\end{tikzcd}
	\end{equation}
	Concretely, here is a non-compact formal curve: \[
	\begin{tikzcd}[row sep=2pt,column sep=1cm]
		\bullet
		\arrow[r,"S\cdot\varG"]
		&
		\circ
		\arrow[r,"D\cdot\varG"]
		&
		\circ
		\arrow[r,"S^2"]
		&
		\circ
		\arrow[r,"D"]
		&
		\circ
		\arrow[r,"S"]
		&
		\bullet
		&
		\bullet
		\arrow[l,"D",swap]
	\end{tikzcd}
	\]

Let \(\FourPuncturedSphere\) denote \(S^2\smallsetminus \EndsAll\). Formal curves admit a geometric interpretation, as follows. 

\begin{definition}
	A compact curve in \(\FourPuncturedSphere\) 
	is a free loop
	\(\gamma\co S^1\rightarrow\FourPuncturedSphere\). 
	We consider compact curves
	up to homotopy and orientation reversal.
	A non-compact curve in \(\FourPuncturedSphere\) 
	is a continuous map
	\(
	\gamma\co [0,1]\rightarrow S^2\smallsetminus \{*\}
	\)
	with
	\(\gamma(0),\gamma(1)\in \EndsThree\) 
	and 
	\(\gamma(0,1)\subseteq\FourPuncturedSphere\).
	We consider non-compact curves
	up to homotopy through non-compact curves and orientation reversal.
\end{definition}

Curves in \(\FourPuncturedSphere\) and 
formal curves are in one-to-one correspondence. 
This correspondence goes through curves on the quiver \(\mathcal{Q}\)
(or, more precisely, the cell complex associated with \(\mathcal{Q}\)). 

\begin{definition}
	A compact curve in \(\mathcal{Q}\) 
	is a free loop
	\(\gamma\co S^1\rightarrow\mathcal{Q}\), 
	 considered up to homotopy and orientation reversal.
	A non-compact curve in \(\mathcal{Q}\)
	is a continuous map
	\(
	\gamma\co [0,1]\rightarrow \mathcal{Q}
	\)
	with
	\(\gamma(0),\gamma(1)\in \{\bullet,\circ\}\),
	considered up to homotopy fixing endpoints
	and orientation reversal.
\end{definition}

Any formal curve \(F\) determines a curve \(\gamma_\mathcal{Q}(F)\) 
in \(\mathcal{Q}\) by concatenation of the paths in \(\mathcal{Q}\) 
associated with individual edges in \(\mathcal{X}\). 
Indeed every curve in \(\mathcal{Q}\)
is of the form \(\gamma_\mathcal{Q}(F)\) 
for some formal curve \(F\), 
and if two curves 
\(\gamma_\mathcal{Q}(F)\) 
and 
\(\gamma_\mathcal{Q}(F')\) 
are equivalent, then so are \(F\) and~\(F'\). 
Consulting \cref{fig:Kh:example} on page~\pageref{fig:Kh:example}, \(\mathcal{Q}\) embeds into \(\FourPuncturedSphere\) so that there is a deformation retract
\(r\co \FourPuncturedSphere\rightarrow\mathcal{Q}\)
where \(r^{-1}(\bullet)\) and \(r^{-1}(\circ)\) are two disjoint arcs \(a_\bullet\) and \(a_\circ\), respectively, starting and ending on the distinguished puncture \(\ast\). 
The map 
\([\gamma]\mapsto [r\gamma]\)
sets up a bijection between 
equivalence classes of
compact curves in~\(\FourPuncturedSphere\) 
and in~\(\mathcal{Q}\).
Observe that \(\gamma\) can be chosen to intersect 
\(a_\bullet \cup a_\circ\) minimally
and such that any self-intersection points of \(\gamma\) are
in the complement of \(a_\bullet \cup a_\circ\).
Then, if 
$F$ %
is a formal curve associated with $\mathcal{X}$, where
\([r\gamma]=\gamma_\mathcal{Q}(F)\),
the vertices in \(\mathcal{X}\)
correspond exactly to the intersection points
\(\gamma\cap (a_\bullet \cup a_\circ)\)
and the edges in \(\mathcal{X}\)
correspond to how these intersection points are connected in~\(\gamma\). 
Similarly, if \(\gamma\) is a non-compact curve 
in \(\FourPuncturedSphere\), 
we can choose \(\gamma\) 
to intersect \(a_\bullet \cup a_\circ\) minimally
and such that any self-intersection points of \(\gamma\) are
in the complement of \(a_\bullet \cup a_\circ\).
We then restrict \(\gamma\) to a maximal subpath \(\gamma|_m\)
with ends on \(a_\bullet \cup a_\circ\). 
The map 
\([\gamma]\mapsto [r\gamma|_m]\)
sets up a bijection between 
equivalence classes of 
non-compact curves in~\(\FourPuncturedSphere\) 
and in~\(\mathcal{Q}\). In summary:

\begin{proposition}\label{prop:functors-curves}
	There is a one-to-one correspondence \(F\leftrightarrow\gamma_F\) between formal curves
	and curves in
	\(\FourPuncturedSphere\), 
	up to the respective notions of equivalence.
	\qed
\end{proposition}

\subsection{%
	Graded curves with local systems 
	and complexes over \texorpdfstring{\(\B\)}{B}}

A bigrading adapted to a formal curve \(F\) 
is a bigrading of objects 
\[
	g\co
	\{x_i\}_{i=0,\dots,\ell}
	\rightarrow
	\Z^2,\quad
	x
	\mapsto
	g(x)=(q(x),h(x))
\]
such that for every edge \(e\co x\rightarrow y\) in \(A\), 
\begin{equation}\label{eqn:grading_on_curves:c}
	q(y)-q(x)+q(F(e))=0
	\quad \text{ and }\quad
	h(y)-h(x)=1.
\end{equation}
Note that the space of such functions is an affine copy of \(\Z^2\) if it is non-empty, 
since any function satisfying \eqref{eqn:grading_on_curves:c} 
is determined by its value at a single point.
Moreover, every non-compact formal curve admits a grading.

\begin{definition}
	A graded curve is a pair \((F,g)\) where 
	$F$ %
	is a formal curve %
	\(g\) is a bigrading adapted to \(F\). 
	Two graded curves 
	$(F,g)$ %
	and 
	$(F,g')$ %
	are equivalent if there exists an isomorphism 
	\(f\co\mathcal{X}'\rightarrow\mathcal{X}\)
	such that 
	\(Ff=F'\) and \(gf=g'\).
\end{definition}

\begin{definition}\label{def:functor-to-complex}
	Given a graded curve \((F,g)\)
	define a chain complex \(C(F,g)\) over \(\B\) 
	by
	\[
	\left(
	\bigoplus_{i=0}^\ell q^{q(x_i)}h^{h(x_i)}F(x_i),
	\Big(
	\sum_{e} F(e)
	\Big)_{ji}
	\right)
	\]
	where the sum in the differential is over all edges 
	\(e\co x_i\rightarrow x_j\) 
	in \(\mathcal{X}\).
\end{definition}

\begin{remark}
The sum in the differential 
has at most one non-zero term, 
unless \(\mathcal{X}\) is the cyclic graph 
with two vertices.
So, in simple terms, 
\(C(F,g)\) is the chain complex represented
by the labelled graph~\eqref{eqn:curve-like_functor};
the assumption that the labels \(a_i\) 
alternate between \(S\) and \(D\) 
ensures that 
the differential squares to zero.
Note that the chain complexes 
associated with equivalent graded curves 
are isomorphic.
\end{remark}

The above construction generalizes as follows: 
An \(n\)-dimensional local system 
is a pair \((X,e)\) 
where \(e\) is an edge of \(\mathcal{X}\) and 
\(X\in\GL_n(\CoeffRing)\), where $n$ is a positive integer.
Two local systems \((X,e)\) and \((X',e')\) are equivalent
if \(X^\sigma\) and \(X'\) are conjugate, 
where \(\sigma\in\{\pm 1\}\) is equal to \(+1\) 
if and only if the edges \(e\) and \(e'\) 
point in the same direction (both left or both right). 

\begin{definition}
	A graded curve with local system is a tuple \((F,g,X,e)\) for
	$(F,g)$ a graded curve and \((X,e)\) a local system on~\(\mathcal{X}\), where
	we require \(X=(1)\in\GL_1(\CoeffRing)\)
	if \(F\) is non-compact.
	Two graded curves with local systems
	\((F,g,X,e)\) %
	and 
	\((F',g',X',e')\) %
	are equivalent if there exists an isomorphism 
	\(f\co\mathcal{X}'\rightarrow\mathcal{X}\)
	such that 
	\(Ff=F'\), \(gf=g'\), and \((X,e)\) is equivalent to \((X',f(e'))\).
\end{definition}

When \(\CoeffRing\) is a field \(\CoeffField\), 
an \(n\)-dimensional local system on a bigraded compact curve \((F,g)\)
can be thought of as the monodromy 
of an \(n\)-dimensional \(\CoeffField\)-vector bundle over the curve \(\gamma_F\), 
where \(\CoeffField\) is equipped with the discrete topology. 
Vector bundles over non-compact curves are unique (up to dimensions), hence the
unique equivalence class of local system for non-compact curves.

\begin{definition}
	Given a %
	graded curve with local system \((F,g,X,e)\) define a chain complex \(C(F,g,X,e)\) over \(\B\) 
	by
	\[
	\left(
	\bigoplus_{i=0}^\ell q^{q(x_i)}h^{h(x_i)}
	F(x_i)\otimes \CoeffRing^n,
	\Big(
	\sum_{e'} F(e')\otimes X(e')
	\Big)_{ji}
	\right)
	\]
	where the sum in the differential is over all edges 
	\(e'\co x_i\rightarrow x_j\) 
	in \(\mathcal{X}\), \(X(e)=X\), and 
	\(X(e')\) is the \(n\)-dimensional identity matrix 
	for \(e'\neq e\).
\end{definition}

\begin{remark}
	If \(X=(1)\in\GL_1(\CoeffRing)\) then
	\(
	C(F,g,X,e)=C(F,g)
	\)
	for any edge \(e\) in~\(\mathcal{X}\).
	As such every graded curve may be regarded as a graded curve with trivial local system.%
\end{remark}

\begin{example}
	If \(\mathcal{X}\) is cyclic
	the chain complex \(C(F,g,X,e_0)\) 
	is represented by a labelled oriented graph of the form
	\[
	\begin{tikzcd}[column sep=40pt]
		\halfbullet_0\otimes \CoeffRing^n
		\arrow[leftrightarrow]{r}{a_1\otimes 1_{\CoeffRing^n}}
		&
		\halfbullet_1\otimes \CoeffRing^n
		\arrow[leftrightarrow]{r}{a_2\otimes 1_{\CoeffRing^n}}
		&
		\dots
		\arrow[leftrightarrow]{r}{a_\ell\otimes 1_{\CoeffRing^n}}
		&
		\halfbullet_\ell\otimes \CoeffRing^n
		\arrow[leftrightarrow]{r}{a_0\otimes X}
		&
		\halfbullet_0\otimes \CoeffRing^n
		\arrow[llll,-,
		to path={ ([xshift=-1pt]\tikztostart.south)
			-- ([yshift=-6pt,xshift=-1pt]\tikztostart.south)
			-- ([yshift=-6pt,xshift=1pt]\tikztotarget.south)
			-- ([xshift=1pt]\tikztotarget.south)}]
		\arrow[llll,-,
		to path={ ([xshift=1pt]\tikztostart.south)
			-- ([yshift=-8pt,xshift=1pt]\tikztostart.south)
			-- ([yshift=-8pt,xshift=-1pt]\tikztotarget.south)
			-- ([xshift=-1pt]\tikztotarget.south)}]
	\end{tikzcd}
	\]
\end{example}

Suppose \((F,g,X,e)\) is a compact graded curve and 
\((X',e')\) is another local system for \((F,g)\). 
Then we define a new local system
\[
(X,e)\oplus(X',e')
=
\left(
\left(
\begin{smallmatrix}
	X & 0
	\\
	0 & X''
\end{smallmatrix}
\right),
e
\right)
\]
where \(X''\) is chosen such that 
\((X',e')\) is equivalent to \((X'',e)\). 
The equivalence class of \((X,e)\oplus(X',e')\) 
is independent of the choices. Proofs of the following observations are left to the reader: 
\begin{proposition}
	There exists an isomorphism
	\[
	C(F,g,X,e)\oplus C(F,g,X',e')
	\cong 
	C(F,g,(X,e)\oplus(X',e'))
	\]
	and, moreover, 
	if \((X,e)\) and \((X',e')\) are equivalent then 
	\(
	C(F,g,X,e)\cong C(F,g,X',e').
	\)
	\qed
\end{proposition}

We call a compact formal curve \(F\) primitive if  \(\gamma_F\) defines a primitive element \([\gamma_F]\in\pi_{1}(\FourPuncturedSphere)\).

\begin{proposition}
	If \((F,g,X,e)\) is a bigraded compact curve with local system
then there exists a primitive element 
	\(\tau\in\pi_1(\FourPuncturedSphere)\) 
	such that
	\([\gamma_F]=\tau^n\in\pi_1(\FourPuncturedSphere)\) 
	for some \(n>0\) and,
	moreover, there exists a bigraded compact curve with local system \((F',g',X',e')\)
	with \(\gamma_{F'}=\tau\) and 
	\(
	C(F,g,X,e)
	\cong 
	C(F',g',X',e')
	\). \qed
\end{proposition}

\subsection{%
	Classification of complexes over \texorpdfstring{\(\B\)}{B} 
	and the multicurve invariants}

	A multicurve 
	is a finite (possibly empty)
	multiset of graded curves with local systems,
	which are required to be primitive when they are compact.
	A multicurve is reduced if no two of its compact curves are equivalent 
	as graded curves (without local systems). 
	By taking direct sums of local systems, 
	we can associate with any multicurve 
	a reduced multicurve 
	that is unique up to bijections 
	matching equivalent graded curves with local systems. 
	Two multicurves are considered equivalent
	if their reductions agree. 
The chain complex associated with a multicurve is 
the direct sum of the chain complexes 
associated with all components of the multicurve. 
We note:

\begin{proposition}
	Chain complexes 
	associated with equivalent multicurves
	are isomorphic. 
	\qed
\end{proposition}

The following is a special case of 
\cite[Theorem~4.3]{HKK}; 
see also \cite[Theorem~1.5]{KWZ}.

\begin{theorem}
	\label{thm:classification:complexes_over_B:simplified}
	When $\CoeffRing$ is a field
	there exists a one-to-one correspondence between 
	bigraded chain complexes over \(\B\) 
	up to chain homotopy 
	and 
	multicurves up to equivalence.\qed
\end{theorem}

\begin{notation}\label{notation:multicurve-reps}
	Over fields, the (equivalence class of) multicurves 
	corresponding to 
	\(\DD(T)\) and \(\DD_1(T)\) 
	are denoted by 
	\(\BNr(T)\) and \(\Khr(T)\),
	respectively.
	Chain complexes 
	associated with (representatives of) the multicurves
	\(\BNr(T)\) and \(\Khr(T)\)
	are denoted by
	\(\DD^c(T)\) and \(\DD^c_1(T)\).
\end{notation}

\begin{figure}[t]
	\begin{subfigure}{0.235\textwidth}
		\centering
		\includegraphics[scale=1]{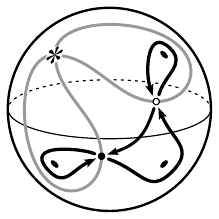}
		\caption{}\label{fig:Kh:example:quiver}
	\end{subfigure}
	\begin{subfigure}{0.235\textwidth}
		\centering
		\includegraphics[scale=1]{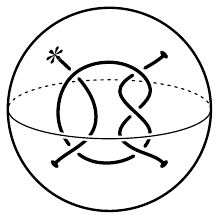}
		\caption{}\label{fig:Kh:example:tangle}
	\end{subfigure}
	\begin{subfigure}{0.235\textwidth}
		\centering
		\includegraphics[scale=1]{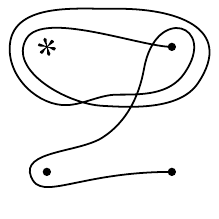}
		\caption{}\label{fig:Kh:example:BNr}
	\end{subfigure}
	\begin{subfigure}{0.235\textwidth}
		\centering
		\includegraphics[scale=1]{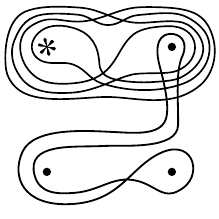}
		\caption{}\label{fig:Kh:example:Khr}
	\end{subfigure}
	\caption{%
		The embedded quiver \(\mathcal{Q}\) (a), 
		the \((2,-3)\)-pretzel tangle (b) and 
		its multicurves (c--d)
		}\label{fig:Kh:example}
\end{figure}

\begin{example}\label{exa:Khr:2m3pt}
	We usually draw the four-punctured sphere \(\FourPuncturedSphere\) 
	as the plane, together with a point at infinity,
	minus the four punctures. 
	\Cref{fig:Kh:example:BNr,fig:Kh:example:Khr}
	show 
	the curves in \(\FourPuncturedSphere\) defined by
	\(\BNr(P_{2,-3})\) and \(\Khr(P_{2,-3})\), 
	respectively, 
	where \(P_{2,-3}\) is the \((2,-3)\)-pretzel tangle 
	from \cref{fig:Kh:example:tangle}, 
	compare \cite[Example~6.7]{KWZ}.
	The local systems on the compact curves 
	are equal to \(-1_{\CoeffRing^1}\). 
	Omitting gradings, 
	the corresponding chain complexes are:
	\begin{align*}
	\DD(P_{2,-3})
	&=
	\left[
	\begin{tikzcd}[row sep=12pt,column sep=16pt,ampersand replacement=\&]
		\bullet
		\&
		\bullet
		\arrow[l,swap,"D"]
		\&
		\circ
		\arrow[l,swap,"S"]
		\&
		\circ
		\arrow[l,swap,"D"]
		\arrow[r,"S"]
		\&
		\bullet
		\arrow[r,"D"]
		\&
		\bullet
		\arrow[r,"S^2"]
		\&
		\bullet
		\arrow[r,"D"]
		\&
		\bullet
		\arrow[r,"S"]
		\&
		\circ
	\end{tikzcd}
	\right]
	\\
	\DD_1(P_{2,-3})
	&=
	\left[
	\begin{tikzcd}[row sep=16pt,column sep=16pt,ampersand replacement=\&]
		\circ
		\arrow{r}{-S}
		\arrow{d}[swap]{D}
		\&
		\bullet
		\arrow{r}{D}
		\&
		\bullet
		\arrow{r}{S^2}
		\&
		\bullet
		\arrow{r}{D}
		\&
		\bullet
		\arrow{r}{S}
		\&
		\circ
		\arrow{d}{D}
		\\
		\circ
		\arrow{r}{S}
		\&
		\bullet
		\arrow{r}{D}
		\&
		\bullet
		\arrow{r}{S^2}
		\&
		\bullet
		\arrow{r}{D}
		\&
		\bullet
		\arrow{r}{S}
		\&
		\circ
	\end{tikzcd}
	\right]
	\oplus
	\left[
	\begin{tikzcd}[row sep=16pt,column sep=16pt,ampersand replacement=\&]
		\circ
		\arrow{r}{-S}
		\arrow{d}[swap]{D}
		\&
		\bullet
		\arrow{r}{D}
		\&
		\bullet
		\arrow{d}{S^2}
		\\
		\circ
		\arrow{r}{S}
		\&
		\bullet
		\arrow{r}{D}
		\&
		\bullet
	\end{tikzcd}
	\right]	
	\end{align*}
	Note in particular that 
	\(\DD(T)\) is homotopic to a chain complex 
	associated with a multicurve over any ring \(\CoeffRing\) in this example, 
	not only fields \(\CoeffRing=\CoeffField\). 
	This is not true in general. 
	Moreover,
	different fields can give rise to different multicurves for the same tangle.
\end{example}

\begin{proposition}[{\cite[Propositions 6.15 and 6.17]{KWZ}}]
	For any pointed tangle~\(T\), 
	\(\Khr(T)\) contains no non-compact curve. 
	Moreover, the number of non-compact components of \(\BNr(T)\)
	is \(2^{|T|}\), where \(|T|\) is the number of closed components of \(T\).\qed
\end{proposition}

\subsection{Mirroring}

Reflection 
in the plane \(\R^2\times\{0\}\subset \R^3\), which 
restricts to both \(D^3\) and~\(\FourPuncturedSphere\), 
associates with any pointed 4-ended tangle \(T\)
a tangle \(-T\)
and with any curve \(\gamma\) in \(\FourPuncturedSphere\)
a curve \(-\gamma\) in \(\FourPuncturedSphere\). 
Appealing to \cref{prop:functors-curves}, 
the curve \(-\gamma\) corresponds to a formal curve described as follows.
Let \(\B^{op}\) be the opposite algebra of \(\B\)
and define an isomorphism
\[
\phi\co \B^{op}\rightarrow\B,
\quad
(D_\bullet)^{op}\mapsto D_\bullet, 
(D_\circ)^{op}\mapsto D_\circ,
(S_\circ)^{op}\mapsto S_\bullet,
(S_\bullet)^{op}\mapsto S_\circ. 
\]
For an oriented graph \(\mathcal{X}\), 
let \(\mathcal{X}^{op}\) be the graph obtained 
by reversing all arrows. 
Then given a formal curve 
\(F\) (associated with \(\mathcal{X}\))
let \(-F\) be the formal curve associated with \(\mathcal{X}^{op}\)
defined
by \((-F)(x)=F(x)\) on vertices and 
by \((-F)(e^{op})=\phi((F(e))^{op})\) on edges. 
Furthermore, if \(C\) is a chain complex over~\(\B\), 
let \(-C\) be obtained by
multiplying the gradings by \(-1\), 
reversing all arrows,
and applying \(\phi\) to the labels. 

\begin{proposition}\label{lem:mirror}
	Let \((F,g,X,e)\) 
	be a graded curve with local system
	and \(C\) its associated chain complex. 
	Then \(-C=C(-F,-g,X^t, e^{op})\)
	and \(\gamma_{-F}=-\gamma_F\). 
	\qed
\end{proposition}

\begin{notation}
	The curve \(\gamma=\gamma_F\) 
	in \(\FourPuncturedSphere\) will often be used
	as shorthand for the data \((F,g,X,e)\).
	Grading shifts \((F,g+(a,b),X,e)\) for \(a,b\in\Z\) are indicated by \(q^ah^b\gamma\). 
\end{notation}

\begin{proposition}[{\cite[Propositions 4.26 and~7.1]{KWZ}}]\label{prop:mirroring:Khr}
	For any pointed Conway tangle \(T\), 
	\(-\DD(T)=\DD(-T)\) and 
	\(-\DD_1(T)=h^{1}\DD_1(-T)\).
	Thus,
	\(-\BNr(T)=\BNr(-T)\) and 
	\(-\Khr(T)=h^{1}\Khr(-T)\).\qed
\end{proposition}

\subsection{Gluing}

The following is a consequence of \cite[Proposition~4.31]{KWZ}; see also \cite[Theorem~7.2]{KWZ}.

\begin{proposition}
	\label{prop:GlueingTheorem:Kh-cx}
	Let \(L=T_1\cup T_2\) 
	be the result of gluing two oriented pointed Conway tangles as in \cref{fig:tanglepairing} such that the orientations match. 
	Then
	\[
	\Khr(L)
	\cong 
	H_*(\Mor(-\DD_1(T_1),\DD(T_2)))
	\cong 
	H_*(\Mor(-\DD(T_1),\DD_1(T_2)))
	\]
	as bigraded $\CoeffRing$-modules.
	Moreover, 
	\[
	\pushQED{\qed}
	\Khr(L)\otimes (q^{-1}h^{-1}\CoeffRing\oplus q^{1}h^{0}\CoeffRing)
	\cong 
	H_*(\Mor(-\DD_1(T_1),\DD_1(T_2))).
	\qedhere
	\popQED
	\]
\end{proposition}

\begin{figure}[bt]
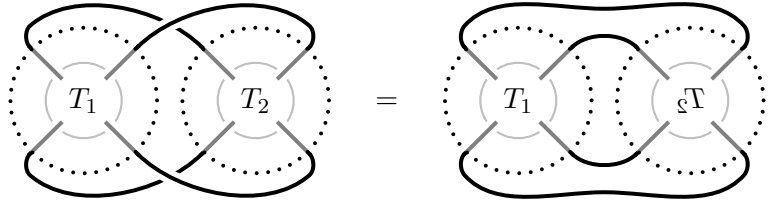

	\centering
	\(
	\tanglepairingI
	\quad = \quad
	\tanglepairingII
	\)
	\caption{Two Conway tangle decompositions defining the link \(T_1\cup T_2\). The tangle \protect\reflectbox{\(T_2\)} is the result of rotating \(T_2\) around the vertical axis. By rotating the entire link on the right-hand side around the vertical axis, we can see that \(T_1\cup T_2=T_2\cup T_1\).}
	\label{fig:tanglepairing} 
\end{figure}

We now reinterpret \cref{prop:GlueingTheorem:Kh-cx} 
in terms of the multicurve invariants in the case when $\CoeffRing$ is some field \(\CoeffField\).
Given two graded curves with local systems \(\gamma\) and \(\gamma'\), 
one can define a bigraded vector space 
\(\HF(\gamma,\gamma')\) 
over \(\CoeffField\) via
the Lagrangian Floer homology of \(\gamma\) and \(\gamma'\);
its definition is irrelevant for the present paper.
Instead, we take the next result %
as a definition. 

\begin{theorem}[{\cite[Theorem 1.5; see also Theorem 5.22]{KWZ}}]
	\label{thm:classification-morphisms}
	Let
	\(\gamma\) and \(\gamma'\) 
	be graded curves with local systems and 
	let \(C\) and \(C'\) be the associated complexes, respectively. 
	Then 
	\(\HF(\gamma,\gamma')\)
	is isomorphic to the homology of
	the morphism space \(\Mor(C,C')\).\qed
\end{theorem}

The following property, which is %
a consequence of 
\cite[Theorem~5.25]{KWZ}, is also needed.

\begin{lemma}
	\label{rem:local-systems-gluing}
	Let \(\gamma_1\) and \(\gamma_2\) be 
	two graded primitive compact curves 
	that are identical except for their local systems. 
	If 
	\[
	\HF(\gamma_2,\gamma)\not\cong
	\HF(\gamma_1,\gamma)
	\]
	for some graded curve with local system \(\gamma\)
	then \(\gamma\) is compact. 
	If we furthermore assume that \(\gamma\) is primitive 
	then its underlying curve agrees 
	with the underlying curve of \(\gamma_i\).\qed
\end{lemma}

For multicurves
\(\Gamma=\{q^{a_i}h^{b_i}\gamma_i\}_{i}\)
and 
\(\Gamma'=\{q^{a'_j}h^{b'_j}\gamma'_j\}_{j}\)
define
\[
\HF(\Gamma,\Gamma')
=
\bigoplus_{i,j}
q^{a_i-a'_j}h^{b_i-b'_j}
\HF(\gamma_i,\gamma'_j).
\]
Thus, 
\(\HF(\Gamma,\Gamma')\cong H_*\Mor(C,C')\)
if \(C\) and \(C'\) 
are the bigraded chain complexes 
associated with \(\Gamma\) and \(\Gamma'\). 
Combining \cref{thm:classification-morphisms} 
with \cref{prop:GlueingTheorem:Kh-cx}:

\begin{theorem}[{\cite[Theorem~1.9]{KWZ}}]
	\label{thm:GlueingTheorem:Kh}
	Let \(L=T_1\cup T_2\) 
	be the result of gluing two oriented pointed Conway tangles 
	as in \cref{fig:tanglepairing} 
	such that the orientations match. 
	Then  
	\[
	\Khr(L)
	\cong
	\HF\left(-\Khr(T_1),\BNr(T_2)\right)
	\cong 
	\HF\left(-\BNr(T_1),\Khr(T_2)\right)
	\]
	as bigraded vector spaces over \(\CoeffField\).\qed
\end{theorem}

\section{\texorpdfstring{New features of immersed curve invariants}{New features of immersed curve invariants}}\label{sec:new:Kh}

\subsection{The mapping class group of the four-punctured sphere}\label{sec:newKh:mcg}

Let \(\Mod(\FourPuncturedSphere,*)\)
denote 
the group of isotopy classes of orientation preserving
homeomorphisms of the unit sphere \(S^2\subset\R^3\) that map the set of punctures \(\EndsAll=\{(\pm e^{\pm i\pi/4},0)\}\subset S^2\) to itself and 
fix the distinguished puncture \(*=(-e^{-i\pi/4},0)\). 
This mapping class group can be realized 
as a quotient of the braid group on three strands
\[
B_3
=
\langle
\BraidB,
\BraidW
\mid 
\BraidB\BraidW\BraidB
=
\BraidW\BraidB\BraidW
\rangle.
\]
To see this, 
let \(D^*\) be a regular open neighbourhood of the distinguished point~\(*\) and
 denote \(\FourPuncturedSphere\smallsetminus D^*\) by \(\ThreePuncturedDisk\) (the disk with three punctures). %
The mapping class group \(\Mod(\ThreePuncturedDisk)\)
is defined as the group  
of orientation preserving homeomorphisms of \(\ThreePuncturedDisk\)
that restrict to the identity on the boundary, 
modulo isotopies that fix the boundary pointwise.
It is well-known that \(\Mod(\ThreePuncturedDisk)\) can be naturally identified with \(B_3\), 
 and the inclusion 
\(\ThreePuncturedDisk\hookrightarrow \FourPuncturedSphere\) 
induces an epimorphism
\[\varphi_\circledast\co \Mod(\ThreePuncturedDisk)\rightarrow\Mod(\FourPuncturedSphere,*)\]
whose kernel is generated by the Dehn twist around the boundary of \(\ThreePuncturedDisk\); 
see for example \cite[Proposition 3.19]{FarbMargalit}. 
In terms of the generators of \(B_3\) 
this boundary Dehn twist can be written as the full twist
\(
\BraidW\BraidB
\BraidW\BraidB
\BraidW\BraidB
\). 
Writing 
\(\TwistB=\varphi_\circledast(\BraidB)\)
and 
\(\TwistW=\varphi_\circledast(\BraidW)\): %

\begin{lemma}\label{lem:MCG:presentation}
	\(\Mod(\FourPuncturedSphere,*)\)
	has presentation
	\(\langle
	\TwistB,
	\TwistW
	\mid 
	\TwistB\TwistW\TwistB
	=
	\TwistW\TwistB\TwistW,
	\TwistW\TwistB
	\TwistW\TwistB
	\TwistW\TwistB
	\rangle.\)
	\qed
\end{lemma}
By convention the generators 
\(\BraidB\) and \(\BraidW\) 
correspond to the elementary braids
\(\vc{%
\begin{tikzpicture}[scale=0.3,line cap=round]
	\draw[thick] (0,0) .. controls (0.5,0) and (0.5,0.6) .. (1,0.6);
	\draw[white,line width=3pt] (0.6,0.2) -- (0.4,0.4);
	\draw[thick] (0,0.6) .. controls (0.5,0.6) and (0.5,0) .. (1,0);
	\draw[thick] (0,1.2) -- (1,1.2);
\end{tikzpicture}%
}\)
 and 
\(\vc{%
	\begin{tikzpicture}[scale=0.3,line cap=round]
		\draw[thick] (0,0.6) .. controls (0.5,0.6) and (0.5,1.2) .. (1,1.2);
		\draw[white,line width=3pt] (0.6,0.8) -- (0.4,1.0);
		\draw[thick] (0,1.2) .. controls (0.5,1.2) and (0.5,0.6) .. (1,0.6);
		\draw[thick] (0,0) -- (1,0);
	\end{tikzpicture}%
}\)
respectively.
The actions of the generators of \(\Mod(\FourPuncturedSphere,*)\) on \(\FourPuncturedSphere\)
are shown in \cref{fig:covering:conventions}.
This also fixes the identification
\(
B_3
\cong
\Mod(\ThreePuncturedDisk)
\); 
compare \cref{fig:twists:generators}.

A second useful description of \(\Mod(\FourPuncturedSphere,*)\) makes use of a covering space of \(\FourPuncturedSphere\), which is constructed in two steps:
First, let \(\TwoTorusKh\rightarrow \FourPuncturedSphere\) be the two-fold covering corresponding to the kernel of the unique epimorphism \(\pi_1(\FourPuncturedSphere)\rightarrow \ZmodTwo\) sending the meridian around each puncture to 1. As a topological space, \(\TwoTorusKh\) is the complement of four points in a 2-dimensional torus \(\TwoTorus\). Second, let \(\PlanarCover\rightarrow\TwoTorusKh\) be the pullback of the universal covering \(\R^2\rightarrow\TwoTorus\) along the inclusion \(\TwoTorusKh\hookrightarrow\TwoTorus\). Let 
\[p\co \PlanarCover\rightarrow \FourPuncturedSphere\]
be the composition of these two coverings, where our explicit model for \(p\) is such that the labels of lattice points correspond to the labels of the punctures of \(\FourPuncturedSphere\) as indicated in \cref{fig:covering:conventions} and such that \(p\) maps the square \((0,1)^2\) to points on \(\FourPuncturedSphere\) with negative \(z\)-coordinate (the unshaded region in \cref{fig:covering:conventions}).

\begin{figure}[t]
	\centering
	\labellist 
	\scriptsize
	\pinlabel $2$ at 12 118
	\pinlabel $1$ at 55 118
	\pinlabel $2$ at 117 118
	\pinlabel $3$ at 3 66
	\pinlabel $3$ at 107 66
	\pinlabel $2$ at 12 13
	\pinlabel $1$ at 55 13
	\pinlabel $2$ at 117 13
	\pinlabel $1$ at 172 30
	\pinlabel $2$ at 233 30
	\pinlabel $3$ at 233 90
	\pinlabel $\longrightarrow$ at 137 59
	\pinlabel $\TwistB$ at 230 59
	\pinlabel $\TwistW$ at 203 33
	\endlabellist
	\includegraphics[scale=1]{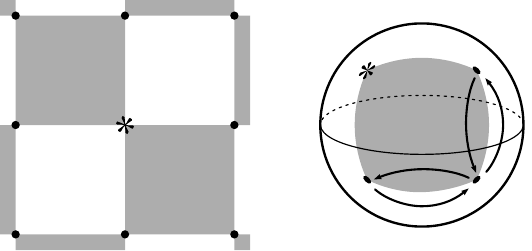}
	\caption{%
		The covering 
		\(p\co \PlanarCover\rightarrow \FourPuncturedSphere\). 
		The points \(\mathtt{e}_i\) in \(\EndsThree\) 
		and their preimages under \(p\) 
		are labelled by their index \(i\).
		The two pairs of arrows indicate the action 
		of the mapping class group generators.
	}
	\label{fig:covering:conventions}
\end{figure}

The group of deck transformations of \(p\) consists of 
translations by elements in \((2\Z)^2\) followed by multiplication by \(\pm1\). 
Any \(\tau\in\operatorname{Mod}(\FourPuncturedSphere,*)\) can be lifted to a homeomorphism \(\PuncturedPlane\rightarrow\PuncturedPlane\) that preserves the puncture at the origin and extends uniquely to a homeomorphism \(\tilde{\tau}\co\R^2\rightarrow\R^2\). Setting \(\tilde{\tau}(1,0)=(a,b)\) and \(\tilde{\tau}(0,1)=(c,d)\), it can be checked that 
the matrix
\[
\psi(\tau)
=
\begin{bmatrix*}[c]
	a & c \\
	b & d
\end{bmatrix*},
\]
as an element of \(\PSL_2(\Z)\), is %
independent of the choices involved. %
In fact: 
\begin{lemma}\label{lem:MCG:PSL}
	The map 
	\(
	\psi\co\operatorname{Mod}(\FourPuncturedSphere,*) \rightarrow \PSL_2(\Z)
	\)
	is an isomorphism with
	\[
	\pushQED{\qed} 
	\psi(\TwistB)=
	\begin{bmatrix*}[c]
		1 & 0 \\
		-1 & 1
	\end{bmatrix*}
	\qquad\text{and}\qquad 
	\psi(\TwistW)=
	\begin{bmatrix*}[c]
		1 & 1 \\
		0 & 1
	\end{bmatrix*}
	\qedhere
	\popQED
	\]
\end{lemma}

\subsection{The action of the mapping class group on multicurves} Given a formal curve $F$ and some $\tau\in \Mod(\FourPuncturedSphere,*)$ there is a formal curve $\tau(F)$ determined by the identity $\tau(\gamma_F)=\gamma_{\tau(F)}$. Towards a direct description of the formal curve $\tau(F)$, consider the following recipe producing \(\TwistB(F)\): First consider each arrow
	\[\begin{tikzcd}[column sep=30pt]
		\circ
		\arrow{r}{S}
		&
		\bullet
	\end{tikzcd}\]
	in \(F\). 
	If no other arrow is incident at the generator~\(\circ\) 
	the arrow and the generator~\(\circ\) is removed;  
	this is case (\( \SaddleBC \)a)
	of \cref{tab:MCG:TwistB}.
	Otherwise, the other arrow incident at the generator~\(\circ\) 
	is labelled by \(D\cdot\varG^k\) for some \(k\in\Z^{\geq0}\): 
	\[
	\begin{tikzcd}[column sep=30pt]
		\circ
		&
		\circ
		\arrow{r}{S}
		\arrow[leftrightarrow,swap]{l}{D\cdot\varG^k}
		&
		\bullet
	\end{tikzcd}
	\]
	If there is an outgoing arrow starting at the leftmost \(\circ\) that is labelled by \(S\) then the labelled subgraph is replaced 
	as shown in case (\( \SaddleBC \)b) 
	of \cref{tab:MCG:TwistB}. 
	Otherwise, the subgraph is replaced
	according to cases (\( \SaddleBC \)c) or (\( \SaddleBC \)d)
	depending on the direction of the arrow labelled \(D\cdot\varG^k\). 
	Finally, all remaining arrows in \(F\) are replaced
	according to the remaining cases in \cref{tab:MCG:TwistB}.

	\begin{proposition}
	For any formal curve $F$ the above recipe produces a formal curve $F'$ satisfying \(\TwistB(\gamma_F)=\gamma_{F'}\). In particular, $F' = \TwistB(F)$.
	\end{proposition}

	\begin{proof}
	If \(F\) does not contain any generators \(\bullet\) 
	only the replacement rules
	(\( S\SaddleBC \)) and (\(\Dw\)) in \cref{tab:MCG:TwistB}
	apply and the identity is immediate.
	Otherwise, 
	split the labelled graph 
	associated with \(F\) 
	at generators \(\bullet\) into pieces of the following types (compare \cite[Section 3]{HanselmanWatson}):
	\begin{enumerate}
		\item 
		\( 
		\begin{tikzcd}[column sep=30pt]
			\bullet
			\arrow{r}{D\cdot\varG^k}
			&
			\bullet
		\end{tikzcd}
		\)
		\item 
		\( 
		\begin{tikzcd}[column sep=30pt]
			\bullet
			\arrow{r}{S^2\cdot\varG^k}
			&
			\bullet
		\end{tikzcd}
		\)
		\item 
		\( 
		\begin{tikzcd}[column sep=30pt]
			\bullet
			\arrow[leftrightarrow]{r}{S\cdot\varG^{k_0}}
			&
			\circ
			\arrow[leftrightarrow]{r}{D\cdot\varG^{k_1}}
			&
			\dots
			\arrow[leftrightarrow]{r}{D\cdot\varG^{k_{n-1}}}
			&
			\circ
			\arrow[leftrightarrow]{r}{S\cdot\varG^{k_{n}}}
			&
			\bullet
		\end{tikzcd}
		\)
		\item
		\( 
		\begin{tikzcd}[column sep=30pt]
			\bullet
			\arrow[leftrightarrow]{r}{S\cdot\varG^{k_0}}
			&
			\circ
			\arrow[leftrightarrow]{r}{D\cdot\varG^{k_1}}
			&
			\dots
			\arrow[leftrightarrow]{r}{}
			&
			\circ
		\end{tikzcd}
		\)
	\end{enumerate}
	By the replacement rule (\(\Db\)) in \cref{tab:MCG:TwistB},
	pieces of the first type remain unchanged.
	For the other three types one checks directly that the replacements 
	occur exactly as prescribed by the identity 
	\(\TwistB(\gamma_F)=\gamma_{\TwistB(F)}\).
	\end{proof}
	
	A similar description for \(\TwistW(F)\) is obtained 
	by swapping the labels of generators and indices of algebra elements.
\begin{example}
	\begin{align*}
	F &=\left[
	\begin{tikzcd}[column sep=20pt,ampersand replacement=\&]
		\circ
		\arrow{r}{S}
		\&
		\bullet
		\arrow{r}{D}
		\&
		\bullet
		\arrow[leftarrow]{r}{S}
		\&
		\circ
		\arrow{r}{D\cdot\varG^2}
		\&
		\circ
		\arrow{r}{S}
		\&
		\bullet
		\arrow{r}{D}
		\&
		\bullet
		\arrow{r}{S}
		\&
		\circ
	\end{tikzcd}
	\right]
	\\
		\TwistB(F)
		&=
		\left[
		\begin{tikzcd}[column sep=20pt,ampersand replacement=\&]
			\bullet
			\arrow{r}{D}
			\&
			\bullet
			\arrow{r}{S^2\cdot\varG^2}
			\&
			\bullet
			\arrow{r}{D}
			\&
			\bullet
			\arrow{r}{S}
			\&
			\circ
			\arrow{r}{D}
			\&
			\circ
		\end{tikzcd}
		\right]
		\\
		\TwistW(F)
		&=
		\left[
		\begin{tikzcd}[column sep=20pt,ampersand replacement=\&]
			\circ
			\arrow{r}{S}
			\&
			\bullet
			\arrow{r}{D}
			\&
			\bullet
			\arrow{r}{S^2}
			\&
			\bullet
			\arrow[leftarrow]{r}{D}
			\&
			\bullet
			\arrow[leftarrow]{r}{S}
			\&
			\circ
			\arrow{r}{D\cdot\varG^2}
			\&
			\circ
			\arrow{r}{S}
			\&
			\bullet
			\arrow{r}{D}
			\&
			\bullet
			\arrow{r}{S}
			\&
			\circ
		\end{tikzcd}
		\right]
	\end{align*}
\end{example}

\begin{table}
	\centering
	\begin{tabular}{ccc}
		cases
		&
		old labelled subgraph
		&
		new labelled subgraph
		\\
		\hline
		(\( \SaddleBC \)a)&
		\( 
		\begin{tikzcd}[column sep=30pt]
			\phantom{\circ}
			\arrow[leftrightarrow,"\times" anchor=center]{r}{}
			&
			\circ
			\arrow{r}{S}
			&
			\bullet
		\end{tikzcd}
		\)
		&
		\(
		\begin{tikzcd}[column sep=30pt]
			\bullet
		\end{tikzcd}
		\)
		\\
		(\( \SaddleBC \)b)&
		\( 
		\begin{tikzcd}[column sep=30pt]
			\bullet
			&
			\circ
			\arrow{r}{D\cdot\varG^k}
			\arrow{l}[swap]{S}
			&
			\circ
			\arrow{r}{S}
			&
			\bullet
		\end{tikzcd}
		\)
		&
		\(
		\begin{tikzcd}[column sep=30pt]
			\bullet
			\arrow{r}{S^2\cdot\varG^k}
			&
			\bullet
		\end{tikzcd}
		\)
		\\
		(\( \SaddleBC \)c)&
		\( 
		\begin{tikzcd}[column sep=30pt]
			\circ
			&
			\circ
			\arrow{r}{S}
			\arrow[swap]{l}{D\cdot\varG^k}
			&
			\bullet
		\end{tikzcd}
		\)
		&
		\( 
		\begin{tikzcd}[column sep=30pt]
			\circ
			&
			\bullet
			\arrow{l}[swap]{ S\cdot\varG^{k+1}}
		\end{tikzcd}
		\)
		\\
		(\( \SaddleBC \)d)&
		\( 
		\begin{tikzcd}[column sep=30pt]
			\circ
			\arrow{r}{D\cdot\varG^k}
			&
			\circ
			\arrow{r}{S}
			&
			\bullet
		\end{tikzcd}
		\)
		&
		\(
		\begin{tikzcd}[column sep=30pt]
			\circ
			\arrow{r}{ S\cdot\varG^k}
			&
			\bullet
		\end{tikzcd}
		\)
		\\
		(\( \SaddleBC \)e)&
		\( 
		\begin{tikzcd}[column sep=30pt]
			\circ
			\arrow{r}{S\cdot\varG^{k+1}}
			&
			\bullet
		\end{tikzcd}
		\)
		&
		\( 
		\begin{tikzcd}[column sep=30pt]
			\circ
			\arrow{r}{D\cdot\varG^k}
			&
			\circ
			&
			\bullet
			\arrow[swap]{l}{ S}
		\end{tikzcd}
		\)
		\\
		(\( \SaddleCB \))&
		\( 
		\begin{tikzcd}[column sep=30pt]
			\bullet
			\arrow{r}{S\cdot\varG^k}
			&
			\circ
		\end{tikzcd}
		\)
		&
		\( 
		\begin{tikzcd}[column sep=30pt]
			\bullet
			\arrow{r}{ S}
			&
			\circ
			\arrow{r}{D\cdot\varG^k}
			&
			\circ
		\end{tikzcd}
		\)
		\\
		(\( S\SaddleBC \))&
		\( 
		\begin{tikzcd}[column sep=30pt]
			\circ
			\arrow{r}{S^2\cdot\varG^k}
			&
			\circ
		\end{tikzcd}
		\)
		&
		\( 
		\begin{tikzcd}[column sep=30pt]
			\circ
			\arrow{r}{D\cdot\varG^k}
			&
			\circ
		\end{tikzcd}
		\)
		\\
		(\( S\SaddleCB \))&
		\( 
		\begin{tikzcd}[column sep=30pt]
			\bullet
			\arrow{r}{S^2\cdot\varG^k}
			&
			\bullet
		\end{tikzcd}
		\)
		&
		\( 
		\begin{tikzcd}[column sep=30pt]
			\bullet
			\arrow{r}{ S}
			&
			\circ
			\arrow{r}{D\cdot\varG^{k}}
			&
			\circ
			&
			\bullet
			\arrow[swap]{l}{ S}
		\end{tikzcd}
		\)
		\\
		(\( \Dw \))&
		\( 
		\begin{tikzcd}[column sep=30pt]
			\circ
			\arrow{r}{D\cdot\varG^k}
			&
			\circ
		\end{tikzcd}
		\)
		&
		\( 
		\begin{tikzcd}[column sep=30pt]
			\circ
			\arrow{r}{S^2\cdot\varG^k}
			&
			\circ
		\end{tikzcd}
		\)
		\\
		(\( \Db \))&
		\( 
		\begin{tikzcd}[column sep=30pt]
			\bullet
			\arrow{r}{D\cdot\varG^k}
			&
			\bullet
		\end{tikzcd}
		\)
		&
		\( 
		\begin{tikzcd}[column sep=30pt]
			\bullet
			\arrow{r}{D\cdot\varG^k}
			&
			\bullet
		\end{tikzcd}
		\)
	\end{tabular}
	\caption{The action of \(\TwistB\) in formal curves; \(k\in\Z^{\geq0}\)}
	\label{tab:MCG:TwistB}
\end{table}

Given a graded curve with local system \(\gamma=(F,g,X,e)\) we define a new graded curve with local system \(\Phi_\mathtt{e}(\TwistB)(\gamma)=(\TwistB(F),g',X,e')\) depending on a choice of endpoint \(\mathtt{e}\in\EndsThree\).
Here \(e'\) is chosen such that 
the edges \(e\) and \(e'\) induce matching orientations 
of the curves \(\gamma(F)\) and \(\gamma(F')\). To define the grading \(g'\), note that \cref{tab:MCG:TwistB} sets up a one-to-one correspondence between vertices \(x_\bullet\) from \(F\) and vertices \(x_\bullet'\) from \(\TwistB(F)\) that are labelled by \(\bullet\).
Then \(g'\) is chosen such that
\[
g'(x_\bullet')=g(x_\bullet)+
\begin{cases*}
	(1,0) & if \(\mathtt{e}=\Endi\)
	\\
	(-2,-1) & otherwise
\end{cases*}
\]
Such a grading exists and is unique unless there is no vertex labelled \(\bullet\). In the remaining case when all vertices are labelled by \(\circ\) there is a one-to-one correspondence between vertices \(x_\circ\) from \(F\) and vertices \(x_\circ'\) from \(\TwistB(F)\). In this case \(g'\) is chosen such that
\[
g'(x_\circ')=g(x_\circ)+
\begin{cases*}
	(3,1) & if \(\mathtt{e}=\Endi\)
	\\
	(0,0) & otherwise
\end{cases*}
\]
Extend \(\Phi_\mathtt{e}(\TwistB)\)
to multicurves componentwise resulting in 
a map \(\Phi_\mathtt{e}(\TwistB)\co\mathcal{C}\rightarrow\mathcal{C}\), where \(\mathcal{C}\) denotes the set of graded curves with local systems, and note that
there is a map \(\Phi_\mathtt{e}(\TwistW)\co\mathcal{C}\rightarrow\mathcal{C}\)
defined in the same way, except that the roles of \(\circ\) and \(\bullet\) are reversed and \(\Endiii\) plays the role of \(\Endi\).

\begin{proposition}\label{prop:twisting-curves}
	There is a left group action of 
	\(\Mod(\FourPuncturedSphere,*)\)
	on
	\(
	\mathcal{C}\times\EndsThree
	\)
	determined by
	\[\TwistL(\Gamma,\mathtt{e})=(\Phi_\mathtt{e}(\TwistL)(\Gamma),\TwistL(\mathtt{e}))\] 
	for \(\halfbullet\in\{\bullet,\circ\}\) and 
	\(\mathtt{e}\in\EndsThree\), where \(\TwistB\) permutes \(\Endii\leftrightarrow\Endiii\) and \(\TwistW\) permutes \(\Endi\leftrightarrow\Endii\).
\end{proposition}

This group action is the lift of a group action on the category of type D structures over \(\B\); see \cref{prop:twisting-complexes}.

\subsection{Naturality under the action of the mapping class group}
\label{subsec:naturalityMCG}
	Given a pointed oriented Conway tangle  \(T\) and a mapping class
	\( 
	\tau\in \Mod(\FourPuncturedSphere,*)
	\)
	let \(\tau(T)\) be the four-ended tangle defined in \cref{fig:rhoT}, 
	where \(\beta\in B_3\) is chosen such that \(\varphi_\circledast(\beta)=\tau\).
Note that \(\tau(T)\) is independent of the choice of \(\beta\in \varphi_\circledast^{-1}(\tau)\), 
since \(\beta\) is unique up to full twists. 
\cref{fig:twists:generators:homeoB,fig:twists:generators:homeoW} illustrate this definition for \(\tau=\TwistB\) and \(\tau=\TwistW\), respectively. 

\begin{figure}
	\centering
	\begin{subfigure}{0.3\textwidth}
		\centering
		\(\braidaction\)
		\caption{\(\tau(T)\)}\label{fig:rhoT}
	\end{subfigure}%
	\begin{subfigure}{0.25\textwidth}\centering		
		\(\braidactionB\)
		\caption{\(\TwistB(T)\)}\label{fig:twists:generators:homeoB}
	\end{subfigure}%
	\begin{subfigure}{0.25\textwidth}\centering		
		\(\braidactionW\)
		\caption{\(\TwistW(T)\)}\label{fig:twists:generators:homeoW}
	\end{subfigure}%
	\caption{The tangle \(\tau(T)\) in \cref{thm:twisting:general-coeff}}
	\label{fig:twists:generators}
\end{figure}

\begin{theorem}\label{thm:twisting:general-coeff}
	Let \(T\) be a pointed oriented Conway tangle and
	\( 
	\tau\in \Mod(\FourPuncturedSphere,*)
	\).
	Let \(\mathtt{e}\in\EndsThree\)
	be the tangle end 
	that has the opposite orientation 
	of the other two ends in \(\EndsThree\). 
	Suppose the chain complex \(\DD(T;\CoeffRing)\)
	is chain homotopic to \(C(\Gamma)\) for some graded multicurve with local system~\(\Gamma\). Then the chain complex \(\DD(\tau(T);\CoeffRing)\) 
	is chain homotopic to \(C(\Phi_{\mathtt{e}}(\tau)(\Gamma))\). Moreover, the same holds with \(\DD_k(T;\CoeffRing)\) in place of \(\DD(T;\CoeffRing)\) for any positive integer \(k\). 
\end{theorem}

Note that the hypothesis on \(\DD(T;\CoeffRing)\) is satisfied 
whenever \(\CoeffRing\) is a field $\CoeffField$. This result will usually be applied in the following more geometric way, which for $\CoeffField=\CoeffFieldTwo$ recovers \cite[Theorem~1.13]{KWZ}.

\begin{corollary}\label{cor:twisting:field-coeff:ungraded}
	Let \(T\) be a pointed oriented Conway tangle and
	\( 
	\tau\in \Mod(\FourPuncturedSphere,*)
	\).
	There is a one-to-one correspondence between graded curves \((F,g)\)
	in \(\BNr(T;\CoeffField)\) and \((F',g')\) in \(\BNr(\tau(T);\CoeffField)\) such that \(\gamma(F')=\tau(\gamma(F))\).
	Moreover, if \((F,g)\) is compact with local system \((X,e)\), then the local system on \((F',g')\) is equal to \((X,e')\), where \(e'\) is chosen such that the edges \(e\) and \(e'\) induce matching orientations of the curves \(\gamma(F)\) and \(\gamma(F')\). The same holds for \(\Khr(T;\CoeffField)\) in place of \(\BNr(T;\CoeffField)\). \qed
\end{corollary}

\begin{remark}\label{rmk:twisting:field-coeff:ungraded}Denoting the underlying multiset of ungraded curves 
of \(\Khr(T)\) and \(\BNr(T)\) by
\(\underline{\Khr}(T)\) and \(\underline{\BNr}(T)\), 
respectively, the concise summary of \cref{cor:twisting:field-coeff:ungraded} is that
\[
\underline{\Khr}(\tau(T))
=
\tau(\underline{\Khr}(T))
\quad
\text{and}
\quad
\underline{\BNr}(\tau(T))
=
\tau(\underline{\BNr}(T)).
\]
\end{remark}

Let 
\(\rho=\TwistB\TwistW\TwistB=\TwistW\TwistB\TwistW \in \Mod(\FourPuncturedSphere,*)\). 
\cref{thm:twisting:general-coeff} for the special case \(\tau=\rho\) has a simple reformulation because the tangle \(\rho(T)\)
is homotopic to the tangle obtained from a given tangle \(T\)
by \(180^\circ\) rotation of the unit 3-ball
around the line through the origin and the point \(*\).
Let \(\rho\co\B\rightarrow\B\) be 
the \(\CoeffRing\)-linear isomorphism
that switches \(\bullet\) and \(\circ\) 
in the indices of all elements of the standard basis of \(\B\). 
The involution \(\rho\) induces an autofunctor 
on the category of chain complexes over \(\B\),
which we denote by the same symbol. 
Moreover, if \(\gamma=(F,g,X,e)\) is a bigraded curve with local system, 
we define  \(\rho\gamma=(\rho\circ F,g,X,e)\). 

\begin{lemma}\label{lem:rho*DD}
	For any bigraded multicurve \(\Gamma\), 
	\(\rho (C(\Gamma))=C(\rho\Gamma)\). 
	In particular, if \(F\) is a formal curve then
	\(\gamma_{\rho\circ F}=\rho(\gamma_F)\). 
	\qed
\end{lemma}

\begin{proposition}\label{prop:rho*Khr}
	If \(T\) is a Conway tangle then 
	\(
	\rho(\DD(T))=\DD(\rho(T))
	\) %
	and %
	\( %
	\rho(\DD_1(T))=\DD_1(\rho(T)).
	\)
	Thus,
	\(\rho(\BNr(T))=\BNr(\rho(T))\) and 
	\(\rho(\Khr(T))=\Khr(\rho(T))\).
\end{proposition}
\begin{proof}
	If \(\Diag\) is a diagram for \(T\) 
	then a diagram \(\Diag'\) for \(\rho(T)\) 
	is obtained by reflection along the diagonal line 
	through \(*\) and the origin and 
	reversing over and under strands at each crossing. 
	Note that this preserves the signs of the crossings. 
	Therefore, the vertices of the cubes of resolutions 
	associated with \(\Diag\) and \(\Diag'\) 
	carry the same gradings and are labelled
	by crossingless tangles
	that are related by reflection along said line; the same is true for 
	the cobordisms along the differentials. 
	Delooping can be done simultaneously and preserves this symmetry. 
	Under the isomorphism \(\omega\) from \cref{thm:OmegaFullyFaithful}, 
	the symmetry corresponds precisely to the involution \(\rho\). 
	This proves the first equality. 
	The second follows from the observation that 
	\(\rho\) fixes the element~\(\varG\), while the third and fourth equalities reinterpret the first two using \cref{lem:rho*DD}.
\end{proof}

\subsection{Rational and special curves}

\begin{definition}
	Let \(\aKh(0)\) be the non-compact curve in \(\FourPuncturedSphere\) 
	connecting the two bottom-most punctures and 
	 lifting to a straight horizontal line segment 
	under the covering map 
	\(p\co\PuncturedPlane\rightarrow\FourPuncturedSphere\). 
	For \(n\in\mathbb N\), 
	let \(\rKh_{n}(0)\) and \(\sKh_{2n}(0)\) be the compact curves in \(\FourPuncturedSphere\) with local systems \((-1)\in\GL_1(\CoeffRing)\) 
	that lift to the curves $\tilde{\rKh}_{n}(0)$ and $\tilde{\sKh}_{2n}(0)$,
	respectively, in \cref{fig:geography:Upstairs}.
\end{definition}

We will refer to the subscripts \(n\) and \(2n\) as the lengths of the curves \(\rKh_{n}(0)\) and \(\sKh_{2n}(0)\), respectively. 
The curves \(\rKh_{n}(0)\) and \(\sKh_{2n}(0)\) for $n=1,2,3$ are illustrated in \cref{fig:geography:Downstairs}.

\begin{definition}\label{def:grading}
	\newcommand{\GGzqh}[4]{\prescript{#3}{}{#1}_{#4}}%
	We equip the curves defined above with the following grading, 
	where the notation \(\GGzqh{\halfbullet}{0}{s}{t}\) means that the bigrading of the generator \(\halfbullet\) is equal to \(q^sh^t\):
	\begin{align*}
		C(\aKh(0))
		&=
		[\GGzqh{\bullet}{0}{0}{0}]
		\\
		C(\rKh_n(0))
		&=
		\left[
		\begin{tikzcd}[row sep=16pt,column sep=16pt,ampersand replacement=\&]
			\GGzqh{\bullet}{0}{-n}{-n}
			\arrow{r}{S^2}
			\arrow{d}[swap]{-D}
			\&
			\bullet
			\arrow{r}{D}
			\&
			\bullet
			\arrow{r}{S^2}
			\&
			\cdots
			\arrow{r}{}
			\&
			\GGzqh{\bullet}{0}{n-2}{-1}
			\arrow{d}{D \textnormal{ or } S^2}
			\\
			\GGzqh{\bullet}{0}{2-n}{1-n}
			\arrow{r}{S^2}
			\&
			\bullet
			\arrow{r}{D}
			\&
			\bullet
			\arrow{r}{S^2}
			\&
			\cdots
			\arrow{r}{}
			\&
			\GGzqh{\bullet}{0}{n}{0}
		\end{tikzcd}
		\right]
		\\
		C(\sKh_{2n}(0))
		&=
		\left[
		\begin{tikzcd}[row sep=16pt,column sep=16pt,ampersand replacement=\&]
			\GGzqh{\circ}{0}{-2n-1}{-n-1}
			\arrow{r}{S}
			\arrow{d}[swap]{-D}
			\&
			\bullet
			\arrow{r}{D}
			\&
			\bullet
			\arrow{r}{S^2}
			\&
			\cdots
			\arrow{r}{S^2}
			\&
			\bullet
			\arrow{r}{D}
			\&
			\bullet
			\arrow{r}{S}
			\&
			\GGzqh{\circ}{0}{2n-1}{n}
			\arrow{d}{D}
			\\
			\GGzqh{\circ}{0}{1-2n}{-n}
			\arrow{r}{S}
			\&
			\bullet
			\arrow{r}{D}
			\&
			\bullet
			\arrow{r}{S^2}
			\&
			\cdots
			\arrow{r}{S^2}
			\&
			\bullet
			\arrow{r}{D}
			\&
			\bullet
			\arrow{r}{S}
			\&
			\GGzqh{\circ}{0}{2n+1}{n+1}
		\end{tikzcd}
		\right]
	\end{align*}
	In the last two complexes the number generators in idempotent \(\iota_\bullet\) is \(2n\) and \(4n\), respectively.
\end{definition}

\begin{remark}
	As a special case, we obtain
	\(
	C(\rKh_1(0))
	=
	\Bigg[
	\begin{tikzcd}[column sep=1cm,ampersand replacement=\&]
		\bullet
		\arrow[r,"-D" above, bend left]
		\arrow[r,"S^2" below, bend right]
		\&
		\bullet
	\end{tikzcd}
	\Bigg]
	=
	\Cone{\varG\cdot1_{C(\aKh(0))}}
	.
	\)
\end{remark}

\begin{figure}[t]
	\centering
	\(\GeographyCovering\)
	\caption{%
		Lifts of the curves \(\rKh_n(0)\) and \(\sKh_{2n}(0)\) 
		on \(\FourPuncturedSphere\) to \(\PuncturedPlane\)
	}\label{fig:geography:Upstairs}
\end{figure}

\begin{figure}[t]
	\centering
	\begin{subfigure}{0.235\textwidth}
		\centering
		\labellist 
		\footnotesize
		\pinlabel $\sKh_{2}(0)$ at 45 75
		\pinlabel $\rKh_{1}(0)$ at 45 35
		\endlabellist
		\includegraphics[scale=1]{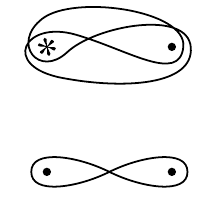}
		\caption{\(n=1\)}\label{fig:geography:Downstairs1}
	\end{subfigure}
	\begin{subfigure}{0.235\textwidth}
		\centering
		\labellist 
		\footnotesize
		\pinlabel $\sKh_{4}(0)$ at 65 76
		\pinlabel $\rKh_{2}(0)$ at 65 27
		\endlabellist
		\includegraphics[scale=1]{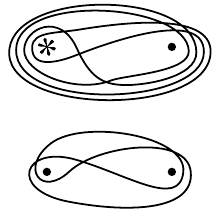}
		\caption{\(n=2\)}\label{fig:geography:Downstairs2}
	\end{subfigure}
	\begin{subfigure}{0.235\textwidth}
		\centering
		\labellist 
		\footnotesize
		\pinlabel $\sKh_{6}(0)$ at 65 78
		\pinlabel $\rKh_{3}(0)$ at 65 17
		\endlabellist
		\includegraphics[scale=1]{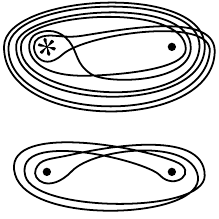}
		\caption{\(n=3\)}\label{fig:geography:Downstairs3}
	\end{subfigure}
	\caption{%
		The curves \(\rKh_n(0)\) and \(\sKh_{2n}(0)\) 
		on \(\FourPuncturedSphere\) 
		for some values of \(n\geq1\)
	}\label{fig:geography:Downstairs}
\end{figure}

Given a slope \(\nicefrac{p}{q}\in\QPI\) 
with \(p\) and \(q\) coprime, 
choose 
\(\tau_{p/q}\in\Mod(\FourPuncturedSphere,*)\) 
such that
its image under the isomorphism \(\psi\) 
from \cref{lem:MCG:PSL}
is equal to
\[
\begin{bmatrix*}[c]
	q & r \\
	p & s
\end{bmatrix*}\in \PSL_2(\Z)
\]
Note that this matrix, 
considered as an automorphism 
of the covering space \(\PuncturedPlane\), 
maps straight lines of slope 0 
to straight lines of slope \(\nicefrac{p}{q}\).
The mapping class \(\tau_{p/q}\) is unique 
up to multiplication by \(\TwistW\) on the right. The following definitions are independent of this choice.

\begin{definition}\label{def:basic_curves}
	Given \(\nicefrac{p}{q}\in\QPI\) and \(\mathtt{e}\in\EndsThree\),
	let \(\mathtt{e}'=\tau_{p/q}^{-1}(\mathtt{e})\). Define 
	\(\aKh(\nicefrac{p}{q};\mathtt{e})\),
	\(\rKh_n(\nicefrac{p}{q};\mathtt{e})\), 
	and
	\(\sKh_{2n}(\nicefrac{p}{q};\mathtt{e})\)
	as the images of 
	\(\aKh(0)\), 
	\(\rKh_n(0)\),
	and
	\(\sKh_{2n}(0)\),
	respectively, 
	under \(\Phi_{\mathtt{e}'}(\tau_{p/q})\). 
\end{definition}

When considering curves only up to a grading shift \(\mathtt{e}\) is ommited from the notation.

\begin{definition}
	Given \(\nicefrac{p}{q}\in\QPI\) the rational tangle of slope \(\nicefrac{p}{q}\) is \(Q_{p/q}=\tau_{p/q}(\Lo)\).
\end{definition}

\begin{corollary}
	Given a slope \(\nicefrac{p}{q}\in\QPI\),
	choose an orientation on \(Q_{p/q}\) 
	and let \(\mathtt{e}\) be the corresponding non-distinguished tangle end 
	(as in \cref{thm:twisting:general-coeff}). 
	Then 
	\[
	\DD(Q_{p/q};\CoeffRing)
	\simeq
	C(\aKh(\nicefrac{p}{q};\mathtt{e});\CoeffRing)
	\quad\text{and}\quad
	\DD_1(Q_{p/q};\CoeffRing)
	\simeq
	C(\rKh_1(\nicefrac{p}{q};\mathtt{e});\CoeffRing)
	\]
	for any ring \(\CoeffRing\).
\end{corollary}

\begin{proof}
	Let \(\mathtt{e}'=\tau_{p/q}^{-1}(\mathtt{e})\) 
	as in \cref{def:basic_curves}. 
	Then \(\mathtt{e}'=\Endi\) or \(\Endii\), 
	depending on the choice of \(\tau_{p/q}\) 
	for our fixed slope \(\nicefrac{p}{q}\in\QPI\).
	Note that \(\DD(Q_0;\CoeffRing)=C(\aKh(0))\), 
	independent of the orientation of \(Q_0\). 
	So by \cref{thm:twisting:general-coeff}
	with \(T=Q_0\) oriented such that 
	\(\mathtt{e}'\) has the opposite orientation 
	to the other two ends in \(\EndsThree\), 
	\(\DD(Q_{p/q};\CoeffRing)=\DD(\tau_{p/q}(Q_0);\CoeffRing)\)
	is chain homotopic to 
	\(
	C(\Phi_{\mathtt{e}'}(\tau_{p/q})(\aKh(0));\CoeffRing)
	=
	C(\aKh(\nicefrac{p}{q};\mathtt{e});\CoeffRing)
	\). 
	The proof for the second identity is similar. 
\end{proof}

Now \cite[Theorem~5.7]{KWZ-thin} can be established for arbitrary coefficients: 
\begin{theorem}
	\label{thm:Kh:rational_tangle_detection}
	Let \(T\) be an oriented Conway tangle 
	and \(\nicefrac{p}{q}\in\QPI\). 
	Then \(\DD_1(T)\) is chain homotopic 
	to \(C(\rKh_1(\nicefrac{p}{q}))\) 
	up to grading shifts if and only if \(T=Q_{p/q}\). 
	The same remains true replacing 
	\(\DD_1(T)\) by \(\DD(T)\) and 
	\(\rKh_1(\nicefrac{p}{q})\) by \(\aKh(\nicefrac{p}{q})\).
\end{theorem}

\begin{proof}
	The if-directions of both statements follow from \cref{thm:twisting:general-coeff}
	and the fact that 
	\(\DD(Q_0)=C(\aKh(0))\) 
	and 
	\(\DD_1(Q_0)=C(\rKh_1(0))\). 
	By \cref{thm:twisting:general-coeff}, 
	it suffices for the only-if-direction 
	to consider the case \(\nicefrac{p}{q}=0\). 
	Moreover, since 
	\(
	C(\rKh_1(0))
	=
	\Cone{\varG\cdot1_{C(\aKh(0))}}
	\),
	the second statement follows from the first. 
	The proof for \(\CoeffRing=\CoeffFieldTwo\) 
	from \cite[Theorem~5.7]{KWZ-thin} generalizes 
	if the two-component unlink is detected 
	by reduced Khovanov homology with coefficients in \(\CoeffRing\). 
	Alternatively, 
	one can modify the argument as follows,
	so that it relies instead on the unknot detection 
	of Khovanov homology over arbitrary coefficients: Consider the link \(L=\rho(T)\cup T\). 
	By \cref{prop:GlueingTheorem:Kh-cx,prop:rho*Khr}, 
	\[
	\Khr(L;\CoeffRing)\otimes (q^{-1}h^{-1}\CoeffRing\oplus q^{1}h^{0}\CoeffRing)
	\cong 
	H_*\Mor(-\!\rho C(\rKh_1(0)),C(\rKh_1(0))).
	\]
	A direct computation of the right-hand side 
	identifies \(\Khr(L;\CoeffRing)\) 
	with \(\Khr(U;\CoeffRing)\cong \CoeffRing\) 
	up to a potential grading shift. 
	By the universal coefficient theorem, we deduce that 
	\(\Khr(L;\Q)\) has rank at most 1. 
	By the spectral sequence to Lee homology \cite{Lee2005}, 
	we deduce that \(\Khr(L;\Q)\) has rank exactly 1 and that \(L\) is a knot.
	Kronheimer and Mrowka stated their unknot detection result as follows:
	a knot \(L\) is the unknot 
	if and only if \(\Khr(L;\Z)\cong \Z\) 
	\cite[Theorem~1.1]{KhDetectsUnknot}. 
	It is known among experts that 
	their proof generalizes to rational coefficients; 
	see \cite[Corollary~1.3 and Proposition~1.4]{KhDetectsUnknot}.
	Thus \(\rho(T)\cup T=U\) and we conclude with \cref{lem:RatTanDet}.
\end{proof}

\begin{lemma}\label{lem:RatTanDet}
	Let \(T_1\) and \(T_2\) be two 4-ended tangles that glue together to form the unknot. 
	Then either \(T_1\) or \(T_2\) is a rational tangle. 
\end{lemma}
\begin{proof}
	Adapting the proof of \cite[Lemma~6.3]{pqMod},
	let \(\FourPuncturedSphere\) be the sphere 
	along which we glue the two tangle complements. 
	Let \(\Delta\) be the disk bounding the unknot.
	Without loss of generality, we may 
	assume that \(\FourPuncturedSphere\) and \(\Delta\) 
	intersect transversely. 
	The intersection \(\FourPuncturedSphere\cap\Delta\) then
	consists of two arcs and a finite number of circles. 
	We proceed by induction on this number of circles. 
	
	If there is no circle in \(\FourPuncturedSphere\cap\Delta\), 
	\(\Delta\) is divided by two arcs into three topological disks. 
	Two of them cobound the union of 
	\(\FourPuncturedSphere\cap\Delta\) and 
	the components of one of the two tangles,
	which is therefore rational. 
	
	If there is a circle in \(\FourPuncturedSphere\cap\Delta\), 
	we can find an innermost circle \(\gamma\) in \(\FourPuncturedSphere\cap\Delta\) (that is, a circle that bounds a disk \(D\) in \(\Delta\) 
	that contains no other circle of \(\FourPuncturedSphere\cap\Delta\)). 
	If \(\gamma\) bounds a disk \(D'\) in \(\FourPuncturedSphere\), 
	then \(D\cup D'\) bounds a 3-ball, 
	which we can use as a homotopy for \(\Delta\) to remove \(\gamma\) 
	(along with any other components of \(\FourPuncturedSphere\cap \Delta\) in \(D'\)). 
	This strictly reduces the number of components of \(\FourPuncturedSphere\cap\Delta\),
	so we are done by the induction hypothesis.
	If \(\gamma\) does not bound a disk in \(\FourPuncturedSphere\), 
	it separates two punctures from the other two. 
	Hence \(D\) separates the two strands in one of the tangles. 
	If those strands were knotted,
	\(T_1\cup T_2\) would be a connected sum of two knots, 
	contradicting the assumption that it is the unknot. 
	Thus the strands must be unknotted and therefore belong to a rational tangle. 
\end{proof}

We can now state a more precise version of \cref{thm:geography:Khr:intro}.  

\begin{theorem}\label{thm:geography:Khr}
	Let \(\field\) be a field, \(T\) a Conway tangle, and \(\gamma\) a component of \(\Khr(T;\field)\). 
	Then, up to an overall shift in bigrading, \(\gamma\) is equal to either \(\rKh_n(\slope)\) or \(\sKh_{2n}(\slope)\) for some positive integer \(n\) and some slope \(\slope\in\QPI\). 
\end{theorem}

In other words, components of \(\Khr(T;\field)\) are classified by their type, slope, length, and bigrading.

\begin{remark}\label{rmk:local_systems_can_be_complicated}
	\Cref{thm:geography:Khr} says in particular that the local systems on the multicurve \(\Khr(T)\) are uninteresting for any tangle \(T\). 
	The analogous statement is not true for the multicurve \(\BNr(T)\). 
	For example, there exists a tangle \(T\) obtained by cutting open the torus knot \(T_{6,11}\)
	such that 
	\(\BNr(T;\Z/7)\) contains a component with indecomposeable 2-dimensional local system \cite{khtpp-non-trivial-local-system}.
\end{remark}

\subsection{Detecting tangle connectivity}
\label{subsec:connectivity}
The invariants described above are sensitive to certain coarse properties of the given tangle. To describe this, set \[
	\connx=\ConnX
	\qquad
	\connz=\ConnZ
	\qquad	
	\conny=\ConnY
	\]
\begin{definition}\label{def:connectivity}
		 A Conway tangle \(T\) has connectivity \(\conn{T}=\conni\) if the two open strands of~\(T\) are connected as in the diagram \(\conni\), for \(i\in\{1,2,3\}\).\end{definition}
		 
The notation is chosen so that \(\conn{T}=\conni\) if and only if the Conway tangle $T$ connects the tangle end \(\mathtt{e}_i\in\EndsThree\) to the distinguished tangle end, for each \(i\in\{1,2,3\}\). In particular, writing $\conns
	=
	\conn{Q_{\slope}}$
for \(\slope\in\QPI\) we calculate that
	\[
	\conns
	=
	\begin{cases*}
		\connx
		&
		if \(p\) is odd and \(q\) is even;
		\\
		\connz
		&
		if \(p\) and \(q\) are both odd;
		\\	
		\conny
		&
		if \(p\) is even and \(q\) is odd.
	\end{cases*}
	\]
Connectivity detection, in the following sense, plays an essential role in the proof of \cref{thm:mutation_invariance:links}:

\begin{theorem}\label{prop:connectivity-detected-by-rationals-of-odd-length}
Let \(\Gamma\) be the multicurve \(\Khr(T)\) or \(\BNr(T)\) associated with a Conway tangle~$T$. 
	If \(\Gamma\) includes a rational curve of slope \(\slope\in\QPI\), of odd length, and carrying a local system that may be non-trivial, then \(\conn{T}=\conns\).
\end{theorem}

\begin{remark}
When combined with \cref{thm:geography:Khr:intro} 
(Theorem \ref{thm:geography:Khr}), 
\cref{prop:connectivity-detected-by-rationals-of-odd-length} 
can be used to prove something stronger and of independent interest, 
namely that \(\Khr(T)\) detects the connectivity \(\conn{T}\) 
when \(T\) has no closed components; 
see \cref{thm:Khr:connectivity-detection}.
A reasonable question is whether \(\Khr(T)\) 
contains an odd-length rational curve for {\em any} tangle $T$; 
if yes then \(\Khr(T)\) detects connectivity for all tangles.
\end{remark}

Also required is an algebraic (and more technical) form of detection. 
To state this, consider the central elements 
	\[
	a_1
	=
	\Db
	\qquad
	a_2
	=
	-S^2
	\qquad
	a_3
	=
	\Dw
	\] of the algebra \(\B\).
The sign in the definition of \(a_2\) is cosmetic; this choice helps avoid signs later
(compare \cref{prop:twists-and-basepoint-action}).

\begin{proposition}\label{prop:connectivity-implies-action-trivial}
	Let \(T\) be a pointed Conway tangle and 
	\(\DD\) a direct summand of \(\DD(T)\) or of \(\DD_n(T)\) 
	for some \(n\geq1\). If \(\conn{T}=\conni\) then \(
		a_i\cdot1_{\DD}
		\simeq 
		0
	\), for each \(i\in \{1,2,3\}\).
		\end{proposition}

\section{%
	Proof of 
	\texorpdfstring{\cref{thm:mutation_invariance:links}}
		{Theorem \ref{thm:mutation_invariance:links}}
	(%
	Mutation invariance%
	)
}\label{sec:mutation}

	Given a tangle \( R \), let \( \mutx(R) \), \( \mutz(R) \) and \( \muty(R) \) be the tangles obtained from \( R \) by  $180^\circ$ rotations about the horizontal axis, the axis pointing out of the drawing plane, and the vertical axis, respectively; see \cref{fig:mutation}.  
	The notation is chosen such that the mutation \(\muti\) maps the tangle end \(\mathtt{e}_i\) to the distinguished tangle end for \(i\in\{1,2,3\}\); the tangles \( \muti(R) \) are called mutants of the tangle \(R\). 
  If \( R \) is oriented these three tangles are oriented such that the orientations agree at the tangle ends. When this requires reversal of the orientation of the two open components of~\( R \) then we also reverse the orientation of all other components; otherwise we do not change any orientation.

\begin{figure}[H]
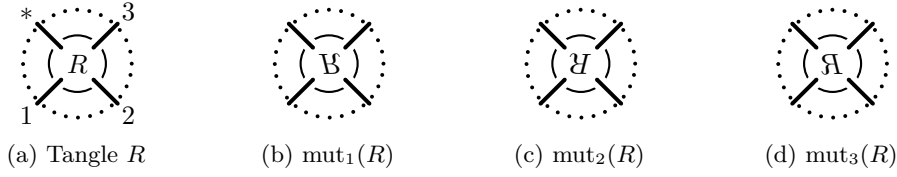

  \centering
  \begin{subfigure}{0.2\textwidth}
    \centering
     $\MutationTangle$
    \caption{Tangle \( R \)}
  \end{subfigure}
  \begin{subfigure}{0.2\textwidth}
    \centering
     $\MutationTangleX$
    \caption{\( \mutx(R) \)}
  \end{subfigure}
  \begin{subfigure}{0.2\textwidth}
    \centering
     $\MutationTangleZ$
    \caption{\( \mutz(R) \)}
  \end{subfigure}
  \begin{subfigure}{0.2\textwidth}
    \centering
     $\MutationTangleY$
    \caption{\( \muty(R) \)}
  \end{subfigure}
  \caption{Conway mutation.}\label{fig:mutation}
\end{figure}

\begin{definition}\label{def:mutation}	
	Two links \(L\) and \(L'\) are Conway mutants
	if they agree outside a 3-ball 
	whose boundary intersects the links transversely in four points, 
	and the tangles inside the 3-ball are related by mutation. 
	In symbols, 
	\(L=T\cup T'\) and 
	\(L'=T\cup\muti(T')\) 
	for some Conway tangles \(T\), \(T'\) and \(i\in\{1,2,3\}\).
\end{definition}

	Let 
	\( \varphi_1 \), 
	\( \varphi_2 \), and 
	\( \varphi_3 \) 
	be the algebra automorphisms of \( \B \) defined by
	\begin{align*}
	\varphi_1\co & 
	\DotcobB \mapsto -\DotcobB,\quad
	\DotcobC \mapsto +\DotcobC,\quad
	\SaddleCB \mapsto \SaddleCB,\quad
	\SaddleBC \mapsto \SaddleBC\\
	\varphi_2\co & 
	\DotcobB \mapsto -\DotcobB,\quad
	\DotcobC \mapsto -\DotcobC,\quad
	\SaddleCB \mapsto \SaddleCB,\quad
	\SaddleBC \mapsto \SaddleBC\\
	\varphi_3\co & 
	\DotcobB \mapsto +\DotcobB,\quad
	\DotcobC \mapsto -\DotcobC,\quad
	\SaddleCB \mapsto \SaddleCB,\quad
	\SaddleBC \mapsto \SaddleBC
	\end{align*}
	and denote the induced functors on type~D structures by \( \mutx \), \( \mutz \), and \( \muty \), respectively.
The following theorem is \cite[Theorem~9.8]{KWZ}. 

\begin{theorem}\label{thm:mutation:DD}
	For any pointed Conway tangle \(T\),
	\(
	\DD(\muti(T);\CoeffRing)=\muti\DD(T;\CoeffRing)
	\) where \( i\in\{1,2,3\} \).\qed
\end{theorem}

In combination with \cref{thm:classification:complexes_over_B:simplified}, we obtain:

\begin{corollary}\label{cor:mutation:BNr-curves}
	The multicurves \(\BNr(T;\CoeffField)\) and \(\BNr(\muti(T);\CoeffField)\) agree up to multiplication of some (possibly all or none) local systems by \(-1\) for \(i\in\{1,2,3\}\), where $T$ is any pointed Conway tangle and $\CoeffField$ is any choice of field. \qed
\end{corollary}

\cref{thm:mutation:DD} is the key result about how Conway mutation affects 
the algebraic tangle invariant \(\DD(T)\). It has several consequence when combined with \cref{thm:GlueingTheorem:Kh}.
First, it recovers mutation invariance of Khovanov homology over characteristic 2, 
which was first proved by Wehrli and Bloom \cite{Bloom,wehrli2010mutation}.
Second, by contrast, this new perspective immediately generalizes to Bar-Natan homology in characteristic 2. 
Third, it gives a conceptual explanation for the failure of mutation invariance
for links. 
Fourth, it allowed us to show that the Rasmussen invariant over any field is preserved under Conway mutation. 
These consequences are described in detail in \cite[Section~9]{KWZ}.
In a first step towards the proof of \cref{thm:mutation_invariance:links} we now show:

\begin{theorem}\label{thm:mutation:Khr}
	If \(T\) is a pointed Conway tangle and \(i\in\{1,2,3\}\) such that \(\conn{T}=\conni\) then \(
	\Khr(\muti(T);\CoeffField)
	\simeq
	\Khr(T;\CoeffField)
	\) for any field \(\CoeffField\). 
\end{theorem}

	Without the condition \(\conn{T}=\conni\), 
	\cref{thm:mutation:Khr} is false; compare \cref{fig:example:mutant:Wehrli} and see \cite{WehrliCounterexample} and \cite[Example~9.11]{KWZ}.

\begin{lemma}\label{lem:mutation:DD1}
	With notation as in \cref{thm:mutation:Khr},
	\(
	\DD_1(\muti(T);\CoeffRing)
	\simeq
	\muti(\DD_1(T);\CoeffRing).
	\)
\end{lemma}

\begin{proof}
	Consider the case \(i=1\) first. 
	Unpacking the definitions:
	\[
	\mutx\DD_1(T)
	=
	\mutx\Cone{(S^2-\Db-\Dw)\cdot1_{\DD(T)}}
	=
	\Cone{(S^2+\Db-\Dw)\cdot1_{\mutx\DD(T)}}
	\]
	By \cref{thm:mutation:DD} the right-hand side is homotopy equivalent to 
	\[
	\Cone{(S^2+\Db-\Dw)\cdot1_{\DD(\mutx(T))}}
	\]
	and, according to \cref{prop:connectivity-implies-action-trivial},
	\(\Db\cdot1_{\DD(\mutx(T))}\simeq0\) so that
	\[
	\Cone{(S^2+\Db-\Dw)\cdot1_{\DD(\mutx(T))}}
	\simeq
	\Cone{(S^2-\Db-\Dw)\cdot1_{\DD(\mutx(T))}}
	\]
	 recovering \(\Cone{\varG\cdot1_{\DD(\mutx(T))}}=\DD_1(\mutx(T))\) as claimed. The proof for the cases \(i=2\) and \(i=3\) are analogous, except that in the case \(i=2\) 
	we obtain \(\Cone{(-\varG)\cdot1_{\DD(\mutz(T))}}\) instead of \(\Cone{\varG\cdot1_{\DD(\mutz(T))}}\) in the last step. However these mapping cones are isomorphic. 
\end{proof}

\begin{proof}[Proof of \cref{thm:mutation:Khr}]Without loss of generality, choose multicurve representatives \(\DD_1^c(\muti(T))\) and \(\DD_1^c(T)\) for \(\DD_1(\muti(T))\) and \(\DD_1(T)\), respectively, for the type D structures in \cref{lem:mutation:DD1}. 
	Then the multicurves  
	\(\Khr(\muti(T))\)
	and 
	\(\Khr(T)\)
	 differ in at most the signs of their local systems. 
	By \cref{thm:geography:Khr},
	components of 
	\(\Khr(\muti(T))\)
	and 
	\(\Khr(T)\)
	only carry the 
	1-dimensional local systems \((-1)\in\GL_1(\field)\), so the signs of the local systems must agree.
\end{proof}

\begin{lemma}\label{lem:Khr-components:number-of-Ds}
	The complex \(C(\mathbf{s}_{2n}(\slope))\) contains 
	\(\max(2,2n|q|)\) differentials labelled~\(D_\bullet\) and 
	\(\max(2,2n|p|)\) differentials labelled~\(D_\circ\).
	Similarly, the complex \(C(\mathbf{r}_{n}(\slope))\) contains
	\(n|q|\) differentials labelled~\(D_\bullet\) and 
	\(n|p|\) differentials labelled~\(D_\circ\).
\end{lemma}
\begin{proof}
	This is immediate by inspection of the lift of the corresponding curve to the covering space \(\PuncturedPlane\rightarrow\FourPuncturedSphere\). 
\end{proof}

The proof of \cref{thm:mutation_invariance:links} requires one more calculation, 
for which we introduce the following notation:
Let 
\(
\DD^{\sKh}_{2n}(\slope;X)
\)
denote the chain complex \(C(\gamma)\),
where \(\gamma\) is equal to the special curve \(\sKh_{2n}(\slope)\) 
as a graded curve and carries the local system \((X,e)\) for some fixed edge \(e\).
Similarly, 
let 
\(
\DD^{\rKh}_{2n}(\slope;X)
\)
denote the chain complex \(C(\gamma)\),
where \(\gamma\) is equal to the rational curve \(\rKh_{2n}(\slope)\) 
as a graded curve and carries the local system \((X,e)\) for some fixed edge \(e\).

\begin{lemma}\label{lem:mutation-on-Khr-special-or-rational}
	Let \(\slope\in\QPI\), \(i\in\{1,2,3\}\), \(n\) and \(m\) two positive integers, and \(X\in\GL_m(\Z)\) a local system of rank \(m\). 
	Then
	\[
	\muti\DD^\mathbf{s}_{2n}(\slope;X)
	\cong
	\DD^\mathbf{s}_{2n}(\slope;X)
	\]
	and
	\[
	\muti\DD^\mathbf{r}_{n}(\slope;X)
	\cong
	\begin{cases*}
		\DD^\mathbf{r}_{n}(\slope;X)
		&
		if \(n\) is even or \(\conni=\conns\)
		\\
		\DD^\mathbf{r}_{n}(\slope;-X)
		&
		otherwise
	\end{cases*}
	\]
\end{lemma}

\begin{proof}
	Assume first that the rank \(m\) of the local system is \(1\). 
	The operation \(\muti\) acts on complexes by changing the signs of some differentials. Since each of the given complexes corresponds to a closed curve with 1-dimensional local system, its isomorphism class is changed by \(\muti\) if and only if the number of sign changes is odd. 
	By \cref{lem:Khr-components:number-of-Ds}, 
	the number of sign changes is even for the complexes \(\DD^\mathbf{s}_{2n}(\slope)\), 
	which proves the first isomorphism. 
	The same applies for \(\DD^\mathbf{r}_{n}(\slope)\) if \(n\) is even. 
	Suppose then that \(n\) is odd. 
	Counting modulo~2, 
	the complex \(\DD^\mathbf{r}_{n}(\slope)\) contains
	\(q\)~differentials labelled~\(D_\bullet\) and
	\(p\)~differentials labelled~\(D_\circ\).
	We now analyse each case for \(i\) separately:
	If \(i=1\) only the differentials labelled~\(D_\bullet\) change their sign. 
	So the isomorphism class of the complex changes 
	if and only if \(q\) is odd,
	that is, when \(\conni\neq \conns\).
	The argument for the other two cases is similar. 
	For local systems of rank \(m>1\) the proof generalises immediately. 
\end{proof}

\begin{proof}[Proof of \cref{thm:mutation_invariance:links}]
	The links \(L\) and \(L'\) are related by Conway mutation, so write 
	\[
	L=
	T\cup T'
	\quad
	\text{and}
	\quad
	L'=
	T\cup \muti(T')
	\]
	for Conway tangles \(T\) and \(T'\) and some \(i\in\{1,2,3\}\).
	By assumption, the mutation preserves components.
	This means that either \(\conni=\conn{T'}\) or \(\conn{T}\neq\conn{T'}\). 
	
	In the case \(\conni=\conn{T'}\), \cref{thm:GlueingTheorem:Kh} gives
	\begin{align*}
		\Khr(L)
		&\cong
		\HF(-\BNr(T),\Khr(T'))
		\\
		\Khr(L')
		&\cong
		\HF(-\BNr(T),\Khr(\muti(T')))
	\end{align*}
	where the right hand sides agree since \(\Khr(\muti(T'))=\Khr(T')\), according to \cref{thm:mutation:Khr}. 
	
	In the case \(\conn{T}\neq\conn{T'}\) consider the alternate isomorphisms given by \cref{thm:GlueingTheorem:Kh}: 
	\begin{align*}
		\Khr(L)
		&\cong
		\HF(-\Khr(T),\BNr(T'))
		\\
		\Khr(L')
		&\cong
		\HF(-\Khr(T),\BNr(\muti(T')))
	\end{align*}
	Write %
	\[
	-\Khr(T)=\coprod_{i=1}^n \gamma_i(X_i)
	\quad
	\text{and}
	\quad
	\BNr(T')=\coprod_{j=1}^m \gamma'_j(X'_j)
	\]
	where \(\gamma_i\) and \(\gamma'_j\) denote primitive curves with their respective local systems \(X_i\) and \(X'_j\). 	
	By \cref{cor:mutation:BNr-curves}, mutation of the tangle \(T'\) changes at most the signs of the local system. 
	In other symbols,
	\[
	\BNr(\muti(T'))=\coprod_{j=1}^m \gamma'_j(\varepsilon_j \cdot X'_j)
	\]
	for some \(\varepsilon_j\in\{\pm1\}\). 
	It suffices to show that
	\begin{equation}\label{eqn:mutation-invariance-components}
		\HF(\gamma_i(X_i),\gamma'_j(X'_j))
		\cong
		\HF(\gamma_i(X_i),\gamma'_j(\varepsilon_j \cdot X'_j))
	\end{equation}
	for all \(i=1,\dots,n\) and \(j=1,\dots,m\).
	When \(\gamma_i\neq \gamma'_j\),
	the isomorphism \eqref{eqn:mutation-invariance-components} follows from \cref{rem:local-systems-gluing}. 
	So suppose \(\gamma=\gamma_i=\gamma'_j\) for some pair \((i,j)\).
	Then applying the geography restrictions for \(\Khr\) (\cref{thm:geography:Khr:intro}), 
	the curve \(\gamma\) is either rational or special. 
	If \(\gamma\) is rational of even length or if it is special 
	then \(\varepsilon_j=+1\) 
	by \cref{lem:mutation-on-Khr-special-or-rational} and 
	so \eqref{eqn:mutation-invariance-components} holds trivially. 
	Finally, if the curve \(\gamma\) is rational of odd length, 
	its slope detects the connectivity of the tangles 
	\(-T\) and \(T'\) 
	(\cref{prop:connectivity-detected-by-rationals-of-odd-length}) since  
	 \(-\Khr(T)=h^{1}\Khr(-T)\) (\cref{prop:mirroring:Khr}). 
	However $\conn{-T}=\conn{T}$ so this case does not occur since \(\conn{T}\neq\conn{T'}\) by assumption. 
\end{proof}

This proof appeals to the classification of objects in \(\C^\B\) in terms of multicurves, which only works over fields, and as such it does not readily adapt to establish mutation invariance of reduced Khovanov homology over the integers. It is also unclear how to adapt the proof to Bar-Natan homology or unreduced Khovanov homology, even restricting to fields: the proof of \cref{thm:mutation_invariance:links} relies in a crucial way on \cref{thm:geography:Khr:intro}, which classifies components of \(\Khr(T)\) into rational and special components. Components of \(\BNr(T)\) and the unreduced invariant \(\Kh(T)\) are more complicated in general; compare in particular \cref{rmk:local_systems_can_be_complicated}. 

\section{Naturality of the mapping class group action}\label{sec:MCGaction}

In this section we prove \cref{prop:twisting-curves,thm:twisting:general-coeff}, extending  the proof over \(\CoeffFieldTwo\) 
from \cite[Theorem 8.1]{KWZ}
to integer coefficients. 

\subsection{Preliminaries on type AD bimodules}\label{sub:bordered}
We first collect material from bordered Floer homology and pin down sign conventions, omitting proofs as these follow from the same routine calculations as the \(\CoeffFieldTwo\) versions in \cite{LOT,LOTBimodules,LOTMor}. Throughout this subsection fix  bigraded algebras \(A\) and \(B\) over some commutative, unital rings \(\CoeffRing\) and \(\CoeffSRing\), respectively. (Later, \(\CoeffRing\) and \(\CoeffSRing\) will be \(\Z\) or some rings of idempotents.) Write \(\mu_A\) and \(\mu_B\) for the multiplication maps in \(A\) and \(B\), respectively.
We follow the Koszul sign rule for tensor products: 

\begin{definition}\label{def:koszul-rule}
Given homogeneous homomorphisms 
\(f\co X_1\rightarrow X_2\) and  
\(g\co Y_1\rightarrow Y_2\) 
between (bi)graded abelian groups 
\(X_1\), \(X_2\), \(Y_1\), \(Y_2\),
and elements 
\(x\in X_1\), \(y\in Y_1\),
we define 
\[
(f\otimes g)
(x\otimes y)
=
(-1)^{|x|\cdot|g|}
(f(x)\otimes g(y))
\] 
where \(|\cdot|\) denotes the homological grading.
\end{definition}

	Let 
	\(\CuAB\)
	be the dg category whose objects are bigraded \(\CoeffRing\)-\(\CoeffSRing\)-bimodules and whose morphisms 
	\(f\in \CuAB(X,Y)\) 
	for 
	\(X,Y\in\CuAB\)
	are sequences 
	\((f_i)_{i\geq0}\) 
	of \(\CoeffRing\)-\(\CoeffSRing\)-bimodule homomorphisms
	\[
	f_i
	\co
	A^{\otimes i}\otimes X
	\rightarrow
	Y\otimes B.
	\] 
	A morphism
	\(f\in \CuAB(X,Y)\) 
	is bounded if \(f_i=0\) for \(i\gg 0\), and it  
	is strictly unital if, for all \(x\in X\), \(f_1(1\otimes x)=x\) and
	\[
	f_k(a_k\otimes \dots\otimes a_1\otimes x)=0
	\quad
	\text{for all \(a_1,\dots,a_k\in A\) with \(a_j=1\) for some \(j=1,\dots,k\).}
	\]
	Moreover, a morphism \(f\) is homogeneous of grading \(\gr(f)=q^rh^s\)
	if \(f_i\) is homogeneous of grading \(\gr(f_i)=q^rh^{s-i}\) for all \(i\geq0\). 
	For any 
	\(X,Y,Z\in\CuAB\), 
	there is an \(\CoeffRing\)-\(\CoeffSRing\)-bimodule homomorphism
	\[
	\CuAB(Y,Z)
	\otimes
	\CuAB(X,Y)
	\rightarrow
	\CuAB(X,Z),
	\quad
	g\otimes f
	\mapsto
	(g\circ f)
	\]
	defined by
	\[
	(g\circ f)_i
	=
	\sum_{i=j+k}
	(-1)^{k|g_j|}
	(1_Z\otimes\mu_B)
	\circ
	(g_j\otimes 1_B)
	\circ
	(
	1_{A^{\otimes j}}
	\otimes
	f_k
	)
	\]
	The identity morphism for any object \(X\)
	is given by \((1_X)_0=1_X\otimes \eval_1\), 
	where \(\eval_1\co \Z\rightarrow B,1\mapsto 1\), 
	and 
	\((1_Y)_{i}=0\)
	for \(i\geq1\). 
	For any 
	\(X,Y\in\CuAB\) the morphism spaces 
	\(\CuAB(X,Y)\)
	carry a differential 
	\[
	(\underline{\partial}(f))_{i}
	=
	\sum_{i=\ell+2+m}
	(-1)^{\ell+1}
	f_{i-1}
	\circ
	(
	1_{A^{\otimes \ell}}
	\otimes
	\mu_A
	\otimes
	1_{A^{\otimes m}\otimes X}
	)
	\]
	for \(i\geq0\). 
	In particular, \((\underline{\partial}(f))_{i}=0\) for \(i=0,1\). 
	
	\begin{definition}\label{def:typeAD}
	The category 
	of type~AD structures over \(A\) and \(B\) is the dg category of chain complexes over
	\(\CuAB\). 
	We will write \(\CAB\) for the full subcategory of strictly unital type~AD structures. 
	More explicitly, the objects in 
	\(\CAB\)
	are pairs \((Y,\delta)\),
	where 
	\(Y\in\CuAB\)
	and
	\(\delta\in\CuAB(Y,Y)\)
	is strictly unital
	with
	\(\gr(\delta)=q^0h^{1}\)
	such that
	\(\delta^2+\underline{\partial}(\delta)=0\). 
	For any 
	\((X,\delta^X),(Y,\delta^Y)\in\CAB\),
	the differential 
	\(\ddCAB\)
	on
	\(\CAB(X,Y)=\CuAB(X,Y)\)
	is defined by 
	\[
	\ddCAB(f)
	=
	\delta^Y\circ f
	-
	(-1)^{|f|}
	f\circ\delta^X
	+
	\underline{\partial}(f)
	\]
	\end{definition}

\begin{lemma}\label{lem:signs-typeAD-partial-well-defined}
	The categories 
	\(\CuAB\)
	and
	\(\CAB\)
	are well-defined. 
	In particular, 
	\[
	\underline{\partial}(g\circ f)
	=
	\underline{\partial}(g)\circ f
	+
	(-1)^{|g|}
	g\circ\underline{\partial}(f)
	\]
	for any 
	\(f\in\CuAB(X,Y)\)
	and
	\(g\in\CuAB(Y,Z)\), 
	where 
	\(X,Y,Z\in\CuAB\).
	Moreover, 
	\[
	\ddCAB(g\circ f)
	=
	\ddCAB(g)\circ f
	+
	(-1)^{|g|}
	g\circ \ddCAB(f)
	\]
	for any 
	\(f\in\CAB(X,Y)\)
	and
	\(g\in\CAB(Y,Z)\), 
	where 
	\(X,Y,Z\in\CAB\). 
	\qed
\end{lemma}

\begin{remark}\label{rem:typeAD:graphical-notation}
	Similar to the graphical notation from \cref{rem:typeD:graphical-notation} for type D structures, objects and morphisms in \(\CuAB\) and \(\CAB\) can often be described by labelled graphs. Suppose
	\(X,Y\in\CuAB\)
	are freely generated
	as \(\CoeffRing\)-\(\CoeffSRing\)-bimodules 
	by \(\{x_i\}_i\) and \(\{y_j\}_j\), 
	respectively. 
	Suppose further that the algebra \(A\) is freely generated by elements \(\{a_\ell\}_\ell\). 
	Then morphisms \(f=(f_k)_{k\geq0}\in\CuAB(X,Y)\) 
	with 
	\[
	f_k
	\left(
	a_{\ell(k,i,k)}\otimes \cdots \otimes a_{\ell(k,i,1)}\otimes x_i
	\right)
	=
	\sum_j y_j \otimes b_{j,k,i}
	\]
	can be uniquely represented by the collection of arrows
	\[
	\begin{tikzcd}[column sep = 4.5cm]
		x_i
		\arrow{r}{%
			\left(\left.
			a_{\ell(k,i,k)},\dots,a_{\ell(k,i,1)}\right|b_{j,k,i}
			\right)
		}
		&
		y_j
	\end{tikzcd}
	\]
	By convention, arrows with \(b_{j,k,i}=0\) are omitted. To simplify the signs, assume that the algebras \(A\) and \(B\) are concentrated in homological degree 0. Then composition of morphisms in \(\CuAB\) can be expressed graphically as follows:
	\[
	\Big(%
	\begin{tikzcd}[column sep = 2cm]
		x
		\arrow{r}{(a_{k},\dots,a_{1}|b_1)}
		&
		y
		\arrow{r}{(a_{k+j},\dots,a_{k+1}|b_2)}
		&
		z
	\end{tikzcd}
	\Big)%
	=
	\Big(%
	\begin{tikzcd}[column sep = 3.5cm]
		x
		\arrow{r}{(a_{k+j},\dots,a_{1}|(-1)^{k\sigma}b_2b_1)}
		&
		z
	\end{tikzcd}
	\Big)%
	\]
	where \(\sigma\) is the grading of the morphism \(y\rightarrow z\). 
	In particular, if the arrow \(y\rightarrow z\) is part of a differential, \(\sigma=1-j\).
To express the differential \(\underline{\partial}\) in this notation, we make the assumption that the basis \(\{a_\ell\}_\ell\) of \(A\) is closed under multiplication, by which we mean that the product of any two basis elements is either equal to zero or \(\pm a_\ell\) for some \(\ell\). 
	Under this assumption, the differential of an arrow
	\[
	\begin{tikzcd}[column sep = 2.5cm]
		x
		\arrow{r}{(a_{i},\dots,a_{1}|b)}
		&
		y
	\end{tikzcd}
	\]
	is equal to the collection of arrows
	\[
	\bigcup_{\ell=0}^{i-1}
	\Big\{
	\begin{tikzcd}[column sep = 6.5cm]
		x
		\arrow{r}{(a_i,\dots,a_{i-\ell+1}
			,a'',a',a_{i-\ell-1},\dots,a_{1}|\pm(-1)^{\ell+1}\cdot b)}
		&
		y
	\end{tikzcd}
	\Big|%
	a''a' = \pm a_{i-\ell}
	\Big\}
	\]
\end{remark}

\begin{warn}
	The standard basis of \(\B\) is closed under multiplication in the sense of the previous remark. 
	By contrast, however, the \(\varG\)-equivariant basis of \(\B\) is not closed under multiplication. 
\end{warn}

Following the notation from \cite{LOTBimodules}, 
	we often use the super- and subscripts \(\typeAD{A}{\cdot}{B}\) 
	to highlight the rings over which a type~AD structure is defined. 
	Notice that a type AD structure
	\((Y,\delta)\in\CAB\)
	with \(\delta_k=0\) for all \(k>0\)
	is a type D structure over $B$. 
	In particular, when $B$ is the Bar-Natan algebra $\B$, 
	we recover the type D structures encountered in \cref{sub:typeD}. 

\begin{definition}\label{def:boxtensoring}
	Given a type D structure 
	\(X^A=(X,\delta_X)\in\C^{A}\),
  there is a recursively defined sequence 
	\((\delta^i_X)_{i\geq0}\)
	of \(\CoeffRing\)-module homomorphisms
	\[
	\delta^i_X
	\co
	X\rightarrow
	X\otimes A^{\otimes i}
	\]
	by setting 
	\(\delta^0_X=1_X\)
	and 
	\(\delta^{i+1}_X=(\delta^i_X\otimes 1_A)\circ \delta_X\). 
	If \(\delta^i_X=0\) for \(i\gg 0\), we say that \(\delta_X\) is bounded. Consider an object 
	\(\typeAD{A}{Y}{B}=(Y,\delta^Y)\) in \(\CAB\)
	and suppose that either
	\(\delta_X\)
	or
	\(\delta^Y\)
	is bounded. 
	We define 
	\((X\boxtimes Y,\delta_X\boxtimes \delta^Y)\in\C^B\) 
	by setting 
	\(X\boxtimes Y= X\otimes Y\)
	and
	\[
	\delta_X\boxtimes \delta^Y
	=
	\sum_{i=0}^\infty
	(1_X\otimes\delta^Y_i)
	\circ
	(\delta^i_X\otimes 1_Y)
	\]
\end{definition}

\begin{remark}\label{rem:bounded}
	Assuming that the algebra \(A\) is concentrated in homological grading 0, the differential of any finitely generated type~D over \(A\) is bounded.  
	The same is true for any finitely generated type~AD structures 
	over algebras with vanishing homological gradings. 
	Moreover, any morphism between two such type~AD structures is bounded. 
\end{remark}

\begin{lemma}\label{lem:signs-boxtimes-diff-well-defined}
	The differential \(\delta_X\boxtimes\delta^Y\) on \(X\boxtimes Y\) is well-defined. 
	\qed
\end{lemma}

\begin{example}
	With notation as in \cref{def:boxtensoring}, suppose that \(A=B=\CoeffRing\), \(\delta^Y_i=0\) for \(i>1\), and \(\delta^Y_1(1\otimes y)=y\otimes 1\) for all \(y\in Y\). 
	Then \((X,\delta_X)\) and \((Y,\delta^Y_0)\) are chain complexes over \(\CoeffRing\), and 
	\(
	(\delta_X\boxtimes \delta^Y)
	=
	(1_X\otimes\delta^Y_0)
	+
	(\delta_X\otimes 1_Y)
	\)
	coincides with the usual differential on the tensor product \((X,\delta_X)\otimes(Y,\delta^Y_0)\).
\end{example}

\begin{remark}	
	We now interpret \cref{def:boxtensoring} graphically, continuing the discussion from \cref{rem:typeD:graphical-notation,rem:typeAD:graphical-notation}. 	
	Suppose
	\(X\in\C^{A}\)
	is freely generated 
	as an \(\CoeffRing\)-module 
	by \(\{x_i\}_i\), 
	and that 
	\(Y\in\CAB\)
	is freely generated	
	as an \(\CoeffRing\)-\(\CoeffSRing\)-bimodule 
	by \(\{y_j\}_j\). 
	Furthermore, suppose that the algebra \(A\) is freely generated by elements \(\{a_\ell\}_\ell\). 
	Write
	\[
	\delta_X^k(x_i)
	=
	\sum_{j}
	x_j\otimes a_{\ell(j,k,i,k)}\otimes\cdots \otimes a_{\ell(j,k,i,1)}.
	\] for any  \(k\geq0\).
	Then for any \(y\in Y\) 
	\begin{align*}
		(\delta_X\boxtimes \delta^Y)(x_i\otimes y)
		=&
		\sum_{k=0}^\infty
		(1_X\otimes \delta^Y_k)
		(\sum_{j}
		x_{j}\otimes a_{\ell(j,k,i,k)}\otimes\cdots \otimes a_{\ell(j,k,i,1)} \otimes y)
		\\
		=&
		\sum_{k=0}^\infty
		\sum_{j}
		(-1)^{|x_j|\cdot(k+1)}%
		x_j\otimes \delta^Y_k(a_{\ell(j,k,i,k)}\otimes\cdots \otimes a_{\ell(j,k,i,1)} \otimes y)
	\end{align*}
	So, choosing graphical representations of \(X^A\)  and \(\typeAD{A}{Y}{B}\) with respect to the chosen bases of \(A\), \(X\), and \(Y\), we obtain the following interpretation of the differential in \(X\boxtimes Y\):
	For every chain of arrows
	\[
	\left(
	\begin{tikzcd}
		x
		\arrow{r}{a_1}
		&
		\cdots
		\arrow{r}{a_k}
		&
		x'
	\end{tikzcd}
	\right)
	\quad
	\text{in \(X^A\)}
	\]
	and for every matching arrow
	\[
	\Big(%
	\begin{tikzcd}[column sep = 3cm]
		y
		\arrow{r}{%
			\left(\left.
			a_k,\dots,a_1\right|b
			\right)
		}
		&
		y'
	\end{tikzcd}
	\Big)%
	\quad
	\text{in \(\typeAD{A}{Y}{B}\)}
	\]
	there is exactly one arrow 
	\[
	\Big(%
	\begin{tikzcd}[column sep = 2.5cm]
		x\boxtimes y
		\arrow{r}{(-1)^{|x'|\cdot(k+1)} b}
		&
		x'\boxtimes y'
	\end{tikzcd}
	\Big)%
	\quad
	\text{in \((X\boxtimes Y)^B\)}
	\]
	If we further assume that the algebra \(A\) is concentrated in homological grading 0 then \(|x'|=|x|+k\) and the sign on this arrow is \((-1)^{|x|\cdot(k+1)}\).
\end{remark}

\begin{remark}\label{rem:boxtensoring-AD-AD}
	\cref{def:boxtensoring} can be extended to boxtensor products of two type AD structures such that 
	\[
	(X\boxtimes Y) \boxtimes Z
	=
	X\boxtimes (Y \boxtimes Z)
	\]
	for any type D structure \(X^A\) and 
	type AD structures \(\typeAD{A}{Y}{B}\) and \(\typeAD{B}{Z}{C}\) 
	over compatible algebras. 
	(For simplicity, we assume that both type AD structures are bounded.)
	For every chain of arrows
	\[
	\Big(%
	\begin{tikzcd}[column sep=2.5cm]
		y
		\arrow{r}{(a^1_{\ell_1},\dots,a^1_1|b^1)}
		&
		\cdots
		\arrow{r}{(a^k_{\ell_k},\dots,a^k_1|b^k)}
		&
		y'
	\end{tikzcd}
	\Big)%
	\quad
	\text{in \(\typeAD{A}{Y}{B}\)}
	\]
	and for every matching arrow
	\[
	\Big(%
	\begin{tikzcd}[column sep = 3cm]
		z
		\arrow{r}{%
			\left(\left.
			b^k,\dots,b^1\right|c
			\right)
		}
		&
		z'
	\end{tikzcd}
	\Big)%
	\quad
	\text{in \(\typeAD{B}{Z}{C}\)}
	\]
	there is exactly one arrow 
	\[
	\Big(%
	\begin{tikzcd}[column sep = 2.5cm]
		y\boxtimes z
		\arrow{r}{(a^k_{\ell_k},\dots,a^1_1|(-1)^{\star} c)}
		&
		y'\boxtimes z'
	\end{tikzcd}
	\Big)%
	\quad
	\text{in \(\typeAD{B}{(Y\boxtimes Z)}{C}\)}
	\]
	where 
	\(
	\star
	=
	|y'|%
	(k+1)
	+
	\sum_{i=1}^k (\ell_k+\dots+\ell_{i+1})(\ell_i+1).
	\)
\end{remark}

The boxtensor products between type~D and type~AD, and (more generally between two type~AD structures) are compatible with notions of homotopy equivalence in the following sense, assuming sufficient boundedness conditions; see \cref{rem:bounded}.

\begin{proposition}\label{prop:boxtensoring-and-homotopies}
	Let \(X_1\simeq X_2\in\C^{A}\) and 
	\(Y_1 \simeq Y_2\in\CAB\).
	Suppose that either 
	\(\delta_{X_1}\) and \(\delta_{X_2}\) are bounded
	or that \(\delta_{Y_1}\) and \(\delta_{Y_2}\) are bounded and the isomorphisms up to chain homotopy between \(Y_1\) and \(Y_2\) are bounded.
	Then
	\(
	X_1\boxtimes Y_1
	\simeq
	X_2\boxtimes Y_2.
	\)
	\qed
\end{proposition}

\cref{prop:boxtensoring-and-homotopies}
follows immediately from
\cref{prop:boxtensoring-morphisms}
below.

\begin{definition}\label{def:boxtimes-morphisms-D}
	Given 
	\(f\in \C^{A}(X_1,X_2)\)
	define a sequence 
	\((f^{i})_{i\geq1}\) 
	of \(\CoeffRing\)-module homomorphisms
	\[
	f^{i}
	\co
	X_1\rightarrow
	X_2\otimes A^{\otimes i}
	\]
	where \(f^1=f\) and
	\[
	f^i
	=
	\sum_{i=j+1+k}
	(-1)^{k|f|+k}
	(\delta^j_{X_2}\otimes1_{A^{\otimes (k+1)}})
	\circ
	(f\otimes 1_{A^{\otimes k}})
	\circ
	\delta^k_{X_1}
	\quad
	\text{for } i>1.
	\]
	Furthermore, let 
	\(Y\in\CAB\)
	and assume that either \(\delta^Y\) is bounded or \(\delta_{X_1}\) and \(\delta_{X_2}\) are bounded. 
	We define 
	\[
	f
	\boxtimes
	1_Y
	=
	\Big(%
	\sum_{i=1}^\infty
	(1_{X_2}\otimes \delta^Y_i)
	\circ
	(f^{i}\otimes 1_{Y})
	\Big)%
	\in
	\C^{B}(X_1\boxtimes Y,X_2\boxtimes Y).
	\]
\end{definition}

\begin{lemma}\label{lem:signs-boxtimes-diff-D}
	With notation as in \cref{def:boxtimes-morphisms-D},
	\(
	\partial_{\C^{B}}(f\boxtimes 1_Y)
	=
	\partial_{\C^A}(f)\boxtimes 1_Y
	\).
	\qed
\end{lemma}

\begin{lemma}\label{lem:signs-boxtimes-composition-D}
	Let 
	\(
	f\in \C^{A}(X_1,X_2)
	\)
	and 
	\(
	f'\in \C^{A}(X_2,X_3)
	\)
	be two morphisms in the kernel of  
	\(\partial_{\C^{A}}\).
	Let
	\(Y\in\CAB\)
	and suppose that either 
	\(\delta_{X_1}\), \(\delta_{X_2}\), and \(\delta_{X_3}\) are bounded
	or \(\delta^Y\) is bounded. 
	Then	
	\[
	\pushQED{\qed} 
	(f'\boxtimes 1_Y)
	\circ
	(f\boxtimes 1_Y)
	\simeq
	(f'\circ f)\boxtimes 1_Y.
	\qedhere
	\popQED
	\]
\end{lemma}

\begin{definition}\label{def:boxtimes-morphisms-AD}
	Let 
	\(
	g\in \CAB(Y_1,Y_2)
	\) 
	and
	let 
	\(X\in\C^{A}\).
	Suppose that either \(\delta_X\) is bounded or
	\(\delta^{Y_1}\), \(\delta^{Y_2}\), and \(g\) are bounded. 
	We define 
	\[
	1_X
	\boxtimes
	g
	=
	\Big(
	\sum_{i=0}^\infty
	(-1)^{i|g|+i}
	(1_{X}\otimes g_i)
	\circ
	(\delta_X^{i}\otimes 1_{Y_1})
	\Big)%
	\in
	\C^{B}(X\boxtimes Y_1,X\boxtimes Y_2).
	\]
\end{definition}

\begin{remark}\label{rem:typeD:identity-boxtensor}
	Since we are assuming that type AD structures are strictly unital,
	the formulas in \cref{def:boxtimes-morphisms-AD,def:boxtimes-morphisms-D}
	for the morphism \(1_X\boxtimes 1_Y\) agree 
	and are both equal to \(1_{X\boxtimes Y}\).
\end{remark}

\begin{lemma}\label{lem:signs-boxtimes-diff-AD}
	With notation as in \cref{def:boxtimes-morphisms-AD}, 
	\(
	\partial_{\C^{B}}(1_X\boxtimes g)
	=
	1_X\boxtimes \ddCAB(g)
	\).
	\qed
\end{lemma}

\begin{lemma}\label{lem:signs-boxtimes-composition-AD}
	Let
	\(g\in \CAB(Y_1,Y_2)\)
	and 
	\(g'\in \CAB(Y_2,Y_3)\).
	Let
	\(X\in\C^{A}\)
	and suppose that either \(\delta_X\) is bounded or 
	\(\delta^{Y_1}\), \(\delta^{Y_2}\), \(\delta^{Y_3}\), \(g\), and \(g'\) are bounded. 
	Then
	\[
	\pushQED{\qed} 
	(1_X\boxtimes g')
	\circ
	(1_X\boxtimes g)
	=
	1_X\boxtimes (g'\circ g).
	\qedhere
	\popQED
	\]
\end{lemma}

\begin{definition}
	Let 
	\(f\in \C^{A}(X_1,X_2)\)
	and 
	\(g\in \CAB(Y_1,Y_2)\)
	be two morphisms, 
	in the kernels of 
	\(\partial_{\C^{A}}\)
	and 
	\(\ddCAB\), 
	respectively. 
	Suppose that either \(\delta_{X_1}\) and \(\delta_{X_2}\) are bounded or
	\(\delta^{Y_1}\), \(\delta^{Y_2}\), and \(g\) are bounded. 
	We then define
	\[
	f\boxtimes g
	=
	(1_{X_2}\boxtimes g)
	\circ
	(f\boxtimes 1_{Y_1})
	\in 
	\C^{B}(X_1\boxtimes Y_1,X_2\boxtimes Y_2).
	\]
\end{definition}

\begin{proposition}\label{prop:boxtensoring-morphisms}
	Let 
	\(f\in \C^{A}(X_1,X_2)\),
	\(f'\in \C^{A}(X_2,X_3)\)
	and 
	\(g\in \CAB(Y_1,Y_2)\),
	\(g'\in \CAB(Y_2,Y_3)\)
	be two pairs of morphisms, 
	in the kernels of 
	\(\partial_{\C^{A}}\)
	and 
	\(\ddCAB\), 
	respectively. 
	Suppose that either \(\delta_{X_1}\), \(\delta_{X_2}\), and \(\delta_{X_3}\) are bounded or
	\(\delta^{Y_1}\), \(\delta^{Y_2}\), \(\delta^{Y_3}\), \(g\), and \(g'\) are bounded. 
	Then
	\[
	\pushQED{\qed}
	(f'\boxtimes g')\circ(f\boxtimes g)
	\simeq
	(-1)^{|g'||f|}
	(f'\circ f)\boxtimes (g'\circ g).
	\qedhere
	\popQED
	\]
\end{proposition}

\begin{definition}\label{def:shifted-complexes}
	For any \(r,s\in\Z\), \(X\in\C^A\), and \(Y\in\CAB\) define
	\begin{align*}
		h^r X
		&=
		(h^r X,(-1)^r\delta_X)
		\\
		h^sY
		&=
		(h^s Y,((-1)^{s(i+1)}\delta^Y_i)_{i\geq0})
	\end{align*}
\end{definition}

\begin{lemma}\label{lem:signs-shifting-complexes}
	With notation as in \cref{def:shifted-complexes}, \(h^r X\) and \(h^sY\) are well-defined objects of \(\C^A\) and \(\CAB\), respectively, and 
	\(
	(h^r X)
	\boxtimes
	Y
	=
	h^{r}
	\left(
	X\boxtimes Y
	\right).
	\)
	Moreover, if \(f\co X_1\rightarrow X_1\) is a homomorphism in \(\C^A\), respectively \(\CAB\), with \(|f|=0\) 
	its mapping cone 
	\[
	\Cone{f}
	=
	\left[
	\begin{tikzcd}
		h^{-1}X_1
		\arrow{r}{f}
		&
		X_2
	\end{tikzcd}
	\right]
	=
	\left(
	h^{-1}X_1\oplus X_2,
	\begin{pmatrix}
		\delta_{h^{-1}X_1} & 0
		\\
		f & \delta_{X_2}
	\end{pmatrix}
	\right)
	\]
	is a well-defined object in \(\C^A\), respectively \(\CAB\).
	\qed
\end{lemma}

\begin{lemma}\label{lem:signs-boxtimes-mapping-cone}
	Let 
	\(f\in \C^{A}(X_1,X_2)\)
	be a homomorphisms with \(|f|=0\)
	for some 
	\(
	X_1,
	X_2
	\in
	\C^{A}
	\). 
	Let
	\(Y\in\CAB\)
	and suppose that either \(\delta^Y\) is bounded or 
	\(\delta_{X_1}\) and \(\delta_{X_2}\) are bounded.
	Then 
	\[
	\pushQED{\qed} 
	\Cone{f}
	\boxtimes Y
	=
	\Cone{f\boxtimes 1_Y}.
	\qedhere
	\popQED
	\] 
\end{lemma}

The following is an improved version of \cite[Lemma~2.17]{KWZ}.

\begin{lemma}[Clean-up Lemma]\label{lem:clean-up:improved}
	Let \(\underline{\C}\) be a differential bigraded category with differential \(\underline{\partial}\) and let \(\C\) the associated category of complexes over \(\underline{\C}\). Let \((X,\delta)\in\C\) and \(a\in\C(X,X)\) with \(\gr(a)=h^0q^0\) such that \(a^2=0\) and \(\underline{\partial}(a)a=0\). Then \((X,\delta)\) is isomorphic to \((X,\delta')\), where 
	\(
	\delta'
	=
	\delta+\partial_{\C}(a)-a\delta a
	=
	\delta+\delta a- a\delta -a\delta a+\underline{\partial}(a)
	=
	(1-a)\delta(1+a)+\underline{\partial}(a).
	\)\qed
\end{lemma}

\subsection{Bimodules and mapping classes}

Let \(M_+(\TwistB)\) and \(M'_+(\TwistB)\)
be the type AD structures defined in 
\cref{fig:Twisting:typeAD:equivariant,fig:Twisting:typeAD:standard}. 
Let \(M_-(\TwistB)\) and \(M'_-(\TwistB)\) 
be the type AD structures obtained from 
\(M_+(\TwistB)\) and \(M'_+(\TwistB)\), respectively,
by shifting the gradings of all three generators by \(q^{-3}h^{-1}\).
Recall the \(\CoeffRing\)-linear isomorphism
\(\rho\co\B\rightarrow\B\)
that switches \(\bullet\) and \(\circ\) 
in the indices of all elements of the standard basis of \(\B\). 
This induces a functor on type D structures over \(\B\), 
or equivalently, a type AD structure that we denote by \(\typeAD{\B}{R}{\B}\). 
We define 
\[
M_\pm(\TwistW)
=
\typeAD{\B}{R}{\B}
\boxtimes 
M_\pm(\TwistB)
\boxtimes 
\typeAD{\B}{R}{\B}
\quad\text{and}\quad
M'_\pm(\TwistW)
=
\typeAD{\B}{R}{\B}
\boxtimes 
M'_\pm(\TwistB)
\boxtimes 
\typeAD{\B}{R}{\B}.
\]

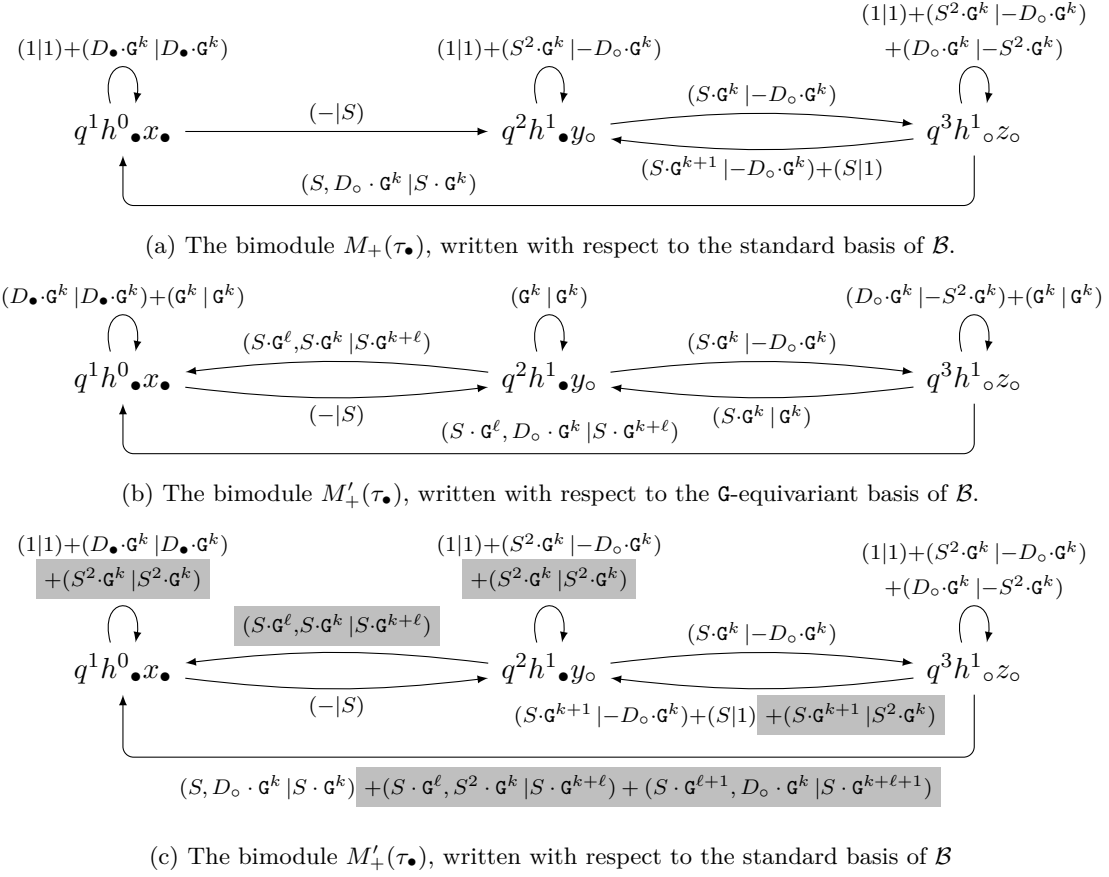
\begin{figure}
	\centering
	\begin{subfigure}{\textwidth}\centering
	\[ 
	\begin{tikzcd}[column sep=4cm]
		q^1h^0
		{}_\bullet x_\bullet
		\arrow{r}{(-| S)}
		\arrow[
		loop above,
		distance=1.75em, 
		start anchor={[xshift=-1ex]north}, 
		end anchor={[xshift=1ex]north}]
		{}{(1|1)+(\Db \cdot \varG^k | \Db \cdot \varG^k)}
		&
		q^2h^1
		{}_\bullet y_\circ
		\arrow[
		loop above,
		distance=1.75em, 
		start anchor={[xshift=-1ex]north}, 
		end anchor={[xshift=1ex]north}]
		{}{(1|1)+(S^{2} \cdot\varG^k | -\Dw\cdot \varG^k)}
		\arrow[bend left=7]{r}[pos=0.5]{(S \cdot \varG^k|-\Dw\cdot \varG^k)}
		&
		q^3h^1
		{}_\circ z_\circ
		\arrow[%
		rounded corners,
		to path={%
			|- ([yshift=-4ex,xshift=3.5cm]\tikztotarget.south) 
			node[above,font=\scriptsize] (x){\((S,\Dw \cdot \varG^k| S\cdot\varG^{k})\)}
			-| (\tikztotarget)}
		]{ll}
		\arrow[%
		loop above,
		distance=1.75em, 
		start anchor={[xshift=-1ex]north}, 
		end anchor={[xshift=1ex]north}]
		{}[align=center]{%
			$(1|1)+(S^2 \cdot \varG^k | -\Dw \cdot \varG^k)$
			\\
			$+(\Dw\cdot \varG^k|-S^2\cdot \varG^k)$
		}
		\arrow[bend left=7]{l}{(S \cdot \varG^{k+1}| -\Dw\cdot\varG^k)+(S|1)}
	\end{tikzcd}
	\]
	\caption{%
		The bimodule \(M_+(\TwistB)\), written with respect to the 
		standard basis of \(\B\). 
	}
	\label{fig:Twisting:typeAD:standard}
	\end{subfigure}%
	
	\begin{subfigure}{\textwidth}\centering
	\centering
	\[
	\begin{tikzcd}[column sep=4cm]
		q^1h^0
		{}_\bullet x_\bullet
		\arrow[bend right=7]{r}[swap]{(-| S)}
		\arrow[
		loop above,
		distance=1.75em, 
		start anchor={[xshift=-1ex]north}, 
		end anchor={[xshift=1ex]north}]
		{}{(\Db \cdot \varG^k | \Db \cdot \varG^k)+(\varG^k | \varG^k)}
		&
		q^2h^1
		{}_\bullet y_\circ
		\arrow[
		loop above,
		distance=1.75em, 
		start anchor={[xshift=-1ex]north}, 
		end anchor={[xshift=1ex]north}]
		{}{(\varG^k | \varG^k)}
		\arrow[bend right=7]{l}[swap]{(S\cdot\varG^\ell,S \cdot \varG^k| S\cdot \varG^{k+\ell})}
		\arrow[bend left=7]{r}[pos=0.5]{(S \cdot \varG^k|-\Dw\cdot \varG^k)}
		&
		q^3h^1
		{}_\circ z_\circ
		\arrow[%
		rounded corners,
		to path={%
			|- ([yshift=-4ex,xshift=5.75cm]\tikztotarget.south) 
			node[above,font=\scriptsize] (x){%
				$(S\cdot \varG^\ell,D_\circ \cdot \varG^k| S\cdot\varG^{k+\ell})$%
			}
			-| (\tikztotarget)}
		]{ll}
		\arrow[%
		loop above,
		distance=1.75em, 
		start anchor={[xshift=-1ex]north}, 
		end anchor={[xshift=1ex]north}]
		{}{(D_\circ\cdot \varG^k|-S^2\cdot \varG^k)+(\varG^k | \varG^k)}
		\arrow[bend left=7]{l}{(S \cdot \varG^k| \varG^k)}
	\end{tikzcd}
	\]
	\caption{%
		The bimodule \(M'_+(\TwistB)\), written with respect to the 
		\(\varG\)-equivariant basis of \(\B\). 
	}
	\label{fig:Twisting:typeAD:equivariant}
\end{subfigure}%

\begin{subfigure}{\textwidth}\centering
	\centering
	\[
	\begin{tikzcd}[column sep=4cm]
		q^1h^0
		{}_\bullet x_\bullet
		\arrow[bend right=7]{r}[swap]{(-| S)}
		\arrow[
		loop above,
		distance=1.75em, 
		start anchor={[xshift=-1ex]north}, 
		end anchor={[xshift=1ex]north}]
		{}[align=center]{%
			$(1|1)+(\Db \cdot \varG^k | \Db \cdot \varG^k)$
			\\ 
			\DWo{$+(S^2\cdot\varG^k | S^2\cdot \varG^k)$}}
		&
		q^2h^1
		{}_\bullet y_\circ
		\arrow[
		loop above,
		distance=1.75em, 
		start anchor={[xshift=-1ex]north}, 
		end anchor={[xshift=1ex]north}]
		{}[align=center]{
			$(1|1)+(S^2\cdot \varG^k | -\Dw\cdot\varG^k)$
			\\
			\DWo{$+(S^2\cdot \varG^k | S^2\cdot\varG^k)$}}
		\arrow[bend right=7]{l}[swap]{%
			\DWo{$(S\cdot\varG^\ell,S \cdot \varG^k| S\cdot \varG^{k+\ell})$}
		}
		\arrow[bend left=7]{r}[pos=0.5]{(S \cdot \varG^k|-\Dw\cdot \varG^k)}
		&
		q^3h^1
		{}_\circ z_\circ
		\arrow[%
		rounded corners,
		to path={%
			|- ([yshift=-5ex,xshift=5.75cm]\tikztotarget.south) 
			node[below,font=\scriptsize] (x){
				$(S,\Dw \cdot \varG^k| S\cdot\varG^{k})$%
				\DWo{$+(S\cdot \varG^{\ell},S^2 \cdot \varG^k| S\cdot\varG^{k+\ell})+(S\cdot \varG^{\ell+1},\Dw \cdot \varG^k| S\cdot\varG^{k+\ell+1})$}%
			}
			-| (\tikztotarget)}
		]{ll}
		\arrow[%
		loop above,
		distance=1.75em, 
		start anchor={[xshift=-1ex]north}, 
		end anchor={[xshift=1ex]north}]
		{}[align=center]{%
			$(1|1)+(S^2 \cdot \varG^k | -\Dw \cdot \varG^k)$
			\\
			$+(\Dw\cdot \varG^k|-S^2\cdot \varG^k)$
		}
		\arrow[bend left=7]{l}[align=center,pos=0.6]{%
			$(S \cdot \varG^{k+1}| -\Dw\cdot\varG^k)+(S|1)$%
			\DWo{$+(S \cdot \varG^{k+1}| S^2\cdot\varG^k)$}
		}
	\end{tikzcd}
	\]
	\caption{%
		The bimodule \(M'_+(\TwistB)\),
		written with respect to the 
		standard basis of \(\B\)
	}
	\label{fig:Twisting:typeAD-SB-raw}
\end{subfigure}
\caption{Definition of the bimodules \(M_+(\TwistB)\) and \(M'_+(\TwistB)\). Labels containing the indices \(k\) or \(\ell\) should be understood as the infinite sum of such labels where \( k,\ell\geq0 \). 	The highlighted labels in \cref{fig:Twisting:typeAD-SB-raw} show the difference between \(M'_+(\TwistB)\) and \(M_+(\TwistB)\).}
\label{fig:Twisting}
\end{figure}

\begin{lemma}\label{lem:iso-twisting-bimodules}
	 The type AD structures \(M'_\pm(\TwistL)\) and \(M_\pm(\TwistL)\) are bigraded isomorphic for \(\halfbullet\in\{\bullet,\circ\}\).
\end{lemma}

\begin{proof}%
	It suffices to show the statement for \(\halfbullet=\bullet\) and \(\pm=+\). 
	We first write \(M'_+(\TwistB)\) in terms of the standard basis. 
	This is shown in \cref{fig:Twisting:typeAD-SB-raw}. 
	We now apply \cref{lem:clean-up:improved} 
	to the endomorphism \(a\) given by
	\[ 
	\begin{tikzcd}[column sep=4cm]
		q^1h^0
		{}_\bullet x_\bullet
		&
		q^2h^1
		{}_\bullet y_\circ
		\arrow{l}[swap]{(S^2\cdot\varG^k|S\cdot\varG^k)}
		&
		q^3h^1
		{}_\circ z_\circ
		\arrow[bend left=5,end anchor={[yshift=-5pt]east}]
		{ll}[pos=0.1,swap]{(S\cdot\varG^{k+1}|S\cdot\varG^k)}
	\end{tikzcd}
	\]
	Making use of \cref{rem:typeAD:graphical-notation}, 
	a routine computation
	shows that the sum of 
	\(\partial_{\typeAD{\B}{\C}{\B}}(a)-a\delta a\) 
	with the endomorphism represented by the highlighted labels 
	in \cref{fig:Twisting:typeAD-SB-raw} 
	is indeed 0. 
\end{proof}

\begin{proposition}\label{prop:AddingASingleCrossing}
	For any pointed oriented Conway tangle \(T\) and \(\halfbullet\in\{\bullet,\circ\}\) 
	\[
	\DD(\TwistL(T))^{\B}
	\simeq
	\DD(T)^{\B}\boxtimes \BB{M'_\pm(\TwistL)}
	\cong
	\DD(T)^{\B}\boxtimes \BB{M_\pm(\TwistL)}
	\]  
	where \(\pm\) is the sign of the added crossing in \(\TwistL(T)\). 
\end{proposition}	

\begin{proof}
	Since \(\DD(T)\) is bounded by \cref{rem:bounded}, the type~D structure \(\DD(T)\boxtimes Y\) is well-defined for any type~AD structure \(Y\).
	We show the homotopy equivalence
	\(
	\DD(\TwistL(T))^{\B}
	\simeq
	\DD(T)^{\B}\boxtimes M'_\pm(\TwistL)
	\)
	for the case \(\halfbullet=\bullet\) first. 
	Making use of Bar-Natan's gluing formalism \cite[Section~5]{BarNatanKhT},
	consider the following planar arc diagram:
	\[
		\Diag=\Diag(T_1,T_2)
		=
		\tanglesum
	\]
	Then \(\TwistB(T)=\Diag(T,Q_{-1})\) 
	and the orientation of \(T\) induces an orientation on 
	\(\Diag\) and \(Q_{-1}\). 
	By \cite[Theorem~2]{BarNatanKhT}
	\[
	\KhTs{\TwistB(T)}
	\simeq
	\Diag(\KhTs{T},\KhT{Q_{-1}})
	\]
	where the right-hand side is interpreted as a tensor product of \(\KhTs{T}\) and \(\KhT{Q_{-1}}\). 
	We first compute this tensor product in the case that \(Q_{-1}\) is a positive crossing. 
	Then 
	\[
	\KhT{Q_{-1}}
	=
	\left[
	\begin{tikzcd}
	q^1h^0
	\No
	\arrow{r}{S}
	&
	q^2h^1
	\Ni
	\end{tikzcd}
	\right]
	\]
	and it suffices to understand the action of \(\Diag(-,\KhT{Q_{-1}})\) on 
	\(\End_{\Cobb}(\Li\oplus\Lo)\). 
	The diagram
	\(\Diag(\Li,\Li)\) contains a single circle, which we deloop using \cref{lem:delooping}.
	Thus, \(\Diag(-,\KhT{Q_{-1}})\) acts as follows:
	\begin{align*}
		q^ih^j
		\Li
		&\mapsto 
		q^{i}h^{j}
		\left[
		\begin{tikzcd}[ampersand replacement =\&]
			q^3h^1
			\Li
			\&
			q^1h^0
			\Li
			\arrow{r}{1}
			\arrow[swap]{l}{S^2}
			\&
			q^1h^1
			\Li
		\end{tikzcd}\right]
		\\
		q^ih^j
		\Lo
		&\mapsto 
		q^{i}h^{j}
		\left[
		\begin{tikzcd}[ampersand replacement =\&]
			q^1h^0
			\Lo
			\arrow{r}{S}
			\&
			q^2h^1
			\Li
		\end{tikzcd}\right]
	\end{align*} 
	Here, we identify the morphisms 
	in \(\End_{\Cobb}(\Li\oplus\Lo)\) and \(\B\)
	according to \cref{thm:OmegaFullyFaithful}. 
	Note that the signs induced by the grading shifts (\cref{def:shifted-complexes}) on these type D structures agree with the Koszul sign rule (\cref{def:koszul-rule}) for the tensor product \(\Diag(-,\KhT{Q_{-1}})\).

	\begin{figure}[t]
		\centering
		\newlength{\braceflex}
		\setlength{\braceflex}{-4pt}
		\newlength{\midflex}
		\setlength{\midflex}{-4pt}
		\begin{align*}
			%
			\left(\hspace*{\braceflex}
			\begin{tikzcd}[ampersand replacement = \&]
				\Li
				\arrow{d}{D}
				\\
				\Li
			\end{tikzcd}\hspace*{\braceflex}\right)
			\mapsto 
			&
			\left(\hspace*{\braceflex}
			\begin{tikzcd}[ampersand replacement = \&,column sep=20pt]
				\Li
				\arrow[out=-90,in=135,pos=0.8]{drr}{1}
				\&
				\Li
				\arrow{r}{1}
				\arrow[swap]{l}{S^2}
				\arrow[swap,pos=0.1]{d}{D}
				\&
				\Li
				\arrow{d}{- \varG}
				\\
				\Li
				\&
				\Li
				\arrow{r}{1}
				\arrow[swap]{l}{S^2}
				\&
				\Li
			\end{tikzcd}\hspace*{\braceflex}\right)\hspace*{\midflex}
			&
			\hspace*{\midflex}\left(\hspace*{\braceflex}
			\begin{tikzcd}[ampersand replacement = \&]
				\Lo
				\arrow{d}{D}
				\\
				\Lo
			\end{tikzcd}\hspace*{\braceflex}\right)
			\mapsto 
			&
			\left(\hspace*{\braceflex}
			\begin{tikzcd}[ampersand replacement = \&,column sep=20pt]
				\Lo
				\arrow[swap]{d}{D}
				\arrow{r}{S}
				\&
				\Li
				\\
				\Lo
				\arrow{r}{S}
				\&
				\Li
			\end{tikzcd}\hspace*{\braceflex}\right)
			\\
			%
			\left(\hspace*{\braceflex}
			\begin{tikzcd}[ampersand replacement = \&]
				\Li
				\arrow{d}{S}
				\\
				\Lo
			\end{tikzcd}\hspace*{\braceflex}\right)
			\mapsto 
			&
			\left(\hspace*{\braceflex}
			\begin{tikzcd}[ampersand replacement = \&,column sep=20pt]
				\Li
				\arrow[out=-90,in=135,pos=0.8]{drr}{1}
				\&
				\Li
				\arrow[pos=0.1,swap]{d}{S}
				\arrow{r}{1}
				\arrow[swap]{l}{S^2}
				\&
				\Li
				\\
				\&
				\Lo
				\arrow{r}{S}
				\&
				\Li
			\end{tikzcd}\hspace*{\braceflex}\right)\hspace*{\midflex}
			&
			\hspace*{\midflex}\left(\hspace*{\braceflex}
			\begin{tikzcd}[ampersand replacement = \&]
				\Lo
				\arrow{d}{S}
				\\
				\Li
			\end{tikzcd}\hspace*{\braceflex}\right)
			\mapsto 
			&
			\left(\hspace*{\braceflex}
			\begin{tikzcd}[ampersand replacement = \&,column sep=20pt]
				\&
				\Lo
				\arrow{r}{S}
				\arrow[pos=0.7]{d}{S}
				\&
				\Li
				\arrow{d}{1}
				\arrow[out=-145,in=90,pos=0.6,swap]{dll}{\varG}
				\\
				\Li
				\&
				\Li
				\arrow{r}{1}
				\arrow[swap]{l}{S^2}
				\&
				\Li		
			\end{tikzcd}\hspace*{\braceflex}\right)
		\end{align*}
		\caption{%
			The action of \(\Diag(-,\KhT{Q_{-1}})\) 
			on the morphisms \(D\) and \(S\) 
			for the proof of \cref{prop:AddingASingleCrossing}
		}
		\label{fig:Twisting:morphisms}
	\end{figure}

	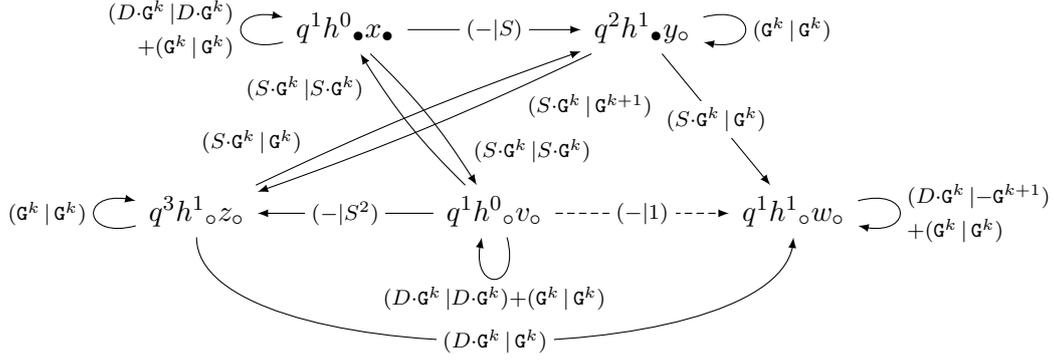
\begin{figure}[t]
		\[ 
		\begin{tikzcd}[column sep=10pt,row sep=50pt,scale=2]
			&
			q^1h^0
			{}_\bullet x_\bullet
			\arrow{rr}[description]{(-|S)}
			\arrow[bend left=7]{dr}[pos=0.9]{(S\cdot \varG^k| S\cdot \varG^k)}
			\arrow[
			loop left,
			distance=1em, 
			start anchor={[yshift=-1ex]west}, 
			end anchor={[yshift=1ex]west}]
			{}[align=right]{%
				$(D \cdot \varG^k | D \cdot \varG^k)$
				\\ 
				$+(\varG^k | \varG^k)$}
			&&
			q^2h^1
			{}_\bullet y_\circ
			\arrow{dr}[description]{(S \cdot \varG^k| \varG^k)}
			\arrow[
			loop right,
			distance=1em, 
			start anchor={[yshift=1ex]east}, 
			end anchor={[yshift=-1ex]east}]
			{}{(\varG^k | \varG^k)}
			\arrow[bend left=3]{dlll}[pos=0.2]{(S \cdot \varG^k| \varG^{k+1})}
			\\
			q^3h^1
			{}_\circ z_\circ
			\arrow[in=-90,out=-90,looseness=0.6]{rrrr}[description]{(D \cdot \varG^k| \varG^k)}
			\arrow[%
			loop left,
			distance=1em, 
			start anchor={[yshift=-1ex]west}, 
			end anchor={[yshift=1ex]west}]
			{}{(\varG^k | \varG^k)}
			\arrow[bend left=3]{urrr}[pos=0.15]{(S \cdot \varG^k| \varG^k)}
			&&
			q^1h^0
			{}_\circ v_\circ
			\arrow[dashed]{rr}[description]{(-| 1)}
			\arrow{ll}[description]{(-| S^2)}
			\arrow[pos=0.7,bend left=7]{ul}[pos=0.9]{(S \cdot \varG^k| S \cdot \varG^k)}
			\arrow[%
			loop below,
			distance=1em, 
			start anchor={[xshift=1ex]south}, 
			end anchor={[xshift=-1ex]south}]
			{}{\substack{(D \cdot \varG^k| D \cdot \varG^k) + (\varG^k | \varG^k)}}
			&&
			q^1h^1
			{}_\circ w_\circ
			\arrow[
			loop right,
			distance=1em, 
			start anchor={[yshift=1ex]east}, 
			end anchor={[yshift=-1ex]east}]
			{}[align=left]{%
				$(D\cdot \varG^k|-\varG^{k+1})$
				\\
				$+(\varG^k | \varG^k)$
			}
		\end{tikzcd}
		\]
		\vspace{-0.5cm}
		\caption{%
			The bimodule \(M''_+(\TwistB)\) from the proof of \cref{prop:AddingASingleCrossing}.
			Labels containing the index \(k\) should be understood as the infinite sum of such labels where \( k\geq0 \).
		}\label{fig:Twisting:typeAD-raw}
	\end{figure}
	
	Similarly, we can compute the action on morphisms; this is done in \cref{fig:Twisting:morphisms}. 
	These actions on objects and morphisms are combined
	in the type AD structure \(M''_+(\TwistB)\) 
	shown in \cref{fig:Twisting:typeAD-raw}, 
	where we have now also identified the objects 
	in \(\End_{\Cobb}(\Li\oplus\Lo)\) and \(\B\). 
	Interpreting 
	\(\KhT{\TwistB(T)}\) and \(\KhT{T}\) 
	as the type~D structures 
	\(\DD(\TwistB(T))\) and \(\DD(T)\), respectively, we obtain
	\[
	\DD(\TwistB(T))^{\B}
	\simeq
	\DD(T)^{\B}
	\boxtimes
	\BB{M''_+(\TwistB)}.
	\]
	Note that, by \cref{prop:boxtensoring-and-homotopies},
	the chain homotopy class of the right hand side only depends on 
	the chain homotopy classes of \(\DD(T)^{\B}\) and \(\BB{M''_+(\TwistB)}\). 
	In particular, we may replace \(M''_+(\TwistB)\) 
	by the homotopy equivalent type AD structure \(M'_+(\TwistB)\);
	the homotopy equivalence is obtained by applying the Cancellation Lemma 
	\cite[Lemma 2.16]{KWZ} 
	to the dashed arrow in \cref{fig:Twisting:typeAD-raw}.
	This concludes the proof for the case that 
	\(Q_{-1}\) is a positive crossing. 
	
	If \(Q_{-1}\) is a negative crossing, then the bigradings of the generators in \(\KhT{Q_{-1}}\) are simply shifted by \(q^{-3}h^{-1}\). 
	The rest of the argument stays the same.
The homotopy equivalence for the case \(\halfbullet=\circ\) can be found similarly. 
	However, there is also the following shortcut:
	\begin{align*}
		\DD(T)^{\B}\boxtimes \BB{M'_\pm(\TwistW)}
		&=
		\DD(T)^{\B}\boxtimes\rho
		\left(
		M'_\pm(\TwistB)
		\right)
		=
		\rho
		\left(
		\rho
		\left(
		\DD(T)^{\B}
		\right)
		\boxtimes M'_\pm(\TwistB)	
		\right)
		\\
		&=
		\rho
		\left(
		\DD(\rho(T))^{\B}\boxtimes M'_\pm(\TwistB)
		\right)
		\simeq
		\rho
		\left(
		\DD((\TwistB\rho)(T))^{\B}
		\right)
		\\
		&=
		\DD((\rho\TwistB\rho)(T))^{\B}
		=
		\DD(\TwistW(T))^{\B}
	\end{align*}
	The third and fifth identity uses \cref{prop:rho*Khr}. 
	The fourth identity uses the assumption that the statement is true for the case \(\halfbullet=\bullet\) and the fact that the signs of the added crossings in \(\TwistW(T)\) and \(\TwistB(\rho(T))\) agree.  
	The last identity uses the fact that 
	\(\rho=\TwistW\TwistB\TwistW\) 
	and 
	\(1=\TwistW\TwistB\TwistW\TwistB\TwistW\TwistB\).
The isomorphism 
	\(
	\DD(T)^{\B}\boxtimes M'_\pm(\TwistL)
	\cong
	\DD(T)^{\B}\boxtimes M_\pm(\TwistL)
	\)
	follows from \cref{lem:iso-twisting-bimodules} combined with \cref{lem:signs-boxtimes-composition-AD,rem:typeD:identity-boxtensor}.
\end{proof}

\begin{table}
	\centering
	\begin{tabular}{ccc}
		\(i\)
		&
		\( X_i \) 
		&
		\( Y_i\)
		\\
		\hline
		(\( \SaddleBC \)a)&
		\( 
		\begin{tikzcd}[column sep=30pt]
			\phantom{\circ}
			\arrow[leftrightarrow,"\times" anchor=center]{r}{}
			&
			\circ
			\arrow{r}{S}
			&
			\bullet
		\end{tikzcd}
		\)
		&
		\(
		\begin{tikzcd}[column sep=30pt]
			\bullet
		\end{tikzcd}
		\)
		\\
		(\( \SaddleBC \)b)&
		\( 
		\begin{tikzcd}[column sep=30pt]
			\bullet
			&
			\circ
			\arrow{r}{D\cdot\varG^k}
			\arrow{l}[swap]{S}
			&
			\circ
			\arrow{r}{S}
			&
			\bullet
		\end{tikzcd}
		\)
		&
		\(
		\begin{tikzcd}[column sep=30pt]
			\bullet
			\arrow{r}{S^2\cdot\varG^k}
			&
			\bullet
		\end{tikzcd}
		\)
		\\
		(\( \SaddleBC \)c)&
		\( 
		\begin{tikzcd}[column sep=30pt]
			\circ
			&
			\circ
			\arrow{r}{S}
			\arrow[swap]{l}{D\cdot\varG^k}
			&
			\bullet
		\end{tikzcd}
		\)
		&
		\( 
		\begin{tikzcd}[column sep=30pt]
			\circ
			&
			\bullet
			\arrow{l}[swap]{\pm S\cdot\varG^{k+1}}
		\end{tikzcd}
		\)
		\\
		(\( \SaddleBC \)d)&
		\( 
		\begin{tikzcd}[column sep=30pt]
			\circ
			\arrow{r}{D\cdot\varG^k}
			&
			\circ
			\arrow{r}{S}
			&
			\bullet
		\end{tikzcd}
		\)
		&
		\(
		\begin{tikzcd}[column sep=30pt]
			\circ
			\arrow{r}{\pm S\cdot\varG^k}
			&
			\bullet
		\end{tikzcd}
		\)
		\\
		(\( \SaddleBC \)e)&
		\( 
		\begin{tikzcd}[column sep=30pt]
			\circ
			\arrow{r}{S\cdot\varG^{k+1}}
			&
			\bullet
		\end{tikzcd}
		\)
		&
		\( 
		\begin{tikzcd}[column sep=30pt]
			\circ
			\arrow{r}{-D\cdot\varG^k}
			&
			\circ
			&
			\bullet
			\arrow[swap]{l}{\pm S}
		\end{tikzcd}
		\)
		\\
		(\( \SaddleCB \))&
		\( 
		\begin{tikzcd}[column sep=30pt]
			\bullet
			\arrow{r}{S\cdot\varG^k}
			&
			\circ
		\end{tikzcd}
		\)
		&
		\( 
		\begin{tikzcd}[column sep=30pt]
			\bullet
			\arrow{r}{\pm S}
			&
			\circ
			\arrow{r}{-D\cdot\varG^k}
			&
			\circ
		\end{tikzcd}
		\)
		\\
		(\( S\SaddleBC \))&
		\( 
		\begin{tikzcd}[column sep=30pt]
			\circ
			\arrow{r}{S^2\cdot\varG^k}
			&
			\circ
		\end{tikzcd}
		\)
		&
		\( 
		\begin{tikzcd}[column sep=30pt]
			\circ
			\arrow{r}{-D\cdot\varG^k}
			&
			\circ
		\end{tikzcd}
		\)
		\\
		(\( S\SaddleCB \))&
		\( 
		\begin{tikzcd}[column sep=30pt]
		\bullet
		\arrow{r}{S^2\cdot\varG^k}
		&
		\bullet
		\end{tikzcd}
		 \)
		&
		\( 
		\begin{tikzcd}[column sep=30pt]
		\bullet
		\arrow{r}{\pm S}
		&
		\circ
		\arrow{r}{-D\cdot\varG^{k}}
		&
		\circ
		&
		\bullet
		\arrow[swap]{l}{\mp S}
		\end{tikzcd}
		 \)
		\\
		(\( \Dw \))&
		\( 
		\begin{tikzcd}[column sep=30pt]
		\circ
		\arrow{r}{D\cdot\varG^k}
		&
		\circ
		\end{tikzcd}
		 \)
		&
		\( 
		\begin{tikzcd}[column sep=30pt]
		\circ
		\arrow{r}{-S^2\cdot\varG^k}
		&
		\circ
		\end{tikzcd}
		 \)
		 \\
		 (\( \Db \))&
		 \( 
		 \begin{tikzcd}[column sep=30pt]
		 	\bullet
		 	\arrow{r}{D\cdot\varG^k}
		 	&
		 	\bullet
		 \end{tikzcd}
		 \)
		 &
		 \( 
		 \begin{tikzcd}[column sep=30pt]
		 	\circ
		 	&
		 	\bullet
		 	\arrow{r}{D\cdot\varG^k}
		 	\arrow[dashed,swap]{l}{\pm S}
		 	&
		 	\bullet
		 	\arrow[dashed]{r}{\mp S}
		 	&
		 	\circ
		 \end{tikzcd}
		 \)
	\end{tabular}
	\caption{Computations for the proof of \cref{prop:twisting:complexes}}
	\label{tab:MCGbimodule}
\end{table}

\begin{definition}
	Given any type D structure \(X\) over \(\B\),
	\(\mathtt{e}\in\EndsThree\),
	and \(\halfbullet\in\{\bullet,\circ\}\)
	define 
	\[
	\Phi_{\mathtt{e}}(\TwistL)(X)
	=
	\begin{cases*}
		X\boxtimes M_+(\TwistL)
		&
		if (\(\halfbullet=\bullet\) and \(\mathtt{e}=\Endi\)) or (\(\halfbullet=\circ\) and \(\mathtt{e}=\Endiii\))
		\\
		X\boxtimes M_-(\TwistL)
		&
		otherwise.
	\end{cases*}
	\]
\end{definition}

\begin{proposition}\label{prop:twisting:complexes}
	Let \(\Gamma\) be a graded multicurve with local system,
	\(\mathtt{e}\in\EndsThree\), 
	and \(\halfbullet\in\{\bullet,\circ\}\).
	Then
	\[
	C(\Phi_{\mathtt{e}}(\TwistL)(\Gamma))
	\simeq
	\Phi_{\mathtt{e}}(\TwistL)(C(\Gamma)).
	\]
\end{proposition}

\begin{proof}
	The proof follows the same line of argument as \cite[Proof of Theorem~8.1]{KWZ}; compare also with \cite{HanselmanWatson}. 
	We first consider the case \(\halfbullet=\bullet\)
	and that \(\Gamma\) consists of a single graded curve \((F,g)\) without a local system. 
	We split \(C(F,g)\) into the basic pieces \(X_i\) shown in the second column of \cref{tab:MCGbimodule}.
	Let \(Y_i\) be the type~D structures shown in the third column of \cref{tab:MCGbimodule}.
	We claim that \(Y_i\simeq X_i\boxtimes M_\pm(\TwistB)\), 
	where the signs are so that 
	\(\pm=(-1)^{|x_\bullet|}=-(\mp)\) where \(x_\bullet\) is 
	the first (and often only) generator in idempotent \(\iota_\bullet\) in \(X_i\). 
	There are four cases, namely \(i=(\SaddleBC\text{a})\) to \(i=(\SaddleBC\text{d})\), 
	in which \(Y_i\neq X_i\boxtimes M_\pm(\TwistB)\).
	In these cases,
	\begin{align*}
		X_{(\SaddleBC\text{a})}\boxtimes M_\pm(\TwistB)
		&=
		\begin{tikzcd}[column sep=30pt,ampersand replacement=\&]
			\circ
			\arrow{r}{1}
			\&
			\circ
			\&
			\bullet
			\arrow[swap]{l}{\pm S}
		\end{tikzcd}
		\\
		X_{(\SaddleBC\text{b})}\boxtimes M_\pm(\TwistB)
		&=
		\begin{tikzcd}[column sep=30pt,ampersand replacement=\&]
			\bullet
			\arrow{r}{\pm S}
			\&
			\circ
			\&
			\circ
			\arrow[swap]{l}{1}
			\arrow{r}{-S^2\cdot\varG^k}
			\arrow[bend right=20]{rrr}{\mp S\cdot\varG^k}
			\&
			\circ
			\arrow{r}{1}
			\&
			\circ
			\&
			\bullet
			\arrow[swap]{l}{\mp S}
		\end{tikzcd}
		\\
		X_{(\SaddleBC\text{c})}\boxtimes M_\pm(\TwistB)
		&=
		\begin{tikzcd}[column sep=30pt,ampersand replacement=\&]
			\circ
			\&
			\circ
			\arrow[swap]{l}{-S^2\cdot\varG^k}
			\arrow{r}{1}
			\&
			\circ
			\&
			\bullet
			\arrow[swap]{l}{\pm S}
		\end{tikzcd}
		\\
		X_{(\SaddleBC\text{d})}\boxtimes M_\pm(\TwistB)
		&=
		\begin{tikzcd}[column sep=30pt,ampersand replacement=\&]
			\circ
			\arrow{r}{-S^2\cdot\varG^k}
			\arrow[bend right=20]{rrr}{\pm S\cdot\varG^k}
			\&
			\circ
			\arrow{r}{1}
			\&
			\circ
			\&
			\bullet
			\arrow[swap]{l}{\pm S}
		\end{tikzcd}
	\end{align*}
	In all cases, however, the cancellation lemma provides the chain homotopy \(Y_i\simeq X_i\boxtimes M_\pm(\TwistB)\).%
	\footnote{%
		Note that the corresponding computation in \cite[Proof of Theorem~8.1]{KWZ} 
		is done with a \(\CoeffFieldTwo\) version of the bimodule \(M'_\pm(\TwistB)\). 
		Since we are using \(M_\pm(\TwistB)\) instead, 
		most of the pieces \(X_i\boxtimes M_\pm(\TwistB)\) are already simplified. 
		Also, our case analysis of arrows labeled by \(\SaddleBC\) 
		is different.
	} 
	Up to signs, \cref{tab:MCGbimodule} agrees with \cref{tab:MCG:TwistB}. 
	Therefore, we can assemble the pieces 
	in the last column of \cref{tab:MCGbimodule} 
	to the complex \(C(\Phi_\mathtt{e}(\TwistB)(F,g))\) 
	along the generators \(\bullet\), 
	ignoring the two dashed arrows in \(X_{(\Db)}\). 
	If the curve is non-compact, 
	all signs can be removed without changing the isomorphism class. 
	If the curve is compact, 
	it remains to check that the total number of signs is even. 
	We claim that in fact the following stronger statement is true:
	Let us cut the graded curve \((F,g)\) into pieces along the generators \(\bullet\). 
	The type~D structure \(X\) corresponding to any one such piece can be built out of copies of the pieces \(X_i\), so \(X\boxtimes M_\pm(\TwistB)\) is homotopy equivalent to a type~D structure \(Y\) built out of copies of the pieces \(Y_i\). 
	Still ignoring the two dashed arrows in \(X_{(\Db)}\), 
	we claim that the number of signs in \(X\boxtimes M_\pm(\TwistB)\) is even. 
	For \(X=X_{(\Db)}\), \(X_{(SS_\bullet)}\), and \(X_\text{(\(\SaddleBC \)b)}\), this is obviously true. 
	So we just need to consider the case 
	when \(X\) is built out of the remaining pieces.
	As the curve is compact, 
	we may ignore the piece \(X_\text{(\(\SaddleBC \)a)}\).
	Thus, \(X\) is of the following form:
	\[
	X=
	\left(
	\begin{tikzcd}[column sep=30pt]
		\bullet
		\arrow[leftrightarrow]{r}{}
		&
		\circ
		\arrow[leftrightarrow]{r}{}
		&
		\cdots
		\arrow[leftrightarrow]{r}{}
		&
		\circ
		\arrow[leftrightarrow]{r}{}
		&
		\bullet
	\end{tikzcd}
	\right)
	\]
	Total number of generators \(\circ\) in \(X\) is always even, because the labels \(S\) and \(D\) alternate. 
	Therefore, the homological gradings of the two generators \(\bullet\) are different modulo 2. 
	This justifies the signs on the first and last arrows in \(Y\):
	\[
	Y=
	\left(
	\begin{tikzcd}[column sep=30pt]
		\bullet
		\arrow[leftrightarrow]{r}{\pm S\cdot\varG^k}
		&
		\circ
		\arrow[leftrightarrow]{r}{}
		&
		\cdots
		\arrow[leftrightarrow]{r}{}
		&
		\circ
		\arrow[leftrightarrow]{r}{\mp S\cdot\varG^\ell}
		&
		\bullet
	\end{tikzcd}
	\right)
	\]
	There is an odd number of arrows \(\circ\leftrightarrow\circ\) and each of them carries a sign. 
	So the total number of signs in \(Y\) is indeed even. 
  
	If the formal curve \(F\) carries a local system \((X,e)\) 
	we may assume, without loss of generality, 
	that the arrow \(e\) either belongs to the case \(\Dw\) 
	if there are no generators \(\bullet\) 
	or \(\Db\) otherwise. 
	The statement for the case \(\halfbullet=\circ\) follows from formal arguments 
	similar to the proof of \cref{prop:AddingASingleCrossing}. 
\end{proof}

Let \(H\C^\B\) denote the set of type D structures over \(\B\) up to chain homotopy. 

\begin{proposition}\label{prop:twisting-complexes}
	There is a left group action of 
	\(\Mod(\FourPuncturedSphere,*)\)
	on
	\(
	H\C^\B\times\EndsThree
	\)
	determined by
	\(\TwistL(X,\mathtt{e})=(\Phi_\mathtt{e}(\TwistL)(X),\TwistL(\mathtt{e}))\) 
	for \(\halfbullet\in\{\bullet,\circ\}\) and 
	\(\mathtt{e}\in\EndsThree\).
\end{proposition}

\begin{proof}
	Define a left group action 
	of the free group generated by 
	two elements
	\(t_\bullet\) and \(t_\circ\) 
	on 
	\(
	H\C^\B\times\EndsThree
	\)
	by
	\(t_{\halfbullet}(X,\mathtt{e})=(\Phi_\mathtt{e}(\TwistL)(X),\TwistL(\mathtt{e}))\) 
	for \(\halfbullet\in\{\bullet,\circ\}\) and 
	\(\mathtt{e}\in\EndsThree\).
	We claim that this group action 
	factors through \(\Mod(\FourPuncturedSphere,*)\), 
	i.e. that 
	\(t_\bullet t_\circ t_\bullet\) and \(t_\circ t_\bullet t_\circ\) 
	act identically
	and that their product acts as the identity. 
	Let us consider the case \(\mathtt{e}=\Endi\) first. 
	Then
	\begin{align*}
		t_\bullet t_\circ t_\bullet(X,\mathtt{e})
		&=
		(\Phi_{\TwistW\TwistB(\mathtt{e})}(\TwistB)
		\Phi_{\TwistB(\mathtt{e})}(\TwistW)
		\Phi_\mathtt{e}(\TwistB)
		(X),
		\TwistB\TwistW\TwistB(\mathtt{e}))
		\\
		&=
		(X
		\boxtimes M_-(\TwistB)
		\boxtimes M_-(\TwistW)
		\boxtimes M_+(\TwistB),
		\rho(\mathtt{e}))
	\end{align*}
	A routine calculation (see \cref{rem:boxtensoring-AD-AD})
	shows that 
	\[
	M_-(\TwistB)\boxtimes M_-(\TwistW)\boxtimes M_+(\TwistB)
	\simeq
	\typeAD{\B}{R}{\B}.
	\]
	Observe that for the other two choices of \(\mathtt{e}\), 
	two opposite signs in the indices of the bimodules change, 
	so that the grading shifts cancel. 
	In summary, 
	\(t_\bullet t_\circ t_\bullet(X,\mathtt{e})=(X\boxtimes\typeAD{\B}{R}{\B},\rho(\mathtt{e}))\)
	for any \(\mathtt{e}\in\EndsThree\).
	Similarly, one can show
	\(t_\circ t_\bullet t_\circ(X,\mathtt{e})=(X\boxtimes\typeAD{\B}{R}{\B},\rho(\mathtt{e}))\). 
	Note that the identity
	\[
	M_-(\TwistW)\boxtimes M_-(\TwistB)\boxtimes M_+(\TwistW)
	\simeq
	\typeAD{\B}{R}{\B}
	\]
	required for this case
	follows from the one above by tensoring with \(\typeAD{\B}{R}{\B}\) on the left and the right. We see that 
	\(t_\bullet t_\circ t_\bullet\) and \(t_\circ t_\bullet t_\circ\) 
	act identically. 
	Since 
	\(\typeAD{\B}{R}{\B}\boxtimes\typeAD{\B}{R}{\B}=1\) 
	and
	\(\rho^2=1\in\Mod(\FourPuncturedSphere,*)\)
	\(t_\bullet t_\circ t_\bullet t_\circ t_\bullet t_\circ\) 
	acts as the identity. 
\end{proof}

\begin{proof}[Proof of \cref{prop:twisting-curves}]
	We follow the same strategy 
	as in the proof of \cref{prop:twisting-complexes}. 
	We define a left group action 
	of the free group generated by 
	two elements
	\(t_\bullet\) and \(t_\circ\) 
	on 
	\(
	\mathcal{C}\times\EndsThree
	\)
	by
	\(t_{\halfbullet}(\Gamma,\mathtt{e})=(\Phi_\mathtt{e}(\TwistL)(\Gamma),\TwistL(\mathtt{e}))\) 
	for \(\halfbullet\in\{\bullet,\circ\}\) and 
	\(\mathtt{e}\in\EndsThree\).
	As in the proof of \cref{prop:twisting-complexes}
	it suffices to show that 
	\(t_\bullet t_\circ t_\bullet\) and \(t_\circ t_\bullet t_\circ\) 
	act on any
	\((\Gamma,\mathtt{e})\in\mathcal{C}\times\EndsThree\)
	as
	\((\rho(\Gamma),\rho(\mathtt{e}))\). 
	The actions on the second component certainly agree,
	and so do the actions on the underlying curves of the first components. 
	We need to check that the local systems and gradings also coincide. 
	Without loss of generality, we may assume that \(\Gamma\) consists of 
	a single graded formal curve with local system \((X,e)\). 
	Then the local systems of \(\rho(\Gamma)\) and the first components 
	of 
	\(t_\bullet t_\circ t_\bullet(\Gamma)\) 
	and 
	\(t_\circ t_\bullet t_\circ(\Gamma)\)
	are all equal to \((X,e')\), where \(e'\) is some arrow 
	that induces the same orientation on the curve as \(e\).
	The gradings of these curves are uniquely determined by the grading of the associated chain complex. 
	By \cref{prop:twisting:complexes},
	\(
	C(t_\bullet t_\circ t_\bullet(\Gamma))
	\simeq 
	\TwistB\TwistW\TwistB(C(\Gamma))
	=
	\rho(C(\Gamma))
	\)
	and similarly
	\(
	C(t_\circ t_\bullet t_\circ(\Gamma))
	\simeq 
	\TwistW\TwistB\TwistW(C(\Gamma))
	=
	\rho(C(\Gamma)).
	\)
	Therefore, the gradings on the three curves agree. 
\end{proof}

\begin{remark}
	Over field coefficients \(\CoeffRing=\CoeffField\), 
	\cref{prop:twisting-curves}
	can be seen as an immediate corollary 
	of \cref{prop:twisting-complexes}
	and \cref{thm:classification:complexes_over_B:simplified}. 
	However, over arbitrary coefficients \(\CoeffRing\), 
	the classification from \cref{thm:classification:complexes_over_B:simplified}
	no longer holds. 
	For the proof of \cref{prop:twisting-curves},
	it would be sufficient to know
	that passing from multicurves \(\Gamma\) 
	to complexes \(C(\Gamma)\) 
	(modulo the respective notions of equivalence)
	is faithful.
	This seems plausible, 
	but the original proof relies on the Frobenius normal form for local systems, 
	so again requires field coefficients \cite[Section 4.7]{pqMod}. 
\end{remark}

\begin{lemma}\label{lem:twists-and-basepoint-action}
	For \(\halfbullet\in\{\bullet,\circ\}\)
	let \(\halfbullet[top]\in\{\bullet,\circ\}\) 
	with \(\halfbullet\neq\halfbullet[top]\) and
	for \(k\geq0\) let 
	\begin{align*}
		b_1
		&=
		\varG^k
		&
		b_2
		&=
		D_{\halfbullet}\cdot\varG^k
		&
		b_3
		&=
		S^2\cdot\varG^k
		\\
		c_1
		&=
		\varG^k
		&
		c_2
		&=
		D_{\halfbullet}\cdot\varG^k
		&
		c_3
		&=
		-D_{\halfbullet[top]}\cdot\varG^k
		\\
		Y_1
		&=
		M'_\pm(\TwistL)
		&
		Y_2
		&=
		M_\pm(\TwistL)
		&
		Y_3
		&=
		M_\pm(\TwistL)
	\end{align*}
	Then 
	\(
	(b_i\cdot 1_X)\boxtimes 1_{Y_i}
	=
	c_i\cdot (1_{X\boxtimes Y_i})
	\) for \(i=1,2,3\).
\end{lemma}

\begin{proof}
	For all \(x\in X\) and all basis elements \(y\) of \(Y_i\)
	\begin{align*}
		((b_i\cdot 1_X)\boxtimes 1_{Y_i})(x\otimes y)
		&=
		(1_X\otimes \delta_1^{Y_i})\circ
		(x\otimes b_i\otimes y)
		\\
		&=
		(x\otimes y \otimes c_i)
		=
		(c_i\cdot(1_{X\boxtimes Y_i}))(x\otimes y)
	\end{align*} 
 	where the first identity follows from the fact that 
 	\(b_i\) does not appear in \(\delta^{Y_i}_j\) for \(j>1\) 
 	and the second identity follows by inspection of the type AD structure \(Y_i\). 
\end{proof}

\begin{proof}[Proof of \cref{thm:twisting:general-coeff}]

	By \cref{lem:MCG:presentation} 
	\(\Mod(\FourPuncturedSphere,*)\) is generated 
	by \(\TwistB\) and \(\TwistW\). 
	In fact, since 
	\(\TwistB\TwistW\TwistB\TwistW\TwistB\TwistW=1\in \Mod(\FourPuncturedSphere)\), 
	any mapping class can be expressed as a word in \(\TwistB\) and \(\TwistW\)
	(without their inverses). 
	It suffices to prove the theorem 
	for \(\tau=\TwistB\) and \(\tau=\TwistW\). 
	The first statement of the theorem follows immediately from 
	\cref{prop:AddingASingleCrossing,prop:twisting:complexes}.
	For the second statement, observe that for all \(k\geq1\) and \(\halfbullet\in\{\bullet,\circ\}\)
	\begin{align*}
		\DD_k(T)^{\B}
		\boxtimes \BB{M'_\pm(\TwistL)}
		&=
		\Big[%
		\begin{tikzcd}[,ampersand replacement=\&]
			h^{-1}q^{-k}\DD(T)
			\arrow{r}{\varG^k\cdot 1_{\DD(T)}}
			\&
			q^k\DD(T)
		\end{tikzcd}
		\Big]%
		\boxtimes M'_\pm(\TwistL)
		\\
		&=
		\Big[%
		\begin{tikzcd}[column sep=2cm,ampersand replacement=\&]
			h^{-1}q^{-k}\DD(T)
			\boxtimes M'_\pm(\TwistL)
			\arrow{r}{\varG^k\cdot 1_{\DD(T)\boxtimes M'_\pm(\TwistL)}}
			\&
			q^k\DD(T)
			\boxtimes M'_\pm(\TwistL)
		\end{tikzcd}
		\Big]%
		\\
		&\simeq
		\Big[%
		\begin{tikzcd}[column sep=2cm,ampersand replacement=\&]
			h^{-1}q^{-k}\DD(\TwistL(T))
			\arrow{r}{\varG^k\cdot 1_{\DD(\TwistL(T))}}
			\&
			q^k\DD(\TwistL(T))
		\end{tikzcd}
		\Big]%
		=
		\DD_k(\TwistL(T))^{\B}
	\end{align*}
	where the second identity follows from  \cref{lem:twists-and-basepoint-action,lem:signs-boxtimes-mapping-cone}
	and the third identity follows from \cref{prop:AddingASingleCrossing,lem:homotopies-of-identity-multiples}.
	Now apply \cref{prop:twisting:complexes}. 
\end{proof}

Recall from the discussion around \cref{prop:twisting-curves}
that every element of \(\Mod(\FourPuncturedSphere,*)\) defines a permutation of the three non-special punctures \(\{\mathtt{e}_1,\mathtt{e}_2,\mathtt{e}_3\}\). 
Specifically, the twist \(\TwistW\) interchanges \(\mathtt{e}_1\) and \(\mathtt{e}_2\), and the twist \(\TwistB\) interchanges \(\mathtt{e}_2\) and \(\mathtt{e}_3\). 
We therefore denote the permutations \((1\,2)\) and \((2\,3)\) by \(\TwistW\) and \(\TwistB\), respectively.
Recalling the labelling of the non-special punctures 
by elements in \(Z(\B)\) 
from \cref{subsec:connectivity}
\[
a_1
=
\Db
\quad
a_2
=
-S^2
\quad
a_3
=
\Dw
\]
we can reinterpret \cref{lem:twists-and-basepoint-action} as follows:

\begin{proposition}\label{prop:twists-and-basepoint-action}
	For \(\halfbullet\in\{\bullet,\circ\}\), \(k\geq0\), and \(j\in\{1,2,3\}\), 
	\[
	((a_j\cdot \varG^k)\cdot 1_X)\boxtimes 1_{M_\pm(\TwistL)}
	\simeq
	(a_{\TwistL(j)}\cdot \varG^k)\cdot 1_{X\boxtimes M_\pm(\TwistL)}.
	\]
\end{proposition}

\begin{proof}
	In view of \cref{lem:twists-and-basepoint-action} for \(i=2\) and \(i=3\), 
	it only remains to show that	
	\[
	((D_{\halfbullet[top]}\cdot \varG^k)\cdot 1_X)\boxtimes 1_{M_\pm(\TwistL)}
	\simeq
	(-S^2\cdot \varG^k)\cdot 1_{X\boxtimes M_\pm(\TwistL)},
	\]
	where \(\halfbullet[top]\in\{\bullet,\circ\}\) 
	with \(\halfbullet\neq\halfbullet[top]\). 
	By \cref{lem:twists-and-basepoint-action} for \(i=1\), we have
	\[
	(\varG^{k}\cdot 1_X)\boxtimes 1_{M'_\pm(\TwistL)}
	=
	\varG^{k}\cdot 1_{X\boxtimes M'_\pm(\TwistL)}
	\]
	for all \(k>0\). 
	By \cref{lem:iso-twisting-bimodules}, \(M'_\pm(\TwistL)\simeq M_\pm(\TwistL)\). 
	In combination with \cref{rem:typeD:identity-boxtensor,lem:signs-boxtimes-composition-AD,prop:boxtensoring-morphisms},
	this implies
	\[
	(\varG^{k}\cdot 1_X)\boxtimes 1_{M_\pm(\TwistL)}
	\simeq
	\varG^{k}\cdot 1_{X\boxtimes M_\pm(\TwistL)}.
	\]
	Combining this with \cref{lem:twists-and-basepoint-action} for \(i=2,3\) we obtain
	\[
	\begin{multlined}[b]
		((D_{\halfbullet[top]}\cdot \varG^k)\cdot 1_X)\boxtimes 1_{M_\pm(\TwistL)}
		=
		((-D_{\halfbullet}\cdot \varG^k-\varG^{k+1}+S^2\cdot \varG^k)\cdot 1_X)\boxtimes 1_{M_\pm(\TwistL)}
		\\
		\simeq
		((-D_{\halfbullet}\cdot \varG^k-\varG^{k+1}-D_{\halfbullet[top]}\cdot \varG^k)\cdot 1_{X\boxtimes M_\pm(\TwistL)}
		=
		((-S^2\cdot \varG^k)\cdot 1_{X\boxtimes M_\pm(\TwistL)}
	\end{multlined}
	\qedhere
	\]
\end{proof}

\section{Connectivity detection}
\label{sec:connectivity}

In this section we prove \cref{prop:connectivity-implies-action-trivial,prop:connectivity-detected-by-rationals-of-odd-length}, which then allow us to show:

\begin{theorem}\label{thm:Khr:connectivity-detection}
	For any Conway tangle \(T\) without closed components, 
	the multicurve \(\Khr(T)\) detects the connectivity~\(\conn{T}\). 
\end{theorem}

	In preparation for the proof of \cref{prop:connectivity-detected-by-rationals-of-odd-length}, we introduce some notation. 
	Let \(\Diag_T\) be a tangle diagram 
	for a pointed Conway tangle \(T\)
	and let \(p\) be a point on a strand 
	of the diagram \(\Diag_T\) away from a crossing.
	Define an endomorphism
	\[
	D_p \cdot 1_{\KhT{\Diag_T}} 
	\in
	\Mor_{\Cobb}\big(\KhT{\Diag_T},\KhT{\Diag_T}\big)
	\]
	as follows: 
	The identity morphism on \(\KhT{\Diag_T}\) 
	can be written as a diagonal matrix of identity cobordisms 
	\(\Diag \times [0,1]\)
	over all resolutions \(\Diag\) of \(\Diag_T\).
	The point \(p\) distinguishes one component 
	of each resolution \(\Diag\) of \(\Diag_T\). 
	We define
	\(D_p \cdot 1_{\KhT{\Diag_T}}\)
	as the endomorphism obtained from \(1_{\KhT{\Diag_T}}\) 
	by placing a dot on the component of \(\Diag \times [0,1]\) 
	containing \(\{p\}\times [0,1]\) 
	for every resolution \(\Diag\) of \(\Diag_T\). 
The following lemma is due to Bar-Natan~\cite{BN_mutation}.

\begin{lemma}[Basepoint Moving Lemma]\label{lem:Basepoint_Moving_Lemma}
	Let \(\Diag_T\) be a diagram for a tangle \(T\) 
	with two basepoints \(p\) and \(p'\) 
	that are separated by a single crossing. 
	Then
	\[
	(D_p+D_{p'}+\varG)\cdot 1_{\KhT{\Diag_T}}
	\simeq
	0.
	\]
\end{lemma}

\begin{proof}
	Since all the maps are equal to the identity 
	away from the crossing that separates 
	\(p\) and \(p'\) and 
	\(\KhT{\Diag_T}\) is natural with respect to gluing 
	\cite[Section~5]{BarNatanKhT}, 
	it suffices to show the Lemma 
	locally at the crossing separating the two basepoints.
	First suppose that
	\(p\) and \(p'\) are separated by an understrand:
	\[
	\begin{tikzpicture}[scale=0.3]\scriptsize
			\draw[thick] (0,0) to (2,2);
			\draw[thick] (2,0) to (1.1,0.9);
			\draw[thick] (0.9,1.1) to (0,2);
			\draw[fill] (0.5,0.5) circle[radius=4pt];
			\draw[fill] (1.5,1.5) circle[radius=4pt];
			\node (p) at (1,2) {\(p\)};
			\node (pp) at (0,1) {\(p'\)};
	\end{tikzpicture}
	\]
	In the following diagram, 
	the morphisms 
	\(D_p\cdot 1_{\KhT{\Diag_T}}\) and 
	\(D_{p'}\cdot 1_{\KhT{\Diag_T}}\) 
	are indicated by the vertical dashed arrows:
	\[
	\begin{tikzcd}[column sep=4cm,row sep=1cm]
		\Ni
		\arrow[r,"\Nil" description]
		\arrow[dashed]{d}[left,solid]{\NiDotL=D_{p'}}
		\arrow[phantom]{d}[right,solid]{\scriptstyle\NiDotR=D_p}
		&
		\No
		\arrow[dotted]{dl}[description,solid]{\Nol}
		\arrow[dashed]{d}[left,solid]{\NoDotB=D_{p'}}
		\arrow[phantom]{d}[right,solid]{\scriptstyle\NoDotT=D_p}
		\\
		\Ni
		\arrow[r,"\Nil" description]
		&
		\No
	\end{tikzcd}
	\]
	Here, \(\Nil\) and \(\Nol\) denote saddle cobordisms and 
	\(\NoDotT\), \(\NiDotR\), \(\NoDotB\), and \(\NiDotL\) dot cobordisms. 
	The dotted arrow is a null-homotopy for 
	\((D_p+D_{p'}+\varG)\cdot 1_{\KhT{\Diag_T}}\),
	which can be checked using the relation 
	\[
	\tube
	=
	\DiscLdot\DiscR+\DiscL\DiscRdot +\varG\cdot \DiscL \DiscR
	\]
	In the case that \(p\) and \(p'\) are separated by an overstrand as in %
	\[
	\begin{tikzpicture}[scale=0.3]\scriptsize
		\draw[thick] (0,0) to (2,2);
		\draw[thick] (2,0) to (1.1,0.9);
		\draw[thick] (0.9,1.1) to (0,2);
		\draw[fill] (1.5,0.5) circle[radius=4pt];
		\draw[fill] (0.5,1.5) circle[radius=4pt];
		\node (p) at (1,2) {\(p\)};
		\node (pp) at (2,1) {~\(p'\)};
	\end{tikzpicture}
	\]
	the homotopy is the same.
\end{proof}

\begin{proof}[Proof of \cref{prop:connectivity-implies-action-trivial}]
	 We claim that it suffices to show the statement for \(\DD=\DD(T)\). 
	Indeed, recalling the notation
	\(
	a_1
	=
	\Db,
	a_2
	=
	-S^2,
	a_3
	=
	\Dw\), suppose that \(\alpha\) is a null-homotopy for \(a_i\cdot 1_{\DD(T)}\). 
	Then 
	\(
	\left(
	\begin{smallmatrix}
		-\alpha & 0
		\\
		0 & \alpha
	\end{smallmatrix}
	\right)
	\)
	is a null-homotopy for 
	\(
	a_i\cdot 1_{\DD_n(T)}
	=
	\left(
	\begin{smallmatrix}
		a_i\cdot 1_{\DD(T)} & 0
		\\
		0 & a_i\cdot 1_{\DD(T)}
	\end{smallmatrix}
	\right)
	\) 
	for any \(n\geq1\). 
	Moreover, given any direct summand \(\DD\) of a chain complex \(X\), 
	a null-homotopy of \(a_i\cdot1_{X}\) restricts to 
	a null-homotopy of \(a_i\cdot1_{\DD}\) 
	by \cref{lem:nullhomotopy-splits-under-direct-sums}.
	
	To show the statement for \(\DD=\DD(T)\) consider the case \(\conn{T}=\ConnY\) first. 
	Given a diagram \(\Diag_T\) for the tangle \(T\), 
	we place the basepoints \(p\) and \(p'\)
	on the top-right and top-left ends, respectively. 
	There is an even number of crossings separating  \(p\) and \(p'\) and	thus \cref{lem:Basepoint_Moving_Lemma} implies 
	\[
	D_p\cdot 1_{\KhT{\Diag_T}} 
	\simeq 
	D_{p'}\cdot 1_{\KhT{\Diag_T}}=0
	\]
	where the second equality is due 
	to the relation \(\planedotstar=0\). 
	After delooping \(\KhT{\Diag_T}\) and 
	recasting it as a type D structure \(\DD(T)^{\B}\), 
	the morphism \(D_p\cdot 1_{\KhT{\Diag_T}}\) 
	becomes the morphism \(\Dw\cdot 1_{\DD(T)}\), 
	which is therefore null-homotopic as well. 
	
	The case  \(\conn{T}=\ConnX\) is analogous to the previous case, 
	except that the basepoint \(p\) 
	is placed on the bottom-left end of \(\Diag_T\). In the case \(\conn{T}=\ConnZ\), we place the basepoint \(p\) 
	on the bottom-right end of \(\Diag_T\). 
	Then there is an odd number of intersections 
	between \(p\) and \(p'\), which 
	 means that 
	\[
	(D_p+\varG)\cdot 1_{\KhT{\Diag_T}} 
	\simeq
	-D_{p'}\cdot 1_{\KhT{\Diag_T}}
	=
	0.
	\]
	The morphism 
	\((D_p+\varG)\cdot 1_{\KhT{\Diag_T}}\)
	corresponds to 
	\((\Db+\Dw+\varG)\cdot 1_{\DD(T)}=S^2\cdot 1_{\DD(T)}\). 
	We now conclude as in the previous cases.
\end{proof}

Recall the following piece of notation from \cref{lem:mutation-on-Khr-special-or-rational}: 
Let 
\(
\DD^{\rKh}_{d}(\slope;X)
\)
denote the chain complex \(C(\gamma)\),
where \(\gamma\) is equal to the rational curve \(\rKh_{d}(\slope)\) 
as a graded curve and carries the local system \((X,e)\) for some fixed edge \(e\) and matrix \(X\in\GL_{m}(\CoeffRing)\).

\begin{lemma}\label{lem:curve_based_detection}
	Let \(d\) be odd. 
	Then, for any \(i\in\{1,2,3\}\), 
	\[
	\conns=\conni
	\iff
	a_i\cdot 1_{\DD^{\rKh}_{d}(\slope;X)} 
	\simeq 
	0.
	\]
\end{lemma}

\begin{proof}%
	The following diagram describes 
	how the two generators \(\TwistB\) and \(\TwistW\) 
	of the mapping class group 
	\(\Mod(\FourPuncturedSphere,*)\)
	act on connectivities (\cref{def:connectivity}):
	\[
	\begin{tikzcd}[column sep=40pt]
		\ConnY
		\arrow[|->,loop left,"\TwistW"]
		\arrow[|->,r,bend left=12,"\TwistB"]
		&
		\ConnZ
		\arrow[|->,l,bend left=12,"\TwistB"]
		\arrow[|->,r,bend left=12,"\TwistW"]
		&
		\ConnX
		\arrow[|->,l,bend left=12,"\TwistW"]
		\arrow[|->,loop right,"\TwistB"]
	\end{tikzcd}
	\]
	This action agrees with the action of \(M_\pm(\TwistL)\) on \(a_i\cdot 1_{\DD}\) observed in \cref{prop:twists-and-basepoint-action}:
	\[
	\begin{tikzcd}[column sep=40pt]
		\Db
		\arrow[|->,loop left,"\TwistW"]
		\arrow[|->,r,bend left=12,"\TwistB"]
		&
		-S^2
		\arrow[|->,l,bend left=12,"\TwistB"]
		\arrow[|->,r,bend left=12,"\TwistW"]
		&
		\Dw
		\arrow[|->,l,bend left=12,"\TwistW"]
		\arrow[|->,loop right,"\TwistB"]
	\end{tikzcd}
	\]
	It suffices to check the statement for \(\slope=0\). 
	We have \(\Dw\cdot 1_{\DD}\simeq 0\), since \(\DD\) does not contain any generators in idempotent \(\iota_\circ\), and \(\conns=\conny\). 
	So for \(i=3\), both statements are true.
	For \(i=1\) and \(i=2\), we need to prove that 
	\(\Db\cdot 1_{\DD} \not\simeq 0\)
	and 
	\(S^2\cdot 1_{\DD}\not\simeq 0\).
	
	First consider the case of a local system 
	\((c)\in\GL_1(\CoeffRing)\) of rank 1. 
	(Note that the element \(c\) is invertible in \(\CoeffRing\).)
	We label the generators in \(\DD\) as follows:
		\[
		\begin{tikzcd}[column sep=20pt]
				\bullet_{d+1}
				\arrow[r, "D" above]
				\arrow[%
				rounded corners,
				to path={%
						|- ([yshift=-2ex,xshift=-7cm]\tikztotarget.south) 
						node[above,font=\scriptsize] (x){$c\cdot S^2$}
						-| (\tikztotarget)}
				]{rrrrrrrr}
				& 
				\bullet_{d+2} 
				\arrow[r, "S^2" above]   
				& 
				\bullet_{d+3}  
				\arrow[r, "D" above]  
				& 
				\cdots  
				\arrow[r, "S^2" above] 
				&  
				\bullet_{2d} 
				\arrow[r, "D" above]  
				&  
				\bullet_{1}  
				&  
				\bullet_{2} 
				\arrow[l, "S^2" above] 
				&  
				\cdots 
				\arrow[l, "D" above] 
				& 
				\bullet_{d} 
				\arrow[l, "D" above]
			\end{tikzcd}
		\]
		We consider the type D structure $Y$ and the homomorphism \(f\in \Mor(Y,\DD)\) defined by
		\[
		Y=
		\left[
		\begin{tikzcd}[column sep=20pt]
				\bullet^{1}  
				&  
				\bullet^{2} 
				\arrow[l, "S^2" above] 
				&  
				\cdots 
				\arrow[l, "D" above] 
				& 
				\bullet^{d} 
				\arrow[l, "D" above] 
			\end{tikzcd}
		\right]
		\qquad
		\text{and} 
		\qquad 
		f(\bullet^i)=\bullet_i \otimes 1.
		\]
		Let \(a\in\{\Db,S^2\}\) and
		suppose for contradiction that 
		\(a\cdot 1_{\DD}\) 
		is null-homotopic. 
		Then so too is the morphism 
		\[
		f_a
		= 
		(\Db\cdot 1_{\DD} \circ f)
		\in
		\Mor(Y,\DD),
		\quad
		f_a (\bullet^i)=\bullet_i \otimes a
		\]
		which is indicated by the solid vertical arrows 
		in the following diagram:
		\[
		\begin{tikzcd}[column sep=20pt, row sep=25pt]
				&&&&
				\bullet^{1}
				\arrow[d, "a" right, near end]
				\arrow[dl, dashed,looseness=0.2, in=45,out=-135]
				&
				\bullet^{2}
				\arrow[d, "a" right, near end]
				\arrow[l, "S^2" above]
				\arrow[%
				dashed,
				rounded corners,
				to path={%
						-- ([yshift=-1ex,xshift=-1ex]\tikztostart.south west) 
						-| (\tikztotarget)}
				]{dlll}
				&
				\cdots
				\arrow[l, "D" above]
				&
				\bullet^{d-1}
				\arrow[d, "a", near end]
				\arrow[l,"S^2",swap]
				\arrow[%
				dashed,
				rounded corners,
				to path={%
						-- ([yshift=-1.5ex,xshift=-1.5ex]\tikztostart.south west) 
						-| (\tikztotarget)}
				]{dllllll}
				&
				\bullet^{d}
				\arrow[d, "a"]
				\arrow[l, "D",swap]
				\arrow[%
				dashed,
				rounded corners,
				to path={%
						-- ([yshift=-2ex,xshift=-2ex]\tikztostart.south west) 
						-| (\tikztotarget)}
				]{dllllllll}
				\\
				\bullet_{d+1} 
				\arrow[r, "D" above]
				\arrow[%
				rounded corners,
				to path={%
						|- ([yshift=-2ex,xshift=-3.25cm]\tikztotarget.south) 
						node[above,font=\scriptsize] (x){$c\cdot S^2$}
						-| (\tikztotarget)}
				]{rrrrrrrr}
				& 
				\bullet_{d+2}
				\arrow[r, "S^2" above]   
				& 
				\cdots  
				\arrow[r, "S^2" above] 
				&  
				\bullet_{2d} 
				\arrow[r, "D" above]  
				& 
				\bullet_{1}
				&
				\bullet_{2}
				\arrow[l, "S^2",swap]
				&
				\cdots
				\arrow[l, "D",swap]
				&
				\bullet_{d-1}
				\arrow[l,"S^2",swap]
				&
				\bullet_{d}
				\arrow[l, "D",swap]
			\end{tikzcd}
		\]
		Now consider the case \(a=\Db\). 
		Any null-homotopy for \(f_a\), or more specifically for the component 
		\[
		\begin{tikzcd}[column sep=25pt]
			\bullet^1
			\arrow[r,"a"]
			& 
			\bullet_1
		\end{tikzcd}
		\]
		of the homomorphism 
		\(f_a\), contains components 
		\[
		\begin{tikzcd}[column sep=25pt]
				\bullet^{i}
				\arrow[dashed,r,"(-1)^{i+1}"]
				& 
				\bullet_{2d-i+1}
			\end{tikzcd}
		\quad
		\text{for \(i=1,\dots,d\) }
		\]
		which can be seen by induction on \(i\). 
		Let \(h_1\) be the sum of these morphisms for \(i=1,\dots,d\). 
		Then \(f_a - d_{\DD}\circ h_1 - h_1\circ d_Y\) 
		contains a component 
		\[
		\begin{tikzcd}[column sep=25pt]
				\bullet^{d}
				\arrow[r,"-c\cdot S^2"]
				& 
				\bullet_{d}
			\end{tikzcd}
		\]
		which cannot be further homotoped away. 
		So $f_a \not\simeq 0$, contradicting our assumption.
		
		The argument for the case \(a=S^2\) is similar. 
		In this case, any null-homotopy for \(f_a\), or more specifically for the component
		\[
		\begin{tikzcd}[column sep=25pt]
			\bullet^{d}
			\arrow[dashed,r,"a"]
			& 
			\bullet_{d}
		\end{tikzcd}
		\]
		of the homomorphism \(f_a\), contains components 
		\[
		\begin{tikzcd}[column sep=65pt]
			\bullet^{i}
			\arrow[dashed,r,"(-1)^{i+1}c^{-1}"]
			& 
			\bullet_{2d-i+1}
		\end{tikzcd}
		\quad
		\text{for \(i=1,\dots,d\).}
		\] 
		These are the same components as above, 
		except that they are labelled by \((-1)^{i+1}c^{-1}\). 
		So let \(h_2=c^{-1}\cdot h_1\). 
		Then \(f_a - d_{\DD}\circ h_2 - h_2\circ d_Y\) 
		contains a component 
		\[
		\begin{tikzcd}[column sep=25pt]
			\bullet^{1}
			\arrow[r,"-c^{-1}\cdot D"]
			& 
			\bullet_{1}
		\end{tikzcd}
		\]
		which cannot be further homotoped away. 
		Hence $f_a \not\simeq 0$, again contradicting our assumption.
		These arguments immediately generalise to higher-dimensional local systems. 
\end{proof}

\begin{proof}[Proof of \cref{prop:connectivity-detected-by-rationals-of-odd-length}]
	By hypothesis, there exists some odd integer \(d\) and some local system \(X\) such that \(\DD=\DD^\mathbf{r}_d(\slope;X)\) is a component of \(C(\Gamma)\), up to some grading shift. 
	Let \(i\in\{1,2,3\}\) such that \(\conn{T}=\conni\).
	Then \(a_i\cdot 1_{\DD}\simeq0\)
	by \cref{prop:connectivity-implies-action-trivial}. 
	\cref{lem:curve_based_detection} implies \(\conni=\conns\). 
\end{proof}

\begin{lemma}\label{lem:odd_number_of_odd_length_rational_curves}
	Let \(T\) be a pointed Conway tangle 
	without any closed component.
	Then the total number of rational components 
	of \(\Khr(T)\) of odd length is odd. 
	In particular, it is non-zero.
\end{lemma}

\begin{proof}%
	By \cref{cor:twisting:field-coeff:ungraded}
	we may assume without loss of generality 
	that \(\conn{T}=\ConnX\). 
	Then \(K=Q_0\cup T\) is a knot 
	since \(T\) does not contain a closed component. 
	By \cref{thm:GlueingTheorem:Kh,prop:mirroring:Khr}
	\[
	\Khr(Q_0\cup T)
	\cong
	\HF(\BNr(Q_0),\Khr(T))
	\cong
	\bigoplus
	\HF(\BNr(Q_0),\gamma)
	\]
	where the direct sum on the right is 
	over all components \(\gamma\) of \(\Khr(T)\). 
	By \cref{thm:geography:Khr}, 
	\(\Khr(T)\) consists of rational and special components only.  
	If the slope of a component \(\gamma\) is 0,  
	a simple calculation shows that
	\(\dim\HF(\BNr(Q_0),\gamma)\) 
	is even 
	(namely 0	if \(\gamma\) is special and 2 if \(\gamma\) is rational).
	If the slope of \(\gamma\) is non-zero, 
	we can apply \cite[Lemma~2.10]{KLMWZ}
	and obtain
	\[
	\dim\HF(\BNr(Q_0),\gamma)
	=
	\ell\cdot |p|
	\]
	where \(\ell\) is the length of \(\gamma\) 
	and \(p\) is the numerator of the slope \(\slope\) of \(\gamma\), 
	with \(p,q\) coprime. 
	
	By \cref{prop:connectivity-detected-by-rationals-of-odd-length}, 
	the slope \(\slope\) of any rational component \(\gamma\) 
	of odd length satisfies 
	\(\conns=\ConnX\). 
	In particular, \(\slope\neq0\) and \(|p|\) is odd. 
	Therefore,
	\(\dim\HF(\BNr(Q_0),\gamma)\)
	is odd for any rational component \(\gamma\) 
	of odd length. 
	Components of odd length are necessarily rational. 
	So for all other components \(\gamma\) of \(\Khr(T)\),
	\(\dim\HF(\BNr(Q_0),\gamma)\)
	is even. 
	We conclude that
	\(\dim\Khr(K)\)
	has the same parity as the number 
	of rational components of \(\Khr(T)\) 
	of odd length.
	Since \(K\) is a knot, i.e.\ a link with one component,
	\(\dim\Khr(K)\) is odd. 
\end{proof}

\begin{proof}[Proof of \cref{thm:Khr:connectivity-detection}]
	Let \(T\) be a Conway tangle without closed components. 
	By \cref{lem:odd_number_of_odd_length_rational_curves}, 
	\(\Khr(T)\) contains an odd (and hence non-zero) number of rational components of odd length. 
	By \cref{prop:connectivity-detected-by-rationals-of-odd-length}, 
	these components detect the connectivity \(\conn{T}\).  
\end{proof}

	The Heegaard Floer invariant \(\HFT\) \cite{pqMod,pqSym} also detects the connectivity of Conway tangles without closed components. 
	The only difference between 
	the connectivity detection results 
	for \(\HFT\) and \(\Khr\) 
	is that all rational components of \(\HFT\) 
	have the same length and 
	that they may carry non-trivial local systems. 
	Note, however, that at the time of writing, 
	no tangle is known whose invariant \(\Khr\) 
	contains a rational component of length \(\geq 3\) or 
	whose invariant \(\HFT\) contains 
	a rational component with non-trivial local system. 
	
	It is also interesting to compare the proofs of the detection results in these two settings. 
	Structurally, they are identical; the proof for \(\HFT\) is based upon two results that are analogous to \cref{lem:odd_number_of_odd_length_rational_curves,prop:connectivity-detected-by-rationals-of-odd-length}. 
	However, while our proof of \cref{lem:odd_number_of_odd_length_rational_curves} follows the same line of argument used for \(\HFT\) \cite[Proposition~3.13]{LMZ}, the fact that odd length rational components detect the connectivity is proved rather differently:
	Our proof of 
	\cref{prop:connectivity-detected-by-rationals-of-odd-length} for \(\Khr\) 
	relies on the basepoint action on Khovanov homology. 
	The analogous result for \(\HFT\)
	follows from a simple observation 
	about the Alexander grading on \(\HFT\) 
	\cite[Observation~6.1]{pqMod}. 
\section{%
	Proof of 
	\texorpdfstring{\cref{thm:geography:Khr:intro}}
		{Theorem \ref{thm:geography:Khr:intro}}
	(%
	Classification of components of
	\texorpdfstring{\(\Khr(T)\)}{Khr(T)}%
	)
}\label{sec:geography:Khr}

Throughout this section we assume that \(\CoeffRing=\CoeffField\) is a field. 

\subsection{Terminology}\label{subsec:geography:Khr:terminology}
The standard Riemannian metric on \(\PuncturedPlane\) 
induces a Riemannian metric on \(\FourPuncturedSphere\) 
via the covering 
\(\PuncturedPlane\rightarrow\FourPuncturedSphere\) 
from \cref{sec:newKh:mcg}.
This allows for a ``normal form''  for curves in \(\FourPuncturedSphere\), 
which is inspired by \cite[Section~7.1]{HRW}.

\begin{definition}\label{def:peg_board_rep}
	Given \(\varepsilon\in(0,\nicefrac{1}{2})\)
	an \(\varepsilon\)-peg-board representative 
	of a compact curve \(\gamma\) in \(\FourPuncturedSphere\) 
	is a representative of the homotopy class of \(\gamma\) 
	that has minimal length among all representatives of distance \(\varepsilon\) 
	to all four punctures in \(\FourPuncturedSphere\). 
	If \(\gamma\) is a non-compact curve in \(\FourPuncturedSphere\)
  its \(\varepsilon\)-peg-board representative is defined in the same way, 
	except that we consider those representatives \(\gamma\co [0,1]\rightarrow S^2\smallsetminus\{\ast\}\)
	for which the preimage
	of the \(\varepsilon\)-neighbourhoods 
	around the punctures is of the form 
	\([0,\varepsilon)\cup(\varepsilon,1]\). 
\end{definition}

The intuition behind this definition is to think 
of the four punctures of \(\FourPuncturedSphere\) 
as pegs of radii~\(\varepsilon\) and 
then to imagine pulling the curve \(\gamma\) ``tight'', 
like a rubber band (the ends of non-compact curves need to be tied 
to the pegs before pulling ``tight''). 
Lifts of \(\varepsilon\)-peg-board representatives 
to the covering space \(\PuncturedPlane\) 
are made up of linear curve segments 
in the complement of the \(\varepsilon\)-neighbourhoods 
around the punctures
and arc segments of radius \(\varepsilon\)
around the punctures. 

Up to reparametrization, 
the \(\varepsilon\)-peg-board representative 
of a curve \(\gamma\) in \(\FourPuncturedSphere\) 
is unique unless \(\gamma\) is compact and 
lifts to a straight line in \(\PuncturedPlane\) of some rational slope. 
However, such curves are not graded, 
which can be seen by inspection for slope 0 and 
then generalizing to arbitrary slopes 
using \cref{cor:twisting:field-coeff:ungraded}. 
So we note:

\begin{lemma}
	Graded curves in \(\FourPuncturedSphere\) 
	have a unique \(\varepsilon\)-peg-board representative.
	\qed
\end{lemma}

\begin{definition}\label{def:RationalVsSpecialKht}
	A graded curve \(\gamma\) in \(\FourPuncturedSphere\) 
	wraps around a puncture if there exists some angle \(\alpha>0\) and  $\delta>0$ 
	such that, for all \(\varepsilon\in(0,\delta)\), 
	a lift of the \(\varepsilon\)-pegboard representative of \(\gamma\)
	to the covering space \(\PuncturedPlane\)
	changes its direction at a lift of this puncture by an angle \(\geq\alpha\). 
	A graded curve \(\gamma\) is called linear 
	if it does not wrap around any puncture.
	We call a linear curve \(\gamma\) almost special 
	if its \(\varepsilon\)-peg-board representative 
	contains points of distance \(\varepsilon\) 
	from the special puncture $\ast$ 
	for all \(\varepsilon\in(0,\nicefrac{1}{2})\), 
	and almost rational otherwise.
\end{definition}

\begin{example}
	The invariant \(\BNr(P_{2,-3})\) 
	shown in \cref{fig:Kh:example} 
	is not linear, 
	since it wraps around the non-special puncture \(\Endiii\).
	Special curves are examples of almost special curves and, similarly, rational curves are examples of almost rational curves. 
\end{example}

The most challenging step in the proof of \Cref{thm:geography:Khr} (\cref{thm:geography:Khr:intro}) is the following result.

\begin{theorem}\label{thm:no_wrapping_around_special}
	Let \(T\) be a pointed Conway tangle. 
	No component of \(\BNr(T)\) wraps around the special puncture.
	The same is true for \(\Khr(T)\) in place of \(\BNr(T)\).
\end{theorem}

The proof of \cref{thm:no_wrapping_around_special} is defered
to \cref{sec:no_wrapping_around_special}. This is the result that requires 
the appeal to the homological mirror symmetry of the three-punctured sphere 
as alluded to in the introduction.

\subsection{Components are linear}\label{subsec:geography:Khr:linearity}

\begin{theorem}\label{thm:Khr:linearity}
	\(\Khr(T)\) consists of only linear components for any pointed Conway tangle \(T\).
\end{theorem}

\begin{lemma}\label{lem:H_is_nullhomotopic_on_Khr}
	Let \(\DD_1\simeq\Cone{\varG\cdot 1_{\DD}}\)
	for some \(\DD\in\C^\B\). 
	Then
	\(\varG\cdot 1_{\DD'}\simeq0\) 
	for any direct summand \(\DD'\) of \(\DD_1\). 
\end{lemma}

\begin{proof}
	By
	\cref{lem:homotopies-of-identity-multiples,lem:nullhomotopy-splits-under-direct-sums},
	it suffices to see that 
	\(\varG\cdot 1_{\Cone{\varG\cdot 1_{\DD}}}\), given by the following vertical solid arrows, is null-homotopic: 
	\[
	\begin{tikzcd}[column sep=1.5cm]
		\Big[
		h^{-1}q^{-1}
		\DD
		\arrow{r}{\varG\cdot 1_{\DD}}
		\arrow{d}[swap]{\varG\cdot 1_{\DD}}
		&
		h^{0}q^{1}
		\DD
		\Big]
		\arrow{d}{\varG\cdot 1_{\DD}}
		\arrow[dashed, ld, "1_{\DD}",swap]
		\\
		\Big[
		h^{-1}q^{-1}
		\DD
		\arrow{r}{\varG\cdot 1_{\DD}}
		&
		h^{0}q^{1}
		\DD
		\Big]
	\end{tikzcd}
	\]
	The diagonal dashed arrow indicates such a null-homotopy.
\end{proof}

\begin{lemma}\label{lem:higher_powers_when_wrapping}
	Let \(\DD\in\C^\B\) and let \(\DD^c_1\simeq\Cone{\varG\cdot 1_{\DD}}\) correspond to a multicurve.
	Then the differential of \(\DD^c_1\) only contains linear combinations of \(D,S,S^2\in\B\). 
\end{lemma}

\begin{proof}
	Let \(\DD^c\simeq \DD\) correspond to a multicurve. 
	Then
	\[
	\DD^c_1
	\simeq 
	\Cone{\varG\cdot 1_{\DD}}
	\simeq
	\Cone{\varG\cdot 1_{\DD^c}}
	\]
	since the action by 
	the central element \(\varG\) of \(\B\) 
	commutes with morphisms in \(\C^\B\),
	as described in \cref{sub:typeD}. 
	So, without loss of generality, 
	we may assume that \(\DD\) itself corresponds 
	to a multicurve and that,
	moreover, 
	this multicurve consists 
	of a single component \(\gamma\). 
	
	Consider \(\DD_1=\Cone{\varG\cdot 1_{\DD}}\). 
	We first assume that the local system on \(\gamma\) 
	is 1-dimensional and equal to \((1)\in\GL_1(\field)\).
	If the differential of \(\DD\) contains any component 
	\(a\cdot \varG^k\) for \(a\in\{D,S,S^2\}\) and \(k\geq1\), 
	then the corresponding portions of the mapping cone 
	look as follows (without the dotted arrows)
	\[
	\begin{tikzcd}
		h^{-1}q^{-1}\DD
		\arrow{d}[swap]{\varG\cdot 1_{\DD}}
		&
		\cdots
		\arrow[leftrightarrow,dashed]{r}{}
		&
		\halfbullet[left]
		\arrow{r}{-a\cdot \varG^k}
		\arrow{d}[swap]{\varG}
		&
		\halfbullet[right]
		\arrow[leftrightarrow,dashed]{r}{}
		\arrow{d}{\varG}
		&
		\cdots
		\\
		h^{0}q^{1}\DD
		&
		\cdots
		\arrow[leftrightarrow,dashed]{r}{}
		&
		\halfbullet[left]
		\arrow{r}{a\cdot \varG^k}
		\arrow[dotted]{ru}[description]{-a\cdot\varG^{k-1}}
		&
		\halfbullet[right]
		\arrow[leftrightarrow,dashed]{r}{}
		&
		\cdots
	\end{tikzcd}
	\] 	where \(\halfbullet[left],\halfbullet[right]\in\{\bullet,\circ\}\).
	(The two pairs of dashed arrows 
	on the left and the right
	need not be there; 
	if a pair of arrows exists, 
	the arrows are labelled by elements 
	that compose with \(a\) to 0.) 
	By applying the Clean-Up Lemma (\cref{lem:clean-up:improved})
	to the dotted arrows, 
	we see that \(\DD_1\) is isomorphic 
	to the type D structure obtained 
	by a local change of the above to the following: 
	\[
	\begin{tikzcd}
		h^{-1}q^{-1}\DD
		\arrow{d}[swap]{\varG\cdot 1_{\DD}}
		&
		\cdots
		\arrow[leftrightarrow,dashed]{r}{}
		&
		\halfbullet[left]
		\arrow{d}{\varG}
		&
		\halfbullet[right]
		\arrow[leftrightarrow,dashed]{r}{}
		\arrow{d}[swap]{\varG}
		&
		\cdots
		\\
		h^{0}q^{1}\DD
		&
		\cdots
		\arrow[leftrightarrow,dashed]{r}{}
		&
		\halfbullet[left]
		&
		\halfbullet[right]
		\arrow[leftrightarrow,dashed]{r}{}
		&
		\cdots
	\end{tikzcd}
	\]
	Denote by \(\DD'_1\) the type~D structure 
	obtained from \(\DD_1\) by applying the above procedure 
	to all such components of the differential of \(\DD\).
	The differential of \(\DD'_1\) 
	has the special property that all its components
	are linear combinations of \(D,S,S^2\in\B\). 
	It remains to see that this property is preserved 
	under the algorithm that turns \(\DD'_1\) into \(\DD_1^c\). 
	The arrow-pushing algorithm from~\cite{HRW} 
	that is applied for this purpose 
	in~\cite[Section~5]{KWZ} 
	only modifies the curve in the neighbourhoods of the arcs 
	corresponding to the two objects \(\bullet\) and \(\circ\), 
	and thus preserves the property. 
	Therefore, it suffices to see that the simply-faced precurve 
	associated with \(\DD'_1\)
	also has this special property. 
	
	As a brief reminder, a precurve is 
	the auxiliary object introduced in~\cite[Section~5.3]{KWZ} 
	to intermediate between type D structures and multicurves.
	Essentially, precurves are type D structures
	with two particular choices of bases \(B_D\) and \(B_S\).
	A precurve is called fully cancelled if 
	the differential does not contain any identity components. 
	In this case, the differential can be written as a sum \(d_D+d_S\), 
	where \(d_D\) and \(d_S\) consist of those components 
	labelled by linear combinations of non-trivial powers 
	of \(D\) and \(S\), respectively. 
	A precurve is simply-faced if it is fully cancelled and
	any basis element of \(B_D\) is connected to 
	at most one other basis element via the differential \(d_D\), 
	and the same for \(B_S\) and \(d_S\). 
	(Readers familiar with knot Floer homology may recognize 
	\(B_D\) and \(B_S\) as the analogues 
	of vertically and horizontally simplified bases.)
	
	The algorithm in the proof of \cite[Proposition 5.10]{KWZ} 
	unfortunately does not, in general, preserve the property 
	that the components of the differential
	are linear combinations of \(D,S,S^2\in\B\). 
	For example, here are the same type~D structures written with respect to two different bases \(B_S\):
	\[
	\Bigg[
	\begin{tikzcd}
		\bullet
		\arrow{r}{S^2}
		&
		\bullet
		\arrow[leftarrow]{r}{S}
		&
		\circ
		\arrow{r}{S^2}
		&
		\circ
	\end{tikzcd}
	\Bigg]
	\cong
	\Bigg[
	\begin{tikzcd}
		\bullet
		\arrow[bend right=10,swap]{rrr}{S^3}
		&
		\bullet
		\arrow[leftarrow]{r}{S}
		&
		\circ
		&
		\circ
	\end{tikzcd}
	\Bigg]
	\] 
	However, for the type D structure \(\DD'_1\), 
	one can easily construct a corresponding simply-faced precurve by hand 
	and verify that it has the desired property: 
	First, observe that the type D structure \(\DD'_1\) can be built out of the following three pieces (without the dotted arrows)
	\[
	\begin{gathered}
		\begin{tikzcd}[row sep=35pt,column sep=45pt]
			\cdots
			\arrow[leftrightarrow,dashed]{r}{-S^a}
			&
			\halfbullet[left]
			\arrow{r}{-D}
			\arrow[bend left=15]{d}{-D}
			\arrow[bend right=15]{d}[swap]{S^2}
			&
			\halfbullet[left]
			\arrow[leftrightarrow,dashed]{r}{-S^b}
			\arrow[bend right=15]{d}[swap]{-D}
			\arrow[bend left=15]{d}{S^2}
			&
			\cdots
			\\
			\cdots
			\arrow[leftrightarrow,dashed]{r}{S^a}
			&
			\halfbullet[left]
			\arrow{r}{D}
			\arrow[dotted,leftarrow]{ru}[description]{-1}
			&
			\halfbullet[left]
			\arrow[leftrightarrow,dashed]{r}{S^b}
			&
			\cdots
		\end{tikzcd}
		\quad
		\begin{tikzcd}[row sep=35pt,column sep=45pt]
			\cdots
			\arrow[leftrightarrow,dashed]{r}{-D}
			&
			\halfbullet[left]
			\arrow{r}{-S^2}
			\arrow[bend left=15]{d}{S^2}
			\arrow[bend right=15]{d}[swap]{-D}
			&
			\halfbullet[left]
			\arrow[leftrightarrow,dashed]{r}{-D}
			\arrow[bend right=15]{d}[swap]{S^2}
			\arrow[bend left=15]{d}{-D}
			&
			\cdots
			\\
			\cdots
			\arrow[leftrightarrow,dashed]{r}{D}
			&
			\halfbullet[left]
			\arrow{r}{S^2}
			\arrow[dotted,leftarrow]{ru}[description]{1}
			&
			\halfbullet[left]
			\arrow[leftrightarrow,dashed]{r}{D}
			&
			\cdots
		\end{tikzcd}
		\\
		\begin{tikzcd}[row sep=35pt,column sep=45pt]
			\cdots
			\arrow[leftrightarrow,dashed]{r}{-D}
			&
			\halfbullet[left]
			\arrow{r}{-S}
			\arrow[bend left=15]{d}{S^2}
			\arrow[bend right=15]{d}[swap]{-D}
			&
			\halfbullet[right]
			\arrow[leftrightarrow,dashed]{r}{-D}
			\arrow[bend right=15]{d}[swap]{S^2}
			\arrow[bend left=15]{d}{-D}
			&
			\cdots
			\\
			\cdots
			\arrow[leftrightarrow,dashed]{r}{D}
			&
			\halfbullet[left]
			\arrow{r}{S}
			\arrow[leftarrow,dotted]{ru}[description]{S}
			&
			\halfbullet[right]
			\arrow[leftrightarrow,dashed]{r}{D}
			&
			\cdots
		\end{tikzcd}
	\end{gathered}
	\] 
	where \(\halfbullet[left],\halfbullet[right]\in\{\bullet,\circ\}\) 
	and \(a,b\in\{1,2\}\). 
	(As above, the two pairs of dashed arrows 
	on the left and the right
	need not be present.) 
	In the case of the third piece, 
	we can apply the Clean-Up Lemma to the dotted arrow 
	to remove the vertical components \(S^2\) of the differential. 
	In the case of the first and second piece, 
	we perform a base change along the dotted arrow, 
	thereby simultaneously adding a crossover arrow to the precurve and 
	removing the vertical arrows \(D\) and \(S^2\), respectively. 
	The result is a simply-faced precurve, 
	and its differential only consists of components 
	that are linear combinations of \(D,S,S^2\in\B\). 
	
	It remains to adapt the above argument to the case 
	that \(\gamma\) carries a local system.
	We represent the local system by a matrix 
	\(X\in\GL_n(\field)\), where \(n\geq1\). 
	If the differential of \(\DD\) contains any component 
	\(a\cdot \varG^k\cdot 1_{\field^n}\) for \(a\in\{D,S,S^2\}\) and \(k\geq1\), 
	we place \(X\) on such a component. 
	In this case \(\DD'_1\) is equal to the
	direct sum of \(n\) identical type~D structures and 
	we can argue as before. 
	If \(\DD_1=\DD'_1\) it can be built from the following four pieces:
	The first three pieces are obtained from the three pieces above by 
	tensoring them by \(\field^n\).  
	The fourth piece is identical to one of the previous three pieces, 
	except that \(1_{\field^n}\) on the horizontal arrows is replaced by \(X\). 
	In this case, we can apply the Clean-Up Lemma and 
	base changes along the same dotted arrows,  
	after tensoring their labels by \(1_{\field^n}\) or \(X\).
\end{proof}

\begin{figure}[t]
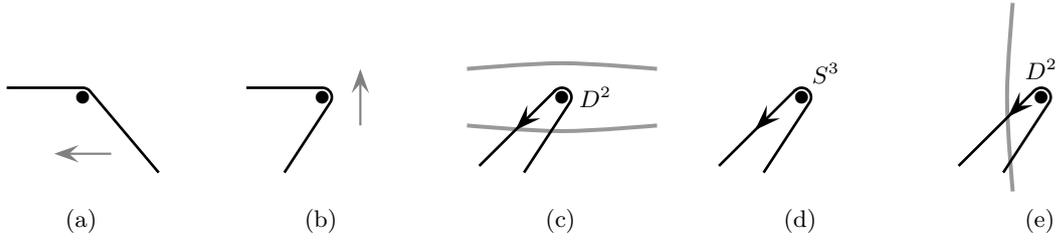

	\centering
	\begin{subfigure}{0.19\textwidth}
		\centering
		\(\WindingNonSpecialA\)
		\caption{}\label{fig:almost_no_wrapping:a}
	\end{subfigure}
	\begin{subfigure}{0.19\textwidth}
		\centering
		\(\WindingNonSpecialB\)
		\caption{}\label{fig:almost_no_wrapping:b}
	\end{subfigure}
	\begin{subfigure}{0.19\textwidth}
		\centering
		\(\WindingNonSpecialC\)
		\caption{}\label{fig:almost_no_wrapping:c}
	\end{subfigure}
	\begin{subfigure}{0.19\textwidth}
		\centering
		\(\WindingNonSpecialD\)
		\caption{}\label{fig:almost_no_wrapping:d}
	\end{subfigure}
	\begin{subfigure}{0.19\textwidth}
		\centering
		\(\WindingNonSpecialE\)
		\caption{}\label{fig:almost_no_wrapping:e}
	\end{subfigure}
	\caption{%
		An illustration for the proof of 
		\cref{prop:almost_no_wrapping}. 
		The arrows in (a) and (b) indicate the directions 
		of the two shearing transformations to get 
		from (a) to (b), and 
		from (b) to (c), (d), or (e).
	}\label{fig:almost_no_wrapping}
\end{figure}

\begin{proposition}\label{prop:almost_no_wrapping}
	A curve that wraps around a non-special puncture 
	cannot be a component of \(\Khr(T)\) 
	for any pointed Conway tangle \(T\). 
\end{proposition}

\begin{proof}
	Suppose a curve changes its direction at a non-special puncture. 
	Then, by \cref{cor:twisting:field-coeff:ungraded}, 
	we can assume that for sufficiently small \(\varepsilon>0\), 
	the \(\varepsilon\)-peg board representative 
	has a linear curve segment of slope 0 
	adjacent to the non-special puncture, 
	as shown in \cref{fig:almost_no_wrapping:a}. 
	Another application of \cref{cor:twisting:field-coeff:ungraded} shows that, 
	by adding twists corresponding to
	\[
	\begin{bmatrix}
		1 & n \\
		0 & 1
	\end{bmatrix}\in\PSL_2(\Z)
	\]
	for sufficiently large \(n\), 
	the curve can be made to change its direction 
	by more than $90^\circ$, as shown in 
	\cref{fig:almost_no_wrapping:b}. 
	By adding one additional twist corresponding to
	\[
	\begin{bmatrix}
		1 & 0 \\
		1 & 1
	\end{bmatrix}\in\PSL_2(\Z)
	\] 
	we can make the first curve segment have slope 1 and 
	the angle of the change of direction be more than $135^\circ$.
	\cref{fig:almost_no_wrapping:c,fig:almost_no_wrapping:d,fig:almost_no_wrapping:e} 
	show the three cases depending on 
	which non-special puncture the curve wraps around. 
	In all three cases,
	the differential of the corresponding type~D structure
	either contains a component \(D^n\) for \(n>1\) 
	or a component \(S^m\) for \(m>2\). 
	This contradicts \cref{lem:higher_powers_when_wrapping}.
\end{proof}

\begin{proof}[Proof of \cref{thm:Khr:linearity}]
	By \cref{thm:no_wrapping_around_special,prop:almost_no_wrapping}, 
	no component of \(\Khr(T)\) can wrap around any puncture.
	Therefore, every component of \(\Khr(T)\) is linear.
\end{proof}

\subsection{%
	The proof of 
	\texorpdfstring{\cref{thm:geography:Khr:intro}}
	{Theorem \ref{thm:geography:Khr:intro}}
}

\Cref{thm:geography:Khr} (\cref{thm:geography:Khr:intro}) 
is immediate from \cref{thm:Khr:linearity} in combination with the following two statements. 

\begin{theorem}\label{thm:geography:special_curves}
	Every almost special component of \(\Khr(T)\) is special.
\end{theorem}

The proof of \cref{thm:geography:special_curves} is given at the end of \cref{sec:no_wrapping_around_special:extension_prop}. 

\begin{theorem}\label{thm:geography:rational_curves}
  Every almost rational component of \(\Khr(T)\) is rational.
\end{theorem}

\begin{proof}%
By \cref{cor:twisting:field-coeff:ungraded} we may assume, without loss of generality, 
that the slope of the almost rational component \(\gamma\) is 0. Therefore, the type~D structure \(C(\gamma)\) only contains generators 
in the idempotent \(\Ib\) and each component of the differential is a linear combination of \(\Db\) and \(S^2_\bullet\). 
It suffices to show that \(C(\gamma)\cong q^rh^sC(\rKh_n(0))\) for some $n\geq 1$ and \(r,s\in\Z\).

	Assume first that the local system is 1-dimensional so the type D structure $C(\gamma)$ may be represented as a graph with two-valent vertices $\bullet$ and edges whose labels alternate between non-zero multiples of $D$ and $S^2$. 
	In the following, we will use super- and subscripts to distinguish vertices of this graph. 
	Every edge 
	\(\!\!
	\begin{tikzcd}[column sep=12pt]
	\bullet_1 
	\arrow[r]
	& 
	\bullet_2
	\end{tikzcd}
	\!\!\)
	implies the quantum grading relation $q(\bullet_2)=q(\bullet_1)+2$. 
	Consequently, there has to be a pair of vertices of the form 
	\[
	\begin{tikzcd}[column sep=12pt]
	\phantom{}
	& 
	\bullet
	\arrow[l]
	\arrow[r]
	&
	\phantom{}
	\end{tikzcd}
	\quad
	\text{and}
	\quad
	\begin{tikzcd}[column sep=12pt]
	\phantom{}
	& 
	\bullet
	\arrow[leftarrow,l]
	\arrow[leftarrow,r]
	&
	\phantom{}
	\end{tikzcd}
	\] 
	Choose a shortest sequence of identically oriented consecutive arrows that connect two such vertices:
	\[
	\begin{tikzcd}[column sep=12pt, row sep =1.3cm]
	\bullet_0 
	\arrow[r]
	& 
	\bullet_1   
	& 
	\bullet_2  
	\arrow[l]
	& 
	\cdots  
	\arrow[l] 
	&  
	\bullet_{n} 
	\arrow[l]  
	&  
	\bullet_{n+1} 
	\arrow[l] 
	\arrow[r] 
	&  
	\bullet_{n+2}
	\end{tikzcd}
	\]
	First assume that the arrow 
	\begin{tikzcd}[column sep=12pt]
	\bullet_1 
	& 
	\bullet_2
	\arrow[l]
	\end{tikzcd}
	is labelled by \(D\).   
	By the minimality hypothesis, the generator $\bullet_{n+1}$ is followed by (at least) $n$ arrows pointing to the right: 
	\[
	\begin{tikzcd}[column sep=12pt, row sep =1.3cm]
	\bullet_0 
	\arrow[r]
	& 
	\bullet_1   
	& 
	\bullet_2  
	\arrow{l}[swap]{D}
	& 
	\cdots  
	\arrow[l] 
	&  
	\bullet_{n} 
	\arrow[l]  
	&  
	\bullet_{n+1} 
	\arrow[l] 
	\arrow[r] 
	&  
	\bullet_{n+2} 
	\arrow[r] 
	&  
	\cdots 
	\arrow[r] 
	& 
	\bullet_{2n}
	\arrow{r}{}
	& 
	\bullet_{2n+1}
	\end{tikzcd}
	\]
	Up to isomorphism, we may assume that all arrows except the last are labelled by \(D\) or \(S^2\).
	Let the last arrow be labelled by \(c\cdot S^2\) for some non-zero element \(c\in\field\). 
	We claim that $\bullet_{2n}=\bullet_{0}$ and $\bullet_{2n+1}=\bullet_{1}$. 
	To show this, let $Y$ be the type D structure defined by the full subgraph consisting of the vertices \(\bullet_{i}\) for \(i=1,\dots,n\), which we relabel \(\bullet^{i}\):
	\[
	Y=
	\left[
	\begin{tikzcd}[column sep=12pt, row sep =1.3cm]
	\bullet^1   
	& 
	\bullet^2  
	\arrow{l}[swap]{D}
	& 
	\bullet^3 
	\arrow[l]   
	& 
	\cdots  
	\arrow[l] 
	&  
	\bullet^{n} 
	\arrow[l]
	\end{tikzcd}
	\right]
	\]
	There is an obvious inclusion map 
	$f\co Y\hookrightarrow C(\gamma)$ 
	given by 
	\(f(\bullet^i)=\bullet_i \otimes 1\). 
	By \cref{lem:H_is_nullhomotopic_on_Khr},
	\(\varG\cdot 1_{C(\gamma)}\) is null-homotopic, 
	and hence so too is the morphism 
	$$
	f_{\varG}
	= 
	(\varG\cdot 1_{C(\gamma)} \circ f)
	\in
	\Mor(Y,C(\gamma)),
	\quad
	f_{\varG} (\bullet^i)
	=
	\bullet_i \otimes \varG
	$$
	Any null-homotopy for \(f_{\varG}\) necessarily contains components 
	\(\!\!
	\begin{tikzcd}[column sep=25pt]
	\bullet^i 
	\arrow[dashed]{r}{(-1)^{i}}
	& 
	\bullet_{i+1}
	\end{tikzcd}
	\!\!\)
	for \(i=1,\dots, n-1\):
	\[
	\begin{tikzcd}[column sep=20pt, row sep=25pt]
	&
	\bullet^1 
	\arrow[d,"\varG"]
	\arrow[rd, dashed]
	&
	\bullet^2 
	\arrow[d,"\varG"]
	\arrow[rd, dashed]  
	\arrow{l}[swap]{D}
	&
	\cdots  
	\arrow[l]
	\arrow[d,"\varG"]
	\arrow[rd, dashed ]
	&
	\bullet^{n} 
	\arrow[l] 
	\arrow[d,"\varG"]   
	\\
	\bullet_0 
	\arrow[r]
	&
	\bullet_1   
	&
	\bullet_2  
	\arrow{l}[swap]{D} 
	&
	\cdots  
	\arrow[l] 
	&
	\bullet_{n} 
	\arrow[l]  
	&
	\bullet_{n+1} 
	\arrow[l] 
	\arrow[r] 
	&
	\bullet_{n+2} 
	\arrow[r] 
	&
	\cdots 
	\arrow[r] 
	&
	\bullet_{2n}
	\arrow{r}{c\cdot S^2} 
	&
	\bullet_{2n+1}
	\end{tikzcd}
	\]
	If \(h_1\) denotes the morphism given by all dashed arrows above, the morphism $f_{\varG}-d_{C(\gamma)}\circ h_1 - h_1 \circ d_Y$ consists of the two solid vertical arrows below
	\[
	\begin{tikzcd}[column sep=20pt, row sep=25pt]
	&
	\bullet^1 
	\arrow[d,"S^2" swap] 
	\arrow[%
	dashed,
	rounded corners,
	to path={%
		-- ([yshift=-2ex,xshift=2ex]\tikztostart.south east) 
		-| (\tikztotarget)}
	]{drrrrrrr}
	&
	\bullet^2  
	\arrow{l}[swap]{D}
	\arrow[%
	dashed,
	rounded corners,
	to path={%
		-- ([yshift=-1.5ex,xshift=1.5ex]\tikztostart.south east) 
		-| (\tikztotarget)}
	]{drrrrr}
	&
	\cdots  
	\arrow[l] 
	\arrow[%
	dashed,
	rounded corners,
	to path={%
		-- ([yshift=-1ex,xshift=1ex]\tikztostart.south east) 
		-| (\tikztotarget)}
	]{drrr}
	&
	\bullet^{n} 
	\arrow[l] 
	\arrow[d,"a" left,pos=0.7] 
	\arrow[dr, dashed,looseness=0.2, in=135,out=-45] 
	\\
	\bullet_0 
	\arrow[r]
	&
	\bullet_1   
	&
	\bullet_2  
	\arrow{l}[swap]{D}  
	&
	\cdots  
	\arrow[l] 
	&
	\bullet_{n} 
	\arrow[l]  
	&
	\bullet_{n+1} 
	\arrow[l] 
	\arrow[r] 
	&
	\bullet_{n+2} 
	\arrow[r] 
	&
	\cdots 
	\arrow[r] 
	&
	\bullet_{2n} 
	\arrow{r}{c\cdot S^2} 
	&
	\bullet_{2n+1}
	\end{tikzcd}
	\]
	where \(a=-D\) if \(n\) is odd and \(a=S^2\) if \(n\) is even. 
	In order to eliminate the component
	\(\!\! 
	\begin{tikzcd}[column sep=25pt]
	\bullet^n
	\arrow[r,"a"]
	& 
	\bullet_{n}
	\end{tikzcd}
	\!\!\),
	the null-homotopy also needs to include the components 
	\(\!\!
	\begin{tikzcd}[column sep=25pt]
	\bullet^i 
	\arrow[dashed]{r}{(-1)^{i}}
	& 
	\bullet_{2n+1-i}
	\end{tikzcd}
	\!\!\)
	for \(i=1,\dots, n\). 
	Let $h_2$ denote the sum of all these components with \(h_1\). 
	The morphism 
	$f_{\varG}-d_{C(\gamma)}\circ h_2 - h_2 \circ d_Y$ 
	consists of two arrows below:
	$$
	\begin{tikzcd}[column sep=20pt, row sep=25pt]
	&
	\bullet^1 
	\arrow[d,"S^2" left] 
	\arrow[%
	rounded corners,
	to path={%
		-- ([yshift=-10pt,xshift=10pt]\tikztostart.south east) 
		-- ([yshift=-10pt,xshift=7cm]\tikztostart.south east) 
		node [above](x) {\scriptsize\(c\cdot S^2\)}
		-| (\tikztotarget)}
	]{drrrrrrrr}
	&
	\bullet^2    
	\arrow{l}[swap]{D}
	&
	\cdots  
	\arrow[l]  
	&
	\bullet^{n} 
	\arrow[l]   
	\\
	\bullet_0 
	\arrow[r]
	&
	\bullet_1   
	&
	\bullet_2  
	\arrow{l}[swap]{D}
	&
	\cdots  
	\arrow[l] 
	&
	\bullet_{n} 
	\arrow[l]  
	&
	\bullet_{n+1} 
	\arrow[l] 
	\arrow[r] 
	&
	\bullet_{n+2} 
	\arrow[r] 
	&
	\cdots 
	\arrow[r] 
	&
	\bullet_{2n} 
	\arrow{r}{c\cdot S^2} 
	&
	\bullet_{2n+1}
	\end{tikzcd}
	$$ 
	The arrow 
	\(\!\!
	\begin{tikzcd}[column sep=25pt]
	\bullet^1 
	\arrow{r}{c\cdot S^2}
	& 
	\bullet_{2n+1}
	\end{tikzcd}
	\!\!\)
	cannot be homotoped away, since the label to the right of $\bullet_{2n+1}$ is labelled by $D$. 
	Thus the sum of the arrows
	\(\!\!
	\begin{tikzcd}[column sep=25pt]
	\bullet^1 
	\arrow{r}{c\cdot S^2}
	& 
	\bullet_{2n+1}
	\end{tikzcd}
	\!\!\)
 	and
	\(\!\!
	\begin{tikzcd}[column sep=25pt]
	\bullet^1 
	\arrow{r}{S^2}
	& 
	\bullet_{1}
	\end{tikzcd}
	\!\!\)
	must be zero. 
	In other words, \(c=-1\), \(h_2\) is a null-homotopy for \(f_{\varG}\) and \(\bullet_0=\bullet_{2n}\) and \(\bullet_1=\bullet_{2n+1}\). 
	Hence $C(\gamma)=C(\rKh_{n})$. 
	
	For indecomposable local systems of dimension \(m>1\), the above argument can be adapted as follows. 
	We replace each vertex \(\bullet_i\) of the graph representing \(C(\gamma)\) by \(\field^m\otimes\bullet_i\) and each arrow labelled \(a\in\{D,S^2\}\) by $1_{\field^m}\otimes\,a$, except for the arrow 
	\(\!\!
	\begin{tikzcd}[column sep=15pt]
	\bullet_{2n}
	\arrow{r}{}
	& 
	\bullet_{2n+1}
	\end{tikzcd}
	\!\!\)
	which we replace by
	\[
	\begin{tikzcd}[column sep=30pt]
	\field^m \otimes \bullet_{2n}
	\arrow{r}{X\otimes\,a}
	& 
	\field^m \otimes \bullet_{2n+1}
	\end{tikzcd}
	\]
	where \(X\in \GL_m(\field)\) corresponds to the local system of the curve. 
	With these changes, the above proof goes through: 
	In the end, the sum of the two arrows 
	\[
	\begin{tikzcd}[column sep=40pt]
	\field^m \otimes \bullet^1
	\arrow{r}{1_{\field^m}\otimes S^2}
	& 
	\field^m \otimes \bullet_{1}
	\end{tikzcd}
	\quad
	\text{and}
	\quad
	\begin{tikzcd}[column sep=40pt]
	\field^m \otimes \bullet^1
	\arrow{r}{X\otimes S^2}
	& 
	\field^m \otimes \bullet_{2n+1}
	\end{tikzcd}
	\]
	needs to be zero, so $X=-1_{\field^m}$.
	This contradicts the assumption that the local system was indecomposable. 
	
	If \(m\geq 1\) and the arrow 
	\begin{tikzcd}[column sep=12pt]
		\bullet_1 
		& 
		\bullet_2
		\arrow[l]
	\end{tikzcd}
	is labelled by \(S^2\), the above argument can be adapted as follows: 
	We replace \(h_1\) by \(-h_1\) and \(h_2\) by \(-h_2\). 
	Then the morphism 
	$f_{\varG}-d_{C(\gamma)}\circ h_2 - h_2 \circ d_Y$ 
	consist of two arrows 
	\(\!\!
	\begin{tikzcd}[column sep=25pt]
		\field^m \otimes\bullet^1 
		\arrow{r}{-X\otimes D}
		& 
		\field^m \otimes\bullet_{2n+1}
	\end{tikzcd}
	\!\!\)
	and
	\(\!\!
	\begin{tikzcd}[column sep=25pt]
		\field^m \otimes\bullet^1 
		\arrow{r}{-1_\field^m \otimes D}
		& 
		\field^m \otimes\bullet_{1}
	\end{tikzcd}
	\!\!\)
	whose sum must be zero. 
	So \(m=1\) and \(X=(-1)\in\GL_1(\field)\). 
\end{proof}

\section{Matrix factorizations for tangles and their delooping}
\label{sec:mfs}

Matrix factorizations were first introduced 
in the context of link homology theories 
in the works of Khovanov and Rozansky~\cite{KR_mf_I,KR_mf_II}.
Rasmussen and Ballinger used this framework
to reveal additional structure on Khovanov homology \cite{Some_diff,Ballinger}.
In our story, matrix factorizations serve a similar purpose, 
and much of what follows is inspired 
by the works of these authors. 
However, there are also important differences. 
For example, Ballinger uses the Frobenius algebra 
$\mathcal{F}_3=\CoeffRing[\varG,X]/(X^2-t)$ in the sense of \cite{Kh_frob}, 
whereas our setup corresponds to the Frobenius algebra 
$\mathcal{F}_7=\CoeffRing[\varG,X]/(X^2+\varG\cdot X)$.
We compare our setup to Rasmussen's works 
\cite{Some_diff,JakeSInvariant} 
in \cref{rem:mfs-well-defined,rmk:regions}. 

\subsection{Preliminaries on matrix factorizations}
\label{subsec:conventionsMF}

Following \cite[Section 2.1]{KR_convolutions}, 
we define the category of matrix factorizations as 
the category of \(\ZmodTwo\)-graded curved chain complexes. 
More explicitly:

\begin{definition}\label{def:mf}
	Let \(R\) be a commutative algebra over \(\CoeffRing\) and let \(w\in R\). 
	A matrix factorization over \(R\) with potential \(w\) is a pair \((M, d)\), 
	where \(M\) is a free \(R\)-module equipped with a \(\ZmodTwo\)-grading 
	and \(d\) is a degree one \(R\)-module endomorphism of \(M\) 
	satisfying \(d^2=w\cdot 1_{M}\).
\end{definition}

\begin{notation}
	In this and the following sections \(\hslash\) denotes a \(\ZmodTwo\)-grading.
\end{notation}

Note that the module $M$ in the above definition need not be finitely generated. 
Given two matrix factorizations \((M,d)\) and \((M',d')\) over \(R\) with the same potential, 
the space of morphisms from \((M,d)\) to \((M',d')\) 
is the space of \(R\)-module morphisms from \(M\) to \(M'\). 
We equip this morphism space with a differential 
defined by \(\partial(f)=d' \circ f - (-1)^{\hslash(f)} f\circ d\). 
This promotes the category of matrix factorizations 
(over $R$ and with potential $w$) to a dg category, 
which we denote by \(\MF_w(R)\). 

\begin{remark}
	We have the usual notions of homotopy equivalence for matrix factorizations: 
	Two maps are called \emph{homotopic} $f\simeq g$ if $f-g=\partial(h)$. 
	Two matrix factorizations are called \emph{homotopy equivalent} $M\simeq M'$ 
	if there exist homomorphisms $f\co M\to M'$ and $g\co M'\to M$ 
	such that $f\circ g\simeq 1_{M'}$ and $g\circ f\simeq 1_{M}$. 

\end{remark}

\begin{definition}\label{def:bifiltration}
	Suppose \(R\) carries a bigrading with \(\gr(w)=q^{-6}h^0\). 
	A bifiltration of a matrix factorization \((C,d)\) over \(R\) 
	is a bigrading on \(C\) such that 
	\(d=\sum_{i=0}^\infty d_i\), where \(\gr(d_i)=\hslash^1q^{3(i-1)}h^i\).
	A bifiltered matrix factorization is a matrix factorization together with a bifiltration.
	A morphism \(f\) between two bifiltered matrix factorizations is a morphism of the underlying matrix factorizations.
\end{definition}

\begin{definition}\label{def:shifted-mfs}
	For \((M,d)\in\MF_w(R)\)
	we define
	\[
	\hslash^ih^j(M,d)
	=
	(\hslash^ih^j M, (-1)^i d)
	\in\MF_w(R).
	\] 
	Given a homomorphism $f\co (M,d)\rightarrow (M',d')$ 
	of matrix factorizations,
	we define
	\[ 
	\Cone{f}
	=
	\left( 
	\hslash^1h^{-1} M\oplus M', 
	\left(\begin{smallmatrix}-d&0\\f&d'\end{smallmatrix}\right)
	\right)
	\in\MF_w(R).
	\]
\end{definition}

The following %
is inspired by \cite[Definition 2.6 and Proposition 2.7]{Ballinger}.

\begin{definition}
	A special deformation retract of a matrix factorization \((M,d)\) over \(R\) is a matrix factorization \((M',d')\) together with three maps
	\[
	\begin{tikzcd}
		(M,d)
		\arrow[loop left,dashed]{lr}{\varepsilon}
		\arrow[bend left=10]{r}{\rho}
		&
		(M',d')
		\arrow[bend left=10]{l}{\iota}
	\end{tikzcd}
	\]
	of gradings 
	\(\hslash(\rho)=0=\hslash(\iota)\) 
	and 
	\(\hslash(\varepsilon)=1\) 
	satisfying 
	\(\partial(\iota)=0=\partial(\rho)\), 
	\(\rho \iota=1_{M'}\), 
	\(\iota\rho= 1_{M}+\partial(\varepsilon)\), and 
	\(\varepsilon^2=0\). 
	If we further assume that 
	\((M,d)\) is bifiltered, we will ask that 
	\((M',d')\)
	is also bifiltered, that
	\(\gr(\varepsilon)=q^3\hslash^1h^0\) 
	and that 
	\(\rho=\sum_{i=0}^\infty\rho_i\) and \(\iota=\sum_{i=0}^\infty\iota_i\) 
	with
	\(\gr(\rho_i)=\gr(\iota_i)=q^{3i}\hslash^0h^i\). 
\end{definition}

We can regard special deformation retracts 
as ``local'' homotopy equivalences 
that extend to ``global'' homotopy equivalences, 
via the following result. 

\begin{lemma}\label{lem:sdr-new-from-old}
	Let \((M,d)\in \MF_w(R)\) and 
	\(
	M\cong N\oplus M'
	\)
	a direct sum of bigraded \(R\)-modules
	such that 
	\((M',d')\in   \MF_w(R)\), 
	where \(d'\) is the restriction of \(d\) to \(M'\).
	Express the other components of the differential \(d\) as follows:
	\[
	(M,d)
	=
	\left(
	\begin{tikzcd}
		N
		\arrow[bend left=10]{r}{c}
		\arrow[loop left]{lr}{a}
		&
		M'
		\arrow[bend left=10]{l}{b}
		\arrow[loop right]{lr}{d'}
	\end{tikzcd}
	\right)
	\]
	Suppose further that there is a special deformation retraction from \(M'\) to \(M''\)
	\[
	\begin{tikzcd}
		(M',d')
		\arrow[loop left,dashed]{lr}{\varepsilon}
		\arrow[bend left=10]{r}{\rho}
		&
		(M'',d'')
		\arrow[bend left=10]{l}{\iota}
	\end{tikzcd}
	\]
	Then there is a special deformation retraction from \((M,d)\) to
	\[
	\left(
	\begin{tikzcd}
		N
		\arrow[bend left=10]{r}{\rho c}
		\arrow[loop left]{lr}{a+b\varepsilon c}
		&
		M''
		\arrow[bend left=10]{l}{b\iota}
		\arrow[loop right]{lr}{d''}
	\end{tikzcd}
	\right)
	\]
	The same is true in the bifiltered setting.
\end{lemma}
\begin{proof}
	Observe that \(cb=0\) since \(d^2=w\cdot 1_M\) and \(d'^2=w\cdot 1_{M'}\). 
	It is now straightforward to check that the new matrix factorization is well-defined and that together with the morphisms
	\[
	\left(
	\begin{tikzcd}[row sep=40pt,column sep=60pt]
		(N,a)
		\arrow[yshift=2pt]{r}{1}
		\arrow[bend left=10]{d}{c}
		&
		(N,a+b\varepsilon c)
		\arrow[bend left=10]{d}{\rho c}
		\arrow[yshift=-2pt]{l}{1}
		\arrow[yshift=-1.4pt,xshift=1.4pt]{dl}{\varepsilon c}
		\\
		(M',d')
		\arrow[loop left,dashed]{lr}{\varepsilon}
		\arrow[bend left=10]{u}{b}
		\arrow[yshift=2pt]{r}{\rho}
		\arrow[yshift=1.4pt,xshift=-1.4pt]{ru}{b\varepsilon}
		&
		(M'',d'')
		\arrow[bend left=10]{u}{b\iota}
		\arrow[yshift=-2pt]{l}{\iota}
	\end{tikzcd}
	\right)
	\]
	it constitutes the desired special deformation retraction.
\end{proof}

Observe that if a matrix factorization \((M,d)\) is homotopy equivalent to the zero matrix factorization \(0\in \MF_w(R)\), then 0 is a special deformation retract of \((M,d)\).
Therefore, \cref{lem:sdr-new-from-old} can be regarded as a generalization of the cancellation lemma \cite[Lemma~2.16]{KWZ}.

\begin{definition}\label{def:mf:tensorproduct}
	Given two matrix factorizations \((M,d)\in\MF_w(R)\) and \((M',d')\in\MF_{w'}(R)\) for some \(w,w'\in R\), we define their tensor product
	\[
	(M,d)
	\otimes
	(M',d')
	=
	(M\otimes M',d\otimes 1+1\otimes d')
	\]
	which is an object of \(\MF_{w+w'}(R)\).
\end{definition}

Note that the Kozsul sign rule ensures that \(d\otimes 1+1\otimes d'\) is indeed a differential on \(M\otimes M'\) with curvature \(w+w'\). 

\begin{definition}\label{def:mf:Koszul}
	Given \(a,b\in R\), we define a matrix factorization  
	\[
	K(a;b)
	=
	\left[
	\begin{tikzcd}[column sep=1cm, ampersand replacement=\&]
		\hslash^1 R 
		\arrow[bend left=10]{d}[right]{a}
		\\
		\hslash^0 R
		\arrow[bend left=10]{u}[left]{b}
	\end{tikzcd}
	\right]
	\]
	over \(R\) with potential \(ab\).
	More generally, given two vectors \(\mathbf{a},\mathbf{b}\in R^n\) for some positive integer \(n\), we define a matrix factorization 
	\[
	K(\mathbf{a};\mathbf{b})
	=
	\bigotimes_{i=1}^{n}
	K(a_i;b_i)
	\]
	over \(R\) with potential \(\mathbf{a}.\mathbf{b}\).
\end{definition}

\begin{remark}
	If \(R\) carries a bigrading and \(a,b\in R\) are homogeneous 
	with \(\gr(a)=q^{i-3}h^0\) and \(\gr(b)=q^{-i-3}h^0\),
	we equip \(K(a;b)\) with the following bifiltration
	\[
	K(a;b)
	=
	\left[
	\begin{tikzcd}[column sep=1cm, ampersand replacement=\&]
		h^0q^i\hslash^1 R 
		\arrow[bend left=10]{d}[right]{a}
		\\
		h^0q^0\hslash^0 R
		\arrow[bend left=10]{u}[left]{b}
	\end{tikzcd}
	\right].
	\]
	Observe that 
	\(
	h^0\hslash^1q^{-i} K(a;b)
	\cong
	K(-b;-a)
	\).
\end{remark}

The following is a bigraded version of \cite[Theorem~2.2]{KR_convolutions}.

\begin{theorem}\label{thm:mfs-special deformation retraction}
	Let \(R\) be a free polynomial ring over \(\CoeffRing\) in finitely many variables and \(w\in R\). 
	Let \(p,q\in R[x]\) and \(n\) an integer strictly bigger 
	than the degree of the polynomial \(p=p(x)\). 
	Let \((M,d)\) be a matrix factorization over \(R[x]\) with potential \(w-(x^n-p)q\). 
	Then 
	\[
	R[x]/(x^n-p)\otimes_{R[x]} (M,d)
	\]
	is a matrix factorization over \(R\) with potential \(w\) and, 
	in fact, a special deformation retract of
	\[
	K(x^n-p;q)\otimes_{R[x]} (M,d).
	\]
	The same is true in the bifiltered setting.
\end{theorem}

\begin{proof}
	We recall the proof from \cite{KR_convolutions}, 
	which generalizes to the bifiltered setting. 
	The \(R\)-module 
	\(R[x]/(x^n-p)\) 
	is freely generated 
	by monomials \(\{x^m\}_{m=0}^{n-1}\) 
	because of the choice of \(n\).  
	Thus, 
	\(R[x]/(x^n-p)\otimes_{R[x]} (M,d)\)
	is a matrix factorization. 
	
	Moreover, there is an inclusion
	\[
	\begin{tikzcd}[column sep=1cm, ampersand replacement=\&]
		K(1;(x^n-p)q)\otimes_{R[x]} (M,d)
		\arrow[hookrightarrow]{r}{\iota}
		\&
		K(x^n-p;q)\otimes_{R[x]} (M,d)
	\end{tikzcd}
	\]
	given by
	\[
	\begin{tikzcd}[column sep=3cm, ampersand replacement=\&]
		\hslash^1 R[x]\otimes_{R[x]} (M,d)
		\arrow[bend left=10]{d}[right]{1\otimes 1}
		\arrow{r}{1\otimes 1}
		\&
		\hslash^1 R[x]\otimes_{R[x]} (M,d)
		\arrow[bend left=10]{d}[right]{(x^n-p)\otimes 1}
		\\
		\hslash^0 R[x]\otimes_{R[x]} (M,d)
		\arrow[bend left=10]{u}[left]{(x^n-p)q\otimes 1}
		\arrow{r}{(x^n-p)\otimes 1}
		\&
		\hslash^0 R[x]\otimes_{R[x]} (M,d)
		\arrow[bend left=10]{u}[left]{q\otimes 1}
	\end{tikzcd}
	\]
	The subcomplex 
	\(K(1;(x^n-p)q)\otimes_{R[x]} (M,d)\) 
	is null-homotopic. 
	We apply \cref{lem:sdr-new-from-old} and 
	obtain a new matrix factorization \(Q\)  
	that is a special deformation retract of 
	\(K(x^n-p;q)\otimes_{R[x]} (M,d)\). 
	Its underlying \(R\)-module 
	agrees with the quotient of 
	\(K(x^n-p;q)\otimes_{R[x]} (M,d)\) 
	by the image of~\(\iota\). 
	Moreover, 
	the differential is the induced differential, 
	since the map \(b\) in 
	\cref{lem:sdr-new-from-old} 
	is zero in this case. 
	Finally observe that 
	\(Q\) is equal to 
	\(R[x]/(x^n-p)\otimes_{R[x]} M\)
	with the induced differential. 
\end{proof}

The following lemmas, 
both of which are specializations of 
\cref{thm:mfs-special deformation retraction}, 
will be our sources of special deformation retracts. 
See also~\cite[Section~2.2]{Ballinger}. 

\begin{lemma}\label{lem:lin_sp_def_retr}
	Let \(n\) and \(\mathbf{a},\mathbf{b}\in R^n\) be as in \cref{def:mf:Koszul}.
	Suppose that \(R=R_0[a_n]\) for some freely generated \(\field\)-algebra \(R_0\) and that the potential \(w\) of \(K(\mathbf{a},\mathbf{b})\) lies in \(R_0\). 
	Then there is a special deformation retract of matrix factorizations over \(R_0\) with potential \(w\):
	\[
	K(\mathbf{a},\mathbf{b}) 
	\rightarrow
	K\left(
	(a_1\big|_{a_n=0},\ldots,a_{n-1}\big|_{a_n=0}),
	(b_1\big|_{a_n=0},\ldots,b_{n-1}\big|_{a_n=0})
	\right)
	\]
	In fact, the map is given by the horizontal quotient homomorphism below
	\[
	\begin{tikzcd}
		\hslash^1  K(\mathbf{a}',\mathbf{b}')
		\arrow[bend left=10]{d}{a_n}
		\\
		\hslash^0 K(\mathbf{a}',\mathbf{b}')
		\arrow[in=180,out=0]{r}{}
		\arrow[bend left=10]{u}{b_n}
		&
		K\left(
		(a_1\big|_{a_n=0},\ldots,a_{n-1}\big|_{a_n=0}),
		(b_1\big|_{a_n=0},\ldots,b_{n-1}\big|_{a_n=0})
		\right)
	\end{tikzcd}
	\]
	where \(\mathbf{a}'= (a_1,\dots,a_{n-1})\) and \(\mathbf{b}'=(b_1,\dots,b_{n-1})\).
\end{lemma}
\begin{proof}
	Apply \cref{lem:sdr-new-from-old} to \(x=a_n\), \(p=0\), and \(q=b_n\). 
\end{proof}

\begin{lemma}\label{lem:quad_sp_def_retr}
	Let \(n\) and \(\mathbf{a},\mathbf{b}\in R^n\) be as in \cref{def:mf:Koszul}.
	Suppose that \(R=R_0[x]\) for some freely generated \(\KK\)-algebra \(R_0\) and some element \(x\in R\). 
	Suppose further that the potential \(w\) of \(K(\mathbf{a},\mathbf{b})\) lies in \(R_0\). 
	Let \(p(x)\in R_0[x]\) be a monic polynomial in \(x\) over \(R_0\) of degree 2, i.e.\ \(p(x)=x^2 + C x + D\) for some \(C,D\in R_0\). 
	Then there is a special deformation retract 
	\[
	K((a_1,\ldots,a_n,0),(b_1,\ldots,b_n,p(x))) \to
	K(\mathbf{a},\mathbf{b}) \otimes_{R_0[x]} \hslash^1 R_0[x]/(p(x))
	\]
	where the matrix factorizations are considered over \(R_0\) with potential \(w\).  
	In fact, this map is given by the horizontal quotient homomorphism below
	\[
	\begin{tikzcd}
		\hslash^1 K(\mathbf{a},\mathbf{b})
		\arrow[in=180,out=0]{r}{}
		&
		K(\mathbf{a},\mathbf{b}) \otimes_{R_0[x]} \hslash^1 R_0[x]/(p(x))
		\\
		\hslash^0 K(\mathbf{a},\mathbf{b})
		\arrow{u}{p(x)}
	\end{tikzcd}
	\]
	In the bifiltered setting, if \(\gr(x)=q^{-2}h^0\), the gradings in the special deformation retract are as follows:
	\[
	K((a_1,\ldots,a_n,0),(b_1,\ldots,b_n,p(x))) \to
	K(\mathbf{a},\mathbf{b}) \otimes_{R_0[x]} h^0q^1\hslash^1 R_0[x]/(p(x)).
	\]
\end{lemma}

\begin{proof}
	Observe that 
	\[
	K((a_1,\ldots,a_n,0),(b_1,\ldots,b_n,p(x)))
	\cong 
	q^{1}\hslash^1K((a_1,\ldots,a_n,p(x)),(b_1,\ldots,b_n,0))
	\]
	and apply \cref{lem:sdr-new-from-old} to \(p=x^2-p(x)=-C x-D\) and \(q=0\). 
\end{proof}

\subsection{Matrix factorizations for labelled diagrams}\label{sub:lab-diag} 

In this subsection matrix factorizations will be assigned 
to a particular class of planar diagrams. 
While we specialize to the case relevant 
to the study of 4-ended tangles at the outset, 
much of the initial setup works for general tangles 
with minor notational changes. 

A singular diagram \(D\) is 
an embedded, connected, oriented graph in the 2-dimensional unit disk,
satisfying the following conditions, 
illustrated in \cref{fig:singular-diagram:singular-diagram}:
The graph meets the boundary of the disk in precisely 
the set \(\EndsAll\) from \cref{def:Cob}, 
which are all valence~1 vertices.
The remaining edges have valence~4. 
We call them the nodes of \(D\)
and write \(\nodes(D)\) for the set of all nodes.
Since we assume that \(D\) is connected,
$\nodes(D)\ne\varnothing$.
We require the edges adjacent to any given node \(a\) 
to be oriented as in the following diagram, 
which also introduces the notation for these edges:
\[
\vc{
	\begin{tikzpicture}[scale=0.4]
		\node(x) at (0,0){$\vc{\includegraphics[scale=0.5]{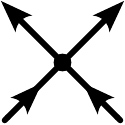}}$};
		\node(x0) at (-1.8, 1.8){\(e^a_0\)};
		\node(x1) at ( 1.8, 1.8){\(e^a_3\)};
		\node(x2) at (-1.8,-1.8){\(e^a_1\)};
		\node(x3) at ( 1.8,-1.8){\(e^a_2\)};
	\end{tikzpicture}
}
\]
Additionally, we fix an ordering of the set $\nodes(D)$ 
and an ordering of the set of edges 
\(\{e_0, e_1,\dots,e_N\}\). 
By convention, 
the edge \(e_0\) is adjacent to the distinguished end \(*\)
and 
the edges \(e_i\) are adjacent to the ends \(\mathtt{e}_i\)
for \(i=1,2,3\). 
We call these the boundary edges of \(D\) and call the remaining edges the interior edges of \(D\). We leave the proof of the following simple observation to the reader:

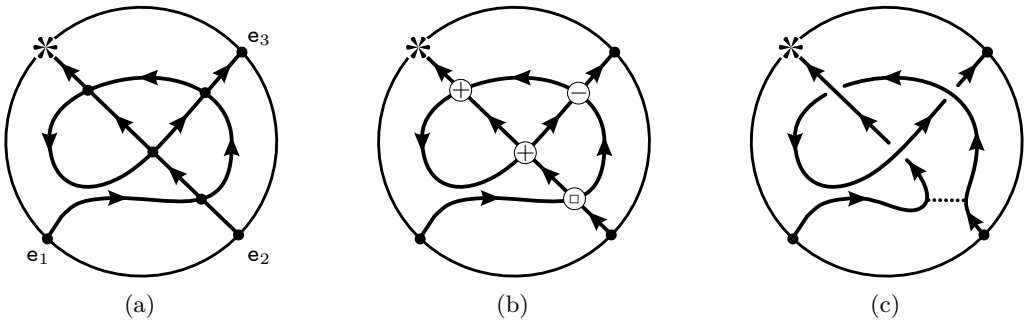
\begin{figure}[b]
	\centering
	\begin{subfigure}{0.3\textwidth}
		\centering
		\labellist \footnotesize
		\pinlabel $\Endi$ at 25 15
		\pinlabel $\Endii$ at 190 15
		\pinlabel $\Endiii$ at 190 180
		\endlabellist
		\includegraphics[scale=0.5]{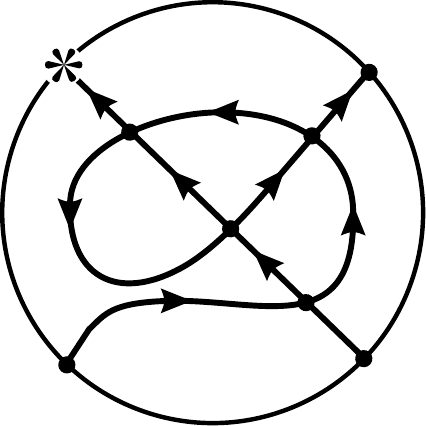}
		\caption{}
		\label{fig:singular-diagram:singular-diagram}
	\end{subfigure}
	\begin{subfigure}{0.3\textwidth}
		\centering
		\labellist \footnotesize
		\pinlabel $+$ at 62.5 140.8
		\pinlabel $-$ at 150.5 138.5
		\pinlabel $+$ at 110.5 93.8
		\pinlabel $\arcD$ at 147.2 58.5
		\endlabellist
		\includegraphics[scale=0.5]{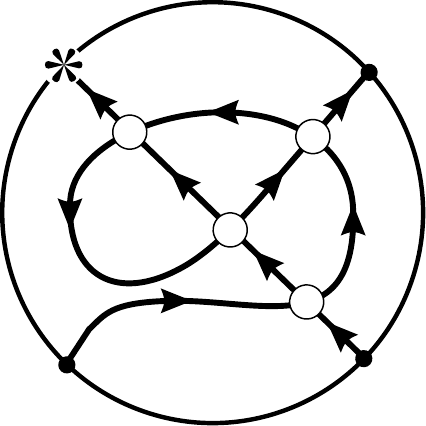}
		\caption{}
		\label{fig:singular-diagram:with-labelling}
	\end{subfigure}
	\begin{subfigure}{0.3\textwidth}
		\centering
		\includegraphics[scale=0.5]{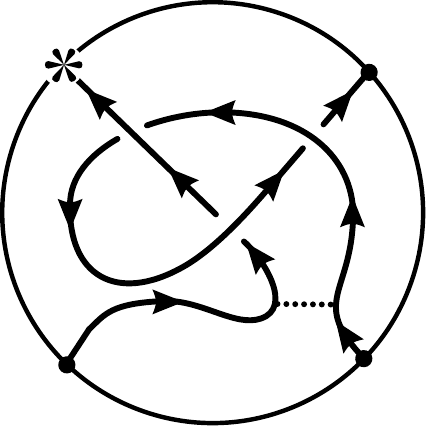}
		\caption{}
		\label{fig:singular-diagram:realized-labelling}
	\end{subfigure}
	\caption{%
		A singular diagram (a), 
		a labelling of this diagram (b), 
		and the pictorial realization of the labelling (c).
		The underlying tangle diagram is obtained by deleting the dashed arc.
	}
\label{fig:singular-diagram}
\end{figure} 

\begin{lemma}\label{lem:edge-signs}
	For any singular diagram \(D\)
	there exists a unique map
	\[
	\sigma\co \{\text{edges of }D\}\to\{\pm1\},  
	\quad
	e\mapsto\sigma^e
	\]
	such that
	\(\sigma^{e_0}=+1\)
	and 
	\(\sigma^{e^a_0}=\sigma^{e^a_1}=-\sigma^{e^a_2}=-\sigma^{e^a_3}\)
	for all \(a\in\nodes(D)\). 
	\qed
\end{lemma}

\cref{lem:edge-signs} allows us to define 
\(\sigma^a=\sigma^{e^a_0}\) for any \(a\in\nodes(D)\).
The role of \(\sigma\) will become clearer 
in the proof of \cref{prop:mfs-well-defined}; 
see \cref{rem:mfs-well-defined}.
Pictorially, 
the second condition on \(\sigma\) means that 
the signs at each node are equal to 
\[
\vc{
	\begin{tikzpicture}[scale=0.4]
		\node(x) at (0,0){$\vc{\includegraphics[scale=0.5]{figures/inkscape/Label-singular.pdf}}$};
		\node(x0) at (-1.8, 1.8){\(+\)};
		\node(x1) at ( 1.8, 1.8){\(-\)};
		\node(x2) at (-1.8,-1.8){\(+\)};
		\node(x3) at ( 1.8,-1.8){\(-\)};
	\end{tikzpicture}
}
\quad 
\text{or}
\quad
\vc{
	\begin{tikzpicture}[scale=0.4]
		\node(x) at (0,0){$\vc{\includegraphics[scale=0.5]{figures/inkscape/Label-singular.pdf}}$};
		\node(x0) at (-1.8, 1.8){\(-\)};
		\node(x1) at ( 1.8, 1.8){\(+\)};
		\node(x2) at (-1.8,-1.8){\(-\)};
		\node(x3) at ( 1.8,-1.8){\(+\)};
	\end{tikzpicture}
}
\]

We assign variables $X_i$ to edges of \(D\) 
according to their ordering.  
In particular,
$X_0,X_1,X_2,X_3$ are the variables assigned to the four boundary edges, 
with $X_0$ assigned to the edge adjacent to the distinguished end \(*\). 
It will often be useful to write $X^e$ 
for the variable associated with an edge $e$, 
so that $X_i=X^{e_i}$. 
For any \(a\in\nodes(D)\), 
we write \(X^a_i=X^{e^a_i}\) 
for the label associated with the edge \(e^a_i\). 
These variables will be considered 
as elements of the polynomial ring 
\(\CoeffRing[\varG,X_0,\dots,X_N]\)
which we equip with a bigrading by setting 
\[
\text{gr}(\varG)=\text{gr}(X_i)=h^0q^{-2}
\]
for all $0\le i\le N$. 
For a node $a$ define the homogeneous element 
\[
\rho^a=X^a_0-X^a_1+X^a_2-X^a_3
\]
and let
\(I(D)\subset\CoeffRing[\varG,X_0,\dots,X_N]\) 
be the ideal generated by \(X_0\) and the \(\rho^a\)
for each \(a\in\nodes(D)\).

\begin{definition}\label{def:edge_ring}
	Given a singular diagram \(D\) as above
	define the edge ring \(R(D)\) associated with \(D\) as the quotient  
	\[
	R(D)
	=
	\CoeffRing[\varG,X_0,\dots,X_N]/I(D),
	\]
	which is bigraded.
	The boundary edge ring \(\Rd(D)\) is the subring of \(R(D)\) 
	generated by \(\varG\) and \(X_0\), \(X_1\), \(X_2\), \(X_3\).
\end{definition}

\begin{proposition}\label{prop:boundary-edge-ring}
	For any singular diagram \(D\) 
	\[
	\Rd(D)=\CoeffRing[\varG,X_0,X_1,X_2,X_3]/(X_0,\rho^\partial)\subseteq R(D)
	\]
	where 
	\[
	\rho^\partial
	=
	\smashoperator[r]{\sum_{a\in\nodes(D)}}
	\sigma^a\rho^a
	=
	\sum_{i=0}^3
	\epsilon^{e_i}\sigma^{e_i} X_i
	\]
	and \(\epsilon^e=+1\) if the edge \(e\) points out of the diagram 
	and \(\epsilon^e=-1\) if \(e\) points into the diagram.
\end{proposition}

\begin{proof}
	For every node \(a\) in \(D\) it is easy to check that
	\[
	\sigma^a\rho^a
	=
	\sum_{i=0}^3 \epsilon^a_{i}\sigma^{e^a_i} X^a_i 
	\]
	where \(\epsilon^a_{i}=+1\) if the edge \(e^a_i\) points away from \(a\) 
	and \(\epsilon^a_{i}=-1\) if \(e^a_i\) points into \(a\).
	The identity for \(\rho^\partial\) can now be seen as follows:
	Every variable corresponding to an interior edge of \(D\) 
	appears either in one or in two summands \(\sigma^a\rho^a\). 
	In the first case, it appears twice, with opposite signs, 
	and in the second case, it appears once in each summand and again with opposite signs. 
	This observation also implies 
	that \((X_0,\rho^\partial)\) is the restriction of the ideal \(I(D)\) 
	to \(\CoeffRing[\varG,X_0,X_1,X_2,X_3]\).
\end{proof}	

\begin{definition}\label{def:diagram}
	A label of node in a singular diagram 
	is an element of \(\{\arcD,\arcT,+,-\}\). 
	A labelled diagram \(\Diag\) is a pair \((D,\lab)\), 
	where \(D\) is a singular diagram and	
	\begin{align*}
	\lab\co \nodes(D)&\to \{\arcD,\arcT,+,-\}\\
	a&\mapsto \lab(a)
	\end{align*} 
	is a labelling of its nodes. 
\end{definition}

There is a special set of labelled diagrams 
associated with singular diagrams having exactly one node. 
We call them elementary diagrams and represent them pictorially as follows:
\begin{align*}
	\DiagD
	&=
	\vc{\includegraphics[scale=0.5]{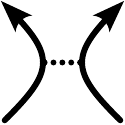}}
	&
	\DiagT
	&=
	\vc{\includegraphics[scale=0.5]{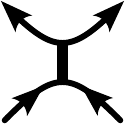}}
	&
	\Diag_+
	&=
	\vc{\includegraphics[scale=0.5]{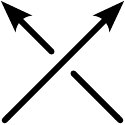}}
	&
	\Diag_-
	&=
	\vc{\includegraphics[scale=0.5]{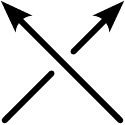}}
\end{align*}
For a given labelled diagram $\mathcal{D}$, 
in pictures, we record the labels 
by changing a small neighbourhood of each \(a\in\nodes(D)\) 
to the diagram \(\mathcal{D}_{\lab(a)}\) 
agreeing with the corresponding elementary diagram, 
such that the orientations of the adjacent edges match up. 
Based on these conventions, 
given a labelled diagram, 
notions of positive and negative crossings 
(nodes carrying labels $+$ and $-$ respectively) 
as well as dotted and thick arcs 
(nodes carrying labels $\arcD$ and $\arcT$ respectively) 
match the terminology found in other treatments 
on link invariants derived from matrix factorizations. 

Finally, the tangle diagram \(|\Diag|\) 
associated with a labelled diagram \(\Diag=(D,\lab)\) 
is the unoriented pointed tangle diagram 
obtained by removing all dotted and thick arcs in \(\Diag\) 
pictured according to the formalism above.
A component of $\Diag$ is a component of the tangle diagram $|\Diag|$.
This makes it possible to distinguish between 
the two open components of $\Diag$ and any closed components. 
For any oriented, pointed tangle \(T\) in the 3-dimensional ball,
 we can always find a labelled diagram \(\Diag_T\) 
 such that the associated tangle diagram \(|\Diag_T|\) 
 is a tangle diagram for \(T\). 
An example of \(\Diag_T\) 
for the rational tangle \(T=Q_{-3/2}\) 
is shown in \cref{fig:singular-diagram:realized-labelling,fig:singular-diagram:with-labelling}. 
It will later be convenient to assume 
that there are additional arcs in \(\Diag_T\); 
see \cref{assums:tangle_diagram}.

\begin{remark}\label{exa:diagrams:basic}
	In the case of elementary diagrams,
	the edge ring is isomorphic to the boundary edge ring, 
	which is expressed as 
	\[
	\CoeffRing[\varG,X_0,X_1,X_2,X_3]/(X_0,X_0-X_1+X_2-X_3)
	\]
	This can be identified with
	\begin{equation}\label{eq:base_ring}
		\Rd
		=
		\CoeffRing[x,y,z]
	\end{equation}
	by setting 
	\begin{equation}\label{eq:xyz}
		x=X_1
		\qquad
		y=X_3
		\qquad
		z=
		x+y+\varG
	\end{equation}
	For future reference: 
	\[
	\text{gr}(x)=\text{gr}(y)=\text{gr}(z)=\text{gr}(\varG)=h^0q^{-2}
	\]
\end{remark}

\begin{definition}\label{def:mf_of_tangle}
	For any given labelled tangle diagram \(\Diag=(D,\lab)\) 
	and any $a\in\nodes(D)$ define the following elements 
	of the edge ring \(R=R(D)\):
	\begin{align*}
		x^a_\arcD
		=
		y^a_\arcT
		&=
		\sigma^a(X^a_1-X^a_0)
		\\
		y^a_\arcD
		=
		x^a_\arcT
		&=
		(X^a_3-X^a_0)
		\\
		z^a
		&=
		X^a_1+X^a_3+\varG
	\end{align*}
	For every label \(\lab\in\{\arcD,\arcT,+,-\}\) 
	define a matrix factorization \(M^a_\lab\) over \(R\):
	\begin{align*}
		M^a_\arcD
		&=
		\left[
		\begin{tikzcd}[column sep=1cm, ampersand replacement=\&]
			h^0 \hslash^1 q^1 R
			\arrow[bend left=10]{d}[right]{x^a_\arcD}
			\\
			h^0 \hslash^0 q^0 R
			\arrow[bend left=10]{u}[left]{y^a_\arcD z^a}
		\end{tikzcd}
		\right]
		&
		M^a_+
		&=
		\left[
		\begin{tikzcd}[column sep=1.9cm, ampersand replacement=\&]
			h^0 \hslash^0 q^2 R
			\arrow[bend left=10]{d}[right]{-x^a_\arcD}
			\arrow[out=0,in=180]{dr}[description,pos=0.2,tight]{1}
			\&
			h^1 \hslash^0 q^3 R
			\arrow[bend left=10]{d}[right]{x^a_\arcT}
			\\
			h^0 \hslash^1 q^1 R
			\arrow[bend left=10]{u}[left]{-y^a_\arcD z^a}
			\arrow[out=0,in=180]{ur}[description,pos=0.2,tight]{z^a}
			\&
			h^1 \hslash^1 q^2 R
			\arrow[bend left=10]{u}[left]{y^a_\arcT z^a}
		\end{tikzcd}
		\right]
		\\
		M^a_\arcT
		&=
		\left[
		\begin{tikzcd}[column sep=1cm, ampersand replacement=\&]
			h^0 \hslash^0 q^1 R
			\arrow[bend left=10]{d}[right]{x^a_\arcT}
			\\
			h^0 \hslash^1 q^0 R
			\arrow[bend left=10]{u}[left]{y^a_\arcT z^a}
		\end{tikzcd}
		\right]
		&
		M^a_-
		&=
		\left[
		\begin{tikzcd}[column sep=1.5cm, ampersand replacement=\&]
			h^{-1} \hslash^0 q^{-1} R
			\arrow[bend left=10]{d}[right]{x^a_\arcT}
			\arrow[out=0,in=180]{dr}[description,pos=0.2,tight]{1}
			\&
			h^0 \hslash^0 q^0 R
			\arrow[bend left=10]{d}[right]{-x^a_\arcD}
			\\
			h^{-1} \hslash^1 q^{-2} R
			\arrow[bend left=10]{u}[left]{y^a_\arcT z^a}
			\arrow[out=0,in=180]{ur}[description,pos=0.2,tight]{z^a}
			\&
			h^0 \hslash^1 q^{-1} R
			\arrow[bend left=10]{u}[left]{-y^a_\arcD z^a}
		\end{tikzcd}
		\right]
	\end{align*}
	Define 
	the matrix factorization \(M(\Diag)\) 
	over \(R\) by
	\[
	M(\Diag) = \bigotimes_{a\in\nodes(D)} M^a_{\lab(a)}
	\]
\end{definition}

\begin{remark}
	\Cref{tab:elementary_mf} shows the matrix factorizations 
	associated with elementary diagrams, 
	regarded as labelled diagrams, 
	as in \cref{exa:diagrams:basic}. 
	The potential is equal to \(xyz\) in all cases, 
	giving the prototypical examples 
	of elements in the category 
	$\MF_{xyz}(\Rd)$
	of matrix factorizations over \(\Rd\)
	with potential \(xyz\).
	Note that, using notation from \cref{def:mf:Koszul},
	$
	M^a_\arcD
	=
	K(x^a_\arcD;y^a_\arcD z^a)
	$ 
	and 
	$
	M^a_\arcT
	=
	\hslash^1 K(-x^a_\arcT;-y^a_\arcT z^a)
	$.
\end{remark}

\begin{proposition}\label{prop:mfs-well-defined}
For any labelled diagram \(\Diag\), 
\(M(\Diag)\in\MF_{xyz}(\Rd)\). 
\end{proposition}
\begin{proof}
	We first verify that \(M(\Diag)\) 
	satisfies the conditions 
	in \cref{def:mf}. 
	It suffices to check that
	\(R(D)\) is free as an 
	\(\Rd(D)\)-module. 
	This is a well-known fact and 
	may be seen as a corollary of the much stronger statement 
	of \cref{thm:edge-ring-generator-criterion-alt-all}.
	The potential of \(M(\Diag)\) 
	is computed as the sum of the potentials of the individual factors 
	(\cref{def:mf:tensorproduct}).
	First consider the case \(\CoeffRing=\Z\). 
	
	Write $x^a=x^a_{\lab(a)}$ and $y^a=y^a_{\lab(a)}$ 
	for each $a\in\nodes(D)$ 
	so that the matrix factorization $M^a_{\lab(a)}$ has potential $x^ay^az^a$. 
	Now consider the expression 
	\[
	E^a=p(X_0^a)-p(X_1^a)+p(X_2^a)-p(X_3^a)
	\]
	where $p(X)=3\!\varG\cdot X^2 + 2 X^3$. 
	Substituting $X_2^a=X_1^a+X_3^a-X_0^a$ 
	and expanding terms shows that 
	\[
	E^a
	=
	6\cdot
	(X_1^a-X_0^a)(X_3^a-X_0^a)(X_1^a+X_3^a+\varG)
	\] 
	In particular, 
	\begin{equation}\label{eq:explain_some_mf_signs}
	6\cdot\sigma^ax^ay^az^a
	=
	\sigma^a\left(p(X_0^a)-p(X_1^a)+p(X_2^a)-p(X_3^a)\right)
	\end{equation}
	so the potential \(P\) of $M(\Diag)$ satisfies
	\[
	6P
	=
	6\cdot
	\smashoperator{\sum_{a\in\nodes(D)}}
	\sigma^ax^ay^az^a 
	= 
	\smashoperator[r]{\sum_{a\in\nodes(D)}}
	\sigma^a E^a  
	= 
	\sum_{a\in\nodes(D)}
	\smashoperator[r]{\sum_{\substack{\text{edges }e\\\text{\ incident\ to\ }a}}}
	\epsilon^e\sigma^e p(X^e)
	\]
	where \(\epsilon^e=+1\) if the edge \(e\) points away from the node $a$ 
	and \(\epsilon^e=-1\) if \(e\) points towards the node $a$. 
	Note that the last of these expressions gives rise a term $p(X^e)-p(X^e)$ 
	for every internal edge $e$ in the diagram. 
	Hence 
	\[
	6P
	= 
	p(X_0)-p(X_1)+p(X_2)-p(X_3) 
	= 
	6\cdot
	X_1X_3(X_1+X_3+\varG)
	=
	6\cdot
	xyz.
	\] 
	Since 6 is not a zero-divisor in \(\Rd(D)\), 
	we conclude \(P=xyz\). 
	Since the matrix factorization $M(\Diag)$ 
	over arbitrary coefficients \(\CoeffRing\) 
	can be obtained from the matrix factorization 
	with integer coefficients by simply tensoring with \(\CoeffRing\), 
	its potential is also \(xyz\). 
\end{proof}

\begin{table}[t]
	\centering
	\begin{tabular}{c|c|c|c|c}
		\toprule
		\(\lab\)
		&
		\(\arcD\)
		&
		\(\arcT\)
		&
		\(+\)
		&
		\(-\)
		\\
		\midrule
		\(\Diag_{\lab}\)
				&
		\(\vc{\includegraphics[scale=0.5]{figures/inkscape/Label-dotted.pdf}}\)
		&
		\(\vc{\includegraphics[scale=0.5]{figures/inkscape/Label-thick.pdf}}\)
		&
		\(\vc{\includegraphics[scale=0.5]{figures/inkscape/Label-plus.pdf}}\)
		&
		\(\vc{\includegraphics[scale=0.5]{figures/inkscape/Label-minus.pdf}}\)
		\\
		\midrule
		\raisebox{30pt}{\(M(\Diag_{\lab})\)}
		&
		\begin{tikzpicture}[scale=0.8]
			\node[](R1) at (0,1){\(h^0 \hslash^1 q^1 \Rd\)};
			\node[](R2) at (0,-1){\(h^0 \hslash^0 q^0 \Rd\)};
			\draw[->] (R1) to[bend left=7,right] node[tight]{\(x\)} (R2);
			\draw[->] (R2) to[bend left=7,left] node[tight]{\(yz\)} (R1);
		\end{tikzpicture}
		&
		\begin{tikzpicture}[scale=0.8]
			\node[](R1) at (0,1){\(h^0 \hslash^0 q^1 \Rd\)};
			\node[](R2) at (0,-1){\(h^0 \hslash^1 q^0 \Rd\)};
			\draw[->] (R1) to[bend left=7,right] node[tight]{\(y\)} (R2);
			\draw[->] (R2) to[bend left=7,left] node[tight]{\(xz\)} (R1);
		\end{tikzpicture} 
		&
		\begin{tikzpicture}[scale=0.8]
			\node(R1) at (-1.5,1)	{\(h^0 \hslash^0 q^2 \Rd\)};
			\node(R2) at (-1.5,-1)	{\(h^0 \hslash^1 q^1 \Rd\)};
			\node(R3) at (1.7,1)	{\(h^1 \hslash^0 q^3 \Rd\)};
			\node(R4) at (1.7,-1)	{\(h^1 \hslash^1 q^2 \Rd\)};
			\draw[->] (R3) to[bend left=7,right] node[tight]{\(y\)} (R4);
			\draw[->] (R4) to[bend left=7,left] node[tight]{\(xz\)} (R3);
			\draw[->] (R1) to node[tight,pos=0.2,fill=white]{\(1\)} (R4);
			\draw[->] (R2) to node[tight,pos=0.2,fill=white]{\(z\)} (R3);
			\draw[->] (R1) to[bend left=7,right] node[tight]{\(-x\)} (R2);
			\draw[->] (R2) to[bend left=7,left] node[tight]{\(-yz\)} (R1);
		\end{tikzpicture} 
		&
		\begin{tikzpicture}[scale=0.8]
			\node(R1) at (-1.5,1)  {\(h^{-1} \hslash^0 q^{-1} \Rd\)};
			\node(R2) at (-1.5,-1) {\(h^{-1} \hslash^1 q^{-2} \Rd\)};
			\node(R3) at (1.7,1)   {\(h^0    \hslash^0 q^0    \Rd\)};
			\node(R4) at (1.7,-1)  {\(h^0    \hslash^1 q^{-1} \Rd\)};
			\draw[->] (R3) to[bend left=7,right] node[tight]{\(-x\)} (R4);
			\draw[->] (R4) to[bend left=7,left] node[tight]{\(-yz\)} (R3);
			\draw[->] (R1) to node[tight,fill=white,pos=0.2]{\(1\)} (R4);
			\draw[->] (R2) to node[tight,fill=white,pos=0.2]{\(z\)} (R3);
			\draw[->] (R1) to[bend left=7,right] node[tight]{\(y\)} (R2);
			\draw[->] (R2) to[bend left=7,left] node[tight]{\(xz\)} (R1);
		\end{tikzpicture}  
		\\
		\bottomrule
	\end{tabular}
	\medskip
	\caption{%
		Matrix factorizations associated with the elementary diagrams. 
}\label{tab:elementary_mf}
\end{table}

\begin{remark}\label{rem:mfs-well-defined}
	In \cref{def:edge_ring}, 
	we chose the relation between edge variables at each node \(a\) 
	such that adjacent edge variables are identified 
	as soon as the opposite adjacent edge variables are identified. 
	This is a subtle, but important difference to the setup 
	in \cite{Some_diff}. 
	The signs in \eqref{eq:explain_some_mf_signs} 
	are forced (up to an overall sign) 
	by the condition that it be divisible 
	by \(X^a_1-X^a_0\) and \(X^a_3-X^a_0\). 
	These signs are precisely \(\epsilon^e\sigma^e\). 
	For the contributions from interior edges to cancel, 
	the factor \(\sigma^a\) is needed. 
	This accounts for the appearance of the sign \(\sigma^a\) 
	in the variables \(x^a_\arcD = y^a_\arcT\). 
\end{remark}

\begin{remark}\label{rmk:regions}
	It is possible to set up matrix factorizations 
	by working with variables associated 
	with regions rather than with edges. 
	The edge ring would then be defined 
	as the ring generated by the difference 
	between variables \(x_i\) labelling all regions of \(D\) 
	(modulo identifying the regions 
	on either side of the edge near the distinguished end). 
	One can translate between these two perspectives 
	by setting, for any edge \(e\)
	\[
	X^e=
	\begin{cases*}
		x^e_\ell-x^e_r & if \(\sigma^e=+1\)
		\\
		-(x^e_\ell-x^e_r)-\varG & if \(\sigma^e=-1\)
	\end{cases*}
	\]
	where \(x^e_r\) and \(x^e_\ell\) 
	are the labels of the two regions 
	on the right and left of the edge \(e\), 
	respectively. 
	This identity is inspired by the definition 
	of canonical generators in~\cite{JakeSInvariant}. 
	Note that Rasmussen's definition 
	uses a certain mod 2 count of nesting circles and their orientation 
	which is equivalent to our sign \(\sigma^e\). 
\end{remark} 

\subsection{Cube of resolutions of matrix factorizations}
\label{sec:no_wrapping_around_special:mfs_cubical} 

Ultimately, 
we will relate the matrix factorizations 
introduced in \cref{def:mf_of_tangle} 
to the tangle invariants \(\DD(T)\) 
from \cref{sec:review:Kh}. 
In this subsection, 
we take the first step in this identification process
by reinterpreting these matrix factorizations 
as cubes of resolutions. 

\begin{definition}\label{def:cubical}
	Let \(n\) be a non-negative integer. 
	A chain complex or, more generally, a matrix factorization \((M,d)\) 
	is \(n\)-cubical
	if \(M\) admits a decomposition 
	\[
	M=\bigoplus_{v\in\{0,1\}^n}M_v
	\]
	such that the restriction of the differential to 
	\(d^v_{v'}\co M_v\rightarrow M_{v'}\)
	is zero unless \(v\leq v'\). 
	The matrix factorization \((M,d)\) is strictly \(n\)-cubical if, in addition, \(d^v_{v'}=0\) unless \(v\rightarrow v'\) is an edge or \(v=v'\). 
\end{definition}

Let \(\Diag=(D,\lab)\) be a labelled diagram with \(n\) crossings (recall that the crossings correspond to nodes labelled with $+$ or $-$).
To describe the $n$-cubical structure on \(M(\Diag)\) some further notation is required. 
Given a label \(\lab\in\{+,-\}\) define
\[
\lab^{0}
=
\begin{cases*}
	\arcD & if \(\lab=+\)
	\\
	\arcT & if \(\lab=-\)
\end{cases*}
\quad
\text{and}
\quad
\lab^{1}
=
\begin{cases*}
	\arcT & if \(\lab=+\)
	\\
	\arcD & if \(\lab=-\)
\end{cases*}
\]
This has a familiar graphical representation in terms of $0$- and $1$-resolutions 
(compare \cref{sub:complexes}): 
\[
\begin{tikzcd}
	\vc{\includegraphics[scale=0.5]{figures/inkscape/Label-dotted.pdf}}
	&
	\vc{\includegraphics[scale=0.5]{figures/inkscape/Label-plus.pdf}}
	\arrow{l}[swap]{0}
	\arrow{r}{1}
	&
	\vc{\includegraphics[scale=0.5]{figures/inkscape/Label-thick.pdf}}
	&
	\vc{\includegraphics[scale=0.5]{figures/inkscape/Label-thick.pdf}}
	&
	\vc{\includegraphics[scale=0.5]{figures/inkscape/Label-minus.pdf}}
	\arrow{l}[swap]{0}
	\arrow{r}{1}
	&
	\vc{\includegraphics[scale=0.5]{figures/inkscape/Label-dotted.pdf}}
\end{tikzcd}
\]
Given a decorated diagram $(D,\lab)$ and a set of values $v_i\in\{0,1\}$ we can define a new labeling $\lab^v\co \nodes(D)\to \{\arcD,\arcT\}$ using \[
\lab^v(a)
=
\begin{cases*}
	\lab(a)
	&
	if \(\lab(a)\in\{\arcD,\arcT\}\)
	\\
	\lab^{v_i}(a)
	&
	if \(a\) is the \(i^{\text{th}}\) crossing in \(\Diag\)
\end{cases*}
\] where $1\le i \le n$ indexes crossings in the diagram. (Note that crossings inherit an order from the ordering on $\nodes(D)$ that was fixed at the outset as part of the singular diagram $D$.) In particular, for any choice of vertex $v\in\{0,1\}^n$ we obtain a labelling $\lab^v$ from the given $\lab$.

\begin{definition}
For \(\Diag=(D,\lab)\) a labelled diagram with \(n\) crossings, 
the complete resolution \(\Diag(v)\) of \(\Diag\) 
with respect to a vertex \(v\in\{0,1\}^n\) 
is the labelled tangle diagram \((D,\lab^v)\).
\end{definition}

A complete resolution of the diagram 
in \cref{fig:singular-diagram:realized-labelling}
is shown in \cref{fig:arc-system} on page \pageref{fig:arc-system}.

Every $v,v'\in\{0,1\}^n$ that differ in exactly one coordinate determine an edge \(v\rightarrow v'\) in the hypercube \([0,1]^n\) from $v$ (with \(j^\text{th}\) coordinate $0$) to $v'$ (with \(j^\text{th}\) coordinate $1$). For any such pair $v,v'\in\{0,1\}^n$ write \(a^v_{v'}\) for the node in \(D\) corresponding to the \(j^{\text{th}}\) crossing of \(\Diag\).

To describe \(M(\Diag)\) 
as an \(n\)-dimensional cube of resolutions
let \(n_-\) be the number of negative crossings of $\Diag$
and \(n_+\) the number of positive crossings $\Diag$, so that $n=n_++n_-$.
For each vertex \(v\in\{0,1\}^n\) consider the 
matrix factorization
\[
M_v(\Diag)
=
h^{|v|-n_-}\hslash^{|v|+n_+}q^{|v|+n_+-2n_-}
M(\Diag(v))
\] where $\Diag(v)$ is the associated complete resolution of $\Diag$ and $|v|=\sum v_i$ is the height of the vertex. 
Observe that the matrix factorizations 
\(M^a_+\) and \(M^a_-\) from \cref{def:mf_of_tangle}
can be viewed as mapping cones
\begin{equation}\label{eq:mcs}
	M^a_+=
	\left[
	\begin{tikzcd}[column sep=1cm]
		h^0\hslash^1q^1 M^a_\arcD 
		\arrow{r}{d^a_+}
		&
		h^{1}\hslash^0 q^{2} M^a_\arcT
	\end{tikzcd}
	\right]
	\qquad 
	M^a_-=
	\left[
	\begin{tikzcd}[column sep=1cm]
		h^{-1} \hslash^0 q^{-2} M^a_\arcT
		\arrow{r}{d^a_-}
		&
		 h^{0}\hslash^1q^{-1} M^a_\arcD
	\end{tikzcd}
	\right]
\end{equation} 
More precisely, following \cref{def:shifted-mfs},
\begin{equation}
	\label{eq:mcs-precise-maps}
	M^a_+
	= 
	\Cone{f^a_+\co h^1q^1M^a_\arcD \to h^1q^2 M^a_\arcT}
	\qquad 
	M^a_-
	=
	\Cone{f^a_-\co \hslash^1q^{-2}M^a_\arcT \to \hslash^1q^{-1} M^a_\arcD}
\end{equation}
With every edge \(v\rightarrow v'\) we associate the map
\[
d^{v}_{v'}(\Diag)
\co
M_v(\Diag)
\rightarrow
M_{v'}(\Diag)
\]
given by \(d^a_{\lab(a)}\) from 
\eqref{eq:mcs} 
in the component corresponding to the node \(a=a^v_{v'}\) 
tensored with the identity maps in all other components.

\begin{proposition}\label{prop:mf_is_cubical}
	For any labelled diagram \(\Diag\) with \(n\) crossings, 
	\(M(\Diag)\) is strictly \(n\)-cubical
	with \(((M(\Diag))_v,d^v_v)=M_v(\Diag)\) and 
	differentials along edges \(v\rightarrow v'\) equal to \(\pm d^{v}_{v'}(\Diag)\). 
\end{proposition}
\begin{proof}
	Immediate from the definitions. 
	Note that the signs along the edge maps 
	come from the Koszul sign rule for the tensor product. 
\end{proof}

\subsection{Delooping the cube of resolutions of matrix factorizations}
\label{sec:no_wrapping_around_special:delooping_mfs}

From now, we let \(T\) denote a pointed Conway tangle 
and \(\Diag_T=(D_T,\lab_T)\) a labelled diagram associated with it. 

\newcommand{\AssumptionT}{\(\mathrm{(A}T\mathrm{)}\)}
\newcommand{\AssumptionD}{\(\mathrm{(A}\Diag\mathrm{)}\)}
\begin{assumptions}\label{assums:tangle_diagram}
	We will make the following assumptions about \(T\) and \(\Diag_T\): 
\begin{enumerate}
	\item[\AssumptionT]%
	The tangle ends \(\Endiii\) and \(\ast\) of \(T\) point outwards.
	\item[\AssumptionD] For any complete resolution of \(\Diag_T\), there exists some thick or dotted arc connecting the open components.
\end{enumerate}
\end{assumptions}

Any tangle can be made to satisfy condition \AssumptionT\
by potentially reversing the orientation on some components
and/or applying the rotation \(\rho\) from \cref{subsec:naturalityMCG}.
Moreover, we note:

\begin{lemma}\label{lem:many-good-diagrams}
	For every tangle \(T\) satisfying \AssumptionT,
	there is a diagram \(\Diag_T\) for \(T\) satisfying \AssumptionD. 
\end{lemma}

\begin{proof}
	Take any connected diagram for \(T\). 
	We perform a Reidemeister~II move between the edges \(e_2\) and \(e_3\)
	and then add a dotted arc parallel to the boundary between the edges \(e_1\) and \(e_2\).
	The new diagram satisfies condition \AssumptionD.
\end{proof}

Note that the example from \cref{fig:singular-diagram} satisfies \AssumptionT\ and \AssumptionD.
We expect that these assumptions are not essential for what is to follow.
However, they simplify the proof of \cref{lem:delooping:mf} below. 

\begin{definition}\label{def:diagram:components}
	Given a complete resolution \(\Diag\) of \(\Diag_T\),
	denote the open component 
	connected to the distinguished end by \(o^*\) 
	and the other open component by \(o\). 
	Furthermore, in analogy with \cref{def:connectivity}, let 
	\[\begin{aligned}
		\conn{\Diag}&= \begin{cases}
			\arcD\text{ if the open components in \(\Diag\) connect the ends as in \(\DiagD\) }\\
			\arcT\text{ if the open components in \(\Diag\) connect the ends as in \(\DiagT\) }
		\end{cases} %
	\end{aligned}
	\] 
	We define $o(\Diag)= \Diag_{\conn{\Diag}}$.
\end{definition}

Let \(\diamond\) denote a copy of the boundary edge ring \(\Rd\cong\CoeffRing[x,y,z]\) from \cref{exa:diagrams:basic}.
\begin{definition}\label{def:mf_algebra}
	 Consider the full subcategory of the dg category of matrix factorizations $\MF_{xyz}(\CoeffRing[x,y,z])$ 	generated by the matrix factorizations
	\[
	M_\arcD
	=
	\left[
	\vc{\begin{tikzcd}[column sep=1cm, ampersand replacement=\&]
			\hslash^1q^1\diamond
			\arrow[d,bend left=10, "x"right]
			\\
			\hslash^0q^0\diamond
			\arrow[u,bend left=10, "yz"left]
	\end{tikzcd}}
	\right]
	\quad
	M_\arcT
	=
	\left[
	\vc{\begin{tikzcd}[column sep=1cm, ampersand replacement=\&]
			\hslash^0q^1\diamond
			\arrow[d,bend left=10, "y"right]
			\\
			\hslash^1q^0\diamond
			\arrow[u,bend left=10, "xz"left]
	\end{tikzcd}}
	\right]
	\] 
	and let $\mathcal{A}$ denote the dg algebra of endomorphisms $\operatorname{End}(M_\arcD\oplus M_\arcT)$. 
	Note that $M_\arcD=M(\Diag_\arcD)$ and $M_\arcT=M(\Diag_\arcT)$, which are supported in \(h\)-grading 0.
	\cref{fig:saddles_and_dots_in_A} 
	defines four distinguished homomorphisms 
	\(\sW,\sB,\dW,\dB \in \mathcal{A}\).
	The notation is informed by the isomorphism 
	between the algebras
	$H_0\mathcal{A}$ and $\mathcal{B}$,
	which we make explicit in \cref{sec:quasi-iso}; 
	see in particular \cref{thm:quasi-iso}.
\end{definition}

\begin{figure}[t]
	\centering
	\begin{subfigure}{0.25\textwidth}
		\centering
		\(
		\begin{tikzcd}[ampersand replacement=\&]
			\diamond
			\arrow[d,bend left=10, "y"right]
			\arrow[in=180,out=0,looseness=0.7,leftarrow]{rd}[inner sep=1pt, outer sep=1pt, description,fill=white, pos=0.3]{z}
			\&
			\diamond
			\arrow[d,bend left=10, "x"right]
			\\
			\diamond
			\arrow[u,bend left=10, "xz"left]
			\arrow[in=180,out=0,looseness=0.7,leftarrow]{ru}[inner sep=1pt, outer sep=1pt, description,fill=white, pos=0.3]{1}
			\&
			\diamond
			\arrow[u,bend left=10, "yz"left]
		\end{tikzcd}
		\)
		\caption{\(\sW\)}
	\end{subfigure}%
	\begin{subfigure}{0.25\textwidth}
		\centering
		\(
		\begin{tikzcd}[ampersand replacement=\&]
			\diamond
			\arrow[d,bend left=10, "x"right]
			\arrow[in=180,out=0,looseness=0.7,leftarrow]{rd}[inner sep=1pt, outer sep=1pt, description,fill=white, pos=0.3]{z}
			\&
			\diamond
			\arrow[d,bend left=10, "y"right]
			\\
			\diamond
			\arrow[u,bend left=10, "yz"left]
			\arrow[in=180,out=0,looseness=0.7,leftarrow]{ru}[inner sep=1pt, outer sep=1pt, description,fill=white, pos=0.3]{1}
			\&
			\diamond
			\arrow[u,bend left=10, "xz"left]
		\end{tikzcd}
		\)
		\caption{\(\sB\)}
	\end{subfigure}%
	\begin{subfigure}{0.25\textwidth}
		\centering
		\(
		\begin{tikzcd}[ampersand replacement=\&]
			\diamond
			\arrow[d,bend left=10, "x"right]
			\arrow[leftarrow]{r}[inner sep=1pt, outer sep=1pt, description,fill=white]{y}
			\&
			\diamond
			\arrow[d,bend left=10, "x"right]
			\\
			\diamond
			\arrow[u,bend left=10, "yz"left]
			\arrow[leftarrow]{r}[inner sep=1pt, outer sep=1pt, description,fill=white]{y}
			\&
			\diamond
			\arrow[u,bend left=10, "yz"left]
		\end{tikzcd}
		\)
		\caption{\(\dW\)}
	\end{subfigure}%
	\begin{subfigure}{0.25\textwidth}
		\centering
		\(
		\begin{tikzcd}[ampersand replacement=\&]
			\diamond
			\arrow[d,bend left=10, "y"right]
			\arrow[leftarrow]{r}[inner sep=1pt, outer sep=1pt, description,fill=white]{x}
			\&
			\diamond
			\arrow[d,bend left=10, "y"right]
			\\
			\diamond
			\arrow[u,bend left=10, "xz"left]
			\arrow[leftarrow]{r}[inner sep=1pt, outer sep=1pt, description,fill=white]{x}
			\&
			\diamond
			\arrow[u,bend left=10, "xz"left]
		\end{tikzcd}
		\)
		\caption{\(\dB\)}
	\end{subfigure}%
	\caption{Four elements of \(\A\) generating the algebra \(H_0\mathcal{A}\), where each $\diamond$ is a copy of $\Rd\cong\CoeffRing[x,y,z]$. Morphisms are written right-to-left; compare \cref{rmk:right-to-left}}\label{fig:saddles_and_dots_in_A}
\end{figure}
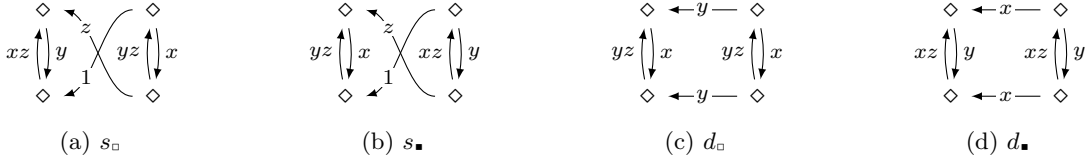

Let $c$ be a closed component of a complete resolution $\Diag_T(v)$ associated with a  tangle diagram $\Diag_T$. Consider the $\CoeffRing$-algebra
\[
	V(c)= q^1 h^0 \hslash^0\CoeffRing[\varG,\mu]/(\mu^2 + \varG \mu)
	\]
where \(\text{gr}(\mu)=h^0\hslash^0q^{-2}\), and notice that $V(c)\cong q^1 \CoeffRing[\varG]\oplus q^{-1}\CoeffRing[\varG]$ supported in $h=\hslash=0$, which arises from the basis $V(c)= 
		\CoeffRing[\varG]
		\big\langle
		1,\mu
		\big\rangle$. 
		These objects are the algebraic ingredients necessary for delooping a matrix factorization $M_v(\Diag_T)$, in the sense of Bar-Natan; see \cref{lem:delooping}.

	\begin{definition}\label{def:delooping-algebra}
	If \(|\Diag_T(v)|\) has \(m\) closed components labelled \(c_1, \dots, c_m\) write 	\[
	\bigotimes_{i=1}^m V(c_i)
	=
	q^{m}h^0\hslash^0\CoeffRing[\varG,\mu_1,\ldots,\mu_m]/(\mu_i^2 + \varG \mu_i)_{i=1,\dots,m}
	\]
	and set $V(\Diag)
	=
	\bigotimes_{i=1}^m V(c_i)$
	for $\Diag=\Diag_T(v)$, where the tensor product is taken over \(\CoeffRing[\varG]\). 
	\end{definition}
	
	If \(\Diag'=\Diag_T(v')\) is a second complete resolution of \(\Diag_T\) with closed components \(c'_1,\dots, c'_{m'}\) for some integer \(m'\), equip
	\begin{equation}\label{eqn:mf:only_circles}
		\Mor_{\CoeffRing[\varG]}\big(V(\Diag), V(\Diag')\big)
		= 
		\Big(
		\bigotimes_{i=1}^{m}
		V^*(c_i)
		\Big)
		\otimes_{\CoeffRing[\varG]}
		\Big(
		\bigotimes_{j=1}^{m'}
		V(c'_j)
		\Big)
	\end{equation}
	with the basis induced by the following bases on the tensor factors:
	\begin{align*}
		V^*(c_i)
		= 
		\CoeffRing[\varG]
		\big\langle
		1^*_i,\mu^*_i
		\big\rangle 
		\qquad \quad
		V(c'_j)
		= 
		\CoeffRing[\varG]
		\big\langle
		1_j,\mu_j
		\big\rangle
	\end{align*}
	Given an edge \(v\rightarrow v'\) in \(\{0,1\}^n\),
	 the diagrams \(\Diag\) and \(\Diag'\) 
	 away from the node \(a^v_{v'}\) 
	 gives rise to a identification 
	 of the closed components $c_i=c'_{\sigma(i)}$ 
	 that avoid \(a^v_{v'}\), 
	 where $\sigma$ is some bijection 
	 between subsets of \(\{1,\dots,m\}\) and \(\{1,\dots,m'\}\). 
	 To ease notation, 
	 we will write
	\[
 I
	=\bigotimes_{\ell} 
	(1_{\sigma(\ell)}1^*_\ell + \mu_{\sigma(\ell)} \mu^*_\ell)
	\in
	\Big(
	\bigotimes_{\ell}
	V^*(c_\ell)
	\Big)
	\otimes_{\CoeffRing[\varG]}
	\Big(
	\bigotimes_{\ell}
	V(c'_{\sigma(\ell)})
	\Big)
	\]

\begin{definition}\label{def:induced_saddle_maps:mf} For every edge \(v\rightarrow v'\) in \(\{0,1\}^n\), we define a map of matrix factorizations over~$\Rd$ 
	\[
	\mathcal{S}^{v}_{v'}
	\co
	V(\Diag)
	\otimes_{\CoeffRing[\varG]}
	M(o(\Diag))
	\rightarrow
	V(\Diag')
	\otimes_{\CoeffRing[\varG]}
	M(o(\Diag'))
	\] according to three cases: 
	If  \(\conn{\Diag}\neq\conn{\Diag'}\) set
	\[
	\mathcal{S}^{v}_{v'}=
	\begin{cases*}
		I\otimes \sW
		&
		if \(\conn{\Diag}=\arcD\) and \(\conn{\Diag'}=\arcT\)
		\\
		I\otimes \sB
		&
		if \(\conn{\Diag}=\arcT\) and \(\conn{\Diag'}=\arcD\)
	\end{cases*}
	\]
	On the other hand, when \(\conn{\Diag}=\conn{\Diag'}=\operatorname{x}\), if \(\Diag'\) is obtained from \(\Diag\) by merging two components set
	\[
	\mathcal{S}^{v}_{v'}=
	\begin{cases*}
		\big(
		1_k 1_i^*1_j^* + \mu_k \mu_i^*1_j^* + \mu_k 1_i^*\mu_j^* -\varG\cdot \mu_k \mu_i^*\mu_j^*
		\big) 
		\otimes 
		I
		\otimes 
		1
		&
		if \(c_i\) and \(c_j\) merge to \(c_k\)
		\\
		1_i^* 
		\otimes
		I
		\otimes 
		\iW
		+ 
		\mu_i^*
		\otimes
		I
		\otimes 
		\dW
		&
		if \(c_i\) merges with \(o\)  and \(\operatorname{x}=\arcD\)
		\\
		1_i^* 
		\otimes
		I
		\otimes 
		\iB
		+ 
		\mu_i^*
		\otimes
		I
		\otimes 
		\dB
		&
		if \(c_i\) merges with \(o\)  and \(\operatorname{x}=\arcT\)
		\\
		1_i^*
		\otimes 
		I
		\otimes
		1 
		&
		if \(c_i\) merges with \(o^*\)
	\end{cases*}
	\]
	and if \(\Diag'\) is obtained from \(\Diag\) by splitting two components set
	\[
	\mathcal{S}^{v}_{v'} =
	\begin{cases*}
		\big(
		\mu_i1_j1_k^* 
		+ 
		1_i\mu_j 1_k^*
		+
		\varG\cdot 1_i1_j1_k^*
		+
		\mu_i\mu_j\mu_k^*
		\big)
		\otimes
		I
		\otimes 
		1
		&
		if \(c_k\) splits into \(c_i\) and \(c_j\)
		\\
		\mu_i
		\otimes 
		I
		\otimes
		\iW
		+
		1_i
		\otimes
		I 
		\otimes
		\sB\sW
		&
		if \(c_i\) splits from \(o\) and \(\operatorname{x}=\arcD\)
		\\
		\mu_i
		\otimes 
		I
		\otimes
		\iB
		+
		1_i
		\otimes
		I 
		\otimes
		\sW\sB
		&
		if \(c_i\) splits from \(o\) and \(\operatorname{x}=\arcT\)
		\\
		\mu_i
		\otimes 
		I
		\otimes
		1 
		+
		\varG
		\cdot
		1_i
		\otimes 
		I
		\otimes
		1 
		&
		if \(c_i\) splits from \(o^*\)
	\end{cases*}
	\]
\end{definition}

\begin{lemma}\label{lem:delooping:mf}
	For every \(v\in\{0,1\}^n\)
	there exists a special deformation retract
	\[
	\varphi_v
	\co
	M(\Diag_T(v))
	\rightarrow
	\hslash^{\mathtt{i}_v}
	V(\Diag_T(v))
	\otimes_{\CoeffRing[\varG]}
	M(o(\Diag_T(v)))
	\]
	of matrix factorizations over \(\Rd\) 
	with right inverse \(\psi_v\) 
	where 
	\[
	\mathtt{i}_v
	= 
	\#\{\text{thick arcs in \(\Diag_T(v)\)}\}
	+
	\#\{\text{closed components of \(\Diag_T(v)\)}\}
	+ 
	\begin{cases*}
		1
		&
		if \(\conn{\Diag_T(v)}=\arcT\)
		\\
		0
		&
		if \(\conn{\Diag_T(v)}=\arcD\)
	\end{cases*}
	\]
	such that for every edge \(v\rightarrow v'\) in \(\{0,1\}^n\) 
	the composition 
	\(\varphi_{v'}\circ \boldnowarningf \circ \psi_v\) 
	is chain homotopic to \(\mathcal{S}^{v}_{v'}\) (up to sign), where $\boldnowarningf$ is the  map $f_\pm^a$ from \cref{eq:mcs-precise-maps} tensored appropriately with the requisite identity factors and $a=a^{v}_{v'}$ is the node corresponding to the entry where $v$ and $v'$ differ.
\end{lemma}

This lemma is well-known in the setting of closed link diagrams. 
However we require a version for tangles as well as some control over morphisms induced by this process. 
The latter represents the bulk of the work. 
We defer the proof of \cref{lem:delooping:mf} to \cref{sec:delooping:mf}, as it is technically involved and the details are irrelevant to the rest of this section. 

Observe that, for any fixed diagram \(\Diag_T\), 
the quantity \(\mathtt{i}=\mathtt{i}_v\) 
is constant modulo 2.  

\begin{theorem}\label{thm:delooping:mf}
The matrix factorization \(M(\Diag_T)\) 
is chain homotopy equivalent, 
as a matrix factorization over \(\Rd\), 
to an \(n\)-cubical matrix factorization \((M^s(\Diag_T),d^s)\). 
Moreover, 
writing \((M^s(\Diag_T))_v= M^s_v(\Diag_T)\) 
for every vertex \(v\in\{0,1\}^n\), 
we have 
\[
M^s_v(\Diag_T) 
= 
h^{|v|-n_-}
\hslash^{|v|+n_++\mathtt{i}}
q^{|v|+n_+-2n_-}
V(\Diag_T(v))
\otimes_{\CoeffRing[\varG]}
M(o(\Diag_T(v))) 
\] 
and, 
if \(v\) and \(v'\) are related by an edge \(v\rightarrow v'\), 
the component \((d^s)^v_{v'}\) of the differential 
is chain homotopic 
to the homomorphism \(\pm\mathcal{S}^{v}_{v'}\) 
(considered as a homomorphism 
between the appropriately shifted matrix factorizations).
\end{theorem}

\begin{proof}
	We iteratively construct 
	filtered matrix factorizations \(M_k\) 
	for \(k=-1,0,\dots,n\), 
	where 
	\(M_{-1}=M(\Diag_T)\) and 
	\(M_{n}=M^s(\Diag_T)\), 
	with the following properties:
	\begin{itemize}
		\item 
		For each \(k=-1,0,\dots,n\),
		\(M_k\) is \(n\)-cubical with 
		\[
		(M_k)_v
		=
		\begin{cases}
			M^s_v(\Diag_T)
			& |v|\leq k
			\\
			M_v(\Diag_T)
			& \text{otherwise}.
		\end{cases}
		\]
		\item If \(v\rightarrow v'\) is an edge of the cube, 
		the restriction of the differential of \(M_k\) 
		to \((M_k)_v\rightarrow (M_k)_{v'}\) 
		is equal to
		\[
		\begin{cases}
			\varphi_{v'}\circ d^v_{v'}\circ \psi_v
			& |v|<k
			\\
			d^v_{v'}\circ \psi_v
			& |v|=k
			\\
			d^v_{v'}
			& |v|>k
		\end{cases}
		\]
		where \(\varphi_v\) and \(\psi_v\) are the maps 
		from \cref{lem:delooping:mf} and \(d^v_{v'}\) is the restriction of the differential of \(M(\Diag_T)\) 
		to \(M_v(\Diag_T)\rightarrow M_{v'}(\Diag_T)\).
		\item \(M_k\) is a special deformation retract of \(M_{k-1}\) 
		for each \(k=0,\dots,n\).
	\end{itemize} 
First note that the restriction of the differential of \(M(\Diag_T)\) to \(M_v(\Diag_T)\rightarrow M_{v'}(\Diag_T)\) is zero unless \(v\rightarrow v'\) is an edge of the cube or \(v=v'\) since \(M(\Diag_T)\) is $n$-cubical (compare \cref{prop:mf_is_cubical}).
	In particular, \(M_{-1}=M(\Diag_T)\) satisfies the above properties. Now suppose we have constructed \(M_{k-1}\) for a fixed \(k\leq n\). 
	We then apply \cref{lem:sdr-new-from-old} to \(M_{k-1}\) 
	and the special deformation retraction
	\[
	\sum_{k=|v|}\varphi_v
	\co
	\bigoplus_{k=|v|}
	M_v(\Diag_T)
	\longrightarrow
	\bigoplus_{k=|v|}
	M^s_v(\Diag_T)
	\]
	to obtain \(M_{k}\) 
	with the above properties. %
	In particular, the matrix factorization \(M_n\) 
	is chain homotopy equivalent to \(M(\Diag_T)\). 
	It is \(n\)-cubical with \((M_n)_v=M^s_v(\Diag_T)\) 
	for all \(v\in\{0,1\}^n\) and 
	moreover, the differentials 
	along edges \(v\rightarrow v'\) 
	are given by
	\(\varphi_{v'}\circ d^v_{v'}\circ \psi_v\). 
	In the notation of \cref{lem:delooping:mf},
	these maps are equal to 
	\(\varphi_{v'}\circ \boldnowarningf \circ \psi_v\)
	up to grading shifts of the source and target.
	The claim then follows from that lemma.
\end{proof}

\begin{remark}
	The representative \(M^s(\Diag_T)\) 
	of the matrix factorization \(M(\Diag_T)\) 
	associated with the tangle diagram $\Diag_T$ 
	plays the role of a {\it fully delooped} complex in \(\Cobb\);
	compare \cref{lem:delooping}.
	We will return to it in \cref{sec:no_wrapping_around_special:mf2twcx}.
\end{remark}
\section{A quasi-isomorphism of algebras from homological mirror symmetry}
\label{sec:quasi-iso}

In this section, 
we recall a quasi-isomorphism 
that is well-known to experts 
and has been studied in the context of homological mirror symmetry 
\cite{AAEKO,Bocklandt,HKK}. 
It relates an \(A_\infty\)-structure on the algebra \(\B\), 
denoted \(\B^\infty\) and described presently, 
to the endomorphism algebra \(\A\) 
of matrix factorizations over \(\field[x,y,z]\) with potential \(xyz\) 
described in \cref{def:mf_algebra}. 
This quasi-isomorphism underpins 
the homological mirror symmetry 
of the sphere with three punctures 
in the following sense: 

\begin{theorem}[Abouzaid {\it et al} \cite{AAEKO}]\label{thm:hms}
	There is a quasi-isomorphism between the $A_\infty$-algebras \(\A\) and $\B^\infty$, described in \Cref{thm:quasi-iso}, inducing an equivalence of triangulated categories 
	\[
	\HMS
	\co 
	\operatorname{Tw}\A
	\to 
	\operatorname{Tw}\Binf
	\] 
\end{theorem}

Here,
\(\operatorname{Tw}\A\) 
and 
\(\operatorname{Tw}\Binf\)
denote  
the categories of twisted complexes 
over \(\A\) and \(\Binf\), respectively;
see \cref{def:twisted_complex}.

\subsection{Preliminaries on \texorpdfstring{\(A_\infty\)}{A\_∞}-structures}\label{subsec:conventions:A-infty}

An \(A_\infty\)-category \(\C\) over \(\CoeffRing\) 
is a category in which composition is associative 
only up to a system of higher homotopies. 
More precisely,
for each pair of objects \(X,Y\) 
we have a \(\Z\)-graded \(\CoeffRing\)-module \(\C(X,Y)\), 
and the structure maps
\[
\mu_n\colon 
\C(X_{n-1},X_n)\otimes\cdots\otimes\C(X_0,X_1)
\longrightarrow \C(X_0,X_n) %
\]
have degree \(2-n\) for \(n\ge1\). 
Here \(\mu_1\) is a differential, 
\(\mu_2\) is the (not strictly associative) composition, 
and units \(1_X\in\C(X,X)\) 
satisfy the usual identities with respect to~\(\mu_2\).
The higher compositions satisfy the \(A_\infty\)-relations
\begin{equation*}\label{eq:a-infty-relations-Seidel}
	0
	=
	\sum_{\substack{i,k\geq0,j\geq1:\\ i+j+k=n}}
	(-1)^{\maltese}
	\mu_{i+1+k} (a_n,\dots,a_{n-i+1},\mu_j(a_{n-i},\dots,a_{k+1}),a_{k},\dots,a_1)
\end{equation*}
for all composable sequences \(a_1,\dots,a_n\), 
where \(\maltese=h(a_k)+\dots+h(a_1)+k\) and \(h\) denotes the \(\Z\)-grading.
An \(A_\infty\)-algebra is the special case of a single object.  

\begin{remark}
	Several sign conventions appear in the literature.  
	We follow the sign conventions in \cite{Seidel}.
	Other works, including
	\cite{Keller,Huebschmann,AAEKO}, 
	use %
	\begin{equation*}\label{eq:a-infty-relations}
		0
		=
		\sum_{\substack{i,k\geq0,j\geq1:\\ i+j+k=n}}
		(-1)^{i+jk}
		\mu_{i+1+k}\circ (1^{\otimes i}\otimes \mu_j\otimes 1^{\otimes k}).
	\end{equation*}
	The two conventions are different but give rise to equivalent notions of \(A_\infty\)-categories. 
	More explicitly, expanding the Koszul signs and replacing
	\[
	\mu_n(a_n,\dots,a_1)
	\longmapsto
	(-1)^{n\cdot h(a_n)+(n-1)\cdot h(a_{n-1})+\cdots+h(a_1)}\,\mu_n(a_n,\dots,a_1)
	\]
	transforms one sign convention into the other, up to a global factor common to all summands; 
	see also \cite{Polishchuk:fieldguide}.  
\end{remark}

An \(A_\infty\)-functor \(\F\colon\C\to\C'\) 
consists of a map on objects together
with multilinear maps
\[
\F^n\colon
\C(X_{n-1},X_n)\otimes\dots\otimes\C(X_0,X_1)
\longrightarrow
\C'(\F(X_0),\F(X_n))
\]
of degree \((1-n)\), 
satisfying the following identity for every integer \(n\geq1\):
\begin{align*}
	&
	\sum_{\substack{r,i_1,\dots,i_r\geq1:\\ i_1+\dots+i_r=n}}
	\mu_{r}\big(\F^{i_r}(a_n,\dots,a_{n-i_r+1}),\dots,\F^{i_1}(a_{i_1},\dots,a_1)\big)
	\\
	=
	&
	\sum_{\substack{i,k\geq0,j\geq1:\\ i+j+k=n}}
	(-1)^{\maltese}
	\mathcal{F}_{i+1+k} (a_n,\dots,a_{n-i+1},\mu_j(a_{n-i},\dots,a_{k+1}),a_{k},\dots,a_1)%
\end{align*}
A morphism of $A_\infty$-algebras is a functor between $A_\infty$-algebras.
Given an \(A_\infty\)-category \(\C\) over \(\CoeffRing\), 
its homology \(H\C\) is the graded \(\CoeffRing\)-linear category
with the same objects and
\[
H\C(X,Y)=H(\C(X,Y),\mu_1)%
\]
with composition induced by \(\mu_2\).
Every \(A_\infty\)-functor \(\mathcal{F}\co\C\rightarrow\C'\)
induces a functor \(H\mathcal{F}\co H\C\rightarrow H\C'\); 
the functor \(\F\) is a quasi-isomorphism if
\(H\F\) is an isomorphism of categories.

If \(\C\) is a graded \(\CoeffRing\)-linear category, an
\(A_\infty\)-structure on \(\C\) is a system \((\mu_k)_{k\ge2}\) as above such
that \(H(\C,\mu)=\C\).
Following \cite[Section~2]{AAEKO}, two \(A_\infty\)-structures on \(\C\) are
strictly homotopic if there exists an \(A_\infty\)-functor between the
associated \(A_\infty\)-categories that acts as the identity on objects and on
individual morphisms; such a functor is necessarily a quasi-isomorphism
inducing the identity on homology.

\begin{definition}\label{def:additive_enlargement_Seidel}
	Given an $A_\infty$-category~\(\mathcal{A}\)
	over a commutative ring \(\CoeffRing\), 
	the additive enlargement \(\Sigma\mathcal{A}\) 
	is the $A_\infty$-category defined as follows. 
	Its objects are formal direct sums
	\[ \bigoplus_{i\in I}h^{r_i}X_i \]
	where 
	\( I \) is a finite index set and 
	\( h^{r_i}X_i \) denotes the object 
	\(X_i\) shifted by 
	\( r_i \) in homological grading. 
	The morphism spaces of degree \(n\) are
	\begin{equation*}\label{eq:morphism_spaces_in_Mat}
		\Mor\bigg(\bigoplus_{i\in I}h^{r_i}X_i,\bigoplus_{j\in J}h^{r_j}X_j\bigg)
		=
		\bigoplus_{(i,j) \in I\times J} 
		h^{r_j-r_i}\Mor(X_i,X_j)
	\end{equation*}
	as \(\CoeffRing\)-modules. 
	The structure maps \(\mu_d^{\Sigma\mathcal{A}}\)
	are given by 
	\begin{align*}
		\Big(
		\mu_d^{\Sigma\mathcal{A}}
		(a_d,\dots,a_p=
		(
		h^{r_\ell}X_\ell
		\xrightarrow{a_p^{k,\ell}}
		h^{r_k}X_k
		)_{k,\ell},\dots,a_1)
		\Big)_{i_d,i_0}
		&
		=
		\smashoperator[r]{\sum_{i_1,\dots,i_{d-1}}}
		(-1)^{\triangleleft}
		\mu_d^{\mathcal{A}}(a_d^{i_d,i_{d-1}},\dots, a_1^{i_1,i_{0}})
	\end{align*}
	where 
	\(
	\triangleleft
	=
	\sum_{p<q}
	(r_{i_p}-r_{i_{p-1}})
	\cdot
	(h(a_q^{i_q,i_{q-1}})-1)
	\)
	\cite[Equation (3.17)]{Seidel}.
\end{definition}

\begin{definition}\label{def:twisted_complex}
	A twisted complex 
	over an $A_\infty$-category~\(\mathcal{A}\) 
	is a pair \((X,\delta)\), 
	where \(X\in\Sigma\A\)
	and \(\delta\) is a strictly lower-triangular endomorphism of \(X\) of degree 1
	satisfying
	\[
	\sum_{i=1}^\infty\mu^{\Sigma\A}_i(\delta,\dots,\delta)=0.
	\]
\end{definition}

\begin{remark}\label{rem:twisted_complex}
	Note that in the special case 
	that the algebra \(\A\) is supported in even homological gradings only 
	\(\mu^\A_d=0\) for odd \(d\). 
	Moreover, 
	the formula for \(\triangleleft\) simplifies to 
	\[
	\triangleleft
	=
	\sum_{p=1}^d p\cdot (r_{i_p}-r_{i_{p-1}}).
	\]
	Thus, the structure relation for a twisted complex \((X,\delta)\) translates to checking that 
	\[
	\sum(-1)^i\mu^{\A}_{2i}(\delta_{2i},\dots,\delta_1)=0
	\]
	for all pairs of objects \(X_j,X_k\in\A\) in \(X\)
	where the sum is over all possible paths \(\delta_{2i},\dots,\delta_1\) in the differential \(\delta\) from \(X_j\) to \(X_k\). 
\end{remark}

Twisted complexes over \(\mathcal{A}\) form an \(A_\infty\)-category,
denoted \(\Tw\mathcal{A}\);
for details, see \cite[Section~I.3l]{Seidel}.
We follow Seidel's conventions 
for grading shifts of twisted complexes 
\cite[Section~I.3o]{Seidel}. 
Specialized to $A_\infty$-categories
supported in even homological grading, 
these read 
\[
h^i(X,\delta)=(h^i X, (-1)^i\delta)
\]
for any twisted complex \((X,\delta)\) and \(i\in\Z\); 
compare with \cref{def:shifted-mfs}.

\begin{remark}
	\(A_\infty\)-categories in this section  may 	carry one or several of the following three gradings:
	a \(\Z\)-grading \(h\), 
	a \(\ZmodTwo\)-grading \(\hslash\),
	and a \(\Z\)-grading \(q\). 
	The gradings \(h\) and \(\hslash\) serve as homological gradings,
	as in the definitions above. 
	If both are present, 
	\(\hslash\) is the modulo 2 reduction of \(h\).
	The \(h\)-grading used in this section 
	should be regarded as internal to this section 
	and unrelated to the \(h\)-grading appearing in 
	\cref{sec:mfs,sec:no_wrapping_around_special}. 
	Only the \(\hslash\)- and \(q\)-gradings will be used elsewhere.
	The quantum grading \(q\) decreases by 3 
	along the differential of any twisted complex. 
	More generally, 
	the structure maps \(\mu_d\) have quantum degree \(3(d-2)\).
	A functor \(\F\) 
	between quantum-graded \(A_\infty\)-categories is called 
	\(q\)-filtered if
	\(q(\F^d)=3(d-1)\). 
\end{remark}

\subsection{An 
	\texorpdfstring{\(A_\infty\)}{A\_∞} 
	structure on the algebra 
	\texorpdfstring{\(\B\)}{B}
	and the main result
}
\label{sec:quasi-iso:Binf}

\begin{definition}\label{def:Binf_algebra} Disk sequences \(\{E_{2m}\}_{m>1}\) are sequences of pure algebra elements in \(\B\), defined inductively. Let
	\[
	E_4=
	\{({}_{\circ}\Sb,{}_{\bullet}\Db,{}_{\bullet}\Sw,{}_{\circ}\Dw),~ ({}_{\bullet}\Db,{}_{\bullet}\Sw,{}_{\circ}\Dw,{}_{\circ}\Sb), ~ ({}_{\bullet}\Sw,{}_{\circ}\Dw,{}_{\circ}\Sb,{}_{\bullet}\Db), ~({}_{\circ}\Dw,{}_{\circ}\Sb,{}_{\bullet}\Db,{}_{\bullet}\Sw) \}
	\]
	so that, given a set  \(E_{2m}\) of disk sequences of length \(2m\geq4\), the elements of the set \(E_{2m+2}\) are constructed by interposing sequences from $E_{4}$ into any disk sequence in \(E_{2m}\) as follows:
	\begin{align*}
		(\ldots,D^k_\sol,{}_\sol S^\ell,\ldots) &\mapsto (\ldots,D^{k+1}_\sol,{}_\sol S_\hol,{}_\hol D_\hol,{}_\hol S^{\ell+1},\ldots) \\  (\ldots,D^k_\hol,{}_\hol S^\ell,\ldots) &\mapsto (\ldots,D^{k+1}_\hol,{}_\hol S_\sol,{}_\sol D_\sol,{}_\sol S^{\ell+1},\ldots)\\ 
		(\ldots,S^k_\sol,{}_\sol D^\ell,\ldots) &\mapsto (\ldots,S^{k+1}_\hol,{}_\hol D_\hol,{}_\hol S_\sol,{}_\sol D^{\ell+1},\ldots)  \\ (\ldots,S^k_\hol,{}_\hol D^\ell,\ldots) &\mapsto (\ldots,S^{k+1}_\sol,{}_\sol D_\sol,{}_\sol S_\hol,{}_\hol D^{\ell+1},\ldots) 
	\end{align*}
\end{definition}
These disk sequences correspond to the geometric disk sequences from~\cite[Section~3.3]{HKK}. 
The same disk sequences can arise in different ways.
The following is essentially a reformulation of \cite[Lemma~3.1]{HKK}.

\begin{lemma}\label{lem:disk-sequences-unique-factorization}
	Suppose a sequence of pure algebra elements fits into one of the three patterns
	\[
	(a_\ell,\dots,a_1b),
	\quad
	(ba_\ell,\dots,a_1),
	\quad
	\text{and}
	\quad
	(a_\ell,\dots,a_1)
	\]
	with \((a_\ell,\dots,a_1)\) a disk sequence and \(b\) a pure algebra element that is not an idempotent. 
	Then the sequence does not fit into either of the other two patterns. 
	Moreover, the disk sequence \((a_\ell,\dots,a_1)\) and the algebra element \(b\) are unique. 
	\qed
\end{lemma}

For each disk sequence \((a_{2m},\ldots,a_{1}) \in E_{2m}\) define 
\begin{equation}\label{eq:Binf:higher-multiplications}
\mu^{\Binf}_{2m}(a_{2m},\ldots,a_{1}b)=(-1)^{m-1} \cdot b
\qquad\text{and}\qquad 
\mu^{\Binf}_{2m}(b a_{2m},\ldots,a_{1})=(-1)^{m-1} \cdot b
\end{equation}
for all \(b\in\{1,S^r,D^r\mid r>1\}\) such that \(a_{1}b \neq 0\) and \(ba_{2m}\neq 0\), respectively.
\Cref{lem:disk-sequences-unique-factorization} ensures that these expressions are well-defined.
We extend them multilinearly to maps 
\[
\mu^{\Binf}_{k}
\co
\B^{\otimes k}
\rightarrow
\B
\]
where \(\mu^{\Binf}_{k}=0\) for odd integers \(k\). 

\begin{lemma}\label{lem:a-infty-reltations-for-Binf}
	The set of maps \(\mu^{\Binf}_{k}\) satisfy the \(A_\infty\)-relations. 
\end{lemma}
\begin{proof}
	Define maps \(\nu_n\) for \(n>0\) by
	\(
	\nu_n 
	=
	(-1)^{m-1}\cdot\mu^{\Binf}_{n}
	\).
	These maps
	satisfy the \(A_\infty\)-relations by
	\cite[Theorem~5.3, Part (1)]{Bocklandt}. 
	Indeed, 
	the maps \((\nu_n)_{n>0}\) agree with the \(A_\infty\)-operations considered in 
	\cite{Bocklandt}
	when \(Q\) is the quiver from the definition of the algebra \(\B\) embedded into the three-punctured sphere. 
	Since \(\mu^{\Binf}_{n}=0\) for odd integers \(n\),
	the \(A_\infty\)-relations for \((\mu^{\Binf}_{n})_{n>0}\)
	are trivially satisfied on even length sequences of algebra elements. 
	Each term of a given odd length sequence is of the form 
	\[
	\mu^{\Binf}_{2r}(\dots,\mu^{\Binf}_{2s}(\dots),\dots)
	=
	(-1)^{r+s}\cdot\nu_{2r}(\dots,\nu_{2s}(\dots),\dots).
	\]
	The length of the input sequence is \(2r+2s-1\), 
	hence \(r+s\) is constant. 
	Thus,
	the maps \(\mu^{\Binf}_{k}\) also satisfy the \(A_\infty\)-relations. 

	For an alternative proof strategy that is more geometric, see \cite[Proposition~3.1]{HKK}.
\end{proof}

\begin{definition}
	We define the \(\hslash\)-graded \(A_\infty\)-algebra \(\Binf\) 
	to be the algebra \(\B\) 
	together with the higher multiplications \(\mu^{\Binf}_{k}\)
	and the trivial \(\hslash\)-grading
	i.e.\ the \(\ZmodTwo\)-grading which is identically equal to 0. 
	We also define a \(h\)-grading by 
	\(
	h({}_{\circ}\Sb)
	=
	h({}_{\bullet}\Db)
	=
	h({}_{\bullet}\Sw)
	=0
	\)
	and \(h({}_{\circ}\Dw)=2\).
\end{definition}

\begin{remark}
	The grading \(\hslash\) is equal to the grading \(h\) modulo 2. 
	The grading \(h\) on \(\Binf\) will only be needed within this section. 
	The choice of this lift of \(\hslash\) is somewhat arbitrary:
	Any \(\Z\)-grading on \(\B\) 
	for which the sum of the gradings 
	of the four algebra generators 
	\({}_{\circ}\Sb\), 
	\({}_{\bullet}\Db\), 
	\({}_{\bullet}\Sw\), and 
	\({}_{\circ}\Dw\) 
	is equal to \(2\) 
	turns \(\B\) into a \(\Z\)-graded \(A_\infty\)-algebra. 
\end{remark}

\begin{remark}
	\label{rem:quantum_grading_on_Binf}
	The \(A_\infty\)-algebra \(\Binf\) also inherits 
	a quantum grading from the algebra \(\B\).  
	The sum of the quantum gradings 
	of all elements in any given disk sequence of length \(2m>0\) 
	is \(6(1-m)\), 
	which can be seen by induction on \(m\). 
	Therefore, 
	\(q(\mu^{\Binf}_{2m})=6(m-1)\).
\end{remark}

\begin{theorem}\label{thm:quasi-iso}
	Let \(\A\) be the algebra from \cref{def:mf_algebra}. There is a \(q\)-filtered quasi-isomorphism
	\(
	\A
	\rightarrow
	\Binf
	\)
	with  
	\[
	M_\arcD\mapsto\circ
	\quad\text{and}\quad
	M_\arcT\mapsto\bullet
	\]
	inducing the following isomorphism on homology:
	\begin{align*}
		[\sW]
		&
		\mapsto
		\Sw
		&
		[\sB]
		&
		\mapsto
		\Sb
		&
		[\dW]
		&
		\mapsto
		\Dw
		&
		[\dB]
		&
		\mapsto
		\Db
	\end{align*}
\end{theorem}

\begin{remark}\label{rem:quasi-iso}
	In \Cref{thm:quasi-iso} both \(\A\) and \(\Binf\) may be viewed
	as \(h\)-graded categories
	or as \(\hslash\)-graded categories.
	For the proof the \(h\)-grading on the algebras is required. 
	However, 
	when we apply \Cref{thm:quasi-iso} in \cref{sec:no_wrapping_around_special:mf2twcx},
	only the \(\hslash\)-grading is needed.
\end{remark}

Two auxiliary algebras play an essential role in the proof of \cref{thm:quasi-iso}.

\begin{definition}
	Let \(\Ccone\) be the path algebra over the quiver
	in \cref{fig:wrapped_fukaya_category}
	modulo the relations
	\begin{equation}\label{eq:relations:C}
		u_{i(i+1)}u_{(i-1)i}=0
		\quad
		\text{and}
		\quad
		v_{i(i-1)}v_{(i+1)i}=0
	\end{equation}
	where \(i=1,2,3\) and indices are taken modulo 3. 
	Equip \(\Ccone\) with an \(h\)-grading determined by
	\begin{align*}
		h(u_{12})
		&
		= 
		0
		&
		h(u_{23})
		&
		= 
		0
		&
		h(u_{31})
		&
		= 
		1
		\\
		h(v_{21})
		&
		= 
		0
		&
		h(v_{32})
		&
		= 
		0
		&
		h(v_{13})
		&
		= 
		1
	\end{align*}
	and define a quantum grading by setting 
	\(q(u_{i(i+1)})=q(v_{i(i-1)})=-1\). 
	Finally, let \(\C\) be the subalgebra of \(\Ccone\) 
	consisting of linear combinations of paths 
	that neither start nor end at \(L_3\).
\end{definition}

\begin{figure}[t]
	\[
	\vc{%
	\begin{tikzpicture}
		\footnotesize
		\draw[ultra thick,dashed] ($(210:2)+(120:0.5)$) .. controls ($(210:2)+(120:0.5)+(30:1)$) and (-0.5,1) .. (-0.5,2);
		\draw[ultra thick] ($(-30:2)+(60:0.5)$) .. controls ($(-30:2)+(60:0.5)+(150:1)$) and (0.5,1) .. (0.5,2);
		\draw[ultra thick] ($(-30:2)+(60:-0.5)$) .. controls ($(-30:2)+(60:-0.5)+(150:1)$) and ($(210:2)+(120:-0.5)+(30:1)$) .. ($(210:2)+(120:-0.5)$);
		
		\node (L1) at (30:0.8) {$L_1$};
		\node (L1) at (-90:0.8) {$L_2$};
		\node (L1) at (150:0.8) {$L_3$};

		\draw[thick,rotate around={60:(-30:2)},fill = white] (-30:2) ellipse (1 and 0.4);
		\draw[thick,rotate around={-60:(210:2)},fill = white] (210:2) ellipse (1 and 0.4);
		\draw[thick,fill = white] (90:2) ellipse (1 and 0.4);

		\draw[rotate around={60:(-30:1.8)},->] (-30:1.8) ++(62:1 and 0.4) arc [start angle=62, end angle=120, x radius=1, y radius=0.4];
		\draw[rotate around={60:(-30:1.8)},->] (-30:1.8) ++(0:1 and 0.4) arc [start angle=0, end angle=50, x radius=1, y radius=0.4];
		\draw[rotate around={60:(-30:1.8)}] (-30:1.8) ++(132:1 and 0.4) arc [start angle=132, end angle=180, x radius=1, y radius=0.4];
		\draw[rotate around={60:(-30:1.8)}] (-30:1.8) ++(200:1 and 0.4) arc [start angle=200, end angle=340, x radius=1, y radius=0.4];

		\draw[rotate around={-60:(210:1.8)},->] (210:1.8) ++(62:1 and 0.4) arc [start angle=62, end angle=120, x radius=1, y radius=0.4];
		\draw[rotate around={-60:(210:1.8)},->] (210:1.8) ++(0:1 and 0.4) arc [start angle=0, end angle=50, x radius=1, y radius=0.4];
		\draw[rotate around={-60:(210:1.8)}] (210:1.8) ++(132:1 and 0.4) arc [start angle=132, end angle=180, x radius=1, y radius=0.4];
		\draw[rotate around={-60:(210:1.8)}] (210:1.8) ++(200:1 and 0.4) arc [start angle=200, end angle=340, x radius=1, y radius=0.4];

		\draw[rotate around={180:(90:1.8)},->] (90:1.8) ++(62:1 and 0.4) arc [start angle=62, end angle=120, x radius=1, y radius=0.4];
		\draw[rotate around={180:(90:1.8)},->] (90:1.8) ++(0:1 and 0.4) arc [start angle=0, end angle=50, x radius=1, y radius=0.4];
		\draw[rotate around={180:(90:1.8)}] (90:1.8) ++(132:1 and 0.4) arc [start angle=132, end angle=180, x radius=1, y radius=0.4];
		\draw[rotate around={180:(90:1.8)}] (90:1.8) ++(200:1 and 0.4) arc [start angle=200, end angle=340, x radius=1, y radius=0.4];
		
		\node[label] (v21) at (-30:2) {$v_{21}$};
		\node (u12) at (-30:1.1) {$u_{12}$};
		\node[label] (v32) at (-150:2) {$v_{32}$};
		\node (u23) at (-150:1.1) {$u_{23}$};
		\node (v13) at (90:2) {$v_{13}$};
		\node (u31) at (90:1.2) {$u_{31}$};

		\draw[thick] ($(210:2)+(120:1)$) .. controls ($(210:2)+(120:1)+(30:1)$) and (-1,1) .. (-1,2);
		\draw[thick] ($(-30:2)+(60:1)$) .. controls ($(-30:2)+(60:1)+(150:1)$) and (1,1) .. (1,2);
		\draw[thick] ($(-30:2)+(60:-1)$) .. controls ($(-30:2)+(60:-1)+(150:1)$) and ($(210:2)+(120:-1)+(30:1)$) .. ($(210:2)+(120:-1)$);
		
	\end{tikzpicture}%
	}
	\hspace{1cm}
	\vc{%
	\begin{tikzpicture}[scale=0.7]
		\draw (  30:2) node (L1) {$L_1$};
		\draw (-90:2) node (L2) {$L_2$};
		\draw ( 150:2) node (L3) {$L_3$};
		
		\draw[->] (40:2.1) arc [start angle=40, end angle=140,radius=2.1];
		\draw[<-] (40:1.9) arc [start angle=40, end angle=140,radius=1.9];
		
		\draw[->] (160:2.1) arc [start angle=160, end angle=260,radius=2.1];
		\draw[<-] (160:1.9) arc [start angle=160, end angle=260,radius=1.9];
		
		\draw[->] (-80:2.1) arc [start angle=-80, end angle=20,radius=2.1];
		\draw[<-] (-80:1.9) arc [start angle=-80, end angle=20,radius=1.9];
		
		\small
		\draw (-30:2.6) node (L3) {$v_{21}$};
		\draw (-30:1.4) node (L3) {$u_{12}$};
		\draw (-150:2.6) node (L3) {$v_{32}$};
		\draw (-150:1.4) node (L3) {$u_{23}$};
		\draw (90:2.4) node (L3) {$v_{13}$};
		\draw (90:1.5) node (L3) {$u_{31}$};
	\end{tikzpicture}%
	}
	\]
	\caption{%
		Three Lagrangians in the wrapped Fukaya category \(\W(S^2_3)\)
		generating the algebra \(\C\) and the corresponding quiver.
	}
	\label{fig:wrapped_fukaya_category}
\end{figure}
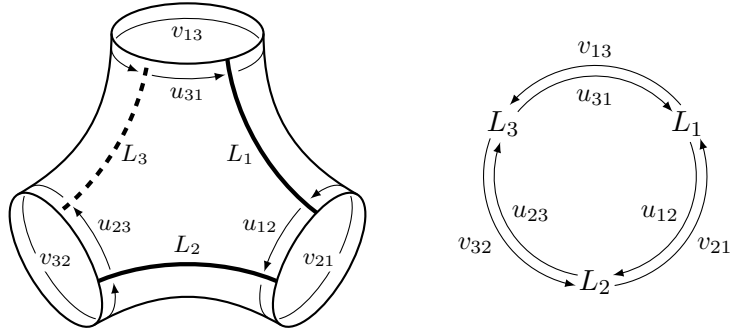

\subsection{The strategy for the proof of \texorpdfstring{\cref{thm:quasi-iso}}{Theorem \ref{thm:quasi-iso}}}
\label{sec:quasi-iso:pair-of-pants}
Denote the wrapped Fukaya category of the three-punctured sphere \(S^2_3\) by \(\W(S^2_3)\).
The (ordinary) algebra \(\Ccone\) is a model for the full subcategory of \(\W(S^2_3)\) 
consisting of the three Lagrangians \(L_1\), \(L_2\), and \(L_3\)
in \cref{fig:wrapped_fukaya_category}; 
see for example \cite{Bocklandt}.
The algebra \(\Ccone\) and its subalgebra \(\C\) inherit \(A_\infty\)-structures from \(\W(S^2_3)\). 

Observe that \(\C\) and \(\B\) are isomorphic as ordinary algebras. 
In fact, the particular \(A_\infty\)-structure on \(\C\) inherited from \(\W(S^2_3)\) makes \(\C\) quasi-isomorphic to \(\Binf\), which can be seen using results from \cite{HKK}. 
This is, essentially, because \(\{L_1,L_2\}\) constitutes a full arc system for \(S^2_3\) and \(L_3\) can be viewed as a mapping cone of \(L_1\) and \(L_2\). 
By \cite[Proposition 7.12]{AAEKO} there is a quasi-isomorphism between the \(A_\infty\)-algebra \(\Ccone\) and a dg algebra containing the algebra \(\A\) of matrix factorizations from \cref{def:mf_algebra}. 
Moreover, this quasi-isomorphism restricts to a quasi-isomorphism between \(\C\) and~\(\A\). 
Combining these two results gives rise to a quasi-isomorphism between \(\A\) and \(\Binf\). 

Note, however, that \cref{thm:quasi-iso} makes a stronger claim 
than simply the existence of some such quasi-isomorphism;
we also need control over the induced map on homology. 
In principle, \cref{thm:quasi-iso} could be deduced from a careful analysis of the quasi-isomorphisms from \cite{HKK,AAEKO}; we proceed with the following direct, perhaps elementary, strategy that makes the required  quasi-isomorphism explicit.

First, we enlarge \(\Binf\) to an \(A_\infty\)-algebra \(\Bcone\)
using the mapping cone construction. 
Second, we apply homological perturbation to show that 
\(\Ccone\) can be equipped with an \(A_\infty\)-structure 
such that it becomes quasi-isomorphic to \(\Bcone\). 
We then repeat the process with the dg algebra~\(\A\):
We enlarge it to a dg algebra \(\Acone\) 
and show that \(\Ccone\) can be equipped with an \(A_\infty\)-structure 
such that it becomes quasi-isomorphic to \(\Acone\). 
Next, we show that those two \(A_\infty\)-structures on \(\Ccone\) are strictly homotopic. 
Here, we rely on the main algebraic result from \cite{AAEKO} 
underlying \cite[Proposition 7.12]{AAEKO}, 
namely, the classification of \(A_\infty\)-structures on \(\Ccone\) 
up to strict homotopy equivalence.
This gives a quasi-isomorphism
between \(\Acone\) and \(\Bcone\) 
which we understand explicitly.
Finally, this quasi-isomorphism is restricted 
to a quasi-isomorphism between \(\A\) and \(\Binf\).

The proof of \cref{thm:quasi-iso} concludes this section in \cref{sec:quasi-iso:identification}. This relies on the aforementioned homological perturbation calculations in \cref{sec:quasi-iso:tw} and \cref{sec:quasi-iso:mfs}, the details of which are included for completeness and in the interest of carefully tracking signs, but which may be skipped and taken as a black-box by the impatient reader.

\subsection{A quasi-isomorphism between \texorpdfstring{\(\Bcone\)}{B\^{}∞\_cone} and \texorpdfstring{\(\Ccone\)}{C\_cone}}
\label{sec:quasi-iso:tw}

{%
	\tikzcdset{diagrams={nodes={inner sep=1pt},row sep=2pt, column sep=10pt}}
	\begin{definition}
		Consider the category of \((h,q)\)-graded twisted complexes over \(\Binf\) whose differential has quantum grading \(-3\). 
		Denote by \(\Binf_{\operatorname{cone}}\) 
		its full \(A_\infty\)-subcategory generated by the following three objects:
		\[
		X_1
		=
		\left[h^0q^0\circ\right]
		\qquad
		X_2
		=
		\left[h^0q^0\bullet\right]
		\qquad
		X_3
		=
		\left[\vc{%
			\begin{tikzcd}
				h^{-1}q^{1}\circ
				\arrow{dd}{S}
				\\
				\phantom{X}
				\\
				h^0q^{-1}\bullet
		\end{tikzcd}}\right]
		\]
	\end{definition}
	
	\begin{remark}\label{rmk:right-to-left}%
		The differential in the twisted complex \(X_3\) will be written as a vertical arrow, while
		arrows with a non-zero horizontal component represent morphisms in \(\Binf_{\operatorname{cone}}\). 
		For example, in the proof of \cref{prop:quasi-iso:BinfCone} below, we will encounter expressions of the form 
		\[\mu_3 \Bigg(
		\begin{tikzcd}
			\phantom{X}
			&
			\circ
			\arrow{dd}[swap]{S}
			&
			&
			\\
			\circ
			\arrow[leftarrow]{ru}{1}
			&
			&
			\bullet
			\arrow[leftarrow]{r}{S}
			\arrow{dl}{1}
			&
			\circ
			\\
			&
			\bullet
			&
			&
			\phantom{X}
		\end{tikzcd}
		\Bigg)
		\] 
		These can be converted into generic symbols (encoding objects and morphisms as usual) via $\mu_3(X_1\xleftarrow{a_3} X_3\xleftarrow{a_2} X_2\xleftarrow{a_1} X_1)$, which can be rewritten $\mu_3(a_3,a_2,a_1)$ as is more common (see Seidel's conventions, for instance) but perhaps less informative.  
	\end{remark}

	\begin{proposition}\label{prop:quasi-iso:BinfCone}
		There exists an \(A_\infty\)-structure on \(\Ccone\) 
		and
		a \(q\)-filtered quasi-isomorphism
		\(\Phi\co\Ccone\rightarrow\Bcone\)
		such that
		\begin{enumerate}
			\item \label{prop:quasi-iso:BinfCone:mu-iii}
			\(\mu_{3}(u_{31},u_{23},u_{12})=1\) 
			and 
			\(\mu_{3}(v_{21},v_{32},v_{13})=1\),
			\item 
			\(\Phi(L_i)=X_i\) for \(i=1,2,3\), and 
			\item \label{prop:quasi-iso:BinfCone:morphisms}
			the following identities hold:
			\begin{align*}
				\Phi^1(u_{12})
				&=
				S_\circ
				&
				\Phi^1(u_{31}v_{13})
				&=
				D_\circ
				&
				\Phi^1(v_{21})
				&=
				S_\bullet
				&
				\Phi^1(u_{32}u_{23})
				&=
				D_\bullet
			\end{align*}
		\end{enumerate}	
	\end{proposition}
	
	\newcommand{\phantomheight}[1]{\vphantom{\parbox[c]{1cm}{\rule{1cm}{#1}}}}
	\newcommand{\cxbox}[2][lightgray]{\colorbox{#1}{$\vc{%
				\begin{tikzcd}[cxbases]#2\end{tikzcd}}\phantomheight{38px}$}}
	\newcommand{\cxboxS}[2][lightgray]{\colorbox{#1}{$\vc{%
				\begin{tikzcd}[cxbases]#2\end{tikzcd}}\phantomheight{43px}$}}	
			
	\begin{figure}[t]
		\centering
		\tikzstyle{cxbases} = [ampersand replacement=\&,column sep=15pt]
		\begin{subfigure}{\textwidth}
			\centering
			\begin{gather*}
				%
				\cxboxS{
					\circ
					\arrow[swap]{dd}{S}
					\arrow[leftarrow]{rdd}[sloped,below,pos=0.9]{S^{2n+1}}
					\arrow[phantom]{r}{\phantom{S^n}}
					\&
					\circ
					\arrow{dd}{S}
					\\
					\phantom{\circ}
					\\
					\bullet
					\arrow[phantom]{r}[below]{\phantom{S^n}}
					\&
					\bullet
				}
				\mapsto
				\cxboxS{
						\circ
						\arrow[leftarrow]{r}{-S^{2n+2}}
						\arrow[swap]{dd}{S}
						\&
						\circ
						\arrow{dd}{S}
						\\
						\phantom{\circ}
						\\
						\bullet
						\arrow[leftarrow]{r}[below,pos=0.55]{-S^{2n+2}}
						\&
						\bullet
				}
				\phantom{\mapsto}
				%
				\cxboxS{
						\circ
						\arrow[leftarrow]{r}{S^{2n}}
						\arrow[swap]{dd}{S}
						\arrow[phantom]{r}[below]{\phantom{S^n}}
						\&
						\circ
						\arrow{dd}{S}
						\\
						\phantom{\circ}
						\\
						\bullet
						\arrow[phantom]{r}[below]{\phantom{S^n}}
						\&
						\bullet
				}
				\mapsto
				\cxboxS{
						\circ
						\arrow[swap]{dd}{S}
						\arrow[phantom]{r}[above]{\phantom{S^n}}
						\&
						\circ
						\arrow{dd}{S}
						\\
						\phantom{\circ}
						\\
						\bullet
						\arrow[phantom]{r}[below]{\phantom{S^n}}
						\arrow[leftarrow]{ruu}[sloped,pos=0.8]{S^{2n+1}}
						\&
						\bullet
				}
%
				%
				\cxbox[white]{
						\circ
						\arrow[leftarrow]{r}{1}
						\arrow[swap]{dd}{S}
						\arrow[phantom]{r}[above]{\phantom{S^n}}
						\&
						\circ
						\arrow{dd}{S}
						\\
						\phantom{\circ}
						\\
						\bullet
						\arrow[phantom]{r}[below]{\phantom{S^n}}
						\arrow[leftarrow]{r}[below]{1}
						\&
						\bullet
				}
				\cxbox[white]{
						\circ
						\arrow[leftarrow]{r}{D^{n+1}}
						\arrow[swap]{dd}{S}
						\arrow[phantom]{r}[above]{\phantom{S^n}}
						\&
						\circ
						\arrow{dd}{S}
						\\
						\phantom{\circ}
						\\
						\bullet
						\arrow[phantom]{r}[below]{\phantom{S^n}}
						\&
						\bullet
				}
				\cxbox[white]{
						\circ
						\arrow[swap]{dd}{S}
						\arrow[phantom]{r}[above]{\phantom{S^n}}
						\&
						\circ
						\arrow{dd}{S}
						\\
						\phantom{\circ}
						\\
						\bullet
						\arrow[phantom]{r}[below]{\phantom{S^n}}
						\arrow[leftarrow]{r}[below]{D^{n+1}}
						\&
						\bullet
				}
			\end{gather*}
			\caption{%
				A basis of the space of morphisms \(X_3\rightarrow X_3\). 
			}
			\label{fig:Bcone:morphisms:33}
		\end{subfigure}
		\\
		\begin{subfigure}{0.5\textwidth}
		\centering
		\begin{gather*}
			%
			\cxbox{
					\circ
					\arrow{dd}[swap]{S}
					\&
					\phantom{X}
					\\
					\&
					\circ
					\arrow{ul}[sloped]{S^{2n}}
					\\
					\bullet
					\&
					\phantom{X}
			}
			\mapsto
			\cxbox{
					\circ
					\arrow{dd}[swap]{S}
					\&
					\phantom{X}
					\\
					\&
					\circ
					\arrow{dl}[sloped,swap,pos=0.3]{-S^{2n+1}}
					\\
					\bullet
					\&
					\phantom{X}
			}
			\phantom{\mapsto}
			%
			\cxbox[white]{
					\circ
					\arrow{dd}[swap]{S}
					\&
					\phantom{X}
					\\
					\&
					\circ
					\arrow{ul}[sloped]{D^{n+1}}
					\\
					\bullet
					\&
					\phantom{X}
			}
		\end{gather*}
		\caption{%
			A basis of the space of morphisms \(X_1\rightarrow X_3\). 
		}
		\label{fig:Bcone:morphisms:13}
	\end{subfigure}%
	\begin{subfigure}{0.5\textwidth}
		\centering
		\begin{gather*}
			%
			\cxbox{
					\phantom{X}
					\&
					\circ
					\arrow{dd}{S}
					\\
					\circ
					\&
					\\
					\phantom{X}
					\&
					\bullet
					\arrow{ul}[swap,sloped]{S^{2n+1}}
			}
			\mapsto
			\cxbox{
					\phantom{X}
					\&
					\circ
					\arrow{dd}{S}
					\arrow{dl}[sloped]{-S^{2n+2}}
					\\
					\circ
					\&
					\\
					\phantom{X}
					\&
					\bullet
			}
			\phantom{\mapsto}
			%
			\cxbox[white]{
					\phantom{X}
					\&
					\circ
					\arrow{dd}{S}
					\arrow{dl}[sloped]{-D^{n}}
					\\
					\circ
					\&
					\\
					\phantom{X}
					\&
					\bullet
			}
		\end{gather*}
		\caption{%
			A basis of the space of morphisms \(X_3\rightarrow X_1\). 
		}
		\label{fig:Bcone:morphisms:31}
		\end{subfigure}
		\\
		\begin{subfigure}{0.5\textwidth}
		\centering
		\begin{gather*}
			%
			\cxbox{
					\circ
					\arrow{dd}[swap]{S}
					\&
					\phantom{X}
					\\
					\&
					\bullet
					\arrow{ul}[sloped]{S^{2n+1}}
					\\
					\bullet
					\&
					\phantom{X}
			}
			\mapsto
			\cxbox{
					\circ
					\arrow{dd}[swap]{S}
					\&
					\phantom{X}
					\\
					\&
					\bullet
					\arrow{dl}[sloped,swap,pos=0.3]{-S^{2n+2}}
					\\
					\bullet
					\&
					\phantom{X}
			}
			\phantom{\mapsto}
			%
			\cxbox[white]{
					\circ
					\arrow{dd}[swap]{S}
					\&
					\phantom{X}
					\\
					\&
					\bullet
					\arrow{dl}[sloped,swap]{D^{n}}
					\\
					\bullet
					\&
					\phantom{X}
			}
		\end{gather*}
		\caption{%
			A basis of the space of morphisms \(X_2\rightarrow X_3\). 
		}
		\label{fig:Bcone:morphisms:23}
		\end{subfigure}%
		\begin{subfigure}{0.5\textwidth}
			\centering
			\begin{gather*}
				%
				\cxbox{
						\phantom{X}
						\&
						\circ
						\arrow{dd}{S}
						\\
						\bullet
						\&
						\\
						\phantom{X}
						\&
						\bullet
						\arrow{ul}[sloped,swap]{S^{2n}}
				}
				\mapsto
				\cxbox{
					\phantom{X}
					\&
					\circ
					\arrow{dl}[sloped]{-S^{2n+1}}
					\arrow{dd}{S}
					\\
					\bullet
					\&
					\\
					\phantom{X}
					\&
					\bullet
				}
				\phantom{\mapsto}
				%
				\cxbox[white]{
					\phantom{X}
					\&
					\circ
					\arrow{dd}{S}
					\\
					\bullet
					\&
					\\
					\phantom{X}
					\&
					\bullet
					\arrow{ul}[sloped,swap]{D^{n+1}}
				}
			\end{gather*}
			\caption{%
				A basis of the space of morphisms \(X_3\rightarrow X_2\). 
			}
			\label{fig:Bcone:morphisms:32}
		\end{subfigure}
		\caption{%
			A basis of the morphism spaces in \(\Bcone\) and the action of the differential (\(n\geq0\)).
		}
		\label{fig:Bcone:morphisms}
	\end{figure}

	\begin{proof}
		We first compute bases of the morphism spaces as vector spaces.  
		We choose these bases such that 
		the differential induces a partial matching of basis elements 
		that cancel each other.
		For some morphism spaces 
		such bases are shown in \cref{fig:Bcone:morphisms}. 
		(The signs are computed using the sign conventions 
		from \cref{def:additive_enlargement_Seidel}.) 
		For the remaining morphism spaces, 
		we take the standard basis of \(\B\subset\Bcone\) from 
		\cref{def:standard_basis}.
		Thus, for any two objects \(X,Y\) of \(\Bcone\),
		we can write \(\Bcone(X,Y)\) as a direct sum of 
		\(H\Bcone(X,Y)\) and some contractible chain complex. 
		Consider the following diagram
		\[
		\begin{tikzcd}[column sep=3cm]
			H\Bcone(X,Y)
			\arrow[bend left=5]{r}{\mathcal{F}^1}
			&
			\Bcone(X,Y)
			\arrow[bend left=5]{l}{\mathcal{G}^1}
			\arrow[loop right, distance=3em, start anchor={[yshift=1ex]east}, end anchor={[yshift=-1ex]east}]{rl}{T^1}
		\end{tikzcd}
		\]
		where 
		\(\mathcal{F}^1\) %
		and 
		\(\mathcal{G}^1\) %
		are the inclusion and projection map, respectively, 
		and
		\(T^1\) %
		is the chain homotopy between \(\mathcal{F}^1\mathcal{G}^1\) and the identity:
		\begin{equation}\label{eq:mfs:indirect:def:T1}
			T^1\mu_1^{\Bcone}
			+
			\mu_1^{\Bcone}T^1
			=
			\mathcal{F}^1\mathcal{G}^1-1
		\end{equation}
		More explicitly, 
		\(T^1\)	is the negative of the inverse of the differential, 
		i.e.\ it is defined by reversing all arrows \(\mapsto\) 
		in \cref{fig:Bcone:morphisms}
		and multiplying the new map by \(-1\). 

		We apply 
		Seidel's homological perturbation lemma 
		to this diagram 
		\cite[Section~I~(1i)]{Seidel}
		and obtain an \(A_\infty\)-structure on \(H\Bcone\) 
		together with a quasi-isomorphism to \(\Bcone\) 
		extending the maps \(\mathcal{F}^1\) and \(\mathcal{G}^1\). 
		Observe that
		\(q(\mathcal{F}^1)=0\) 
		and 
		\(q(T^1)=3\). 
		By induction, it follows 
		from Seidel's formula for \(\mathcal{F}^d\) 
		that 
		\(q(\mathcal{F}^d)=3(d-1)\).
		Together with 
		\(q(\mathcal{G}^1)=0\),
		Seidel's formula for \(\mu^{H\Bcone}_d\) gives
		\(q(\mu^{H\Bcone}_d)=3(d-2)\).
		
		To determine the induced multiplication on \(H\Bcone\), define 
		\begin{align*}
		\Phi^1(u_{12})
		&=
		\vc{%
			\begin{tikzcd}[ampersand replacement=\&]
				\phantom{X}
				\&
				\\
				\bullet
				\arrow[leftarrow]{r}{S}
				\&
				\circ
				\\
				\&
				\phantom{X}
		\end{tikzcd}}
		&
		\Phi^1(u_{23})
		&=
		\vc{%
			\begin{tikzcd}[ampersand replacement=\&]
				\circ
				\arrow{dd}[swap]{S}
				\&
				\phantom{X}
				\\
				\&
				\bullet
				\arrow{dl}{1}
				\\
				\bullet
				\&
				\phantom{X}
		\end{tikzcd}}
		&
		\Phi^1(u_{31})
		&=
		\vc{%
			\begin{tikzcd}[ampersand replacement=\&]
				\phantom{X}
				\&
				\circ
				\arrow[swap]{dl}{-1}
				\arrow{dd}{S}
				\\
				\circ
				\&
				\\
				\phantom{X}
				\&
				\bullet
		\end{tikzcd}}
		\\
		\Phi^1(v_{21})
		&=
		\vc{%
			\begin{tikzcd}[ampersand replacement=\&]
				\phantom{X}
				\&
				\\
				\circ
				\arrow[leftarrow]{r}{S}
				\&
				\bullet
				\\
				\&
				\phantom{X}
		\end{tikzcd}}
		&
		\Phi^1(v_{32})
		&=
		\vc{%
			\begin{tikzcd}[ampersand replacement=\&]
				\phantom{X}
				\&
				\circ
				\arrow{dd}{S}
				\\
				\bullet
				\&
				\\
				\phantom{X}
				\&
				\bullet
				\arrow{ul}{D}
		\end{tikzcd}}
		&
		\Phi^1(v_{13})
		&=
		\vc{%
			\begin{tikzcd}[ampersand replacement=\&]
				\circ
				\arrow{dd}[swap]{S}
				\&
				\phantom{X}
				\\
				\&
				\circ
				\arrow[swap]{ul}{D}
				\\
				\bullet
				\&
				\phantom{X}
		\end{tikzcd}}
		\end{align*}
		Applying \cite[Formula~(1.18)]{Seidel},
		one checks that the morphisms
		\(\mu_{2}(\Phi^1(u_{i,i+1}),\Phi^1(u_{i-1,i}))\)
		and 
		\(\mu_{2}(\Phi^1(v_{i,i-1}),\Phi^1(v_{i+1,i}))\)
		are null-homotopic
		for all \(i=1,2,3\), 
		and that 
		\(\Phi^1\) determines an algebra isomorphism 
		between \(\Ccone\) and \(H\Bcone\).
		In particular, 
		the identities 
		from \eqref{prop:quasi-iso:BinfCone:morphisms} 
		hold. 
		
		Finally, the map \(\Phi^1\) induces an \(A_\infty\)-structure on \(\Ccone\). 
		Using \cite[Formula~(1.18)]{Seidel}
		we verify the identities 
		from \eqref{prop:quasi-iso:BinfCone:mu-iii}:
		\begin{multline*}
			\mu^{\Ccone}_{3}(u_{31},u_{23},u_{12})
			=
			\mu^{H_*\Bcone}_{3}(\Phi^1(u_{31}),\Phi^1(u_{23}),\Phi^1(u_{12}))
			=
			\mu_3^{H_*\Bcone}\Bigg(
			\vc{%
				\begin{tikzcd}[ampersand replacement=\&]
					\phantom{X}
					\&
					\circ
					\arrow{dd}[swap]{S}
					\&
					\&
					\\
					\circ
					\arrow[leftarrow]{ru}{-1}
					\&
					\&
					\bullet
					\arrow[leftarrow]{r}{S}
					\arrow{dl}{1}
					\&
					\circ
					\\
					\&
					\bullet
					\&
					\&
					\phantom{X}
			\end{tikzcd}}
			\Bigg)
			\\
			=
			\mathcal{G}^1\Bigg(\mu_3^{\mathcal{B}^\infty_{\operatorname{cone}}}\Bigg(
			\vc{%
				\begin{tikzcd}[ampersand replacement=\&]
					\phantom{X}
					\&
					\circ
					\arrow{dd}[swap]{S}
					\&
					\&
					\\
					\circ
					\arrow[leftarrow]{ru}{-1}
					\&
					\&
					\bullet
					\arrow[leftarrow]{r}{S}
					\arrow{dl}{1}
					\&
					\circ
					\\
					\&
					\bullet
					\&
					\&
					\phantom{X}
			\end{tikzcd}}
			\Bigg)\Bigg)
			+\mathcal{G}^1\Bigg(
			\mu_2^{\mathcal{B}^\infty_{\operatorname{cone}}}
			\Bigg(
			\mathcal{F}^2\Bigg(
			\vc{%
				\begin{tikzcd}[ampersand replacement=\&]
					\phantom{X}
					\&
					\circ
					\arrow{dd}[swap]{S}
					\&
					\\
					\circ
					\arrow[leftarrow]{ru}{-1}
					\&
					\&
					\bullet
					\arrow{dl}{1}
					\\
					\&
					\bullet
					\&
					\phantom{X}
			\end{tikzcd}}
			\Bigg)
			,
			\vc{%
				\begin{tikzcd}[ampersand replacement=\&]
					\phantom{X}
					\&
					\\
					\bullet
					\arrow[leftarrow]{r}{S}
					\&
					\circ
					\\
					\&
					\phantom{X}
			\end{tikzcd}}
			\Bigg)
			\Bigg)
			\\
			+\mathcal{G}^1\Bigg(
			\mu_2^{\mathcal{B}^\infty_{\operatorname{cone}}}
			\Bigg(
			\vc{%
				\begin{tikzcd}[ampersand replacement=\&]
					\phantom{X}
					\&
					\circ
					\arrow{dd}[swap]{S}
					\\
					\circ
					\arrow[leftarrow]{ru}{-1}
					\&
					\\
					\phantom{X}
					\&
					\bullet
			\end{tikzcd}}
			,
			\mathcal{F}^2\Bigg(
			\vc{%
				\begin{tikzcd}[ampersand replacement=\&]
					\circ
					\arrow{dd}[swap]{S}
					\&
					\phantom{X}
					\&
					\\
					\&
					\bullet
					\arrow[leftarrow]{r}{S}
					\arrow{dl}{1}
					\&
					\circ
					\\
					\bullet
					\&
					\&
					\phantom{X}
			\end{tikzcd}}
			\Bigg)
			\Bigg)
			\Bigg)
			\end{multline*}
			Note that the symbols \(\mathcal{F}^1\) 
			have been dropped 
			to lighten the expression above, 
			since the map \(\mathcal{F}^1\) 
			is simply the inclusion of representatives of homology classes. 
			Observe that
			\[
			\mu_3^{\mathcal{B}^\infty_{\operatorname{cone}}} 
			\Bigg(
			\vc{%
				\begin{tikzcd}[ampersand replacement=\&]
					\phantom{X}
					\&
					\circ
					\arrow{dd}[swap]{S}
					\&
					\&
					\\
					\circ
					\arrow[leftarrow]{ru}{-1}
					\&
					\&
					\bullet
					\arrow[leftarrow]{r}{S}
					\arrow{dl}{1}
					\&
					\circ
					\\
					\&
					\bullet
					\&
					\&
					\phantom{X}
			\end{tikzcd}}
			\Bigg)=0
			\] since %
			there is no 
			composable sequence of arrows. 
			Therefore, the first term in the sum above vanishes. 
			The second term vanishes because
			\begin{align*}
				\mathcal{F}^2\Bigg(
				\vc{%
					\begin{tikzcd}[ampersand replacement=\&]
						\phantom{X}
						\&
						\circ
						\arrow{dd}[swap]{S}
						\&
						\\
						\circ
						\arrow[leftarrow]{ru}{-1}
						\&
						\&
						\bullet
						\arrow{dl}{1}
						\\
						\&
						\bullet
						\&
						\phantom{X}
				\end{tikzcd}}
				\Bigg)
				&=
				T^1
				\Bigg(
				\mu_2^{\mathcal{B}^\infty_{\operatorname{cone}}}
				\Bigg(
				\vc{%
					\begin{tikzcd}[ampersand replacement=\&]
						\phantom{X}
						\&
						\circ
						\arrow{dd}[swap]{S}
						\\
						\circ
						\arrow[leftarrow]{ru}{-1}
						\\
						\phantom{X}
						\&
						\bullet
				\end{tikzcd}}
				,
				\vc{%
					\begin{tikzcd}[ampersand replacement=\&]
						\circ
						\arrow{dd}[swap]{S}
						\&
						\phantom{X}
						\\
						\&
						\bullet
						\arrow{dl}{1}
						\\
						\bullet
						\&
						\phantom{X}
				\end{tikzcd}}
				\Bigg)
				\Bigg)
				=
				T^1(0)
				=
				0.
			\end{align*}
			We obtain
			\[
			\mu^{\Ccone}_{3}(u_{31},u_{23},u_{12})
			= 
			\mathcal{G}^1\Bigg(
			\mu_2^{\mathcal{B}^\infty_{\operatorname{cone}}}
			\Bigg(
			\vc{%
				\begin{tikzcd}[ampersand replacement=\&]
					\phantom{X}
					\&
					\circ
					\arrow{dd}[swap]{S}
					\\
					\circ
					\arrow[leftarrow]{ru}{-1}
					\&
					\\
					\phantom{X}
					\&
					\bullet
			\end{tikzcd}}
			,
			\mathcal{F}^2\Bigg(
			\vc{%
				\begin{tikzcd}[ampersand replacement=\&]
					\circ
					\arrow{dd}[swap]{S}
					\&
					\phantom{X}
					\&
					\\
					\&
					\bullet
					\arrow[leftarrow]{r}{S}
					\arrow{dl}{1}
					\&
					\circ
					\\
					\bullet
					\&
					\&
					\phantom{X}
			\end{tikzcd}}
			\Bigg)
			\Bigg)
			\Bigg)
			\]
			and compute 
			\begin{align*}
				\mathcal{F}^2\Bigg(
				\vc{%
					\begin{tikzcd}[ampersand replacement=\&]
						\circ
						\arrow{dd}[swap]{S}
						\&
						\&
						\phantom{X}
						\\
						\&
						\bullet
						\arrow[leftarrow]{r}{S}
						\arrow{dl}{1}
						\&
						\circ
						\\
						\bullet
						\&
						\&
						\phantom{X}
				\end{tikzcd}}
				\Bigg)
				=
				T^1
				\Bigg(
				\vc{%
					\begin{tikzcd}[ampersand replacement=\&]
						\circ	
						\arrow{dd}[swap]{S}
						\&
						\phantom{X}
						\\
						\&
						\circ
						\arrow{dl}{S}
						\\
						\bullet
						\&
						\phantom{X}
				\end{tikzcd}}
				\Bigg)
				=T^1\mu^{\mathcal{B}^\infty_{\operatorname{cone}}}_1
				\Bigg(
				\vc{%
					\begin{tikzcd}[ampersand replacement=\&]
						\circ	
						\arrow{dd}[swap]{S}
						\&
						\phantom{X}
						\\
						\&
						\circ
						\arrow[swap]{ul}{-1}
						\\
						\bullet
						\&
						\phantom{X}
				\end{tikzcd}}
				\Bigg)
				=
				\vc{%
					\begin{tikzcd}[ampersand replacement=\&]
						\circ	
						\arrow[leftarrow]{rd}{1}
						\arrow{dd}[swap]{S}
						\&
						\phantom{X}
						\\
						\&
						\circ
						\\
						\bullet
						\&
						\phantom{X}
				\end{tikzcd}}
			\end{align*}
			In summary, the above shows that
			\[
			\mu^{\Ccone}_{3}(u_{31},u_{23},u_{12})
			= 
			\mathcal{G}^1\Bigg(
			\mu_2^{\mathcal{B}^\infty_{\operatorname{cone}}}
			\Bigg(
			\vc{%
				\begin{tikzcd}[ampersand replacement=\&]
					\phantom{X}
					\&
					\circ
					\arrow{dd}[swap]{S}
					\\
					\circ
					\arrow[leftarrow]{ru}{-1}
					\&
					\\
					\phantom{X}
					\&
					\bullet
			\end{tikzcd}}
			,
			\vc{%
				\begin{tikzcd}[ampersand replacement=\&]
					\circ	
					\arrow[leftarrow]{rd}{1}
					\arrow{dd}[swap]{S}
					\&
					\phantom{X}
					\\
					\&
					\circ
					\\
					\bullet
					\&
					\phantom{X}
			\end{tikzcd}}
			\Bigg)
			\Bigg)\\
			=\mathcal{G}^1\Bigg(
			\vc{%
				\begin{tikzcd}[ampersand replacement=\&]
					\phantom{X}
					\&
					\\
					\circ
					\arrow[leftarrow]{r}{1}
					\&
					\circ
					\\
					\&
					\phantom{X}
			\end{tikzcd}}
			\Bigg)
			=1
			\]
			Similarly,
			\begin{multline*}
				\mu^{\Ccone}_{3}(v_{21},v_{32},v_{13})
				=
				\mu_3^{H_*\mathcal{B}^\infty_{\operatorname{cone}}}\Bigg(
				\vc{%
					\begin{tikzcd}[ampersand replacement=\&]
						\phantom{X}
						\&
						\&
						\circ
						\arrow{dd}[swap]{S}
						\arrow[leftarrow]{dr}{D}
						\&
						\\
						\circ
						\arrow[leftarrow]{r}{S}
						\&
						\bullet
						\&
						\&
						\circ
						\\
						\&
						\&
						\bullet
						\arrow{ul}{D}
						\&
						\phantom{X}
				\end{tikzcd}}
				\Bigg)
				=
				\mathcal{G}^1\Bigg(\mu_3^{\mathcal{B}^\infty_{\operatorname{cone}}}\Bigg(
				\vc{%
					\begin{tikzcd}[ampersand replacement=\&]
						\phantom{X}
						\&
						\&
						\circ
						\arrow{dd}[swap]{S}
						\arrow[leftarrow]{dr}{D}
						\&
						\\
						\circ
						\arrow[leftarrow]{r}{S}
						\&
						\bullet
						\&
						\&
						\circ
						\\
						\&
						\&
						\bullet
						\arrow{ul}{D}
						\&
						\phantom{X}
				\end{tikzcd}}
				\Bigg)\Bigg)
				\\ %
				+\mathcal{G}^1\Bigg(
				\mu_2^{\mathcal{B}^\infty_{\operatorname{cone}}}
				\Bigg(
				\mathcal{F}^2\Bigg(
				\vc{%
					\begin{tikzcd}[ampersand replacement=\&]
						\phantom{X}
						\&
						\&
						\circ
						\arrow{dd}[swap]{S}
						\\
						\circ
						\arrow[leftarrow]{r}{S}
						\&
						\bullet
						\&
						\\
						\&
						\phantom{X}
						\&
						\bullet
						\arrow{ul}{D}
				\end{tikzcd}}
				\Bigg)
				,
				\vc{%
					\begin{tikzcd}[ampersand replacement=\&]
						\circ
						\arrow{dd}[swap]{S}
						\arrow[leftarrow]{dr}{D}
						\&
						\\
						\&
						\circ
						\\
						\bullet
						\&
						\phantom{X}
				\end{tikzcd}}
				\Bigg)
				\Bigg)
				+\mathcal{G}^1\Bigg(
				\mu_2^{\mathcal{B}^\infty_{\operatorname{cone}}}
				\Bigg(
				\vc{%
					\begin{tikzcd}[ampersand replacement=\&]
						\phantom{X}
						\&
						\\
						\circ
						\arrow[leftarrow]{r}{S}
						\&
						\bullet
						\\
						\&
						\phantom{X}
				\end{tikzcd}}
				,
				\mathcal{F}^2\Bigg(
				\vc{%
					\begin{tikzcd}[ampersand replacement=\&]
						\phantom{X}
						\&
						\circ
						\arrow{dd}[swap]{S}
						\arrow[leftarrow]{dr}{D}
						\&
						\\
						\bullet
						\&
						\&
						\circ
						\\
						\&
						\bullet
						\arrow{ul}{D}
						\&
						\phantom{X}
				\end{tikzcd}}
				\Bigg)
				\Bigg)
				\Bigg)
				\end{multline*}
				In this case, the second and third terms vanish because in each case, $\mathcal{F}^2$ is evaluated on a composition that vanishes. 
				This leaves 
				\begin{align*}
					\mu^{\Ccone}_{3}(v_{21},v_{32},v_{13})
					&= 
					\mathcal{G}^1\Bigg(\mu_3^{\mathcal{B}^\infty_{\operatorname{cone}}}\Bigg(
					\vc{%
						\begin{tikzcd}[ampersand replacement=\&]
							\phantom{X}
							\&
							\&
							\circ
							\arrow{dd}[swap]{S}
							\arrow[leftarrow]{dr}{D}
							\&
							\\
							\circ
							\arrow[leftarrow]{r}{S}
							\&
							\bullet
							\&
							\&
							\circ
							\\
							\&
							\&
							\bullet
							\arrow{ul}{D}
							\&
							\phantom{X}
					\end{tikzcd}}
					\Bigg)\Bigg)
					=
					\mathcal{G}^1\Big(-\mu_4^{\mathcal{B}^\infty}\Big(
					\vc{%
						\begin{tikzcd}[ampersand replacement=\&]
							\circ
							\arrow[leftarrow]{r}{S}
							\&
							\bullet
							\arrow[leftarrow]{r}{D}
							\&
							\bullet
							\arrow[leftarrow]{r}{S}
							\&
							\circ
							\arrow[white,swap]{r}{\phantom{D}}
							\arrow[leftarrow]{r}{D}
							\&
							\circ
					\end{tikzcd}}
					\Big)\Big)
					=1
				\end{align*}
				as desired. 
			\end{proof}
			
		}%
		
		%
		%
		%
		%
		%
		%
		%
		%
		%
		%
		%
		%
		%
		%
		%
		%
		%
		%
		%
		%
		%
		%
		%
		%
		%
		%
		%
		%
		%
		%
		%
		%
		
		\subsection{A quasi-isomorphism between \texorpdfstring{\(\Acone\)}{A\_cone} and \texorpdfstring{\(\Ccone\)}{C\_cone}}
		\label{sec:quasi-iso:mfs} 
				Denote by $R$ the boundary edge ring $\Rd=\field[x,y,z]$ from \cref{exa:diagrams:basic}, 
			and consider the dg category $\mathit{MF}_{xyz}(R)$ of matrix factorizations over \(R\). 
			We define a \(h\)-grading on \(R\) by \(h(y)=2\) and \(h(x)=0=h(z)\).

		\begin{definition}
			Let \(\Acone\) be %
			the full subcategory of $\mathit{MF}_{xyz}(R)$
			generated by the %
			matrix factorizations
			\[
			M_1
			=
			\left[
			\vc{\begin{tikzcd}[column sep=1cm, ampersand replacement=\&]
					h^{-1}q^1R
					\arrow[d,bend left=10, "x"right]
					\\
					h^0q^0R
					\arrow[u,bend left=10, "yz"left]
			\end{tikzcd}}
			\right]
			\quad
			M_2
			=
			\left[
			\vc{\begin{tikzcd}[column sep=1cm, ampersand replacement=\&]
					h^0q^1R
					\arrow[d,bend left=10, "y"right]
					\\
					h^{-1}q^0R
					\arrow[u,bend left=10, "xz"left]
			\end{tikzcd}}
			\right]
			\quad
			M_3
			=%
			\left[
			\vc{\begin{tikzcd}[column sep=1cm, ampersand replacement=\&]
					h^{-1}q^{1}R
					\arrow[d,bend left=10, "z"right]
					\\
					h^0q^{0}R
					\arrow[u,bend left=10, "xy"left]
			\end{tikzcd}}
			\right]
			\]
			noting that $M_1 = M_\arcD$ and $M_2=M_\arcT$. 
			We will often write \(\diamond\) for \(R\) and 
			drop the gradings from the notation. 
		\end{definition}
		
		{%
			\tikzcdset{diagrams={nodes={inner sep=2pt},row sep=12pt, column sep=24pt}}
			
			\begin{proposition}\label{prop:quasi-iso:Acone}
				There exists an \(A_\infty\)-structure on \(\Ccone\) 
				and
				a \(q\)-filtered quasi-isomorphism
				\(\Psi\co\Ccone\rightarrow\Acone\)
				such that
				\begin{enumerate}
					\item \label{prop:quasi-iso:Acone:mu-iii}
					\(\mu_{3}(u_{31},u_{23},u_{12})=1\) 
					and 
					\(\mu_{3}(v_{21},v_{32},v_{13})=1\),
					\item 
					\(\Psi(L_i)=M_i\) for \(i=1,2,3\), and 
					\item \label{prop:quasi-iso:Acone:morphisms}
					the following identities hold:
					\begin{align*}
						\Psi^1(u_{12})
						&=
						\sW
						&
						\Psi^1(u_{31}v_{13})
						&=
						\dW
						&
						\Psi^1(v_{21})
						&=
						\sB
						&
						\Psi^1(u_{32}u_{23})
						&=
						\dB
					\end{align*}
				\end{enumerate}
			\end{proposition}

\newcommand{\phantomheight}[1]{\vphantom{\parbox[c]{1cm}{\rule{1cm}{#1}}}}
\tikzstyle{mfbases} = [ampersand replacement=\&,column sep=45pt]
\tikzstyle{mfdes} = [inner sep=1pt, outer sep=1pt,fill=white]
\tikzstyle{mfarrow} = [in=180,out=0,looseness=0.7]
\tikzstyle{mftop} = [pos=0.2,outer sep=-1.5pt]
\tikzstyle{mfbot} = [pos=0.2,outer sep=-1.5pt,swap]
\newcommand{\mfbox}[2][lightgray]{\colorbox{#1}{$\vc{%
						\begin{tikzcd}[mfbases]#2\end{tikzcd}}\phantomheight{40px}$}}
\begin{figure}[t]
	\centering
	\begin{subfigure}{\textwidth}
		\centering
		\begin{gather*}
			%
			\mfbox{
				\diamond
				\arrow[d,bend left=10, "x"right]
				\arrow[mfarrow,leftarrow]{rd}[pos=0,above right,xshift=-5pt]{x^{a}y^{b}z^{c}}
				\&
				\diamond
				\arrow[d,bend left=10, "x"right]
				\\
				\diamond
				\arrow[u,bend left=10, "yz"left]
				\arrow[phantom]{r}[below]{\phantom{XXX}}
				\&
				\diamond
				\arrow[u,bend left=10, "yz"left]
			}
			\mapsto 
			\mfbox{
				\diamond
				\arrow[d,bend left=10, "x"right]
				\arrow[mfarrow,leftarrow]{r}[above]{-x^{a+1}y^{b}z^{c}}
				\&
				\diamond
				\arrow[d,bend left=10, "x"right]
				\\
				\diamond
				\arrow[u,bend left=10, "yz"left]
				\arrow[mfarrow,leftarrow]{r}[below]{-x^{a+1}y^{b}z^{c}}
				\&
				\diamond
				\arrow[u,bend left=10, "yz"left]
			}
			\phantom{\mapsto}
			%
			\mfbox{
				\diamond
				\arrow[d,bend left=10, "x"right]
				\arrow[phantom]{r}[above]{\phantom{XXX}}
				\&
				\diamond
				\arrow[d,bend left=10, "x"right]
				\\
				\diamond
				\arrow[u,bend left=10, "yz"left]
				\arrow[mfarrow,leftarrow]{ru}[pos=0,below right,xshift=-5pt]{y^{b}z^{c}}
				\&
				\diamond
				\arrow[u,bend left=10, "yz"left]
			}
			\mapsto
			\mfbox{
				\diamond
				\arrow[d,bend left=10, "x"right]
				\arrow[mfarrow,leftarrow]{r}[above]{-y^{b+1}z^{c+1}}
				\&
				\diamond
				\arrow[d,bend left=10, "x"right]
				\\
				\diamond
				\arrow[u,bend left=10, "yz"left]
				\arrow[mfarrow,leftarrow]{r}[below]{-y^{b+1}z^{c+1}}
				\&
				\diamond
				\arrow[u,bend left=10, "yz"left]
			}
			\\
			%
			\mfbox{
				\diamond
				\arrow[d,bend left=10, "x"right]
				\arrow[mfarrow,leftarrow]{r}[above]{x^{a}y^{b}z^{c}}
				\&
				\diamond
				\arrow[d,bend left=10, "x"right]
				\\
				\diamond
				\arrow[u,bend left=10, "yz"left]
				\arrow[phantom]{r}[below]{\phantom{XXX}}
				\&
				\diamond
				\arrow[u,bend left=10, "yz"left]
			}
			\mapsto 
			\mfbox{
				\diamond
				\arrow[d,bend left=10, "x"right]
				\arrow[mfarrow,leftarrow]{rd}[pos=0,above right,xshift=-5pt]{-x^{a}y^{b+1}z^{c+1}}
				\&
				\diamond
				\arrow[d,bend left=10, "x"right]
				\\
				\diamond
				\arrow[u,bend left=10, "yz"left]
				\arrow[mfarrow,leftarrow]{ru}[pos=0,below right,xshift=-5pt]{x^{a+1}y^{b}z^{c}}
				\&
				\diamond
				\arrow[u,bend left=10, "yz"left]
			}
			\phantom{\mapsto}
			%
			\mfbox[white]{
				\diamond
				\arrow[d,bend left=10, "x"right]
				\arrow[mfarrow,leftarrow]{r}[above]{y^{b}}
				\&
				\diamond
				\arrow[d,bend left=10, "x"right]
				\\
				\diamond
				\arrow[u,bend left=10, "yz"left]
				\arrow[mfarrow,leftarrow]{r}[below]{y^{b}}
				\&
				\diamond
				\arrow[u,bend left=10, "yz"left]
			}
			\phantom{\mapsto}
			\mfbox[white]{
				\diamond
				\arrow[d,bend left=10, "x"right]
				\arrow[mfarrow,leftarrow]{r}[above]{z^{c+1}}
				\&
				\diamond
				\arrow[d,bend left=10, "x"right]
				\\
				\diamond
				\arrow[u,bend left=10, "yz"left]
				\arrow[mfarrow,leftarrow]{r}[below]{z^{c+1}}
				\&
				\diamond
				\arrow[u,bend left=10, "yz"left]
			}
		\end{gather*}
		\caption{%
			A basis of the space of morphisms \(M_1\rightarrow M_1\). 
		}
		\label{fig:Acone:morphisms:11}
	\end{subfigure}
	\begin{subfigure}{\textwidth}
		\centering
		\begin{gather*}
			%
			\mfbox{
				\diamond
				\arrow[d,bend left=10, "y"right]
				\arrow[mfarrow,leftarrow]{r}[above]{x^{a}y^{b}z^{c}}
				\&
				\diamond
				\arrow[d,bend left=10, "x"right]
				\\
				\diamond
				\arrow[u,bend left=10, "xz"left]
				\arrow[phantom]{r}[below]{\phantom{XXX}}
				\&
				\diamond
				\arrow[u,bend left=10, "yz"left]
			}
			\mapsto 
			\mfbox{
				\diamond
				\arrow[d,bend left=10, "y"right]
				\arrow[mfarrow,leftarrow]{rd}[pos=0,above right,xshift=-5pt]{-x^{a}y^{b+1}z^{c+1}}
				\&
				\diamond
				\arrow[d,bend left=10, "x"right]
				\\
				\diamond
				\arrow[u,bend left=10, "xz"left]
				\arrow[mfarrow,leftarrow]{ru}[pos=0,below right,xshift=-5pt]{-x^{a}y^{b+1}z^{c}}
				\&
				\diamond
				\arrow[u,bend left=10, "yz"left]
			}
			\phantom{\mapsto}
			%
			\mfbox{
				\diamond
				\arrow[d,bend left=10, "y"right]
				\arrow[phantom]{r}[above]{\phantom{XXX}}
				\&
				\diamond
				\arrow[d,bend left=10, "x"right]
				\\
				\diamond
				\arrow[u,bend left=10, "xz"left]
				\arrow[mfarrow,leftarrow]{r}[below]{x^{a}z^{c}}
				\&
				\diamond
				\arrow[u,bend left=10, "yz"left]
			}
			\mapsto
			\mfbox{
				\diamond
				\arrow[d,bend left=10, "y"right]
				\arrow[mfarrow,leftarrow]{rd}[pos=0,above right,xshift=-5pt]{-x^{a+1}z^{c+1}}
				\&
				\diamond
				\arrow[d,bend left=10, "x"right]
				\\
				\diamond
				\arrow[u,bend left=10, "xz"left]
				\arrow[mfarrow,leftarrow]{ru}[pos=0,below right,xshift=-5pt]{-x^{a+1}z^{c}}
				\&
				\diamond
				\arrow[u,bend left=10, "yz"left]
			}
			\\
			%
			\mfbox{
				\diamond
				\arrow[d,bend left=10, "y"right]
				\arrow[mfarrow,leftarrow]{rd}[pos=0,above right,xshift=-5pt]{x^{a}y^{b}z^{c}}
				\&
				\diamond
				\arrow[d,bend left=10, "x"right]
				\\
				\diamond
				\arrow[u,bend left=10, "xz"left]
				\arrow[phantom]{r}[below]{\phantom{XXX}}
				\&
				\diamond
				\arrow[u,bend left=10, "yz"left]
			}
			\mapsto
			\mfbox{
				\diamond
				\arrow[d,bend left=10, "y"right]
				\arrow[mfarrow,leftarrow]{r}[above]{-x^{a+1}y^{b}z^{c}}
				\&
				\diamond
				\arrow[d,bend left=10, "x"right]
				\\
				\diamond
				\arrow[u,bend left=10, "xz"left]
				\arrow[mfarrow,leftarrow]{r}[below]{x^{a}y^{b+1}z^{c}}
				\&
				\diamond
				\arrow[u,bend left=10, "yz"left]
			}
			%
			\phantom{\mapsto}
			\mfbox[white]{
				\diamond
				\arrow[d,bend left=10, "y"right]
				\arrow[mfarrow,leftarrow]{rd}[pos=0.2,above right,xshift=-5pt]{z^{c+1}}
				\&
				\diamond
				\arrow[d,bend left=10, "x"right]
				\\
				\diamond
				\arrow[u,bend left=10, "xz"left]
				\arrow[mfarrow,leftarrow]{ru}[pos=0.2,below right,xshift=-5pt]{z^{c}}
				\&
				\diamond
				\arrow[u,bend left=10, "yz"left]
			}
		\end{gather*}
		\caption{%
			A basis of the space of morphisms \(M_1\rightarrow M_2\). 
		}
		\label{fig:Acone:morphisms:12}
	\end{subfigure}
	\caption{%
		A basis of some of the morphism spaces in \(\Acone\) and the action of the differential
		(\(a,b,c\geq0\))
	}
	\label{fig:Acone:morphisms}
\end{figure}

			With the exception of 
			the identities \eqref{prop:quasi-iso:Acone:morphisms} and a different choice of grading on the algebra 
			(see \cite[paragraph after Definition~7.2]{AAEKO}), 
			this is essentially \cite[Proposition~7.9]{AAEKO}. 
			
			\begin{proof}[Proof of \cref{prop:quasi-iso:Acone}]
				We follow the same strategy 
				as in the proof of \cref{prop:quasi-iso:BinfCone}. 
				We first compute bases and differentials of the morphism spaces
				\(\Acone(M_1,M_1)\) and \(\Acone(M_1,M_2)\)
				in \cref{fig:Acone:morphisms}. 
				(The signs are computed using the sign conventions 
				from \cref{def:additive_enlargement_Seidel}.) 
				These bases give rise to similar bases 
				on the morphism spaces \(\Acone(M_i,M_j)\) 
				for all \(i,j\in\{1,2,3\}\), 
				by simultaneous permutation 
				of the indices \(1,2,3\) 
				and the variables \(x,y,z\), 
				up to signs on some arrows that stem from the fact 
				that these permutations are not compatible 
				with the homological grading on the three matrix factorizations. 
				
				We now apply Seidel's homological perturbation lemma 
				\cite[Section~I~(1i)]{Seidel} 
				to construct an \(A_\infty\)-structure on \(H\Acone\) 
				together with a quasi-isomorphism to \(\Acone\). 
				For any two objects \(X,Y\) of \(\Acone\), let
				\[
				\begin{tikzcd}[column sep=3cm]
					H\Acone(X,Y)
					\arrow[bend left=5]{r}{\mathcal{F}^1}
					&
					\Acone(X,Y)
					\arrow[bend left=5]{l}{\mathcal{G}^1}
					\arrow[loop right, distance=3em, start anchor={[yshift=1ex]east}, end anchor={[yshift=-1ex]east}]{rl}{T^1}
				\end{tikzcd}
				\]
				be the homomorphisms constructed in the first part of the proof, i.e.\ \(\mathcal{F}^1\) and \(\mathcal{G}^1\) are chain maps and \(T^1\) is the chain homotopy between \(\mathcal{F}^1\mathcal{G}^1\) and the identity:
				\begin{equation}
					T^1\mu_1^{\Acone}
					+
					\mu_1^{\Acone}T^1
					=
					\mathcal{F}^1\mathcal{G}^1-1
				\end{equation}
				As in the proof of \cref{prop:quasi-iso:BinfCone}, 
				\(\mathcal{F}^1\) and \(\mathcal{G}^1\) are the inclusion and projection map, respectively, 
				and
				\(T^1\) is the negative of the inverse of the differential.
				We can choose all maps to preserve the gradings. 
				To determine the induced multiplication on \(H\Acone\), define			
				\begin{align*}
					\Psi^1(u_{12})
					&=
					\vc{%
						\begin{tikzcd}[ampersand replacement=\&]
							\diamond
							\arrow[d,bend left=10, "y"right]
							\arrow[mfarrow,leftarrow]{rd}[mftop]{z}
							\&
							\diamond
							\arrow[d,bend left=10, "x"right]
							\\
							\diamond
							\arrow[u,bend left=10, "xz"left]
							\arrow[mfarrow,leftarrow]{ru}[mfbot]{1}
							\&
							\diamond
							\arrow[u,bend left=10, "yz"left]
					\end{tikzcd}}
					&
					\Psi^1(u_{23})
					&=
					\vc{%
						\begin{tikzcd}[ampersand replacement=\&]
							\diamond
							\arrow[d,bend left=10, "z"right]
							\arrow[mfarrow,leftarrow]{rd}[mftop]{x}
							\&
							\diamond
							\arrow[d,bend left=10, "y"right]
							\\
							\diamond
							\arrow[u,bend left=10, "xy"left]
							\arrow[mfarrow,leftarrow]{ru}[mfbot]{1}
							\&
							\diamond
							\arrow[u,bend left=10, "xz"left]
					\end{tikzcd}}
					&
					\Psi^1(u_{31})
					&=
					\vc{%
						\begin{tikzcd}[ampersand replacement=\&]
							\diamond
							\arrow[d,bend left=10, "x"right]
							\arrow[mfarrow,leftarrow]{rd}[mftop]{y}
							\&
							\diamond
							\arrow[d,bend left=10, "z"right]
							\\
							\diamond
							\arrow[u,bend left=10, "yz"left]
							\arrow[mfarrow,leftarrow]{ru}[mfbot]{\!\!\!-1}
							\&
							\diamond
							\arrow[u,bend left=10, "xy"left]
					\end{tikzcd}}
					\\
					\Psi^1(v_{21})
					&=
					\vc{%
						\begin{tikzcd}[ampersand replacement=\&]
							\diamond
							\arrow[d,bend left=10, "x"right]
							\arrow[mfarrow,leftarrow]{rd}[mftop]{z}
							\&
							\diamond
							\arrow[d,bend left=10, "y"right]
							\\
							\diamond
							\arrow[u,bend left=10, "yz"left]
							\arrow[mfarrow,leftarrow]{ru}[mfbot]{1}
							\&
							\diamond
							\arrow[u,bend left=10, "yz"left]
					\end{tikzcd}}
					&
					\Psi^1(v_{32})
					&=
					\vc{%
						\begin{tikzcd}[ampersand replacement=\&]
							\diamond
							\arrow[d,bend left=10, "y"right]
							\arrow[mfarrow,leftarrow]{rd}[mftop]{x}
							\&
							\diamond
							\arrow[d,bend left=10, "z"right]
							\\
							\diamond
							\arrow[u,bend left=10, "xz"left]
							\arrow[mfarrow,leftarrow]{ru}[mfbot]{1}
							\&
							\diamond
							\arrow[u,bend left=10, "xy"left]
					\end{tikzcd}}
					&
					\Psi^1(v_{13})
					&=
					\vc{%
						\begin{tikzcd}[ampersand replacement=\&]
							\diamond
							\arrow[d,bend left=10, "z"right]
							\arrow[mfarrow,leftarrow]{rd}[mftop]{y}
							\&
							\diamond
							\arrow[d,bend left=10, "x"right]
							\\
							\diamond
							\arrow[u,bend left=10, "xy"left]
							\arrow[mfarrow,leftarrow]{ru}[mfbot]{\!\!\!-1}
							\&
							\diamond
							\arrow[u,bend left=10, "yz"left]
					\end{tikzcd}}
				\end{align*}
				Applying \cite[Formula~(1.18)]{Seidel}, 
				one checks that the morphisms
				\(\mu_{2}(\Psi^1(u_{i,i+1}),\Psi^1(u_{i-1,i}))\)
				and 
				\(\mu_{2}(\Psi^1(v_{i,i-1}),\Psi^1(v_{i+1,i}))\)
				are null-homotopic
				for all \(i=1,2,3\), 
				and that \(\Psi^1\) determines an algebra isomorphism between \(\Ccone\) and \(H\Acone\).
				In particular, 
				the identities 
				from \eqref{prop:quasi-iso:Acone:morphisms} 
				hold. 
				 
				Finally, the map \(\Psi^1\) induces an \(A_\infty\)-structure on \(\Ccone\). 
				Using \cite[Formula~(1.18)]{Seidel},
				we verify the identities 
				from \eqref{prop:quasi-iso:BinfCone:mu-iii}:
				\begin{multline*}
					\mu_{3}(u_{31},u_{23},u_{12})
					=
					\mu_{3}(\Psi^1(u_{31}),\Psi^1(u_{23}),\Psi^1(u_{12}))
					=
					\mu_3^{H_*\Acone}\Bigg(
					\vc{%
						\begin{tikzcd}[ampersand replacement=\&]
							\diamond
							\arrow[d,bend left=10, "x"right]
							\arrow[mfarrow,leftarrow]{rd}[mftop]{y}
							\&
							\diamond
							\arrow[d,bend left=10, "z"right]
							\arrow[mfarrow,leftarrow]{rd}[mftop]{x}
							\&
							\diamond
							\arrow[d,bend left=10, "y"right]
							\arrow[mfarrow,leftarrow]{rd}[mftop]{z}
							\&
							\diamond
							\arrow[d,bend left=10, "x"right]
							\\
							\diamond
							\arrow[u,bend left=10, "yz"left]
							\arrow[mfarrow,leftarrow]{ru}[mfbot]{\!\!\!-1}
							\&
							\diamond
							\arrow[u,bend left=10, "xy"left]
							\arrow[mfarrow,leftarrow]{ru}[mfbot]{1}
							\&
							\diamond
							\arrow[u,bend left=10, "xz"left]
							\arrow[mfarrow,leftarrow]{ru}[mfbot]{1}
							\&
							\diamond
							\arrow[u,bend left=10, "yz"left]
					\end{tikzcd}}
					\Bigg)
					\\
					=\mathcal{G}^1\Bigg(
					\mu_2^{\Acone}
					\Bigg(
					\underbrace{
						\mathcal{F}^2\Bigg(
						\vc{%
							\begin{tikzcd}[ampersand replacement=\&]
								\diamond
								\arrow[d,bend left=10, "x"right]
								\arrow[mfarrow,leftarrow]{rd}[mftop]{y}
								\&
								\diamond
								\arrow[d,bend left=10, "z"right]
								\arrow[mfarrow,leftarrow]{rd}[mftop]{x}
								\&
								\diamond
								\arrow[d,bend left=10, "y"right]
								\\
								\diamond
								\arrow[u,bend left=10, "yz"left]
								\arrow[mfarrow,leftarrow]{ru}[mfbot]{\!\!\!-1}
								\&
								\diamond
								\arrow[u,bend left=10, "xy"left]
								\arrow[mfarrow,leftarrow]{ru}[mfbot]{1}
								\&
								\diamond
								\arrow[u,bend left=10, "xz"left]
						\end{tikzcd}}
						\Bigg)}_{= f_1}
					,
					\vc{%
						\begin{tikzcd}[ampersand replacement=\&]
							\diamond
							\arrow[d,bend left=10, "y"right]
							\arrow[mfarrow,leftarrow]{rd}[mftop]{z}
							\&
							\diamond
							\arrow[d,bend left=10, "x"right]
							\\
							\diamond
							\arrow[u,bend left=10, "xz"left]
							\arrow[mfarrow,leftarrow]{ru}[mfbot]{1}
							\&
							\diamond
							\arrow[u,bend left=10, "yz"left]
					\end{tikzcd}}
					\Bigg)
					\Bigg)
					\\
					+\mathcal{G}^1\Bigg(
					\mu_2^{\Acone}
					\Bigg(
					\vc{%
						\begin{tikzcd}[ampersand replacement=\&]
							\diamond
							\arrow[d,bend left=10, "x"right]
							\arrow[mfarrow,leftarrow]{rd}[mftop]{y}
							\&
							\diamond
							\arrow[d,bend left=10, "z"right]
							\\
							\diamond
							\arrow[u,bend left=10, "yz"left]
							\arrow[mfarrow,leftarrow]{ru}[mfbot]{\!\!\!-1}
							\&
							\diamond
							\arrow[u,bend left=10, "xy"left]
					\end{tikzcd}}
					,
					\underbrace{
						\mathcal{F}^2\Bigg(
						\vc{%
							\begin{tikzcd}[ampersand replacement=\&]
								\diamond
								\arrow[d,bend left=10, "z"right]
								\arrow[mfarrow,leftarrow]{rd}[mftop]{x}
								\&
								\diamond
								\arrow[d,bend left=10, "y"right]
								\arrow[mfarrow,leftarrow]{rd}[mftop]{z}
								\&
								\diamond
								\arrow[d,bend left=10, "x"right]
								\\
								\diamond
								\arrow[u,bend left=10, "xy"left]
								\arrow[mfarrow,leftarrow]{ru}[mfbot]{1}
								\&
								\diamond
								\arrow[u,bend left=10, "xz"left]
								\arrow[mfarrow,leftarrow]{ru}[mfbot]{1}
								\&
								\diamond
								\arrow[u,bend left=10, "yz"left]
						\end{tikzcd}}
						\Bigg)}_{= f_2}
					\Bigg)
					\Bigg)
				\end{multline*}
				As in the proof of \cref{prop:quasi-iso:BinfCone}, 
				the symbols $\mathcal{F}^1$ 
				have been dropped to lighten the expression above. 
				Also note that \(\Acone\) is a dg algebra 
				so \(\mu_d^{\Acone}=0\) for \(d>2\). 
				We calculate
				\begin{align*}
					f_1
					&=
					T^1
					\Bigg(
					\mu_2^{\Acone}
					\Bigg(
					\vc{%
						\begin{tikzcd}[ampersand replacement=\&]
							\diamond
							\arrow[d,bend left=10, "x"right]
							\arrow[mfarrow,leftarrow]{rd}[mftop]{y}
							\&
							\diamond
							\arrow[d,bend left=10, "z"right]
							\\
							\diamond
							\arrow[u,bend left=10, "yz"left]
							\arrow[mfarrow,leftarrow]{ru}[mfbot]{\!\!\!-1}
							\&
							\diamond
							\arrow[u,bend left=10, "xy"left]
					\end{tikzcd}}
					,
					\vc{%
						\begin{tikzcd}[ampersand replacement=\&]
							\diamond
							\arrow[d,bend left=10, "z"right]
							\arrow[mfarrow,leftarrow]{rd}[mftop]{x}
							\&
							\diamond
							\arrow[d,bend left=10, "y"right]
							\\
							\diamond
							\arrow[u,bend left=10, "xy"left]
							\arrow[mfarrow,leftarrow]{ru}[mfbot]{1}
							\&
							\diamond
							\arrow[u,bend left=10, "xz"left]
					\end{tikzcd}}
					\Bigg)
					\Bigg)
					=
					T^1
					\Bigg(
					\vc{%
						\begin{tikzcd}[ampersand replacement=\&]
							\diamond
							\arrow[d,bend left=10, "x"right]
							\arrow[mfarrow,leftarrow]{r}[mfdes,description]{y}
							\&
							\diamond
							\arrow[d,bend left=10, "y"right]
							\\
							\diamond
							\arrow[u,bend left=10, "yz"left]
							\arrow[mfarrow,leftarrow]{r}[mfdes,description]{-x}
							\&
							\diamond
							\arrow[u,bend left=10, "xz"left]
					\end{tikzcd}}
					\Bigg)
					=
					\vc{%
						\begin{tikzcd}[ampersand replacement=\&]
							\diamond
							\arrow[d,bend left=10, "x"right]
							\arrow[mfarrow,leftarrow]{rd}[mftop]{1}
							\&
							\diamond
							\arrow[d,bend left=10, "y"right]
							\\
							\diamond
							\arrow[u,bend left=10, "yz"left]
							\&
							\diamond
							\arrow[u,bend left=10, "xz"left]
					\end{tikzcd}}
				\end{align*}
				and similarly
				\begin{align*}
					f_2
					&=
					T^1
					\Bigg(
					\mu_2^{\Acone}
					\Bigg(
					\vc{%
						\begin{tikzcd}[ampersand replacement=\&]
							\diamond
							\arrow[d,bend left=10, "z"right]
							\arrow[mfarrow,leftarrow]{rd}[mftop]{x}
							\&
							\diamond
							\arrow[d,bend left=10, "y"right]
							\\
							\diamond
							\arrow[u,bend left=10, "xy"left]
							\arrow[mfarrow,leftarrow]{ru}[mfbot]{1}
							\&
							\diamond
							\arrow[u,bend left=10, "xz"left]
					\end{tikzcd}}
					,
					\vc{%
						\begin{tikzcd}[ampersand replacement=\&]
							\diamond
							\arrow[d,bend left=10, "y"right]
							\arrow[mfarrow,leftarrow]{rd}[mftop]{z}
							\&
							\diamond
							\arrow[d,bend left=10, "x"right]
							\\
							\diamond
							\arrow[u,bend left=10, "xz"left]
							\arrow[mfarrow,leftarrow]{ru}[mfbot]{1}
							\&
							\diamond
							\arrow[u,bend left=10, "yz"left]
					\end{tikzcd}}
					\Bigg)
					\Bigg)
					=
					T^1
					\Bigg(
					\vc{%
						\begin{tikzcd}[ampersand replacement=\&]
							\diamond
							\arrow[d,bend left=10, "z"right]
							\arrow[mfarrow,leftarrow]{r}[mfdes,description]{x}
							\&
							\diamond
							\arrow[d,bend left=10, "x"right]
							\\
							\diamond
							\arrow[u,bend left=10, "xy"left]
							\arrow[mfarrow,leftarrow]{r}[mfdes,description]{z}
							\&
							\diamond
							\arrow[u,bend left=10, "yz"left]
					\end{tikzcd}}
					\Bigg)
					=
					\vc{%
						\begin{tikzcd}[ampersand replacement=\&]
							\diamond
							\arrow[d,bend left=10, "z"right]
							\arrow[mfarrow,leftarrow]{rd}[mftop]{1}
							\&
							\diamond
							\arrow[d,bend left=10, "x"right]
							\\
							\diamond
							\arrow[u,bend left=10, "xy"left]
							\&
							\diamond
							\arrow[u,bend left=10, "yz"left]
					\end{tikzcd}}
				\end{align*}
				We obtain
				\begin{align*}
					\mu_{3}(u_{31},u_{23},u_{12})
					&= 
					\mathcal{G}^1\Bigg(
					\vc{%
						\begin{tikzcd}[ampersand replacement=\&]
							\diamond
							\arrow[d,bend left=10, "x"right]
							\arrow[mfarrow,leftarrow]{r}[mfdes,description]{1}
							\&
							\diamond
							\arrow[d,bend left=10, "x"right]
							\\
							\diamond
							\arrow[u,bend left=10, "yz"left]
							\&
							\diamond
							\arrow[u,bend left=10, "yz"left]
					\end{tikzcd}}
					\Bigg)
					+
					\mathcal{G}^1\Bigg(
					\vc{%
						\begin{tikzcd}[ampersand replacement=\&]
							\diamond
							\arrow[d,bend left=10, "x"right]
							\&
							\diamond
							\arrow[d,bend left=10, "x"right]
							\\
							\diamond
							\arrow[u,bend left=10, "yz"left]
							\arrow[mfarrow,leftarrow]{r}[mfdes,description]{1}
							\&
							\diamond
							\arrow[u,bend left=10, "yz"left]
					\end{tikzcd}}
					\Bigg)
					=
					1
				\end{align*}
				Similarly,
				\[
				\begin{multlined}[b]
					\mu_{3}(v_{21},v_{32},v_{13})
					=\mathcal{G}^1\Bigg(
					\mu_2^{\Acone}
					\Bigg(
					\mathcal{F}^2\Bigg(
					\vc{%
						\begin{tikzcd}[ampersand replacement=\&]
							\diamond
							\arrow[d,bend left=10, "x"right]
							\arrow[mfarrow,leftarrow]{rd}[mftop]{z}
							\&
							\diamond
							\arrow[d,bend left=10, "y"right]
							\arrow[mfarrow,leftarrow]{rd}[mftop]{x}
							\&
							\diamond
							\arrow[d,bend left=10, "z"right]
							\\
							\diamond
							\arrow[u,bend left=10, "yz"left]
							\arrow[mfarrow,leftarrow]{ru}[mfbot]{1}
							\&
							\diamond
							\arrow[u,bend left=10, "xz"left]
							\arrow[mfarrow,leftarrow]{ru}[mfbot]{1}
							\&
							\diamond
							\arrow[u,bend left=10, "xy"left]
					\end{tikzcd}}
					\Bigg)
					,
					\vc{%
						\begin{tikzcd}[ampersand replacement=\&]
							\diamond
							\arrow[d,bend left=10, "z"right]
							\arrow[mfarrow,leftarrow]{rd}[mftop]{y}
							\&
							\diamond
							\arrow[d,bend left=10, "x"right]
							\\
							\diamond
							\arrow[u,bend left=10, "xy"left]
							\arrow[mfarrow,leftarrow]{ru}[mfbot]{\!\!\!-1}
							\&
							\diamond
							\arrow[u,bend left=10, "yz"left]
					\end{tikzcd}}
					\Bigg)
					\Bigg)
					\\
					+\mathcal{G}^1\Bigg(
					\mu_2^{\Acone}
					\Bigg(
					\vc{%
						\begin{tikzcd}[ampersand replacement=\&]
							\diamond
							\arrow[d,bend left=10, "x"right]
							\arrow[mfarrow,leftarrow]{rd}[mftop]{z}
							\&
							\diamond
							\arrow[d,bend left=10, "y"right]
							\\
							\diamond
							\arrow[u,bend left=10, "yz"left]
							\arrow[mfarrow,leftarrow]{ru}[mfbot]{1}
							\&
							\diamond
							\arrow[u,bend left=10, "xz"left]
					\end{tikzcd}}
					,
					\mathcal{F}^2\Bigg(
					\vc{%
						\begin{tikzcd}[ampersand replacement=\&]
							\diamond
							\arrow[d,bend left=10, "y"right]
							\arrow[mfarrow,leftarrow]{rd}[mftop]{x}
							\&
							\diamond
							\arrow[d,bend left=10, "z"right]
							\arrow[mfarrow,leftarrow]{rd}[mftop]{y}
							\&
							\diamond
							\arrow[d,bend left=10, "x"right]
							\\
							\diamond
							\arrow[u,bend left=10, "xz"left]
							\arrow[mfarrow,leftarrow]{ru}[mfbot]{1}
							\&
							\diamond
							\arrow[u,bend left=10, "xy"left]
							\arrow[mfarrow,leftarrow]{ru}[mfbot]{\!\!\!-1}
							\&
							\diamond
							\arrow[u,bend left=10, "yz"left]
					\end{tikzcd}}
					\Bigg)
					\Bigg)
					\Bigg)
					\\
					=\mathcal{G}^1\Bigg(
					\mu_2^{\Acone}
					\Bigg(
					\vc{%
						\begin{tikzcd}[ampersand replacement=\&]
							\diamond
							\arrow[d,bend left=10, "x"right]
							\arrow[mfarrow,leftarrow]{rd}[mftop]{1}
							\&
							\diamond
							\arrow[d,bend left=10, "z"right]
							\\
							\diamond
							\arrow[u,bend left=10, "yz"left]
							\&
							\diamond
							\arrow[u,bend left=10, "xy"left]
					\end{tikzcd}}
					,
					\vc{%
						\begin{tikzcd}[ampersand replacement=\&]
							\diamond
							\arrow[d,bend left=10, "z"right]
							\arrow[mfarrow,leftarrow]{rd}[mftop]{y}
							\&
							\diamond
							\arrow[d,bend left=10, "x"right]
							\\
							\diamond
							\arrow[u,bend left=10, "xy"left]
							\arrow[mfarrow,leftarrow]{ru}[mfbot]{\!\!\!-1}
							\&
							\diamond
							\arrow[u,bend left=10, "yz"left]
					\end{tikzcd}}
					\Bigg)
					\Bigg)
					+
					\mathcal{G}^1\Bigg(
					\mu_2^{\Acone}
					\Bigg(
					\vc{%
						\begin{tikzcd}[ampersand replacement=\&]
							\diamond
							\arrow[d,bend left=10, "x"right]
							\arrow[mfarrow,leftarrow]{rd}[mftop]{z}
							\&
							\diamond
							\arrow[d,bend left=10, "y"right]
							\\
							\diamond
							\arrow[u,bend left=10, "yz"left]
							\arrow[mfarrow,leftarrow]{ru}[mfbot]{1}
							\&
							\diamond
							\arrow[u,bend left=10, "xz"left]
					\end{tikzcd}}
					,
					\vc{%
						\begin{tikzcd}[ampersand replacement=\&]
							\diamond
							\arrow[d,bend left=10, "y"right]	
							\arrow[mfarrow,leftarrow]{rd}[mftop]{1}
							\&
							\diamond
							\arrow[d,bend left=10, "x"right]
							\\
							\diamond
							\arrow[u,bend left=10, "xz"left]
							\&
							\diamond
							\arrow[u,bend left=10, "yz"left]
					\end{tikzcd}}
					\Bigg)
					\Bigg)
					\\
					=
					\mathcal{G}^1\Bigg(
					\vc{%
						\begin{tikzcd}[ampersand replacement=\&]
							\diamond
							\arrow[d,bend left=10, "x"right]
							\arrow[mfarrow,leftarrow]{r}[mfdes,description]{1}
							\&
							\diamond
							\arrow[d,bend left=10, "x"right]
							\\
							\diamond
							\arrow[u,bend left=10, "yz"left]
							\arrow[mfarrow,leftarrow]{r}[mfdes,description]{1}
							\&
							\diamond
							\arrow[u,bend left=10, "yz"left]
					\end{tikzcd}}
					\Bigg)
					=
					1
				\end{multlined}
			\qedhere
			\]
			\end{proof}
		}

\subsection{The quasi-isomorphism between \texorpdfstring{\(\Binf\)}{B\^{}∞} and \texorpdfstring{\(\A\)}{A}}
\label{sec:quasi-iso:identification}
		\cref{thm:quasi-iso} now follows from the algebraic setup in \cite[Section 3]{AAEKO}, having first made an adjustment for gradings.

	\begin{lemma}
		\label{lem:a-infty:shift-grading-so-f-objects}
		Given
		an \(A_\infty\)-category \(\C\) 
		and 
		an object \(X_0\in\C\), 
		define a \(\{0,1\}\)-valued map 
		\(\varepsilon\) on objects of \(\C\)
		by setting 
		\(\varepsilon(X)=1\) 
		if and only if 
		\(X=X_0\). 
		There exists an \(A_\infty\)-category \(\C'\) 
		whose objects are the same as the objects in \(\C\), 
		whose morphism spaces are defined by
		\[
		\C'(X,Y)
		=
		h^{\varepsilon(Y)-\varepsilon(X)}\C(X,Y)
		\]
		for any two objects \(X,Y\in\C\), 
		and for which 
		\(\mu^\C_1=0\) 
		if and only if 
		\(\mu^
		{\C'}_1=0\). 
		Furthermore, 
		suppose \(\mu^\C_1=0\) and 
		that \(\mu\) and \(\nu\) are two \(A_\infty\)-structures 
		on \(H\C\). 
		Let 
		\(\mu'\) and \(\nu'\) 
		be the induced \(A_\infty\)-structures 
		on \(H\C'\). 
		Then \(\mu\) and \(\nu\) are strictly homotopic 
		if and only if 
		\(\mu'\) and \(\nu'\) are strictly homotopic.
	\end{lemma}
	
	\begin{proof}
		The proof is based on a construction 
		that we learned from 
		\cite[Section 3.1]{Kajiura}. Write \(a'\in\C'(X,Y)\) 
		for the morphism corresponding to a morphism 
		\(a\in\C(X,Y)\).
		Let \((\mu_k)_{k>0}\) be the structure maps on \(\C\). 
		Then we define the structure maps 
		\((\mu'_k)_{k>0}\) on \(\C'\) 
		as follows:
		Given 
		a sequence of objects \(Y_0,\dots,Y_k\)
		and 
		morphisms \(a_i\in\C(Y_{i-1},Y_i)\) set 
		\[
		\mu'_n(a'_n,\dots,a'_1)
		=
		(-1)^{\varepsilon(Y_0)}
		(\mu_k(a_n,\dots,a_1))'
		\]
		We check that 
		the maps \((\mu'_n)_{n>0}\) satisfy the \(A_\infty\)-relations,
		which are equivalent to
		\[
		0
		=
		\sum_{\substack{i,k\geq0,j\geq1:\\ i+j+k=n}}
		(-1)^{h(a'_k)+\dots+h(a'_1) + k+\varepsilon(Y_0)+\varepsilon(Y_k)}
		(\mu_{i+1+k} (a_n,\dots,\mu_j(a_{n-i},\dots,a_{k+1}),\dots,a_1))'
		\]
		Observe that \(h(a'_i)=h(a_i)+\varepsilon(Y_i)-\varepsilon(Y_{i-1})\) 
		for all \(i=1,\dots,k\). 
		Simplifying the telescoping sum, 
		the sign in each summand is 
		\(
		(-1)^{h(a_k)+\dots+h(a_1)+k}
		\)
		and we recover the structure relations for \((\mu_n)_{n>0}\). 
		By construction,
		\(\mu^\C_1=0\) if and only if \(\mu^
		{\C'}_1=0\). 
		
		For the second part of the lemma,
		denote the \(A_\infty\)-categories
		\((H\C,\mu)\), \((H\C',\mu')\),
		\((H\C,\nu)\) and \((H\C',\nu')\) 
		by 
		\(\C\), \(\C'\),
		\(\D\) and \(\D'\), 
		respectively. 
		Suppose \(\mathcal{F}\co \C\rightarrow\D\) 
		is an \(A_\infty\)-functor
		acting as the identity on objects. 
		Define 
		\(\mathcal{F}'\co \C'\rightarrow\D'\)
		by 
		\(
		\mathcal{F}'_k(a'_k,\dots,a'_1)
		=
		(\mathcal{F}_k(a_k,\dots,a_1))'.
		\)
		A computation very similar to the one above shows that 
		\(\mathcal{F}'\) is indeed an \(A_\infty\)-functor; 
		the only difference is %
		an overall sign 
		\((-1)^{\varepsilon(Y_0)}\).
		Moreover, 
		any functor 
		\(\C'\rightarrow\D'\) 
		is equal to \(\mathcal{F}'\) 
		for some functor 
		\(\mathcal{F}\co \C\rightarrow\D\). 
		Hence %
		\(\mathcal{F}\) is a strict homotopy 
		if and only if 
		\(\mathcal{F}'\) is a strict homotopy. 
	\end{proof}

	\begin{proposition}\label{prop:AAEKO}
			Any \(A_{\infty}\)-structure \(\mu\) on \(\Ccone\) is uniquely determined up to strict homotopy by the coefficient of the identity map \(1_{L_1}\) of 
			\[
			\mu_{3}(u_{31},u_{23},u_{12})
			\quad
			\text{and}
			\quad
			\mu_{3}(v_{21},v_{32},v_{13}).
			\]
		\end{proposition}

	\begin{proof}%
		Let \(\mu\) be an \(A_\infty\)-structure on \(\Ccone\). 
		We would like to apply 
		part (2) of Proposition~3.1 in \cite{AAEKO} 
		for the case \(n=3\).
		However, 
		the hypotheses of this result require 
		a \(\Z\)-grading on \(\Ccone\) satisfying 
		\cite[condition (2.4)]{AAEKO}. 
		This condition consists of two parts, namely
		\[
		h(u_{31})+h(u_{23})+h(u_{12})=h(v_{21})+h(v_{32})+h(v_{13})=1
		\]
		and all terms in this expressions must be odd. 
		While the first part is satisfied 
		the second part is clearly not.  
		To fix this, apply
		\cref{lem:a-infty:shift-grading-so-f-objects}
		to the category \(\Ccone\) with its \(A_\infty\)-structure \(\mu\) and the object \(L_2\), %
		resulting in %
		a %
		category \(\Ccone'\) with %
		\(A_\infty\)-structure \(\mu'\)
		for which \cite[condition (2.4)]{AAEKO} is %
		satisfied. 
		Moreover, 
		the \(A_\infty\)-structures 
		\(\mu\) and \(\mu'\) 
		determine each other, 
		as is clear from the formula for \(\mu'\) 
		in the proof of \cref{lem:a-infty:shift-grading-so-f-objects}. 
		So we can now apply 
		part (2) of Proposition~3.1 in \cite{AAEKO} 
		for the case \(n=3\), 
		combined with the second statement of 
		\cref{lem:a-infty:shift-grading-so-f-objects}.
	\end{proof}		
		
		\begin{proof}[Proof of \cref{thm:quasi-iso}]
			\cref{prop:AAEKO} gives a strict homotopy 
			\(\mathcal{H}\co\Ccone\rightarrow\Ccone\) 
			between the two \(A_\infty\)-structures on \(\Ccone\) defined 
			in \cref{prop:quasi-iso:BinfCone,prop:quasi-iso:Acone}.
			Let \(\Psi^{-1}\) be the inverse of the quasi-isomorphism \(\Psi\) 
			constructed in the proof of \cref{prop:quasi-iso:Acone} 
			and consider the functor
			\(
			\Phi\circ\mathcal{H}\circ\Psi^{-1}
			\co
			\Acone\rightarrow\Bcone
			\).
			As a composition of \(q\)-filtered quasi-isomorphisms, 
			this is again a \(q\)-filtered quasi-isomorphism. 
			Its restriction to the full subcategories 
			\(\A\subset\Acone\)
			and 
			\(\Binf\subset\Bcone\)
			gives the desired quasi-isomorphism
			\(
			\A
			\rightarrow
			\Binf
			\).
			Indeed, by construction, it sends
			\(\sW\mapsto\Sw\), 
			\(\dW\mapsto\Dw\), 
			\(\sB\mapsto\Sb\), and
			\(\dB\mapsto\Db\).
		\end{proof}

\section{%
Extensions and \texorpdfstring{\(\boldsymbol{*}\)}{*}
}
\label{sec:no_wrapping_around_special}

In this section we restrict coefficients to fields
\(\CoeffRing=\CoeffField\). 
This is because multicurve invariants are, in general, only defined over fields
(\cref{thm:classification:complexes_over_B:simplified}),
and because 
\cref{thm:quasi-iso} 
has so far only been established over 
\(\CoeffField\). 
We expect that the latter result can be extended to arbitrary coefficients,
in which case most results in this section
would generalize accordingly. 

\subsection{%
	The extension property}
\label{sec:no_wrapping_around_special:extension_prop}
With the \(A_\infty\)-algebra \(\Binf\) from \cref{sec:quasi-iso:Binf} as a starting point,
we introduce an \(A_\infty\)-algebra~\(\BinfU\) that comes with an \(A_\infty\)-epimorphism \(\BinfU\rightarrow\B\).
The difference between \(\Binf\) and \(\BinfU\) is cosmetic:

\begin{definition}\label{def:BinfU}
	An \(A_\infty\)-algebra is bigraded if 
	it carries a bigrading and 
	its higher multiplication maps \(\mu_d\) 
	have grading \(h^{2-d}q^0\) 
	for all \(d\geq1\).
	Consider the bigraded ring
	\(\field[U]\)
	where 
	\(\text{gr}(U)=h^{-2}q^{-6}\).
	Let \(\BinfU\) be the bigraded \(A_\infty\)-algebra that agrees with \(\B[U]\) as a \(\field[U]\)-algebra and with (higher) products defined as follows.
	For any positive integer \(m>0\), 
	any sequence \((n_{2m},\dots,n_1)\) of non-negative numbers, and 
	any sequence \((a_{2m},\ldots,a_{1})\) of algebra elements in \(\B\), set
	\begin{equation}\label{eq:def:BinfU}
		\mu^{\BinfU}_{2m}(U^{n_{2m}}\cdot a_{2m},\ldots,U^{n_{1}}\cdot a_{1})
		=
		(-1)^{m-1}U^{(m-1)+(n_{2m}+\dots+n_{1})} 
		\cdot 
		\mu^{\Binf}_{2m}(a_{2m},\ldots,a_{1})
	\end{equation}
	and set \(\mu^{\BinfU}_{k}\equiv 0\) for any odd integer \(k\).
	Define the homological grading on the restriction of \(\BinfU\) to \(\B\) to be identically 0. 
\end{definition}

\begin{lemma}
	\(\BinfU\) is a bigraded \(A_\infty\)-algebra. 
\end{lemma}
\begin{proof}
	By linearity in \(U\) of \eqref{eq:def:BinfU},
	it suffices to check the \(A_\infty\)-relations
	for the maps \(\mu^{\BinfU}_m\) 
	on sequences of elements in \(\B\). 
	Doing this follows the same argument 
	as in the proof of 
	\cref{lem:a-infty-reltations-for-Binf}. 
	Indeed, 
	since \(\mu^{\BinfU}_{k}=0\) for any odd integer \(k\),
	the \(A_\infty\)-relations for \((\mu^{\BinfU}_{k})_{k>0}\)
	are trivially satisfied on even length sequences of algebra elements. 
	Each term of the \(A_\infty\)-relation 
	for any odd length sequence of algebra elements in \(\B\) 
	is of the form 
	\[
	\mu^{\BinfU}_{2r}(\dots,\mu^{\BinfU}_{2s}(\dots),\dots)
	=
	(-U)^{r+s}%
	\mu^{\Binf}_{2r}(\dots,\mu^{\Binf}_{2s}(\dots),\dots)
	\]
	where \(r+s\) is constant. 
	Thus
	the maps \(\mu^{\BinfU}_{k}\) indeed satisfy the \(A_\infty\)-relations and,
	moreover, their gradings are equal to \(h^{2-k}q^0\),
	which can be seen using \cref{rem:quantum_grading_on_Binf}.
\end{proof}

\begin{remark}
	Some examples for non-zero higher products in \(\BinfU\), where \(k\geq1\) (idempotents are suppressed):
	\begin{align*}
		\mu_4(D,S,D,S^k)
		&=
		U\cdot S^{k-1}
		&
		\mu_4(D^k,S,D,S)
		&=
		U\cdot D^{k-1}
		\\
		\mu_4(S,D,S,D^k)
		&=
		U\cdot D^{k-1}
		&
		\mu_4(S^k,D,S,D)
		&=
		U\cdot S^{k-1}
		\\
		\mu_6(S,D,S^2,D,S,D^{k+1})
		&=
		U^2\cdot D^{k-1}
		&
		\mu_6(S^k,D,S^2,D,S,D^2 )
		&=
		U^2\cdot S^{k-1} 
	\end{align*}
	In fact, there are no other sequences of basic algebra elements for which \(\mu_4\) is non-vanishing. 
\end{remark}

There is a quotient map  \(\BinfU \rightarrow \B\) sending \(U\) to zero that ``forgets'' the \(A_\infty\)-operations on \(\Binf\),
inducing a functor between the corresponding categories of type D structures over the respective algebras.

\begin{theorem}[Extension property]\label{thm:extension}
	For any oriented pointed Conway tangle \(T\)
	there exist bigraded twisted complexes \(\DD^c(T)^{\BinfU}\) and \(\DD^c_1(T)^{\BinfU}\) such that
	\[
	\DD^c(T)^{\BinfU}\Big|_{U=0} 
	= 
	\DD^c(T)^{\B} 
	\quad 
	\text{and} 
	\quad 
	\DD^c_1(T)^{\BinfU}\Big|_{U=0} 
	= 
	\DD^c_1(T)^{\B}
	\]
\end{theorem}

A bigraded twisted complex is a twisted complex whose differential has grading \(h^1q^0\). 

\begin{remark}
	Although \cref{thm:extension} is purely algebraic (in both statement and proof) its  origin lies in symplectic geometry. 
	As discussed in \cref{sec:quasi-iso:pair-of-pants},
	the algebras~\(\B\) and \(\Binf\) encode the Fukaya categories 
	of the punctured spheres~\(\FourPuncturedSphere\) and \(S^2_3\),
	so that passing from \(\B\) to \(\Binf\) 
	corresponds to filling in the distinguished puncture~\(\ast\) 
	of~\(\FourPuncturedSphere\). 
	The algebra~\(\BinfU\) can then be viewed 
	as a geometric deformation of~\(\Binf\), 
	where each polygon 
	that is counted in the definition of the higher operations
	contributes a factor of~\(U^\lambda\)
	according to its multiplicity~\(\lambda\) at the special puncture. 
	In this interpretation, 
	\cref{thm:extension} asserts 
	that the curves~\(\BNr(T)\) and~\(\Khr(T)\), 
	corresponding  
	to the type~D structures~\(\DD^c(T)\) and~\(\DD^c_1(T)\), respectively,
	define objects in the wrapped Fukaya category of~\(S^2_3\).
	In particular, 
	fishtails (in the sense of~\cite[Lemma~2.2]{AbouzaidSurfaces}) 
	enclosing the special puncture must cancel in pairs.
\end{remark}

Before proceeding to the lengthy proof of the extension property (\cref{thm:extension})
we first explain how this result implies 
\cref{thm:no_wrapping_around_special,thm:geography:special_curves}.

\begin{proof}[%
	Proof of 
	\texorpdfstring{\cref{thm:no_wrapping_around_special}}
	{Theorem \ref{thm:no_wrapping_around_special}}
	(%
	No wrapping around 
	\texorpdfstring{\(\ast\)}{*}%
	)]

	Suppose the multicurve \(\BNr(T)\) has a component that wraps around the special puncture. For simplicity assume that the local system is trivial; the general case follows similarly. We will prove that an extension $\DD^c(T)^{\BinfU}$ of $\DD^c(T)^{\B}$ cannot exist, contradicting \cref{thm:extension}. 
	
	By \cref{cor:twisting:field-coeff:ungraded}, we may assume that one of the two segments where the curve changes direction is horizontal. If the direction changes by less than \(90^\circ\) we can add some twists to the bottom two tangle ends such that the angle becomes greater than \(90^\circ\), as illustrated on the left of \cref{fig:GeographyForSpecials}. In the corresponding type~D structure \(\DD^c(T)\), we then see a sequence of differentials
	\begin{equation}\label{eq:sequence}
		\begin{tikzcd}[column sep=15pt]
			\bullet
			\arrow{r}{D}
			&
			\bullet
			\arrow{r}{S}
			&
			\circ
			\arrow{r}{D}
			&
			\circ
			\arrow{r}{S}
			&
			\bullet
			\arrow{r}{D}
			&
			\bullet
		\end{tikzcd}
	\end{equation}
	Since \(\mu_4(S,D,S,D)=U\), the first four arrows contribute \(U\) to the structure relation of any extension \(\DD^c(T)^{\BinfU}\) of \(\DD^c(T)^\B\). %
	There are only three other sequences of differentials that may contribute \(U\), namely 
	\[
	\begin{tikzcd}[column sep=15pt]
		\bullet
		\arrow{r}{S}
		&
		\circ
		\arrow{r}{D}
		&
		\circ
		\arrow{r}{S}
		&
		\bullet
		\arrow{r}{D}
		&
		\bullet
	\end{tikzcd}
	,\quad
	\begin{tikzcd}[column sep=15pt]
		\bullet
		\arrow{r}{1}
		&
		\bullet
		\arrow{r}{U}
		&
		\bullet
	\end{tikzcd},
	\quad\text{and}\quad
	\begin{tikzcd}[column sep=15pt]
		\bullet
		\arrow{r}{U}
		&
		\bullet
		\arrow{r}{1}
		&
		\bullet
	\end{tikzcd}
	\]
	The first sequence cannot contribute the same \(U\) term, since the last arrow in the sequence~\eqref{eq:sequence} points out of the penultimate generator. 
	The other two sequences do not appear, since the differential does not contain any identity arrows. 
Since the $U$ term in the structure relation cannot be cancelled, the extension \(\DD^c(T)^{\BinfU}\) cannot exist. The same argument works for components of \(\Khr(T)\), since \(\DD_1^c(T)^{\B}\) also extends to some type~D structure over \(\BinfU\) by \cref{thm:extension}.
\end{proof}

\begin{figure}[t]
	\centering
	\(
	\vc{$\LittleSpecialWrapping$}
	\vc{$\GeographyForSpecials$}
	\)
	\caption{%
		(left) 
		The shearing transformation, viewed in \(\PuncturedPlane\), appearing in the proof of 
		\cref{thm:no_wrapping_around_special};
		(right)
		Some curve segments of a special component of slope 0
		appearing in the proof of
		\cref{thm:geography:special_curves}
	}\label{fig:GeographyForSpecials}
\end{figure}

\begin{proof}[%
	Proof of 
	\texorpdfstring{\cref{thm:geography:special_curves}}
	{Theorem \ref{thm:geography:special_curves}}
	(Almost special components are special.)]
	
	Let \(\gamma\) be an almost special component of \(\Khr(T)\) or \(\BNr(T)\). 
	By \cref{cor:twisting:field-coeff:ungraded}, 
	we may assume without loss of generality 
	that the slope of \(\gamma\) is zero. 
	It suffices to show that
	\(C(\gamma)\cong q^{r}h^s C(\sKh_{2n}(0))\) 
	for some \(r,s\in\Z\) 
	and \(n\in\Z_{>0}\). 
	
	First, suppose that the local system on \(\gamma\)
	is 1-dimensional. 
	Some generators of \(C(\gamma)\) belong 
	to the idempotent \(\iota_{\circ}\); 
	otherwise, \(\gamma\) would be 
	an almost rational component of slope \(0\). 
	By \cref{lem:higher_powers_when_wrapping}, 
	the differential of \(C(\gamma)\) only contains 
	linear combinations of \(S\), \(S^2\), and \(D\). 
	Therefore, it contains some differential 
	\(\circ\xrightarrow{c\cdot D}\circ\) 
	for some non-zero \(c\in\CoeffField\). 
	Near this differential,
	\(C(\gamma)\) looks either like
	\begin{equation}\label{eq:geography:special_curves}
		\begin{tikzcd}
			\circ
			\arrow{r}{S}
			\arrow{d}[swap]{c\cdot D}
			&
			\bullet
			\arrow{r}{D}
			&
			\bullet
			\arrow{r}{}
			&
			\cdots
			\\
			\circ
			\arrow{r}{S}
			&
			\bullet
			\arrow{r}{D}
			&
			\bullet
			\arrow{r}{}
			&
			\cdots
		\end{tikzcd}
		\quad
		\text{or}
		\quad
		\begin{tikzcd}
			\circ
			\arrow[leftarrow]{r}{S}
			\arrow[leftarrow]{d}[swap]{c\cdot D}
			&
			\bullet
			\arrow[leftarrow]{r}{D}
			&
			\bullet
			\arrow[leftarrow]{r}{}
			&
			\cdots
			\\
			\circ
			\arrow[leftarrow]{r}{S}
			&
			\bullet
			\arrow[leftarrow]{r}{D}
			&
			\bullet
			\arrow[leftarrow]{r}{}
			&
			\cdots
		\end{tikzcd}
	\end{equation}
	since \(\gamma\) does not wrap around any tangle end. 
	Consider the first case of \eqref{eq:geography:special_curves}. 
	A lift of the corresponding portion of the curve \(\gamma\) 
	to the covering space \(\PuncturedPlane\) 
	is shown on the right of \cref{fig:GeographyForSpecials}. 
	Without loss of generality, 
	we may assume that the labels 
	of all other components of the differential of \(C(\gamma)\) 
	are equal to \(S\), \(S^2\), or \(D\). 
	Let \(R\) and \(L\) be 
	the number of consecutive generators 
	in idempotent \(\iota_{\bullet}\) 
	on the upper and lower legs of this type~D structure, 
	respectively. 
	These integers are equal to 
	the number of intersection points 
	of the vertical parametrizing arcs 
	with the right and left segments of \(\gamma\),
	up to the point where they intersect 
	the horizontal parametrizing arcs again. 
	In particular, this implies that 
	\(R\) and \(L\) are even integers. 
	Suppose \(L\leq R\). 
	If \(L=2\), there is a non-trivial \(\mu_4\)-action 
	contributing a term \(c\cdot U\) 
	from the generator on the top left 
	to the generator on the bottom right 
	in the following diagram:
	\[
	\begin{tikzcd}
		\circ
		\arrow{r}{S}
		\arrow{d}[swap]{c\cdot D}
		&
		\bullet
		\arrow{r}{D}
		&
		\bullet
		\arrow{r}{S}
		&
		\circ
		\arrow[dashed]{d}{D}
		\\
		\circ
		\arrow{r}{S}
		&
		\bullet
		\arrow{r}{D}
		&
		\bullet
		\arrow{r}{S}
		&
		\circ
	\end{tikzcd}
	\]
	This action can only be cancelled 
	if \(c=-1\) and the dashed arrow is part of \(C(\gamma)\), 
	i.e.\ if \(C(\gamma)\) and \(C(\sKh_{2}(0))\)
	agree up to grading shift. 
	If \(L>2\), 
	the type~D structure \(C(\gamma)\) 
	looks as follows 
	(without the dotted and dashed arrows): 
	\[
	\begin{tikzcd}
		\circ
		\arrow{r}{S}
		\arrow{d}[swap]{c\cdot D}
		&
		\bullet
		\arrow{r}{D}
		\arrow[dotted,in=135,out=-45]{drr}[description]{c\cdot U}
		&
		\bullet
		\arrow{r}{S^2}
		\arrow[dotted,in=135,out=-45]{drr}[description]{-c\cdot U}
		&
		\cdots
		&
		\cdots
		\arrow{r}{S^2}
		\arrow[dotted,in=135,out=-45]{drr}[description]{-c\cdot U}
		&
		\bullet
		\arrow{r}{D}
		&
		\bullet
		\arrow[dashed]{r}{S}
		&
		\circ
		\arrow[dashed]{d}{D}
		\\
		\circ
		\arrow{r}{S}
		&
		\bullet
		\arrow{r}{D}
		&
		\bullet
		\arrow{r}{S^2}
		&
		\cdots
		&
		\cdots
		\arrow{r}{S^2}
		&
		\bullet
		\arrow{r}{D}
		&
		\bullet
		\arrow{r}{S}
		&
		\circ
	\end{tikzcd}
	\]
	In this case, the \(\mu_4\)-action 
	forces the existence of a \(c\cdot U\)-differential 
	in the extended twisted complex from \cref{thm:extension}, 
	namely the one indicated in the complex above 
	by the first dotted arrow on the left. 
	This component of the differential 
	forces the existence of a \(-c\cdot U\)-differential
	indicated by the second dotted arrow, and so on.
	The signs on these differentials alternate,
	and there is an even number of them.
	The contribution to the structure relation 
	of the composition of the last \(-c\cdot U\)-differential 
	with the differential \(\bullet\xrightarrow{S}\circ\) 
	on the lower leg of the complex can be cancelled 
	if and only if \(R=L\), \(c=-1\),
	and the dashed arrows are part of the differential, 
	i.e.\ if \(C(\gamma)=C(\sKh_L(0))\). 
	
	If the curve carries an indecomposable \(n\)-dimensional local system for \(n>1\), 
	we may choose the corresponding complex to be the same as before, 
	except that we tensor each generator 
	by an \(n\)-dimensional vector space \(W\), 
	replace the differential \(\circ\xrightarrow{c\cdot D}\circ\) on the left 
	by \(\circ\otimes W\xrightarrow{D\otimes X}\circ\otimes W\), 
	where \(X\in\GL_n(\field)\), and 
	tensor all other differentials of \(C(\gamma)\) by \(1_{W}\). 
	Then the \(U\)-differentials in the extended type~D structure 
	need to be tensored by \(X\). 
	We then deduce that \(X=-1_W\), 
	contradicting the assumption 
	that the local system was indecomposable. 
	
	The second case in \eqref{eq:geography:special_curves} 
	with \(L\leq R\) follows from reversing all arrows 
	in the arguments above. 
	So it remains to consider those two cases for \(L>R\). 
	We claim that a curve containing a portion of this kind 
	cannot be a component of \(\Khr(T)\). 
	To see this, consider a shortest leg of length \(R\) of such a curve. 
	By the previous arguments, 
	the corresponding portion of the type~D structure 
	looks as follows: 
	\[
	\begin{tikzcd}
		\cdots
		&
		\circ
		\arrow{l}[swap]{S}
		&
		\circ
		\arrow[dashed]{l}[swap]{D}
		\arrow{r}{S}
		&
		\bullet
		\arrow{r}{D}
		&
		\bullet
		\arrow{r}{}
		&
		\cdots
		\arrow{r}{}
		&
		\bullet
		\arrow{r}{D}
		&
		\bullet
		\arrow{r}{S}
		&
		\circ
		&
		\circ
		\arrow[dashed]{l}[swap]{D}
		&
		\cdots
		\arrow{l}[swap]{S}
	\end{tikzcd}
	\]
	(note the direction of the two dashed arrows).
	A simple application of the Clean-Up Lemma 
	\cite[Lemma~2.17]{KWZ} 
	shows that the mapping cone of such a complex 
	is chain isomorphic to the direct sum of 
	\(q^rh^sC(\sKh_R(0))\) for some \(r,s\in\Z\)
	and a fully reduced type~D structure \(X\)
	with strictly more generators than \(q^rh^sC(\sKh_R(0))\).
	Thus \(X\) is not chain homotopy equivalent to 
	\(q^rh^sC(\sKh_R(0))\) for any \(r,s\in\Z\). 
	However, this is not possible, 
	since by \cref{lem:H_is_nullhomotopic_on_Khr}, 
	the mapping cone of \(C(\gamma)\) is isomorphic to 
	\(q^{-1}h^{-1}C(\gamma)\oplus q^{+1}h^{0}C(\gamma)\). 
	Any local system may be pushed outside 
	of the relevant region of the type~D structure 
	in which these isotopies are non-trivial, 
	so this argument generalizes to arbitrary local systems. 
	\end{proof}

\subsection{From bifiltered matrix factorizations to bigraded twisted complexes}
\label{sec:no_wrapping_around_special:mf2twcx}
We now construct a bigraded twisted complex \(\DD(\Diag_T)^{\BinfU}\) from the bifiltered matrix factorization \(M(\Diag_T)\); \Cref{fig:overview_mf2twcx} gives an overview. Interpolating between bifiltered matrix factorizations and bigraded twisted complexes uses the notion of a bifiltered twisted complex.
\begin{definition}
	A bifiltered twisted complex over a bigraded \(A_\infty\)-algebra \(\A\) is 
	a \(\hslash\)-graded twisted complex \((C,d)\) over \(\A\)
	together with a bigrading such that 
	\(d=\sum_{i=1}^\infty d_i\), 
	where \(\gr(d_i)=\hslash^1q^{3(i-1)}h^i\).
\end{definition}

\cref{thm:delooping:mf} constructs a bifiltered matrix factorization \(M^s(\Diag_T)\) 
that is bifiltered homotopy equivalent 
to the original matrix factorization \(M(\Diag_T)\)
from \cref{sec:mfs}.

Next, we construct a bifiltered twisted complex \(C(\Diag_T)^{\A}\)
over the endomorphism algebra \(\A\) 
from \cref{def:mf_algebra}. 
This construction is rather tautological and goes as follows:
The inclusion 
\(
\A
\hookrightarrow
\MF_{xyz}(\Rd)
\) 
induces a functor 
\(
\Tw\A
\hookrightarrow
\Tw \MF_{xyz}(\Rd)
\) 
between the associated categories of twisted complexes
\cite[Section~I.3m]{Seidel}. 
Since
the category
\(
\Tw \MF_{xyz}(\Rd)
\) 
is equivalent to 
\(
\MF_{xyz}(\Rd)
\)
we obtain a full and faithful functor
\(
\mathcal{F}\co
\Tw \A
\hookrightarrow
\MF_{xyz}(\Rd)
\).
The following lemma is immediate from the definition of \(\mathcal{F}\).
\begin{lemma}\label{lem:criterion_is_in_image_of_H}
	A matrix factorization  
	\((Z,d_Z)\in \MF_w(\Rd)\)
	is in the essential image of \(\mathcal{F}\)
	if and only if there exists a decomposition
	\(Z\cong \bigoplus_i Z_i\)
	with the following property:
	For each \(i\),
	if \(d_i\co Z_i\rightarrow Z_i\) denotes the restriction of \(d_Z\),
	\((Z_i,d_i)\) is equal to \(M_\arcD\) or \(M_\arcT\) up to a grading shift.
	Moreover, 
	a chain complex over \(\A\) is bifiltered 
	if and only if 
	its image under \(\mathcal{F}\) is bifiltered. 
	\qed
\end{lemma}

\begin{figure}[t]
	\centering
	\tikzstyle{mypath}=[decorate,decoration={snake,post length=3pt,segment length=3pt,amplitude=1pt,post length=5pt,pre length=5pt}]
	\begin{tikzpicture}[xscale=3]
		\node (M) at (0,0) {\(M(\Diag_T)\)};
		\node (Ms) at (1,0) {\(M^s(\Diag_T)\)};
		\node (CA) at (2,0) {\(C(\Diag_T)^{\A}\)};
		\node (CBinf) at (3.1,0) {\(C(\Diag_T)^{\Binf}\)};
		\node (DBinfU) at (4,0) {\(\DD(\Diag_T)^{\BinfU}\)};
		
		\draw[thick,decorate,decoration=brace] (Ms.south east) -- node (y)[yshift=-8pt,below, draw,thick, rounded corners]{%
			\begin{minipage}{5.2cm}
				\centering
				bifiltered matrix factorizations\\
				\((C,d=d_0+d_1+\dots)\)\\
				\(\gr(d_i)=q^{3(i-1)}\hslash^1 h^i\)\\
				\(d^2=xyz\cdot 1_C\)
			\end{minipage}
		} (M.south west);
		\draw[thick,decorate,decoration=brace] (CBinf.south east) -- node (y)[yshift=-8pt,below, draw,thick, rounded corners]{%
			\begin{minipage}{4.7cm}
				\centering
				bifiltered twisted complexes\\
				\((C,d=d_1+d_2+\dots)\)\\
				\(\gr(d_i)=q^{3(i-1)}\hslash^1 h^i\)\\
				\(\phantom{1_C}d^2=0\phantom{1_C}\)
			\end{minipage}
		} (CA.south west);
		\draw[thick,decorate,decoration=brace] (DBinfU.south east) -- node (y)[yshift=-8pt,below, draw,thick, rounded corners]{%
			\begin{minipage}{2.9cm}
				\centering
				bigraded twisted\\
				complex \((C,d)\)\\
				\(\gr(d)=q^0h^1\)\\
				\(\phantom{1_C}d^2=0\phantom{1_C}\)
			\end{minipage}
		} (DBinfU.south west);
		\footnotesize
		\draw[mypath,->] (M) -- node (x)[above]{delooping} (Ms);
		\draw[mypath,->] (CA) -- node (x)[above]{\phantom{q}functor \(\mathcal{F}\)\phantom{q}} (Ms);
		\draw[mypath,->] (CA) -- node (x)[above]{\(\HMS\)} (CBinf);
		\draw[mypath,->] (CBinf) -- node (x)[above]{\(\operatorname{Add}_U\)} (DBinfU);
	\end{tikzpicture}
	\caption{%
		The construction 
		of the twisted complex \(\DD(\Diag_T)^{\BinfU}\) 
		from the matrix factorization \(M(\Diag_T)\)
	}\label{fig:overview_mf2twcx}
\end{figure}

\Cref{thm:delooping:mf} says that at each vertex of the cube of resolutions associated with the matrix factorization \(M^s(\Diag_T)\) is equal to a direct sum of copies of either \(M_\arcD\) or \(M_\arcT\), possibly shifted in bigrading.
So, by \cref{lem:criterion_is_in_image_of_H}, there exists a (up to isomorphism unique) bifiltered complex \(C(\Diag_T)^{\A}\) over \(\A\) with the property that 
\(\mathcal{F}(C(\Diag_T)^{\A})=M^s(\Diag_T)\). 

Next we use the quasi-isomorphism 
between the algebras \(\A\) and \(\Binf\) 
from \cref{thm:quasi-iso}, regarding both algebras as 
\(\hslash\)-graded and \(q\)-filtered \(A_\infty\)-algebras. 
(In particular, the \(h\)-grading is identically zero on both algebras.)
The \(q\)-filtered quasi-isomorphism between \(\A\) and \(\Binf\) 
induces the equivalence of categories
\(\HMS\co\Tw \A\rightarrow \Tw \Binf\) 
from \cref{thm:hms}, which maps bifiltered twisted complexes to bifiltered twisted complexes.
Define
\(C(\Diag_T)^{\Binf}\)
as the image of \(C(\Diag_T)^{\A}\)
under
\(\HMS\). 

The final step 
depends on the following property of \(C(\Diag_T)^{\Binf}\).

\begin{lemma}\label{lem:complex_over_Binf_has_no_even_diffs}
	Let \(d=\sum_i d_i\)
	be the decomposition of the differential of 
	\(C(\Diag_T)^{\Binf}\)
	with \(\gr(d_i)=q^{3(i-1)}\hslash^1 h^i\).
	Then \(d_i=0\) for even \(i\). 
\end{lemma}

\begin{proof}
	The quantum gradings of generators in \(C(\Diag_T)^{\Binf}\) have the same parity if and only if they belong to the same idempotent. 
	This follows from the construction as a cube of resolutions: 
	By \cref{thm:delooping:mf}, the modulo 2 quantum grading of a generator at a vertex \(v\in\{0,1\}^n\) is equal to \(q_v= |v|-n_-+m_v\), 
	where \(m_v\) is the number of circles in the diagram \(\Diag_T(v)\). 
	If two vertices \(v,w\in\{0,1\}^n\) are related by a single edge, 
	the difference \(m_v-m_w\) is odd
	if and only if 
	\(o(\Diag_T(v))=o(\Diag_T(w))\).
	Thus the difference 
	\(q_v-q_w\equiv 1+m_v-m_w \mod 2\)
	between the modulo 2 quantum gradings 
	of two generators at the vertices \(v\) and \(w\), respectively, 
	is even
	if and only if their idempotents agree.
	Since the cube of resolutions is connected, the claim follows.
	Moreover, the quantum grading of a pure algebra element in \(\Binf\) is even if and only if it starts and ends on the same idempotent. 
	Thus the quantum grading of any (non-zero) homogeneous endomorphism of \(C(\Diag_T)^{\Binf}\) must be even. 
	The quantum grading of \(d_i\) is \(3(i-1)\). So if \(d_i\neq0\), \(3(i-1)\) is even, and thus \(i\) must be odd.
\end{proof}

\begin{proposition}\label{prop:AddingU}
	Let \((C,d)\) be 
	a bifiltered complex over \(\Binf\) with \(d_i=0\) for even \(n\). 
	Then there exists a bigraded twisted complex
	\(\operatorname{Add}_U(C,d)=(C_U,d_U)\) 
	over the category \(\BinfU\) 
	such that
	\(C=C_U\) 
	and 
	\(d_1=d_U \mod U\). 
\end{proposition}

\begin{proof}
	Set
	\(d_U=\sum_{i=0}^\infty (-U)^i \cdot d_{2i+1}\) 
	so that \(\gr(d_U)=q^0\hslash^1 h^1\).
	It remains to check the structure relation.
	Note that since \(d_i=0\) for even indices \(i\) 
	the \(h\)-grading of \(d\) modulo 2 
	is equal to the \(\hslash\)-grading of \(d\).  
	Also, 
	the \(\hslash\)-grading of \(\Binf\) is identically zero 
	and 
	\(\BinfU\) is supported only in even degrees.
	So by \cref{rem:twisted_complex} the signs \(\triangleleft\) in the formula for \(\mu^{\Sigma\Binf}_k\) and \(\mu^{\Sigma\BinfU}_k\) are the same.
	We now check the formula in said remark. 
	Consider a sequence of arrows 
	\[
	\begin{tikzcd}[column sep=large]
		x_0
		\arrow{r}{a_1}
		&
		x_1
		\arrow{r}%
		&
		\cdots
		\arrow{r}%
		&
		x_{2n-1}
		\arrow{r}{a_{2n}}
		&
		x_{2n}
	\end{tikzcd}
	\]
	in \((C,d)\). 
	The corresponding sequence in \((C_U,d_U)\) is 
	\[
	\begin{tikzcd}[column sep=large]
		x_0
		\arrow{r}{(-U)^{h_1} a_1}
		&
		x_1
		\arrow{r}%
		&
		\cdots
		\arrow{r}%
		&
		x_{2n-1}
		\arrow{r}{(-U)^{h_{2n}} a_{2n}}
		&
		x_{2n}
	\end{tikzcd}
	\]
	where the integers \(h_i\) are determined by
	\(
	1+2h_i
	=
	h(
	\begin{tikzcd}
		x_{i-1}
		\arrow{r}{a_i}
		&
		x_i
	\end{tikzcd}
	)
	=
	h(x_i)-h(x_{i-1})\)
	for all \(i=1,\dots,2n\).
	Therefore
	\(\mu^{\BinfU}_{2n}\) 
	applied to the second sequence gives
	\((-U)^N\mu^{\Binf}_{2n}\)
	applied to the first sequence, where
	\[
	N
	=
	\sum_{i=1}^{2n}h_i+(n-1)
	=
	\tfrac{1}{2}(h(x_{2n})-h(x_0)-n)+(2n-1)
	=
	\tfrac{1}{2}(h(x_{2n})-h(x_0))-1
	\]
	In particular, the exponent \(N\) only depends on the start and endpoint of the sequence. 
	Thus the structure relation for \((C,d)\) implies the structure relation for \((C_U,d_U)\). 
\end{proof}

By \cref{lem:complex_over_Binf_has_no_even_diffs}, 
the complex \(C(\Diag_T)^{\Binf}\) satisfies the hypotheses of \cref{prop:AddingU}, and we  
define a twisted complex 
\(
\DD(\Diag_T)^{\BinfU}
=
\operatorname{Add}_U(C(\Diag_T)^{\Binf})
\).

\begin{remark}\label{prop:matrix_factorization_is_an_invariant}
	While we only need the existence of extensions, it seems reasonable to expect that the chain homotopy type of 
	\(
	\DD(\Diag_T)^{\BinfU}
	\)
	is a tangle invariant. The question of invariance is pursued in \cite{TomasInvariance}.
\end{remark}

\subsection{Delooping the cube of resolutions in the cobordism category}
\label{sec:no_wrapping_around_special:delooping_cobs}

	We continue to view 
	any tangle diagram of a 4-ended tangle 
	as the tangle diagram \(|\Diag|\) 
	associated with labelled diagram \(\Diag\), 
	as in \cref{sec:mfs}. 
	Thus, the objects of \(\Cobb\)
	are precisely those tangle diagrams \(|\Diag|\) 
	associated with labelled diagram \(\Diag\)
	whose labelling does not contain any element \(+\) or \(-\). 
	This perspective allows us to reuse some of the notation 
	from \cref{def:induced_saddle_maps:mf}.

\begin{definition}\label{def:induced_saddle_maps:B}
	Define the morphism 
	\[
	\mathcal{S}^{v}_{v'}\co
	V(\Diag_T(v))
	\otimes_{\field[\varG]}
	\omega(|o(\Diag_T(v))|)
	\rightarrow
	V(\Diag_T(v'))
	\otimes_{\field[\varG]}
	\omega(|o(\Diag_T(v'))|)
	\]
	in the additive enlargement \(\Sigma\B\) of \(\B\)
	as the image of the map 
	\(\mathcal{S}^{v}_{v'}\) 
	from \cref{def:induced_saddle_maps:mf}
	under the functor \(\HMS\), 
	where \(\omega\) is the isomorphism 
	between \(\B\) and \(\End_{\Cobb}(\Li\oplus\Lo)\)
	from \cref{thm:OmegaFullyFaithful}. 
	Overloading notation further, 
	define the morphism 
	\[
	\mathcal{S}^{v}_{v'}\co
	V(\Diag_T(v))
	\otimes_{\field[\varG]}
	|o(\Diag_T(v))|
	\rightarrow
	V(\Diag_T(v'))
	\otimes_{\field[\varG]}
	|o(\Diag_T(v'))|
	\]
	in the additive enlargement \(\Sigma\Cobb\) of \(\Cobb\)
	as the image of the above map 
	\(\mathcal{S}^{v}_{v'}\) 
	under the functor induced
	by the inverse of the isomorphism \(\omega\).
\end{definition}

In other words,
the morphism \(\mathcal{S}^{v}_{v'}\) 
in \(\Sigma\B\)
is obtained from the rules 
for the morphism \(\mathcal{S}^{v}_{v'}\) 
between matrix factorizations
in \cref{def:induced_saddle_maps:mf} 
by replacing \(\sB\), \(\sW\), \(\dB\), and \(\dW\) 
by \(\Sb\), \(\Sw\), \(\Db\), and \(\Dw\), 
respectively. 
For the morphism \(\mathcal{S}^{v}_{v'}\) 
in \(\Sigma\Cobb\), 
these algebra elements are simply converted 
into cobordisms using \(\omega\).

\begin{corollary}\label{cor:delooping:mf}
	The complex \(\DD(\Diag_T)^{\BinfU}\) 
	is \(n\)-cubical
	with cubical structure
	\[
	(\DD(\Diag_T))_v
	=
	h^{|v|-n_-}
	q^{|v|+n_+-2n_-}
	V(\Diag_T(v))
	\otimes_{\field[\varG]}
	\omega(|o(\Diag_T(v'))|)
	\]
		for every vertex \(v\in\{0,1\}^n\).
	Moreover, 
	if \(v\) and \(v'\) are related by an edge \(v\rightarrow v'\), 
	then the differential is equal to \(\pm\mathcal{S}^{v}_{v'}\).
\end{corollary}
\begin{proof}
	This is simply a reformulation of \cref{thm:delooping:mf} 
	in terms of \(\DD(\Diag_T)^{\BinfU}\) 
	instead of \(M^s(\Diag_T)\). 
	Note that the second part of the statement uses the fact that
	the differential of
	\(
	\DD(\Diag_T)^{\BinfU}
	\)
	modulo \(U\) coincides with the \(d_1\)-differential of \(C(\Diag_T)^{\Binf}\) (\cref{prop:AddingU}). 
\end{proof}

The following is the analogue of \cref{lem:delooping:mf} in the setting of Bar-Natan's cobordism category.%

\begin{lemma}\label{lem:delooping:cob}
	For every \(v\in\{0,1\}^n\) there exists an isomorphism
	\[
	\psi_v
	\co 
	|\Diag_T(v)|
	\rightarrow
	V(\Diag_T(v))
	\otimes_{\field[\varG]}
	|o(\Diag_T(v))|
	\]
	of objects in the additive enlargement of \(\Cobb\) 
	such that, for every edge \(v\rightarrow v'\) in \(\{0,1\}^n\), 
	the diagram
	\[
	\begin{tikzcd}
		{|\Diag_T(v)|}
		\arrow{d}{\psi_v}
		\arrow{r}{d^{v}_{v'}}
		&
		{|\Diag_T(v')|}
		\arrow{d}{\psi_{v'}}
		\\
		V(\Diag_T(v))
		\otimes_{\field[\varG]}
		|o(\Diag_T(v))|
		\arrow{r}{\mathcal{S}^{v}_{v'}}
		&
		V(\Diag_T(v'))
		\otimes_{\field[\varG]}
		|o(\Diag_T(v'))|
	\end{tikzcd}
	\]
	commutes.
\end{lemma}

\begin{proof}
	By \cref{lem:delooping},
	we have inverse isomorphisms
	locally defined by 
	\[
	\begin{tikzcd}[row sep=-0.2cm, column sep=2cm]%
		&
		q^{1}h^0 (\mu_i\otimes\varnothing)
		\arrow{dr} [near start] {\DiscRdot}%
		\\
		\Circle_i
		\arrow{ur} [near end] {\DiscL}%
		\arrow{dr} [swap, near end] {\DiscLdot+\varG\cdot\DiscL}%
		&&
		\Circle_i
		\\
		&
		q^{1}h^0(1_i\otimes\varnothing)
		\arrow{ur} [swap, near start] {\DiscR\phantom{+\varG\cdot\DiscL}}%
	\end{tikzcd}
	\]
	for each closed component \(c_i\) of \(|\Diag_T(v)|\).
	The isomorphism \(\psi_v\) is defined as 
	the tensor product of the left morphism 
	over all \(c_i\).
	
	To prove the second part of the statement recall the following terminology
	from \cite{KWZ}: A cobordism representing an element 
	in a morphism space in \(\Cobb\) is simple
	if all its components are disks, 
	the distinguished component (marked by \(\ast\)) 
	carries no dot, and 
	any other component carries at most one dot. 
	By \cite[Proposition~4.15]{KWZ}
	simple cobordisms form a 
	\(\field[\varG]\)-linear basis 
	of any morphism space in \(\Cobb\).
	Therefore, 
	the morphism space
	\(\Cobb(|\Diag_T(v)|,|\Diag_T(v')|)\) 
	can be written as the tensor product of 
	\(\Cobb(|o(\Diag_T(v))|,|o(\Diag_T(v'))|)\) 
	over \(\field[\varG]\) with
	\begin{equation}\label{eqn:cob:only_circles}
		\Big(
		\bigotimes_{i=1}^{n}
		\Cobb\big(\Circle_i,\varnothing\big)
		\Big)
		\otimes_{\field[\varG]}
		\Big(
		\bigotimes_{j=1}^{m}
		\Cobb\big(\varnothing,\Circle_j\big)
		\Big)
	\end{equation}
	On each of these tensor factors, 
	we modify the basis given by simple cobordisms as follows:
	\begin{align*}
		\Cobb\big(\Circle_i,\varnothing\big)
		&= 
		\field[\varG]
		\Big\langle
		\prescript{}{i}{\DiscLdot}+\varG\cdot\prescript{}{i}{\DiscL}, \prescript{}{i}{\DiscL}
		\Big\rangle 
		\quad\text{for }i=1,\dots,n
		\\
		\Cobb\big(\varnothing,\Circle_j\big)
		&= 
		\field[\varG]
		\Big\langle
		\DiscR_j, \DiscRdot_j
		\Big\rangle 
		\quad\text{for }j=1,\dots,m
	\end{align*}
	This new basis has the advantage 
	that it is compatible with the isomorphism \(\psi_v\) 
	in the sense that the basis on the space of \(\field[\varG]\)-linear maps
	\[
	V(\Diag_T(v))
	\rightarrow
	V(\Diag_T(v'))
	\]
	given by tensor products of 
	\(\mu_i^*\), \(1_i^*\), \(\mu_j\), and \(1_j\)
	corresponds to the new basis on \eqref{eqn:cob:only_circles}
	via the following dictionary:
	\[
	\prescript{}{i}{\DiscL} \leftrightarrow \mu_i^*
	\qquad
	\prescript{}{i}{\DiscLdot} +\varG\cdot\prescript{}{i}{\DiscL} \leftrightarrow 1_i^*
	\qquad
	\DiscR_j \leftrightarrow 1_j
	\qquad
	\DiscRdot_j \leftrightarrow \mu_j
	\]
	It remains to verify that 
	when
	\(\Diag_T(v)\) and \(\Diag_T(v')\) 
	are related by a single saddle cobordism \(C\), 
	the map \(\psi_{v'}^{-1}\circ C\circ \psi_{v}\) 
	agrees with \(\mathcal{S}^{v}_{v'}\) in \(\Sigma\Cobb\). 
	Given the dictionary above 
	this computation is straightforward and 
	similar to the one carried out 
	in \cite[Proposition~4.31]{KWZ}.
\end{proof}

\begin{theorem}\label{thm:delooping:cob}
	The chain complex \(\KhT{\Diag_T}\) is chain homotopy equivalent to a strictly \(n\)-cubical chain complex \(\KhT{\Diag_T}^s\) with cubical structure
	\[
	(\KhT{\Diag_T}^s)_v
	=
	\KhT{\Diag_T}^s_v
	=
	h^{|v|-n_-}q^{|v|+n_+-2n_-}
	V(\Diag_T(v))
	\otimes
	|o(\Diag_T(v))|
	\]
for each \(v\in\{0,1\}^n\).	The differential along each edge \(v\rightarrow v'\) is equal to \(\pm\mathcal{S}^{v}_{v'}\).
\end{theorem}

\begin{proof}
	Apply \cref{lem:delooping:cob} to every vertex of the cube for \(\KhT{\Diag_T}\). 
\end{proof}

\begin{corollary}\label{cor:identification_of_type_D_structures}
	Let \(\BinfU\rightarrow \B\) be the quotient homomorphism setting \(U=0\). 
	Then \(\DD(\Diag_T)^{\BinfU}|_{U=0}\) is isomorphic to \(\KhT{\Diag_T}^s\).
\end{corollary}

\begin{proof}
	As a twisted complex over \(\BinfU\), 
	the differential of the complex
	\(\DD(\Diag_T)^{\BinfU}\)
	increases homological grading by \(1\). 
	Thus, for any two vertices \(v,v'\in\{0,1\}^n\), 
	the component \(d^v_{v'}\) 
	of the differential of \(\DD(\Diag_T)^{\BinfU}\)
	from \(\DD(\Diag_T)^{\BinfU}_v\)
	to \(\DD(\Diag_T)^{\BinfU}_{v'}\)
	has homological degree \(|v|-|v'|+1\). 
	If \(|v|>|v'|-1\), this degree is positive, 
	but \(\BinfU\) is supported in non-positive degrees only, 
	hence, \(d^v_{v'}=0\). 
	If \(|v|<|v'|-1\), the degree is negative, 
	so \(d^v_{v'}\) is a multiple of \(U\). 
	If \(|v|=|v'|-1\), the degree is zero, 
	so \(d^v_{v'}\) has entries in \(\B\). 
	We conclude that the only non-zero components 
	\(d^v_{v'}|_{U=0}\) 
	of the differential 
	of \(M(\Diag_T)^{\BinfU}|_{U=0}\)
	are those for which \(v\rightarrow v'\) is an edge, 
	and \(d^v_{v'}|_{U=0}=d^v_{v'}\). 
	So by \cref{cor:delooping:mf,thm:delooping:cob}, 
	the two complexes agree up to signs along the edges. 
	
	Observe that the composition of two consecutive edge maps in \(\KhT{\Diag_T}\) is a composition of two saddles, which is a non-zero cobordism. 
	Since \(\KhT{\Diag_T}^s\) is obtained from \(\KhT{\Diag_T}\) by composition with isomorphisms at the vertices, the composition of two consecutive edge maps in \(\KhT{\Diag_T}^s\) is also non-zero. 
	We now conclude with \cref{lem:identify_cubes}.
\end{proof}

\begin{lemma}\label{lem:identify_cubes}
	Any two strictly \(n\)-cubical complexes 
	whose cubical structures agree up to signs along the edge maps 
	are isomorphic 
	if the composition of any two consecutive edge maps in one complex is non-zero. 
\end{lemma}

\begin{proof}
	The lemma follows from the same argument as 
	\cite[Lemma 2.2]{OSR}. 
	More explicitly, 
	suppose \(d=(d^v_{v'})_{v,v'}\) 
	is the differential for the first complex and 
	the differential of the second complex 
	is given by 
	\((\varepsilon(v\rightarrow v')\cdot d^v_{v'})_{v,v'}\), 
	where we regard 
	\(\varepsilon\) as a 1-cochain 
	\(C_1([0,1]^n)\rightarrow \{\pm 1\}\) 
	of the space \([0,1]^n\) 
	equipped with a cell complex structure
	whose 0-cells are the vertices and 
	whose 1-cells are the edges of the cube.
	Consider a face of the cube \(\{0,1\}^n\), 
	that is four vertices 
	\(v_{00},v_{10},v_{01},v_{11}\in\{0,1\}^n\)
	related by edges as follows
	\[
	\begin{tikzcd}[column sep=20pt,row sep=1pt]
		&
		v_{10}
		\arrow{dr}
		\\
		v_{00}
		\arrow{dr}
		\arrow{ur}
		&&
		v_{11}
		\\
		&
		v_{01}
		\arrow{ur}
	\end{tikzcd}
	\]
	The \(d^2\) relations show that the compositions of the two pairs of consecutive differentials are identical in both complexes. 
	By hypothesis they are non-zero. 
	This implies that the product of the signs 
	\(\varepsilon(v\rightarrow v')\)
	over the four edges is 1. 
	Thus, the map 
	\(\varepsilon\co C_1([0,1]^n)\rightarrow \{\pm 1\}\) 
	defined by mapping each edge \(v\rightarrow v'\) 
	to \(\varepsilon(v\rightarrow v')\)
	is a 1-cocycle. 
	Since \(H^1([0,1]^n)=0\), 
	there exists a 0-cochain \(\eta\co C_0([0,1]^n)\rightarrow \{\pm 1\}\)
	such that \(\varepsilon = d^*(\eta)\), i.e.\
	\(\varepsilon(v\rightarrow v')=\eta(v)\eta(v')\). 
	Then the diagonal matrix with entries \(\eta(v)\) 
	for each vertex \(v\) is a chain isomorphism between the two complexes. 
\end{proof}

\subsection{Proof of the extension property}
\label{sec:no_wrapping_around_special:proof_extension_prop}

\begin{proof}[Proof of \cref{thm:extension}]
	The tangle operator \(\rho\in \Mod(\FourPuncturedSphere,*)\) from \cref{prop:rho*Khr} 
	does not affect the extendibility of the invariants, 
	since the corresponding automorphism \(\rho\) of the algebra \(\B\) 
	extends to an automorphism of the \(A_\infty\)-algebra \(\BinfU\). 
	Therefore, we may assume without loss of generality 
	that the connectivity of \(T\) is either \(\ConnX\) or \(\ConnZ\). 
	Moreover, changing the orientation of a component of a tangle \(T\) 
	only affects the invariants \(\DD(T)\) and \(\DD_1(T)\) 
	by some overall shift in bigrading \cite[Proposition~4.8]{KWZ}.
	Thus we may assume without loss of generality that \(T\) satisfies condition \AssumptionT\ from \cref{assums:tangle_diagram}. 
	By \cref{lem:many-good-diagrams}, we can find a diagram \(\Diag_T\) for \(T\) that satisfies \AssumptionD. 
	\cref{cor:identification_of_type_D_structures}
	shows that
	\(
	\DD(\Diag_T)^{\BinfU}\big|_{U=0} 
	= 
	\DD(\Diag_T)^{\B}
	\). 
	The statement about \(\DD_1(\Diag_T)^\B\) is established by
	modifying the construction in \cref{sec:mfs} by
	replacing \(M(\Diag_T)\)
	with the matrix factorization 
	\[
	\Big[
	q^{-2}h^{-1}\hslash^1 M(\Diag_T)
	\xrightarrow{\varG} 
	M(\Diag_T)
	\Big]
	\]
	which is chain homotopy equivalent, as a matrix factorization over \(\Rd\),
	to an \((n+1)\)-cubical matrix factorization \(M^s_{\varG}(\Diag_T)\)
	with the following cubical structure:
	For every \(v\in\{0,1\}^{n+1}\)	\[
	(M^s_{\varG}(\Diag_T))_v
	=
	\begin{cases*}
		q^{-2}h^{-1}\hslash^1 M^s_{\hat{v}}(\Diag_T)
		&
		if \(v_{n+1}=0\)
		\\
		M^s_{\hat{v}}(\Diag_T)
		&
		if \(v_{n+1}=1\)
	\end{cases*}
	\]
	where \(\hat{v}\in\{0,1\}^{n}\) 
	is obtained from \(v\) 
	by dropping the last entry.
	Moreover, 
	if \(v\) and \(v'\) are related by an edge \(v\rightarrow v'\), 
	then the differential is chain homotopic to
	\[
	\begin{cases*}
		\pm\mathcal{S}^{\hat{v}}_{\hat{v}'}
		&
		if \(v_{n+1}=v'_{n+1}\)
		\\
		\pm\varG\cdot 1_{M^s_{\hat{v}}(\Diag_T)}
		&
		otherwise.
	\end{cases*}
	\]
	This follows from the same arguments as the proof of \cref{thm:delooping:mf}.
	Next, we carry out the construction in \cref{sec:no_wrapping_around_special:mf2twcx}
	to obtain a twisted complex \(\DD_1(\Diag_T)^{\BinfU}\)
	for which the analogue of \cref{cor:delooping:mf} then also holds. 
	Finally, \(\DD_1(\Diag_T)^{\BinfU}\big|_{U=0} 
	= 
	\DD_1(\Diag_T)^{\B}\)
	is established in the same way as \cref{cor:identification_of_type_D_structures},
	observing that the hypothesis of \cref{lem:identify_cubes} is still satisfied
	since \(\varG\) is not a zero-divisor.
	
	It remains to show that the type~D structure \(\DD^c(T)^{\B}\) representing the curve \(\BNr(T)\) is extendable. Such a type~D structure is obtained from \(\DD(\Diag_T)\) via the so-called arrow sliding algorithm. This algorithm, which is described in~\cite[Section~5]{KWZ}, consists of a sequence of basic homotopies, namely cancellations and clean-ups. 
	\cref{prop:cancel_n_clean} below implies that extendability is preserved under such homotopies.
\end{proof}

\begin{remark}
	It is not the case that, for any twisted complex \((C,d)\) over \(\BinfU\), the map \(\varG\cdot 1_{C}\) is a homomorphism. 
	For example,  \(\varG\cdot 1_{C}\) is not a homomorphism of twisted complexes over \(\BinfU\) when
	\begin{equation*}
		(C,d)
		=
		\Big[
		\begin{tikzcd}[column sep=15pt]
			\circ
			\arrow{r}{S}
			&
			\bullet
			\arrow{r}{D}
			&
			\bullet
			\arrow{r}{S^3}
			&
			\circ
		\end{tikzcd}
		\Big]
	\end{equation*} 
	It is, however, extendable:
	\begin{equation*}
		\begin{tikzcd}[column sep=20pt,row sep=20pt]
			\circ
			\arrow{d}{\varG}
			\arrow{r}{S}
			&
			\bullet
			\arrow{d}{\varG}
			\arrow{r}{D}
			\arrow[in=120,out=-30,dashed]{drr}{S\cdot U}
			&
			\bullet
			\arrow{d}{\varG}
			\arrow{r}{S^3}
			&
			\circ
			\arrow{d}{\varG}
			\\
			\circ
			\arrow{r}{S}
			&
			\bullet
			\arrow{r}{D}
			&
			\bullet
			\arrow{r}{S^3}
			&
			\circ
		\end{tikzcd}
	\end{equation*}
\end{remark}

The following results are straightforward generalizations 
of \cite[Lemmas~2.16 and~2.17]{KWZ} 
to the setting of twisted complexes. 

\begin{lemma}[Clean-Up Lemma for twisted complexes]
	\label{lem:clean-up:ainfty}
	Let \((X,\delta)\) be a twisted complex over \(\BinfU\) and 
	suppose \(a\) is an endomorphism of \(X\) 
	with \(\gr(a)=h^0q^0\) and \(a^2=0\). 
	Furthermore, we assume that
	either \(h(x)\) or \(h(y)\) 
	is the same for any arrow \(x\rightarrow y\) in \(a\). 
	Then \((X,\delta)\) is isomorphic to \((X,\delta')\), 
	where 
	\[
	\pushQED{\qed}
	\delta'
	=
	\delta
	+
	\sum_i\mu^{\Sigma\BinfU}_{i}(\delta,\dots,\delta,a,\delta,\dots,\delta).
	\qedhere
	\popQED
	\]
\end{lemma}

\begin{lemma}[Cancellation Lemma for twisted complexes]
	\label{lem:cancel:ainfty}
	Any twisted complex over \(\BinfU\) of the form
	\[
	\begin{tikzcd}[column sep=20pt,row sep=35pt]
		&
		(C,\delta)
		\arrow[bend left=10]{dl}{a}
		\arrow[bend left=10]{dr}{c}
		\\
		\halfbullet
		\arrow{rr}{\alpha\cdot \iota_{\halfbullet}}
		\arrow[bend left=10]{ur}{b}
		&
		&
		\halfbullet
		\arrow[bend left=10]{ul}{d}
	\end{tikzcd}
	\]
	with non-zero \(\alpha\in\field\)
	is homotopy equivalent to 
	the twisted complex \((C,\delta')\) where
	\[
	\delta'
	=
	\delta
	+
	\sum\mu^{\Sigma\BinfU}_i(\delta,\dots,\delta,b_\alpha,c,\delta,\dots,\delta)
	\]
	and 
	\(
	b_\alpha = \tfrac{1}{\alpha}\cdot b
	\)
	is the morphism with bigrading \(\gr(b_\alpha)=h^0q^0\). 
	\qed
\end{lemma}

	Using the sign formula in \cref{def:additive_enlargement_Seidel}, 
	one can see that
	\[
	\delta'
	=
	\delta
	-
	\sum\mu^{\Sigma\BinfU}_i(\delta,\dots,\delta,b,\prescript{}{\alpha}{c},\delta,\dots,\delta)
	\]
	where 
	\(\prescript{}{\alpha}{c}=\tfrac{1}{\alpha}\cdot c
	\)
	is the morphism with bigrading \(\gr(\prescript{}{\alpha}{c})=h^0q^0\). 

\begin{proposition}\label{prop:cancel_n_clean}
	Suppose a type~D structure \(X^\B\) over \(\B\) 
	is extendable to a twisted complex over \(\BinfU\), 
	and \(Y^\B\) is obtained from \(X^\B\) 
	by a single application of the Cancellation or Clean-Up Lemmas 
	\cite[Lemmas~2.16 and~2.17]{KWZ} 
	in the arrow sliding algorithm~\cite[Section~5]{KWZ}. 
	Then \(Y^\B\) is also extendable 
	to a twisted complex over \(\BinfU\).
\end{proposition}

\begin{proof}
	Let \((X,\delta)\) be a twisted complex over \(\BinfU\)
	with the property that
	\(X^\B=(X,\delta|_{U=0})\).
	Consider applications of the two lemmas in the arrow sliding algorithm separately.

	The Clean-Up Lemma is applied to some \(a=(\mathbf{x}\xrightarrow{a_{\B}}\mathbf{y}) \in \End(X^{\B})\), 
	where \(\mathbf{x}\) and \(\mathbf{y}\) 
	are distinct homogeneous generators of \(X^\B\) 
	and \(\text{gr}(a)=q^0h^0\).
	Let \((X,\delta')\) be the twisted complex
	obtained by applying the generalized version of the Clean-Up Lemma,  \cref{lem:clean-up:ainfty},
	to the twisted complex \((X,\delta)\).
	Then 
	\begin{align*}	
		\delta'|_{U=0}
		&
		=
		\delta|_{U=0}
		+
		\sum_i\mu_{i}(\delta,\dots,\delta,a,\delta,\dots,\delta)|_{U=0}
		=
		\delta|_{U=0}
		+
		\mu_2(\delta,a)|_{U=0}
		+
		\mu_2(a,\delta)|_{U=0}
		\\
		&
		=
		\delta|_{U=0}
		+
		\mu_2(\delta|_{U=0},a)
		+
		\mu_2(a,\delta|_{U=0}).
	\end{align*}
	We observe that 
	\(Y^{\B}=(X,\delta'|_{U=0})\).
	
	The Cancellation Lemma is applied to some component 
	\((\mathbf{x}\xrightarrow{\alpha}\mathbf{y})\) 
	of the differential of \(X^{\B}\), 
	where \(\alpha\) is some invertible algebra element. 
	The component of the extended differential \(\delta\)
	from \(\mathbf{x}\) to \(\mathbf{y}\)
	is also equal to \(\alpha\), 
	since the differentials are homogeneous and 
	the homological grading of \(U^n\) is zero 
	if and only if \(n=0\). 
	So we can apply \cref{lem:cancel:ainfty}
	and obtain a twisted complex \((X,\delta')\)
	that is homotopy equivalent to \((X,\delta)\). 
	Again, one checks that \(\delta'|_{U=0}\) 
	agrees with the differential of \(Y^{\B}\).	
\end{proof}
\section{%
	Proof of 
	\texorpdfstring{\cref{lem:delooping:mf}}{Lemma \ref*{lem:delooping:mf}}
}\label{sec:delooping:mf}

\newcommand{\myitheading}[1]{\medskip\noindent\textit{#1:}}%

	The proof of \cref{lem:delooping:mf} proceeds in three steps: 
	First, a candidate \(\varphi_{A}\) 
	for the special deformation retract \(\varphi_v\) is constructed, 
	which depends on some additional data \(A\) on \(\Diag_T(v)\) 
	 called an arc system. 
	While the maps \(\varphi_{A_1}\) and \(\varphi_{A_2}\) 
	may be distinct for different arc systems 
	\(A_1\) and \(A_2\) on \(\Diag_T(v)\), the second step shows
	their homotopy classes agree (up to signs). 
	In the last step, we show the second part of the lemma
	via choices of arc systems for the diagrams 
	\(\Diag_T(v)\) and \(\Diag_T(v')\) 
	that are adapted to the edge map \(d^{v}_{v'}\). 
	To prepare for this, 
	the first subsection introduces notation and terminology 
	and collects several preliminary results needed for the proof.
	Throughout this entire section let \(D\) be the singular diagram for \(\Diag_T\). 
	We write \(R=R(D)\). 
	
	\subsection{Notation and preliminary observations}
	Throughout this subsection 
	 fix a labelled diagram \(\Diag=(D,\lab)\) 
	with labels taking values in \(\{\arcD,\arcT\}\). 
	We call nodes \(a\) of \(D\) 
	together with their labels \(\lab(a)\) 
	the arcs of \(\Diag\) and write \(\arcs(\Diag)\) for the set of arcs of \(\Diag\). 
	Given any arc \(a\in\arcs(\Diag)\) write \(x^a=x^{a}_{\lab(a)}\) and \(y^a=y^{a}_{\lab(a)}\).
	Finally, we enumerate the closed components of \(\Diag\) and 
	denote them by \(c_1,\dots,c_m\).
	
	\subsubsection{Edge ring generators}
	
	\begin{definition}
		The component graph \(G_{\Diag}\) 
		of the labelled diagram \(\Diag\) 
		is the graph with vertices in one-to-one correspondence 
		with the components of \(\Diag\) and 
		with edges in one-to-one correspondence 
		with the arcs in \(\Diag\); 
		each edge corresponding to an arc \(a\) 
		connects the two vertices corresponding to the components 
		that the endpoints of the arcs lie on. 
		Given a subset \(A\subseteq \arcs(\Diag)\), 
		we denote by \(G_A\) 
		the subgraph of \(G_\Diag\) 
		consisting of all edges in \(A\) and 
		vertices incident to those edges. 
	\end{definition} 

	\begin{definition}
		An arc \(a^*\in \arcs(\Diag)\) 
		is a bonding arc
		if it connects the two open components of \(\Diag\). 
		An arc system of \(\Diag\) 
		is a subset \(A\subseteq \arcs(\Diag)\) 
		containing a bonding arc and 
		for which \(G_A\) is a maximal tree in \(G_{\Diag}\).
	\end{definition}

Assumption~\ref{assums:tangle_diagram}~\AssumptionD\
guarantees that \(\Diag\) has a bonding arc.

\begin{remark}\label{rem:size-of-arc-system-fixed}
		The number %
		of arcs in any arc system \(A\) for \(\Diag\)
		is equal to the number of closed components plus 1. 
		Moreover, any arc system contains exactly one bonding arc, 
		since any two distinct bonding arcs in \(\Diag\) 
		define a cycle in \(G_\Diag\).
	\end{remark}

	\begin{example}
		The labelled diagram shown in \cref{fig:arc-system} 
		contains just one bonding arc labelled \(a^*\). 
		The arc \(a_3\) does not belong to any arc system 
		as it connects an open component with itself, 
		and so \(\{a^*,a_1\}\) and \(\{a^*,a_2\}\) 
		are the only arc systems for this diagram.
	\end{example}
	
	\begin{figure}[t]
	\labellist 
	\tiny
	\pinlabel $a_3$ at 141 51
	\pinlabel $a_2$ at 110 105
	\pinlabel $a_1$ at 74 137
	\pinlabel $a^{\! *}$ at 162 138
	\pinlabel $c_1$ at 48 87
	\endlabellist
		\centering
		\begin{subfigure}{0.333\textwidth}
			\centering
			\includegraphics[scale=0.5]{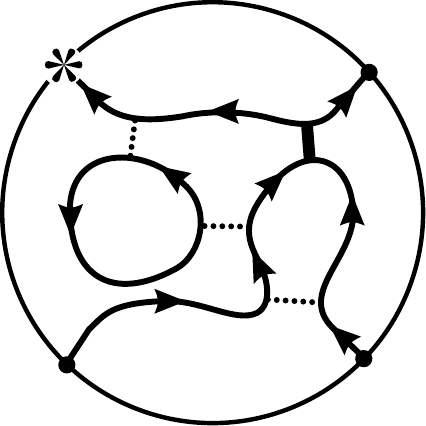}
			\caption{}
			\label{fig:arc-system}
		\end{subfigure}%
		\begin{subfigure}{0.333\textwidth}
\labellist 
	\tiny
	\pinlabel $a_3$ at 128 20
	\pinlabel $a_2$ at 25 29
	\pinlabel $a_1$ at 26 123
	\pinlabel $a^{\! *}$ at 114 77
	\pinlabel $o^{\! *}$ at 81 116
	\pinlabel $o$ at 82 40
	\pinlabel $c_1$ at 10 77
	\endlabellist		
			\centering
				\includegraphics[scale=0.5]{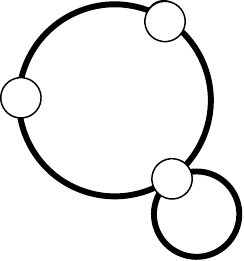}
			\caption{}
			\label{fig:component-graph}
		\end{subfigure}%
		\begin{subfigure}{0.333\textwidth}
\labellist 
	\tiny
	\pinlabel $c_0$ at 55 167
	\pinlabel $c_y$ at 157 165
	\pinlabel $c_{x+y}$ at 174.5 61
	\pinlabel $c_x$ at 51 40
	\pinlabel $a^{\! *}$ at 162 138
	\pinlabel $c_1$ at 48 87
	\endlabellist		
			\centering
			\includegraphics[scale=0.5]{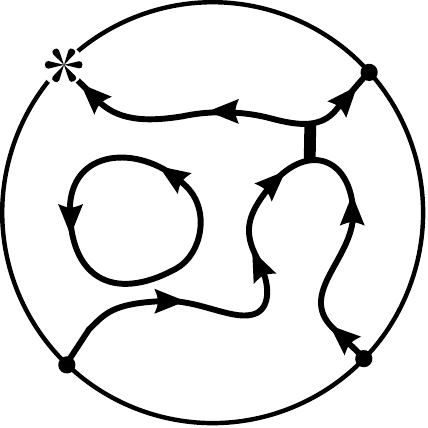}
			\caption{}
			\label{fig:diagram-segments}
		\end{subfigure}
		\caption{%
			(a) A labelled diagram \(\Diag\), 
			(b) its component graph \(G_\Diag\), and 
			(c) its closed components and diagram segments 
			    subordinate to the bonding arc \(a^*\)
		}
	\end{figure}
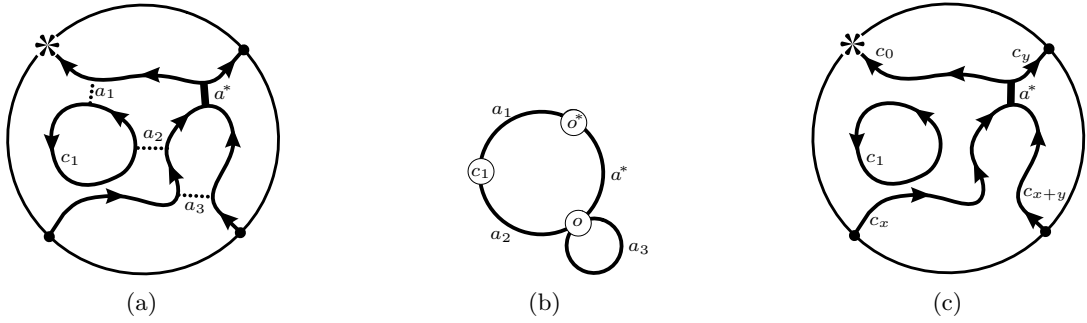

	\begin{proposition}\label{thm:edge-ring-generator-criterion-alt-all}
		Let \(A\) be an arc system for \(\Diag\). 
		Choose some edge \(e_i\) 
		on the closed component \(c_i\) and 
		write \(Y_{i}= X^{e_i}\) for \(i=1,\dots,m\).
		Then \(R\) is freely generated 
		as a polynomial ring over \(\Rd\) 
		by \(Y_i\) for \(i=1,\dots,m\) and 
		\(x^a\) for \(a\in \arcs(\Diag)\smallsetminus A\). 
		\qed
	\end{proposition}

	The strategy for the proof of 
	\cref{thm:edge-ring-generator-criterion-alt-all} 
	is to perform a sequence of base changes in
	\(\CoeffRing[\varG,X_0,X_1,\dots,X_N]\), 
	until one arrives at a basis consisting of 
	the desired basis for \(R\) and 
	a set of generators for the ideal \(I(D)\). 
	We leave the details to the reader.

\subsubsection{Relations for \texorpdfstring{\(x^a\)}{x\^{}a}} 

	\begin{definition}
		Given a subset \(A\subseteq \arcs(\Diag)\) 
		the arc ideal is
		\[
		\mathfrak{a}
		= 
		\mathfrak{a}(A)
		= 		
		\sum_{a\in \arcs(D)\smallsetminus A} 
		(x^a)
		\] in $R$. 
		In the particular case when $A$ consists of a single bonding arc $a^*$ in \(\Diag\) 
		we will write  \(\mathfrak{a}(a^*)\) for this ideal. 
	\end{definition} 
	
	\begin{corollary}\label{cor:edge-ring-generator-criterion}
		Let \(A'\) be a subset of arcs of \(\Diag\) 
		containing an arc system \(A\). If $\mathfrak{a}'$ is the arc ideal associated with $A'$
		then \(R/\mathfrak{a}'\) 
		is freely generated as a polynomial ring over \(\Rd\) 
		by \(Y_i\) for \(i=1,\dots,m\) and 
		\(x^a\) for \(a\in A'\smallsetminus A\). 
		In particular, for any arc system \(A\) with associated arc ideal $\mathfrak{a}$, 
		\[R/\mathfrak{a}=\Rd[Y_1,\dots,Y_m].\]
	\end{corollary}
	
	\begin{proof}
		By definition, the ideal \(\mathfrak{a}'\) is generated 
		by \(x^a\) for \(a\in\arcs(\Diag)\smallsetminus A'\). 
		So taking the quotient of \(R\) by this ideal 
		simply amounts to setting some of the generators of \(R\) 
		from \cref{thm:edge-ring-generator-criterion-alt-all} equal to zero. 
	\end{proof}

	\begin{proposition}\label{thm:ideal-independent-of-arc-system}
		\(
		\mathfrak{a}(a^*)
		=
		\mathfrak{a}(A)
		\)
		for any bonding arc \(a^*\) contained in an arc system \(A\). 

	\end{proposition}

	Towards proving \cref{thm:ideal-independent-of-arc-system} some preliminary observations are needed. Set \(Y_\star=\star\in \Rd\) for \(\star\in\{0,x,x+y,y\}\).

	\begin{definition}\label{def:diagram-segments}
		The endpoints of a bonding arc \(a^*\) of \(\Diag\) 
		split each of the two open components of \(\Diag\) 
		into two connected components. 
		Call these components the 
		open diagram segments subordinate to \(a^*\) and
		denote them by \(c_0\), \(c_{x}\), \(c_{x+y}\), and \(c_{y}\), 
		ordered by their ends on the boundary of \(\Diag\), 
		starting at the top left and 
		going in counter-clockwise direction, 
		as illustrated in \cref{fig:diagram-segments}. 
	\end{definition}

	The next two lemmas follow directly from definitions.

	\begin{lemma}\label{lem:edge-label-boundary}
		 Let 
		\(e_\star\) be the edge on the open diagram segment \(c_\star\)
		that meets the boundary, for \(\star\in\{0,x,x+y,y\}\).
		Then 
		\(Y_{\star}=X^{e_\star}\in\Rd\).
		\qed
	\end{lemma}

	\begin{lemma}\label{lem:x-a-identifies}
		Let \(e\) and \(e'\) be two edges 
		that meet at an endpoint of an arc \(a\) in \(\Diag\). 
		Then \([X^e]=[X^{e'}]\in R/(x^a)\). 
		\qed
	\end{lemma}

	\begin{proposition}\label{prop:edge-label-independent}
		Let \(a^*\) be a bonding arc and 
		\(e\) an edge on a closed component or open diagram segment \(c_\star\). 
		Then \([X^e]=[Y_\star]\in R/\mathfrak{a}(a^*)\).
	\end{proposition}
	\begin{proof}%
		For closed components, 
		this follows directly from iteratively applying \cref{lem:x-a-identifies} 
		to the arcs that we meet along a path on the closed component.
		The same argument works for open diagram segments, using \cref{lem:edge-label-boundary}. 
	\end{proof}

	\begin{lemma}\label{lem:criterion-x-in-mathfrak-a}
		Let \(A\subset\arcs(\Diag)\) be a subset of arcs containing an arc system. 
		If \(a\in A\) is a non-bonding arc 
		such that the graph \(G_{A\smallsetminus\{a\}}\) is disconnected
		then \(x^a\in \mathfrak{a}(A)\).   
	\end{lemma}

	\begin{proof}%
		To simplify index notation below write \(\hat{a}\) for the arc \(a\) from the statement of the lemma. 
		By assumption 
		the graph \(G_{A\smallsetminus\{\hat{a}\}}\) consists of two connected components.
		Let \(G'\) be the connected component 
		whose vertices correspond to closed components of \(\Diag\) only. 
		Let \(A'\subset A\) be the subset of arcs that correspond to the edges in \(G'\).
		Let \(e'\) and \(e''\) be the two edges 
		that meet the arc \(\hat{a}\) and 
		that lie on a closed component corresponding to a vertex in \(G'\). 
		Then \(x^{\hat{a}}=\pm(X^{e'}-X^{e''})\in R\). 
	 	Consider
		\[
		S
		=
		\sum_{a\in A'} \sigma^a\rho^a
		=
		\sum_{a\in A'}
		\sum_{i=0}^3 \epsilon^a_i\sigma^{e^a_i} X^a_i 
		\]
		where 
		\(\epsilon^a_{i}=+1\) 
		if the edge \(e^a_i\) points away from \(a\) and 
		\(\epsilon^a_{i}=-1\) 
		if \(e^a_i\) points into \(a\),
		as in the proof of \cref{prop:boundary-edge-ring}.
		Now consider this expression 
		as an element in \(R/\mathfrak{a}\):
		By \cref{lem:x-a-identifies}, 
		this has the effect of identifying the edge variables of any two edges 
		if the two edges are connected by a path on \(|\Diag|\) 
		that does not meet any endpoint of an arc in \(A\) and,  
		moreover, the signs \(\sigma\) of these edges agree 
		if and only if the edges along the path change orientation an even number of times. 
		Therefore, each edge variable that appears in \([S]\) appears twice, 
		but with opposite sign, 
		except for \([X^{e'}]\) and \([X^{e''}]\), 
		which only appear once, but also with opposite sign. 
		So \([S]=\pm[x^{\hat{a}}]\in R/\mathfrak{a}\). 
		Since \(\rho^{a}=0\) in \(R\) for any arc \(a\) in \(\Diag\), 
		\(S=0\), 
		and hence \([x^{\hat{a}}]\in R/\mathfrak{a}\).
	\end{proof}
		
		\begin{proof}[Proof of \cref{thm:ideal-independent-of-arc-system}]
		For any non-bonding arc  \(a\subset A\) the graph 
		\(G_{A\smallsetminus\{a\}}\) is disconnected
		since \(G_A\) is a tree. 
		Now apply \cref{lem:criterion-x-in-mathfrak-a}.
	\end{proof}
	
	\subsubsection{Relations for \texorpdfstring{\(y^az^a\)}{y\^{}az\^{}a}}
	
Let $a^*$ be the bonding arc in some arc system $A$. 

\begin{definition}\label{def:mathfrak-A}
		Denote by \(\mathfrak{A}=\mathfrak{A}(A)\) 
		the sum of \(\mathfrak{a}(A)\) 
		with the ideal generated by the expressions 
		\(y^{a}z^{a}\) for \(a\in A\smallsetminus\{a^*\}\). 
	\end{definition}
	
	\begin{proposition}\label{thm:R-mod-A-generation}
		\(R/\mathfrak{A}\) is freely generated as an \(\Rd\)-module 
		by the expressions \(\prod_j Y_j^{\varepsilon_j}\) 
		where \((\varepsilon_1,\dots,\varepsilon_m)\in\{0,1\}^m\).
	\end{proposition}
	
	\begin{definition}\label{def:Delta-star}
		For \(\star\in\{0,x,x+y,y\}\) define 
		\[
		\Delta_\star =
		\begin{cases}
			0
			&
			\star=0
			\\
			x(z-y)
			&
			\star=x
			\\
			(x+y)z
			&
			\star=x+y
			\\
			y(z-x)
			&
			\star=y
		\end{cases}
		\]	so that $\Delta_\star=Y_{\star}(\varG+Y_{\star})$.
	\end{definition}

	\begin{lemma}\label{lem:yz}
		For any arc \(a\) in \(\Diag\), 
		\[
		X^a_1(\varG+X^a_1)
		-
		X^a_3(\varG+X^a_3)
		=
		\begin{cases*}
			\sigma^ax^az^a
			-
			y^az^a
			&
			if \(\lab(a)=\arcD\)
			\\
			\sigma^ay^az^a
			-
			x^az^a
			&
			if \(\lab(a)=\arcT\)
		\end{cases*}
		\]
	\end{lemma}
	\begin{proof}
		Let us consider the case \(\lab(a)=\arcD\) first. 
		Then, by definition,
		\begin{align*}
			\sigma^ax^az^a
			-
			y^az^a
			=
			(\sigma^ax^a - y^a)
			\cdot z^a
			&=
			(X^a_1-X^a_3)
			\cdot
			(X^a_1+X^a_3+\varG)
			\\
			&=
			(X^a_1)^2+X^a_1\varG
			-(X^a_3)^2-X^a_3\varG.
		\end{align*}
		The identity for \(\lab(a)=\arcT\) follows by observing that 
		\(
		\sigma^ay^a_\arcT z^a
		-
		x^a_\arcT z^a
		=
		\sigma^ax^a_\arcD z^a
		-
		y^a_\arcD z^a
		\).
	\end{proof}

	\begin{definition}
		Let \((a_1,\dots,a_k)\) be a path in \(G_A\) that 
		starts at a closed component \(c_i\) and 
		ends at an open component and 
		does not meet the other open component. 
		Then \(a_k\) ends on an open diagram segment \(c_\star\). 
		We say that \(c_i\) is connected to \(c_\star\) 
		via the sequence of arcs \(a_1,\dots,a_k\) in \(A\). 
	\end{definition}

	\begin{proposition}\label{prop:yz-sequence}
		Let \(c_i\) be a closed component and 
		\(c_\star\) the open diagram segment 
		that \(c_i\) is connected to
		via a sequence of arcs \(a_1,\dots,a_k\) in \(A\).
		Then there exist some signs \(\tau_i\in\{\pm1\}\) such that 
		\[
		\sum_{i=1}^k\tau_i\cdot [z^{a_i}y^{a_i}]
		=
		[Y_i]^2+\varG\cdot[Y_i]-\Delta_\star\in R/\mathfrak{a}.
		\]  
	\end{proposition}

	\begin{proof}
		By \cref{lem:yz}, 
		we can choose the signs \(\tau_i\) 
		such that the sum on the left is a telescoping sum and 
		such that \([Y_i](\varG+[Y_i])\) appears in the first summand without a sign. 
		We then use the fact that \(x^a\in\mathfrak{a}\) 
		by \cref{thm:ideal-independent-of-arc-system}
		and apply \cref{def:Delta-star}.
	\end{proof}

	\begin{corollary}\label{cor:yz}
		With notation as in \cref{prop:yz-sequence}, 
		\(
		[Y_i]^2+\varG\cdot[Y_i]
		=
		\Delta_\star\in R/\mathfrak{A}.
		\)
	\end{corollary}
	
	\begin{proof}
		Immediate from \cref{prop:yz-sequence,def:mathfrak-A}.
	\end{proof}
	
	\begin{proof}[Proof of \cref{thm:R-mod-A-generation}]
		Using \cref{prop:yz-sequence}, we can easily see that 
		\[
		\{[Y_i]^2+\varG\cdot[Y_i]-\Delta_\star\}_{i=1,\dots,m}
		\]
		is a generating set for the image of the ideal 
		\(\mathfrak{A}\) in \(R/\mathfrak{a}\). 
		By \cref{cor:edge-ring-generator-criterion}, 
		\( R/\mathfrak{a}=\Rd[Y_1,\dots,Y_m]\), 
		so the claim follows.
	\end{proof}

	The following lemmas are recorded for future use.
	
	\begin{lemma}\label{lem:null-homotopies:basic}
		The endomorphisms \(x\cdot\iW\), \(yz\cdot\iW\), \(y\cdot\iB\) and \(xz\cdot\iB\) are null-homotopic.
	\end{lemma}
	\begin{proof}
		The null-homotopy in each case consists of a diagonal arrow labelled \(1\).
	\end{proof}
	
	\begin{lemma}\label{lem:null-homotopies:Delta}
		For \(\star\in\{0,x,x+y,y\}\) and \(\lab\in\{\arcD,\arcT\}\), \(\Delta_\star\cdot\iota_\lab\) is null-homotopic.
	\end{lemma}
	
	\begin{proof}
		This follows from the observation that 
		\(\Delta_\star\cdot\iota_\lab\) 
		can be written as a linear combination 
		of the null-homotopic maps from \cref{lem:null-homotopies:basic}.
	\end{proof}
		
	\begin{table}[t]
		\centering
		\begin{tabular}{c|cccc}
			\toprule
			\begin{tikzpicture}
				\draw[thick] (-1,-1) to ( 1, 1);
				\draw[thick] ( 1,-1) to (-1, 1);
				\draw[fill] (0,0) circle[radius=2pt];
				\node[label] at (-0.45, 0.45){\(Y^{a^*}_0\)};
				\node[label] at ( 0.45, 0.45){\(Y^{a^*}_y\)};
				\node[label] at (-0.45, -0.45 ){\(Y^{a^*}_x\)};
				\node at (1, 1.15){\phantom{\(+\)}};
				\node at (-1,-1.15){\phantom{\(+\)}};
			\end{tikzpicture} 
			&
			\begin{tikzpicture}
				\draw[thick,->] (-1,-1) to ( 1, 1);
				\draw[thick,->] ( 1,-1) to (-1, 1);
				\draw[fill] (0,0) circle[radius=2pt];
				\node[label] at (-0.45, 0.45){\(X^{a^*}_0\)};
				\node[label] at ( 0.45, 0.45){\(X^{a^*}_3\)};
				\node[label] at (-0.45,-0.45){\(X^{a^*}_1\)};
				\node at (-1.15, 1.15){\(+\)};
				\node at ( 1.15, 1.15){\(-\)};
				\node at (-1.15,-1.15){\(+\)};
				\node at (0, 1.15){\phantom{\(+\)}};
				\node at (0,-1.15){\phantom{\(+\)}};
			\end{tikzpicture}
			&
			\begin{tikzpicture}
				\draw[thick,<-] (-1,-1) to ( 1, 1);
				\draw[thick,->] ( 1,-1) to (-1, 1);
				\draw[fill] (0,0) circle[radius=2pt];
				\node[label] at (-0.45, 0.45){\(X^{a^*}_3\)};
				\node[label] at ( 0.45, 0.45){\(X^{a^*}_2\)};
				\node[label] at (-0.45,-0.45){\(X^{a^*}_0\)};
				\node at (-1.15, 1.15){\(+\)};
				\node at ( 1.15, 1.15){\(+\)};
				\node at (-1.15,-1.15){\(-\)};
				\node at (0, 1.15){\phantom{\(+\)}};
				\node at (0,-1.15){\phantom{\(+\)}};
			\end{tikzpicture} 
			&
			\begin{tikzpicture}
				\draw[thick,<-] (-1,-1) to ( 1, 1);
				\draw[thick,<-] ( 1,-1) to (-1, 1);
				\draw[fill] (0,0) circle[radius=2pt];
				\node[label] at (-0.45, 0.45){\(X^{a^*}_2\)};
				\node[label] at ( 0.45, 0.45){\(X^{a^*}_1\)};
				\node[label] at (-0.45,-0.45){\(X^{a^*}_3\)};
				\node at (-1.15, 1.15){\(-\)};
				\node at ( 1.15, 1.15){\(+\)};
				\node at (-1.15,-1.15){\(-\)};
				\node at (0, 1.15){\phantom{\(+\)}};
				\node at (0,-1.15){\phantom{\(+\)}};
			\end{tikzpicture}
			&
			\begin{tikzpicture}
				\draw[thick,->] (-1,-1) to ( 1, 1);
				\draw[thick,<-] ( 1,-1) to (-1, 1);
				\draw[fill] (0,0) circle[radius=2pt];
				\node[label] at (-0.45, 0.45){\(X^{a^*}_1\)};
				\node[label] at ( 0.45, 0.45){\(X^{a^*}_0\)};
				\node[label] at (-0.45,-0.45){\(X^{a^*}_2\)};
				\node at (-1.15, 1.15){\(-\)};
				\node at ( 1.15, 1.15){\(-\)};
				\node at (-1.15,-1.15){\(+\)};
				\node at (0, 1.15){\phantom{\(+\)}};
				\node at (0,-1.15){\phantom{\(+\)}};
			\end{tikzpicture} 
			\\
			\midrule
			\(Y^{a^*}_{x}-Y^{a^*}_{0}\)
			&
			\(X^{a^*}_1-X^{a^*}_0\)
			&
			\DWo{\(X^{a^*}_0-X^{a^*}_3\)}
			&
			\(X^{a^*}_3-X^{a^*}_2\)
			&
			\DWo{\(X^{a^*}_2-X^{a^*}_1\)}
			\\
			\(Y^{a^*}_{y}-Y^{a^*}_{0}\)
			&
			\DWo{\(X^{a^*}_3-X^{a^*}_0\)}
			&
			\(X^{a^*}_2-X^{a^*}_3\)
			&
			\DWo{\(X^{a^*}_1-X^{a^*}_2\)}
			&
			\(X^{a^*}_0-X^{a^*}_1\)
			\\
			\midrule
			\(\tau^{a^*}_{\Diag}\)
			&
			\(+1\)
			&
			\(-1\)
			&
			\(-1\)
			&
			\(+1\)
			\\
			\midrule
			\(\sigma^{a^*}\)
			&
			\(+1\)
			&
			\(-1\)
			&
			\(+1\)
			&
			\(-1\)
			\\
			\bottomrule
		\end{tabular}
		\medskip
		\caption{
			An analysis for the proof of \cref{lem:def:tau_a_star} 
			of the four different ways how the edges meeting a bonding arc \(a^*\) 
			can be connected to the ends of the diagram. 
			The second row is relevant for the case
			\(\conn{\Diag}=\arcD\)
			and the third row for the case
			\(\conn{\Diag}=\arcT\).
			When \(\lab(a^*)=\arcT\), the entries in these rows are highlighted.
		}\label{tab:def:tau_a_star}
	\end{table}

	\subsubsection{Relations for \texorpdfstring{\(x^{a^*}\)}{x\^{}a*}}
	
	\begin{definition}
		Given a bonding arc \(a^*\) of \(\Diag\) and \(\star\in\{0,x,y\}\), 
		let \(e_\star(a^*)\) be the edge that lies on the open diagram segment \(c_\star\).
		We define 
		\(Y^{a^*}_\star=X^{e_\star(a^*)}\)
		and
		\(\tau^{a^*}_{\Diag}=\sigma^{e_x(a^*)}\in\{\pm1\}\).
	\end{definition}
	
	\begin{lemma}\label{lem:def:tau_a_star}
		For any bonding arc \(a^*\) of \(\Diag\), 
		\begin{align*}
			\tau^{a^*}_{\Diag} \cdot x^{a^*}
			&=
			\begin{cases*}
				Y^{a^*}_x-Y^{a^*}_0
				& 
				if \(\conn{\Diag}=\arcD\)
				\\
				Y^{a^*}_y-Y^{a^*}_0
				& 
				if \(\conn{\Diag}=\arcT\)
			\end{cases*}
		\end{align*}
	\end{lemma}
	\begin{proof}
		Consider the diagrams in the first row of \cref{tab:def:tau_a_star}. 
		The top left edge \(e_0(a^*)\) in each diagram lies on \(c_0\).
		Hence \(\sigma^{e_0(a^*)}=+1\)
		if and only if 
		the number of times the orientation of the edges changes along the path from \(e_0(a^*)\) to \(e_0\)
		is even. 
		This way, we can compute the signs on the edges 
		in \cref{tab:def:tau_a_star}. 
		In particular, we can determine 
		\(\sigma^{a^*}\) and \(\tau^{a^*}_{\Diag}\), 
		as shown in the last two rows of \cref{tab:def:tau_a_star}. 
		
		Next, we compute 
		\[
		\tau^{a^*}_{\Diag} \cdot x^{a^*}
		=
		\begin{cases*}
			\tau^{a^*}_{\Diag}\sigma^a\cdot(X^{a^*}_1-X^{a^*}_0)
			=
			\tau^{a^*}_{\Diag}\sigma^a\cdot(X^{a^*}_2-X^{a^*}_3)
			& 
			if \(\lab(a^*)=\arcD\)
			\\
			\tau^{a^*}_{\Diag}\cdot (X^{a^*}_3-X^{a^*}_0)
			=
			\tau^{a^*}_{\Diag}\cdot (X^{a^*}_2-X^{a^*}_1)
			& 
			if \(\lab(a^*)=\arcT\)
		\end{cases*}
		\]
		and compare these values to the expressions 
		\(Y^{a^*}_{x}-Y^{a^*}_{0}\) and \(Y^{a^*}_{y}-Y^{a^*}_{0}\) 
		shown in \cref{tab:def:tau_a_star}.
		Observing that 
		\(\lab(a^*)=\arcD \Leftrightarrow\conn{\Diag}=\arcD\) 
		in the second and fourth column of \cref{tab:def:tau_a_star} 
		and \(\lab(a^*)=\arcD \Leftrightarrow\conn{\Diag}=\arcT\) 
		in the third and fifth column of \cref{tab:def:tau_a_star}, 
		the desired identity can now be easily verified. 
	\end{proof}
	
	\begin{remark}
		The value of \(\tau^{a^*}_{\Diag}\) 
		is not an invariant of the underlying singular diagram of \(D\). 
	\end{remark}

	\begin{definition}
		We define
		\[
		\lambda
		=
		\lambda(\Diag)
		=
		\begin{cases*}
			x & if \(\conn{\Diag}=\arcD\)\\
			y & if \(\conn{\Diag}=\arcT\)
		\end{cases*}
		\]
	\end{definition}

	\begin{corollary}\label{cor:x-a-star}
		For any bonding arc \(a^*\) of \(\Diag\), 
		\(
		[\tau^{a^*}_{\Diag} \cdot x^{a^*}]
		=
		\lambda
		\in R/\mathfrak{a}(a^*). 
		\)
	\end{corollary}

	\begin{proof}
		This follows from \cref{prop:edge-label-independent,lem:def:tau_a_star}. 
	\end{proof}
	
	\subsection{Construction of the maps \texorpdfstring{\(\varphi_A\)}{φ\_A}}
	Let \(\Diag=\Diag_T(v)=(D,\lab)\) for some \(v\in\{0,1\}^n\). 
	As in the previous section, 
	we write \(x^a=x^{a}_{\lab(a)}\) and 
	\(y^a=y^{a}_{\lab(a)}\),
	as well as \(M^a=M^{a}_{\lab(a)}\),
	for any \(a\in\arcs(\Diag)\). 
	Furthermore, let \(\mathtt{i}_a=0\) 
	if \(a\) is a dotted arc 
	and \(\mathtt{i}_a=1\) otherwise. 
	  
	We start the construction of the map \(\varphi_A\) 
	by picking some bonding arc \(a^*\) of \(\Diag\), 
	which always exists by \cref{assums:tangle_diagram}, 
	and some arc system \(A\) containing \(a^*\). 
	Let \(\mathfrak{a}=\mathfrak{a}(A)\) 
	and \(\mathfrak{A}=\mathfrak{A}(A)\).
	We define the special deformation retract \(\varphi_A\) in three steps. 
	First, we define a special deformation retract 
	\[
	\varphi_A^1
	\co
	M(\Diag)
	=
	\bigotimes_{a\in\arcs(D)} M^a
	\longrightarrow
	\bigotimes_{a\in A} M^a
	\otimes 
	\bigotimes_{a\notin A}
	\hslash^{\mathtt{i}_a}
	R/(x^a)
	\cong
	\hslash^{\mathtt{j}}
	\bigotimes_{a\in A} 
	\left(
	M^a
	\otimes 
	R/\mathfrak{a}
	\right)
	\]
	of matrix factorizations over $R/\mathfrak{a}=\Rd[Y_1,\dots,Y_m]$.
	Here, \(\mathtt{j}\) is the number of thick arcs \(a\notin A\). 
	On the tensor factors for \(a\not\in A\), 
	the map is induced by the special deformation retractions 
	from \cref{lem:lin_sp_def_retr}; 
	on the other tensor factors, it is the identity. 
	For each \(a\in A\smallsetminus\{a^*\}\), 
	\(x^a\in\mathfrak{a}\) 
	by \cref{thm:ideal-independent-of-arc-system},
	so 
	\[
	\left(
	M^a
	\otimes 
	R/\mathfrak{a}
	\right)
	\cong
	\left(
	\hslash^{\mathtt{i}_a+1}q^1 R/\mathfrak{a}
	\xleftarrow{y^az^a}
	\hslash^{\mathtt{i}_a}q^0 R/\mathfrak{a}
	\right).
	\]
	Moreover, by \cref{cor:x-a-star}, 
	\[
	\left(
	M^{a^*}
	\otimes 
	R/\mathfrak{a}
	\right)
	=
	\left(
	\begin{tikzcd}[column sep=1cm, ampersand replacement=\&]
	\hslash^{\mathtt{i}_{a^*}+1}q^1  R/\mathfrak{a}
	\arrow[r,bend left=10, "\tau^{a^*}_{\Diag}\cdot\lambda"above]
	\&
	\hslash^{\mathtt{i}_{a^*}}q^0  R/\mathfrak{a}
	\arrow[l,bend left=10, "\tau^{a^*}_{\Diag}\cdot xyz/\lambda"below]
	\end{tikzcd}
	\right)
	\]
	as a matrix factorization over \( R/\mathfrak{a}\), 
	since the overall potential is \(xyz\).
	Applying \cref{lem:quad_sp_def_retr} $|A|-1$ times, 
	we obtain the special deformation retraction 
	\[
	\varphi_A^2
	\co
	\hslash^{\mathtt{j}}
	\bigotimes_{a\in A} 
	\left(
	M^a
	\otimes 
	 R/\mathfrak{a}
	\right)
	\longrightarrow
	\hslash^{\mathtt{j}+\mathtt{k}}
	q^{|A|-1}  R/\mathfrak{A}\otimes M^{a^*}
	\]
	of matrix factorizations over $\Rd$,
	where \(\mathtt{k}\) is the number of dotted arcs in \(A\smallsetminus\{a^*\}\).
	By \cref{thm:R-mod-A-generation}, 
		\[
		f_A
		\co
		q^{|A|-1}  R/\mathfrak{A}
		\longrightarrow
		V(\Diag)\otimes_{\field[\varG]} \Rd,
		\qquad
		\textstyle\prod_j Y_j^{\varepsilon_j}\mapsto \prod_j \mu_j^{\varepsilon_j},
		\]
	defines an \(\Rd\)-module homomorphism,
	which is in fact an isomorphism.	 
	This allows us to define an isomorphism%
	\[
	\varphi_A^3
	\co
	\hslash^{\mathtt{j}+\mathtt{k}}
	q^{|A|-1}  R/\mathfrak{A}\otimes  M^{a^*}
	\longrightarrow
	\hslash^{\mathtt{i}_v}
	V(\Diag)\otimes  M(o(\Diag))
	\]
	by
	\[
	\begin{tikzcd}[column sep=2cm]
	\hslash^{\mathtt{j}+\mathtt{k}+\mathtt{i}_{a^*}+1}
	q^{|A|}  R/\mathfrak{A}
	\arrow[d,bend left=10, "\tau^{a^*}_{\Diag}\cdot\lambda"right]
	\arrow{r}{\tau^{a^*}_{\Diag}\cdot f_A}
	&
	\hslash^{\mathtt{i}_v+\mathtt{m}}
	q^1  V(\Diag) \otimes_{\field[\varG]} \Rd
	\arrow[d,bend left=10, "\lambda"right]
	\\
	\hslash^{\mathtt{j}+\mathtt{k}+\mathtt{i}_{a^*}}
	q^{|A|-1}  R/\mathfrak{A}
	\arrow[u,bend left=10, "\tau^{a^*}_{\Diag}\cdot xyz/\lambda"left]
	\arrow{r}{f_A}
	&
	\hslash^{\mathtt{i}_v+\mathtt{m}+1}
	q^0 V(\Diag) \otimes_{\field[\varG]} \Rd
	\arrow[u,bend left=10, "xyz/\lambda"left]
	\end{tikzcd}
	\]
	where \(\mathtt{m}=1\) if \(\lambda=x\) and \(\mathtt{m}=0\) if \(\lambda=y\). 
	This map preserves \(\hslash\)-grading because, modulo 2,
	\[
	\mathtt{j}+\mathtt{k}+\mathtt{i}_{a^*}+1
	\equiv
	\sum_{a\notin A}\mathtt{i}_a
	+
	\smashoperator[r]{\sum_{a\in A\smallsetminus\{a^*\}}}(1+\mathtt{i}_a)
	+\mathtt{i}_{a^*}+1
	\equiv
	|A|+\smashoperator[r]{\sum_{a\in\arcs(\Diag)}}\mathtt{i}_a
	\equiv
	\mathtt{i}_v+\mathtt{m}.
	\]
	We now define the special deformation retract 
	\(\varphi_v=\varphi_A\) 
	as the composition 
	\(\varphi_A^3\circ\varphi_A^2\circ\varphi_A^1\). 
	As the notation suggests, 
	this map depends on the choice of an arc system \(A\). 
	However, as we will show next, 
	its homotopy type is independent of \(A\) up to signs.

	\subsection{Independence of \texorpdfstring{\(\varphi_A\)}{φ\_A} up to homotopy and signs}
	
	\begin{proposition}\label{prop:retracts-independent-of-arc-system}
		Let \(A_1\) and \(A_2\) be two arc systems 
		for two bonding arcs \(a_1^*\) and \(a_2^*\) 
		of a labelled diagram \(\Diag\), respectively. 
		Then \(\varphi_{A_1}\) is homotopic to 
		\(\varphi_{A_2}\) or \(-\varphi_{A_2}\). 
	\end{proposition}

	\begin{proof}
	As in the previous two subsections, 
	we write \(x^a=x^{a}_{\lab(a)}\) and \(y^a=y^{a}_{\lab(a)}\), 
	as well as \(M^a=M^{a}_{\lab(a)}\), 
	for any \(a\in\arcs(\Diag)\). 
	An induction argument shows 
	that we may assume without loss of generality that 
	the arc systems \(A_1\) and \(A_2\) 
	differ in a single arc. 
	Let \(B= A_1 \cap A_2\) so that 
	\(A_1=B\cup \{a_1\}\) and \(A_2=B\cup \{a_2\}\) 
	for some arcs \(a_1\neq a_2\). 
	We distinguish two cases.
	
	\myitheading{Case 1} 
	Suppose \(a_1=a^*_1\) is a bonding arc; 
	then so is \(a_2=a^*_2\). 
	Without loss of generality, 
	we may assume that when travelling 
	along the open component \(o^*\) of \(\Diag\), 
	starting at the distinguished end, 
	we first meet the arc \(a_1\) and then \(a_2\). %
	Let 
	\(\mathfrak{b}=\mathfrak{a}(A_1\cup A_2)\) 
	so that 
	\[
	\mathfrak{a}_1=\mathfrak{a}(A_1)=\mathfrak{b}+(x^{a_2})
	\quad
	\text{ and }
	\quad
	\mathfrak{a}_2=\mathfrak{a}(A_2)=\mathfrak{b}+(x^{a_1})
	\]
	By \cref{lem:criterion-x-in-mathfrak-a}, 
	\(x^a\in\mathfrak{b}\) for every arc \(a\neq a_1,a_2\) 
	that meets the open component \(o^*\) of \(\Diag\).
	Therefore,  
	by \cref{lem:x-a-identifies,lem:edge-label-boundary},
	\([Y^{a_1}_0]=0\) 
	and 
	\[
	\begin{cases*}
		[Y^{a_1}_x]=[Y^{a_2}_0]\in R/\mathfrak{b}
		\text{ and }
		[Y^{a_2}_x]=x\in R/\mathfrak{b}
		&
		if \(\conn{\Diag}=\arcD\) and
		\\
		[Y^{a_1}_y]=[Y^{a_2}_0]\in R/\mathfrak{b}
		\text{ and }
		[Y^{a_2}_y]=y\in R/\mathfrak{b}
		&
		if \(\conn{\Diag}=\arcT\). 
	\end{cases*}
	\]
	Together with \cref{lem:def:tau_a_star}, 
	this implies that in any case,
	\begin{equation}\label{eq:delooping:mf:delta_x_a:case1}
		[\tau_{\Diag}^{a_1}\cdot x^{a_1}]
		+
		[\tau_{\Diag}^{a_2}\cdot x^{a_2}]
		=
		\lambda(\Diag)
		=
		\lambda
		\in R/\mathfrak{b}. 
	\end{equation}
	By \cref{cor:edge-ring-generator-criterion}, 
	\(R/\mathfrak{b}\) is freely generated 
	as a polynomial ring over \(\Rd\) 
	by \([x^{a_2}]\) and \([Y_i]\) for \(i=1,\dots,m\).
	
	Let \(\mathfrak{B}\) be the sum of \(\mathfrak{b}\) and 
	the ideal in \(R\) generated by \(y^{a}z^{a}\) 
	where 
	\(
	a\in B
	=
	A_1 \cap A_2
	=
	A_1\smallsetminus\{a^*_1\}
	=
	A_2\smallsetminus\{a^*_2\}
	\),
	so that 
	\[
	\mathfrak{A}_1=\mathfrak{A}(A_1)=\mathfrak{B}+(x^{a_2})
	\quad
	\text{ and }
	\quad
	\mathfrak{A}_2=\mathfrak{A}(A_2)=\mathfrak{B}+(x^{a_1}).
	\]
	By \cref{prop:yz-sequence}, 
	the image of \(\mathfrak{B}\) in \(R/\mathfrak{a}_1\) 
	is generated by expressions 
	\([Y_i]^2+\varG [Y_i]+\alpha_i\), 
	where \(\alpha_i\in \Rd\) and \(i=1,\dots,n\).
	Therefore, the image of \(\mathfrak{B}\) in \(R/\mathfrak{b}\) 
	is generated by expressions 
	\([Y_i]^2+\varG [Y_i]+\alpha_i+[x^{a_2}]\cdot\beta_i\), 
	where \(\beta_i\in R/\mathfrak{b}\) and \(i=1,\dots,n\).
	By considering the quantum grading of these expressions, 
	we see that
	\(\beta_i\) is a linear expression in the \([Y_i]\) over \(\Rd[x^{a_2}]\). 
	Therefore, \(R/\mathfrak{B}\) is freely generated 
	as an \(\Rd\)-module by 
	\[
	b_{k,\varepsilon}
	=
	(\tau_{\Diag}^{a_2}\cdot [x^{a_2}])^k\prod_i[Y_i]^{\varepsilon_i}
	\quad
	\text{where }
	k\geq0
	\text{ and }
	\varepsilon 
	=
	(\varepsilon_1,\dots,\varepsilon_m)\in\{0,1\}^m.
	\]

	We are now ready to compare the deformation retractions 
	\(\varphi_{A_1}\) and \(\varphi_{A_2}\). 
	For \(i=1,2\), \cref{lem:lin_sp_def_retr} gives us a map 
	\[
	\pi_i\co
	M_{A_1,A_2}
	=
	R/\mathfrak{B}
	\otimes
	 M^{a_1}
	\otimes
	 M^{a_2}
	\longrightarrow
	 R/\mathfrak{A}_i
	\otimes
	M^{a^*_i},
	\]
	and the map \(\varphi^2_{A_i}\circ\varphi^1_{A_i}\) 
	factors through this map, as shown on the left of 
	\cref{fig:delooping:mf:homotopy:overview}. 
	For simplicity, we drop the gradings from the notation, 
	since they are immaterial;
	without loss of generality 
	we may assume that the Kozsul signs in the differential 
	of \(M_{A_1,A_2}\) 
	are as shown as in \cref{fig:delooping:mf:homotopy:i}
	which shows a more detailed version of the diagram 
	contained in the shaded region of 
	\cref{fig:delooping:mf:homotopy:overview}.
	It suffices to construct a homotopy \(\xi\) 
	between \(\varphi^3_{A_1}\circ\pi_1\) and \(\varphi^3_{A_2}\circ\pi_2\).
	The two pairs of consecutive dashed arrows 
	in the top half of \cref{fig:delooping:mf:homotopy:i}
	are equal to \(\varphi^3_{A_1}\circ\pi_1\). 
	Similarly, the dashed arrows in the bottom half 
	represent the map \(\varphi^3_{A_2}\circ\pi_2\). 
	We calculate the compositions of the four pairs 
	of consecutive dashed arrows in this diagram as follows:%
	\[
	(f_{A_1}\circ\pi_1)(b_{k,\varepsilon})
	=
	\begin{cases}
	\prod_j \mu_j^{\varepsilon_j}
	&
	k=0
	\\
	0 
	&
	k>0
	\end{cases}
	\quad
	\text{and}
	\quad
	(f_{A_2}\circ\pi_2)(b_{k,\varepsilon})
	=
	\lambda^k\textstyle\prod_j \mu_j^{\varepsilon_j}.
	\]
	For the first identity, 
	we use the fact that 
	\([x^{a_2}]=0 \in R/\mathfrak{A}_1\); 
	for the second identity, 
	we use \([x^{a_1}]=0\in R/\mathfrak{A}_2\) 
	together with relation~\eqref{eq:delooping:mf:delta_x_a:case1}.
	The map \(\xi\) represented by the dotted arrow in 
	\cref{fig:delooping:mf:homotopy:i} and defined by 
	\[
	\xi(b_{k,\varepsilon})
	=
	\begin{cases}
	0
	&
	k=0
	\\
	\lambda^{k-1} \prod_j \mu_j^{\varepsilon_j}
	&
	k>0
	\end{cases}
	\]
	establishes the desired homotopy. 
	Indeed, postcomposition with the differential gives
	\[
		d \circ \xi(b_{k,\varepsilon})
		=
		\begin{cases}
		0
		&
		k=0
		\\
		\lambda^{k} \prod_j \mu_j^{\varepsilon_j}
		&
		k>0
		\end{cases}
	\]
	Precomposition with the differential \(d_1\) labelled \(x^{a_1}\) gives
		\begin{align*}
		\xi \circ d_1(b_{k,\varepsilon})
		&=
		\xi ((\tau_{\Diag}^{a_1}\lambda-\tau_{\Diag}^{a_1}\tau_{\Diag}^{a_2}x^{a_2})\cdot b_{k,\varepsilon})
		\\
		&
		=
		\tau_{\Diag}^{a_1}\lambda \cdot \xi (b_{k,\varepsilon})
		-
		\tau_{\Diag}^{a_1} \cdot \xi (b_{k+1,\varepsilon})
		=
		-\tau_{\Diag}^{a_1} (f_{A_1}\circ\pi_1)(b_{k,\varepsilon})
	\end{align*}
	and similarly for the differential \(d_2\) labelled \(x^{a_2}\):
	\[
		\xi \circ d_2(b_{k,\varepsilon})
		=
		\xi (\tau_{\Diag}^{a_2}\cdot b_{k+1,\varepsilon})
		=
		\tau_{\Diag}^{a_2}\cdot \lambda^{k} \prod_j \mu_j^{\varepsilon_j}
		=
		\tau_{\Diag}^{a_2}(f_{A_2}\circ\pi_2)(b_{k,\varepsilon})
	\]
	Since \(\pi_1\) and \(f_{A_1}\) preserve \(\hslash\)-grading, \(\xi\) does not. 
	In summary, 
		\[
		\partial(\xi)
		=
		d \circ \xi - (-1)^{\hslash(\xi)} \xi\circ d=d \circ \xi + \xi\circ d
		=
		(\varphi^3_{A_2}\circ\pi_2)-(\varphi^3_{A_1}\circ\pi_1).
		\]
	
		\begin{figure}[t]
		\begin{subfigure}{\textwidth}	
			\centering
			\begin{tikzpicture}[scale=0.8]
				\path[fill,lightgray,rounded corners=30pt] (-7.5,0) -- (-3,2) -- (3,2) -- (8,0) -- (3,-2) -- (-3,-2) -- cycle;
				\node(L) at (-9,0){\(M(\Diag)\)};
				\node(C) at (-5,0){\(M_{A_1,A_2}\)};
				\node(T) at (0,1.4){\(R/\mathfrak{A}_1\otimes M^{a^*_1}\)};%
				\node(B) at (0,-1.4){\(R/\mathfrak{A}_2\otimes M^{a^*_2}\)};%
				\node(R) at (5,0){\(V(\Diag)\otimes M(o(\Diag))\)};
				\draw[->,dotted] (C) to [tight,above]node{\(\xi\)} (R);
				\draw[->,dashed,bend left=7] (T) to [tight,above right]node{\(\varphi^3_{A_1}\)} (R);
				5		\draw[->,dashed,bend left=7] (C) to [tight,above left]node{\(\pi_1\)} (T);
				\draw[->,dashed,bend right=7] (C) to [tight,below left]node{\(\pi_2\)} (B);
				\draw[->,dashed,bend right=7] (B) to [tight,below right]node{\(\varphi^3_{A_2}\)} (R);
				\draw[->,in=180,out=25] (L) to [tight,above left]node{\(\varphi^2_{A_1}\circ\varphi^1_{A_1}\)} (T);
				\draw[->] (L) to node{} (C);
				\draw[->,in=180,out=-25] (L) to [tight,below left]node{\(\varphi^2_{A_2}\circ\varphi^1_{A_2}\)} (B);
			\end{tikzpicture}
			\caption{%
				The general strategy for constructing the homotopy
				between \(\varphi_{A_1}\) and \(\varphi_{A_2}\).
				In case~2, \(a^*_1=a^*_2=a^*\).
			}
			\label{fig:delooping:mf:homotopy:overview}
		\end{subfigure}
		
		\begin{subfigure}{\textwidth}	
			\centering
			\begin{tikzpicture}[scale=0.7]
				\begin{scope}[shift={(-4,0)}]%
					\node(Lt) at (0,2){\(R/\mathfrak{B}\)};%
					\node(Lb) at (0,-2){\(R/\mathfrak{B}\)};%
					\node(Lr) at (2,0){\(R/\mathfrak{B}\)};%
					\node(Ll) at (-2,0){\(R/\mathfrak{B}\)};%
					\draw[->] (Lt) to[bend left=7,above right,tight]node{\(x^{a_1}\)} (Lr);
					\draw[->] (Lr) to[bend left=7,above right,tight]node{} (Lt);
					\draw[->] (Ll) to[bend left=7,above right,tight]node{\(x^{a_1}\)} (Lb);
					\draw[->] (Lb) to[bend left=7,above right,tight]node{} (Ll);
					\draw[->] (Ll) to[bend right=7,below right,tight]node{\(-x^{a_2}\)} (Lt);
					\draw[->] (Lt) to[bend right=7,below right,tight]node{} (Ll);
					\draw[->] (Lb) to[bend right=7,below right,tight]node{\(x^{a_2}\)} (Lr);
					\draw[->] (Lr) to[bend right=7,below right,tight]node{} (Lb);
				\end{scope}
				\begin{scope}[shift={(0,+1.5)}]%
					\node(Tt) at (0,2){\(R/\mathfrak{A}_1\)};%
					\node(Tr) at (2,0){\(R/\mathfrak{A}_1\)};%
					\draw[->] (Tt) to[bend left=7,above right,tight]node{\(x^{a_1}\)} (Tr);
					\draw[->] (Tr) to[bend left=7,above right,tight]node{} (Tt);
				\end{scope}
				\begin{scope}[shift={(0,-1.5)}]%
					\node(Bb) at (0,-2){\(R/\mathfrak{A}_2\)};%
					\node(Br) at (2,0){\(R/\mathfrak{A}_2\)};%
					\draw[->] (Bb) to[bend right=7,below right,tight]node{\(x^{a_2}\)} (Br);
					\draw[->] (Br) to[bend right=7,below right,tight]node{} (Bb);
				\end{scope}
				\begin{scope}[shift={(+8,0)}]%
					\node(Rr) at (2,0){\(V(\Diag)\otimes \Rd\)};%
					\node(Rl) at (-2,0){\(V(\Diag)\otimes \Rd\)};%
					\draw[->] (Rr) to[bend left=7,below,tight]node{} (Rl);
					\draw[->] (Rl) to[bend left=7,above,tight]node{\(\lambda\)} (Rr);
				\end{scope}
				\draw[dashed,->,out=30,in=-170] (Lt) to [tight,above left]node{\(\pi_1\)} (Tt);
				\draw[dashed,->,out=0,in=90] (Tt) to[tight,above right] node{\(\tau_{\Diag}^{a_1}\cdot f_{A_1}\)} (Rl);
				\draw[dashed,->,out=30,in=-170] (Lr) to [tight,above left]node{\(\pi_1\)} (Tr);
				\draw[dashed,->,out=0,in=90] (Tr) to [tight,above]node{\(f_{A_1}\)} (Rr);
				\draw[dashed,->,out=-30,in=170] (Lb) to [tight,below left]node{\(\pi_2\)} (Bb);
				\draw[dashed,->,out=0,in=-90] (Bb) to[tight,below right] node{\(\tau_{\Diag}^{a_2}\cdot f_{A_2}\)} (Rl);
				\draw[dashed,->,out=-30,in=170] (Lr) to [tight,below left]node{\(\pi_2\)} (Br);
				\draw[dashed,->,out=0,in=-90] (Br) to [tight,below]node{\(f_{A_2}\)} (Rr);
				\draw[->,dotted] (Lr) to [tight,above]node{\(\xi\)} (Rl);
			\end{tikzpicture}
			\caption{%
				The homotopy \(\xi\)
				between \(\varphi^3_{A_1}\circ\pi_1\)
				and \(\varphi^3_{A_2}\circ\pi_2\)
				in case 1
			}
			\label{fig:delooping:mf:homotopy:i}
		\end{subfigure}
		
		\begin{subfigure}{\textwidth}	
			\centering
			\begin{tikzpicture}[scale=0.6]
				\begin{scope}[shift={(-4,0)}]%
					\node(Lt) at (0,3){\(R/\mathfrak{B}\)};%
					\node(Lb) at (0,-3){\(R/\mathfrak{B}\)};%
					\node(Lr) at (3,0){\(R/\mathfrak{B}\)};%
					\node(Ll) at (-3,0){\(R/\mathfrak{B}\)};%
					\draw[->] (Lt) to[bend left=7,above right,tight]node{\(-y^{a_2}z^{a_2}\)} (Lr);
					\draw[->] (Lr) to[bend left=7,below left,tight,pos=0.3]node{\(-x^{a_2}\)} (Lt);
					\draw[->] (Ll) to[bend left=7,above right,tight,pos=0.3]node{\(y^{a_2}z^{a_2}\)} (Lb);
					\draw[->] (Lb) to[bend left=7,below left,tight]node{\(x^{a_2}\)} (Ll);
					\draw[->] (Ll) to[bend right=7,below right,tight,pos=0.3]node{\(y^{a_1}z^{a_1}\)} (Lt);
					\draw[->] (Lt) to[bend right=7,above left,tight]node{\(x^{a_1}\)} (Ll);
					\draw[->] (Lb) to[bend right=7,below right,tight]node{\(y^{a_1}z^{a_1}\)} (Lr);
					\draw[->] (Lr) to[bend right=7,above left,tight,pos=0.3]node{\(x^{a_1}\)} (Lb);
				\end{scope}
				\node(T) at (3,2.5){\(R/\mathfrak{A}_1\)};%
				\node(B) at (3,-2.5){\(R/\mathfrak{A}_2\)};%
				\node(R) at (8,0){\(V(\Diag)\otimes \Rd\)};
				\draw[dashed,->,bend left=7] (Lt) to [tight, above]node{\(\pi_1\)}  (T);
				\draw[dashed,->,bend left=7] (T) to [tight, above right]node{\(f_{A_1}\)} (R);
				\draw[dashed,->,bend right=7] (Lb) to [tight, below]node{\(\pi_2\)} (B);
				\draw[dashed,->,bend right=7] (B) to [tight, below right]node{\(f_{A_2}\)} (R);
				\draw[->,dotted] (Lr) to [tight,above]node{\(\xi\)} (R);
			\end{tikzpicture}
			\caption{%
				The homotopy \(\xi\)
				between \(\varphi^3_{A_1}\circ\pi_1\) 
				and \(\varphi^3_{A_2}\circ\pi_2\)
				in case 2.
				The overall tensor factor corresponding to the bonding arc \(a^*\) is omitted.
			}
			\label{fig:delooping:mf:homotopy:ii}
		\end{subfigure}
		\caption{%
			Commutative diagrams illustrating the second step 
			in the proof of \cref{lem:delooping:mf}
		}\label{fig:delooping:mf:homotopy}
	\end{figure}
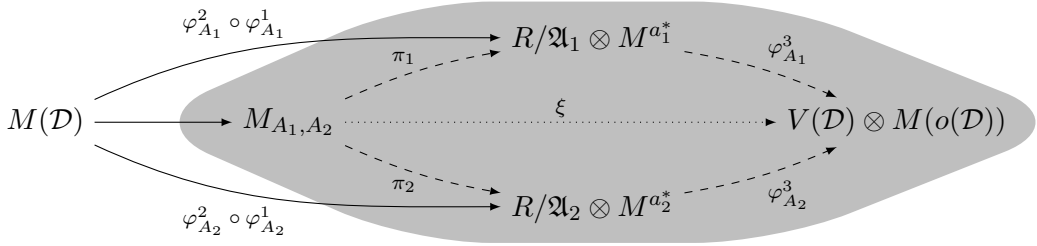
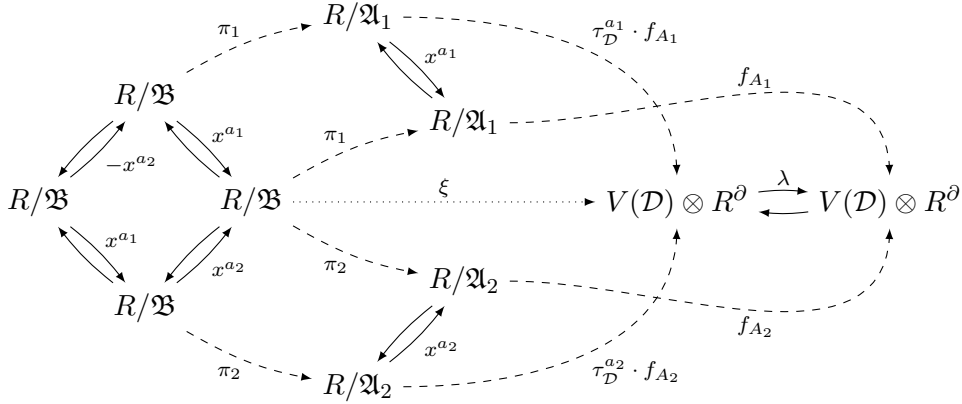
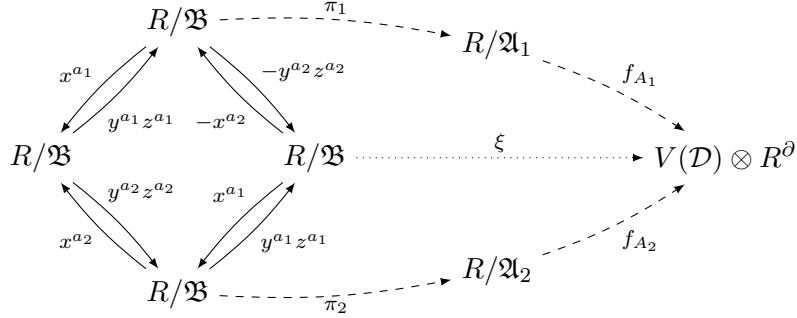
	
	\myitheading{Case 2} 
	Suppose \(a_1\) is not bonding arc; 
	then neither is \(a_2\). 
	Let \(a^*\) be the common bonding arc of \(A_1\) and \(A_2\). 
	The graph \(G_{A_1\cap A_2}\) has two connected components. 
	Let \(G'\) be the connected component 
	whose vertices correspond to
	closed components of \(\Diag\) only. 
	Each of the arcs \(a_1\) and \(a_2\) 
	connects \(G'\)
	to the other connected component of \(G_{A_1\cap A_2}\). 
	In particular, the arc \(a_1\) ends 
	on a closed component \(c_j\) in \(G'\). 
	Let \(e'\) and \(e''\) be the edges on \(c_j\) meeting \(a_1\). 
	By construction, these are different edges. 
	After potentially switching the roles of \(e'\) and \(e''\), 
	we can assume that 
	\(x^{a_1}=X^{e'}-X^{e''}\in R\). 
	As in the previous case, 
	let \(\mathfrak{b}=\mathfrak{a}(A_1\cup A_2)\) 
	so that by \cref{thm:ideal-independent-of-arc-system},
	\[
	\mathfrak{a}
	=
	\mathfrak{a}(a^*)
	=
	\mathfrak{b}+(x^{a_1})
	=
	\mathfrak{b}+(x^{a_2}).
	\]
	By \cref{cor:edge-ring-generator-criterion}, \(R/\mathfrak{b}\) is freely generated as a polynomial ring over \(\Rd\) by \([x^{a_1}]\) and \([Y_i]\) for \(i=1,\dots,m\). 
	
	Let \(\mathfrak{B}\) be the sum of \(\mathfrak{b}\) and the ideal in \(R\) generated by $y^{a}z^{a}$ for $a\in A_1\cap A_2 \smallsetminus\{a^*\}$, so that
	\[
	\mathfrak{A}_1=\mathfrak{A}(A_1)=\mathfrak{B}+(x^{a_2},y^{a_1}z^{a_1})
	\quad
	\text{ and }
	\quad
	\mathfrak{A}_2=\mathfrak{A}(A_2)=\mathfrak{B}+(x^{a_1},y^{a_2}z^{a_2}).
	\]
	We claim that \(R/\mathfrak{B}\) 
	is freely generated as an \(\Rd\)-module by 
	\[
	b_{k,\ell,\varepsilon}
	=
	(x^{a_1})^k
	(y^{a_1}z^{a_1})^\ell
	\prod_j[Y_j]^{\varepsilon_j}
	\quad
	\text{where }
	k,\ell\geq0
	\text{ and }
	\varepsilon = (\varepsilon_1,\dots,\varepsilon_m)\in\{0,1\}^m.
	\] 
	This can be seen as follows. 
	For \(i=1,\dots,m\), 
	suppose that \(c_i\) is connected in \(A_1\) 
	to the open diagram segment \(c_{\star(i)}\). 
	Then by \cref{prop:yz-sequence}, 
	the image of \(\mathfrak{B}\) in \(R/\mathfrak{a}\) 
	is generated by expressions 
	\([Y_i]^2+\varG [Y_i]-\alpha_i\) 
	for \(i\in\{1,\dots,m\}\smallsetminus\{j\}\), 
	where 
	\[
	\alpha_i
	=
	\begin{cases*}
		\Delta_{\star(i)}\pm [y^{a_1}z^{a_1}]
		&
		if \(c_i\) is a vertex in \(G'\),
		\\
		\Delta_{\star(i)}
		&
		otherwise.
	\end{cases*}
	\]
	Therefore, the image of \(\mathfrak{B}\) in \(R/\mathfrak{b}\) 
	is generated by expressions 
	\([Y_i]^2+\varG [Y_i]-\alpha_i+[x^{a_1}]\cdot\beta_i\) 
	for \(i\in\{1,\dots,m\}\smallsetminus\{j\}\), 
	where \(\beta_i\in R/\mathfrak{b}\).
	Similarly, 
	by \cref{lem:yz}
	\begin{align*}
		\pm[y^{a_1}z^{a_1}]
		&=
		[Y_j]^2+\varG[Y_j]-\alpha_j
		\in R/\mathfrak{a},
	\end{align*}
	where 
	\(\alpha_j=\Delta_\star\in \Rd\) 
	if \(a_1\) connects \(c_j\) to an open diagram segment, 
	namely \(c_\star\),
	and 
	\(\alpha_j=[Y_{j'}]^2+\varG [Y_{j'}]\) 
	if \(a_1\) connects \(c_j\) to a closed component, 
	namely \(c_{j'}\). 
	Therefore, 
	\[
	\pm[y^{a_1}z^{a_1}]
	=
	[Y_j]^2+\varG[Y_j]-\alpha_j
	+[x^{a_1}]\cdot \beta_j
	\in R/\mathfrak{b},
	\]
	where \(\beta_j\in R/\mathfrak{b}\).
	By considering the quantum grading of these expressions, 
	we see that for all \(i'\in\{1,\dots,m\}\), 
	\(\beta_{i'}\) is a linear expression in the edge variables \([Y_i]\) over \(\Rd[x^{a_1}]\).
	So by first eliminating all non-linear expressions in \([Y_j]\) 
	and then all non-linear expressions in \([Y_i]\) 
	for \(i\in\{1,\dots,m\}\smallsetminus\{j\}\), 
	we can express any element in \(R/\mathfrak{B}\) 
	as a unique linear combination of elements 
	\(b_{k,\ell,\varepsilon}\). 
	
	Next, we write the image of \(y^{a_2}z^{a_2}\) 
	in \(R/\mathfrak{B}\) in terms of our basis. 
	First, by applying \cref{prop:yz-sequence} 
	to the closed component \(c_j\) twice, 
	once for the arc system \(A_1\) and 
	once for the arc system \(A_2\), 
	we observe that there exists a sign $\tau\in\{\pm1\}$ such that 
	\[
		[y^{a_1}z^{a_1}]
		-
		\tau\cdot [y^{a_2}z^{a_2}]
		=\Delta \in R/(\mathfrak{B}+(x^{a_1})),
	\]
	for some \(\Delta\in \Rd\). 
	So there is some \(\alpha\in R/\mathfrak{B}\) such that 
	\begin{equation}\label{eq:delooping:mf:delta_yz:case2}
		[y^{a_1}z^{a_1}]
		-
		\tau\cdot [y^{a_2}z^{a_2}]
		=
		\Delta
		+
		\alpha\cdot[x^{a_1}]
		\in R/\mathfrak{B}. 
	\end{equation}
	
	We are now ready to compare the deformation retractions 
	\(\varphi_{A_1}\) and \(\varphi_{A_2}\). 
	For \(i=1,2\), 
	the map \(\varphi^2_{A_i}\circ\varphi^1_{A_i}\) 
	factors through some map 
	\[
	\pi_i\co
	M_{A_1,A_2}
	=
	R/\mathfrak{B}
	\otimes
	M^{a_1}
	\otimes
	M^{a_2}
	\otimes
	M^{a^*}
	\rightarrow
	R/\mathfrak{A}_i
	\otimes
	M^{a^*}.
	\]
	As in Case~1, 
	it therefore suffices to construct a homotopy between 
	\(\varphi^3_{A_1}\circ\pi_1\) and \(\varphi^3_{A_2}\circ\pi_2\).
	\cref{fig:delooping:mf:homotopy:ii} 
	shows a more detailed version of the diagram 
	contained in the shaded region of 
	\cref{fig:delooping:mf:homotopy:overview}. 
	On the tensor factors \(M^{a^*}\), 
	the maps \(\varphi^3_{A_i}\circ\pi_i\) 
	are identical. 
	When restricted to the first tensor factors, 
	they are determined by
	\[
	(f_{A_1}\circ\pi_1)\co 
	b_{k,\ell,\varepsilon}
	\mapsto
	\begin{cases}
	\prod_j \mu_j^{\varepsilon_j}
	&
	k=0=\ell
	\\
	0 
	&
	k>0\text{ or } \ell>0
	\end{cases}
	\quad
	\text{and}
	\quad
	(f_{A_2}\circ\pi_2)
	\co
	b_{k,\ell,\varepsilon}
	\mapsto
	\begin{cases}
	\Delta^\ell\prod_j \mu_j^{\varepsilon_j}
	&
	k=0
	\\
	0 
	&
	k>0
	\end{cases}
	\]
	To determine the second map, 
	we are using the 
	relation~\eqref{eq:delooping:mf:delta_yz:case2}.
	The desired homotopy \(\xi\) is now given by the map 
	represented by the dotted arrow in 
	\cref{fig:delooping:mf:homotopy:ii} 
	and defined by 
	\[
	\xi
	\co
	b_{k,\ell,\varepsilon}
	\mapsto
	\begin{cases}
	\Delta^{\ell-1}\prod_j \mu_j^{\varepsilon_j}
	&
	k=0\text{ and }\ell\neq0
	\\
	0 
	&
	\text{otherwise}
	\end{cases}
	\]
	on the first tensor factor and 
	by the identity on the tensor factor \(M^{a^*}\).
	Indeed, 
	precomposition with the differential \(d_2\) 
	labelled \(-y^{a_2}z^{a_2}\) 
	gives
	\begin{align*}
		\xi \circ d_2(b_{k,\ell,\varepsilon})
		&=
		-\tau\cdot \xi (( y^{a_1}z^{a_1}-\Delta -\alpha\cdot [x^{a_1}])\cdot b_{k,\ell,\varepsilon})
		\\
		&
		=
		-\tau\cdot 
		\Big(
			\xi (b_{k,\ell+1,\varepsilon})
			-
			\Delta\cdot \xi (b_{k,\ell,\varepsilon})
			-
			\underbrace{\xi(\alpha\cdot  b_{k+1,\ell,\varepsilon})}_{=0}
		\Big)
		=
		-\tau\cdot (f_{A_1}\circ\pi_1)(b_{k,\ell,\varepsilon})
	\end{align*}
	and precomposition with the differential \(d_1\) labelled \(y^{a_1}z^{a_1}\) gives
		\[
		\xi \circ d_1(b_{k,\ell,\varepsilon})
		=
		\xi (b_{k,\ell+1,\varepsilon})
		=
		(f_{A_2}\circ\pi_2)(b_{k,\ell,\varepsilon})
		\]
		In summary, 
		\[
		\partial(\xi)
		=
		d \circ \xi - (-1)^{\hslash(\xi)} \xi\circ d=d \circ \xi + \xi\circ d
		=
		(f_{A_2}\circ\pi_2)
		-
		\tau\cdot (f_{A_1}\circ\pi_1)
		\]
		Thus, \((f_{A_1}\circ\pi_1)\) is chain homotopic to \(\tau\cdot(f_{A_2}\circ\pi_2)\). 
\end{proof}

	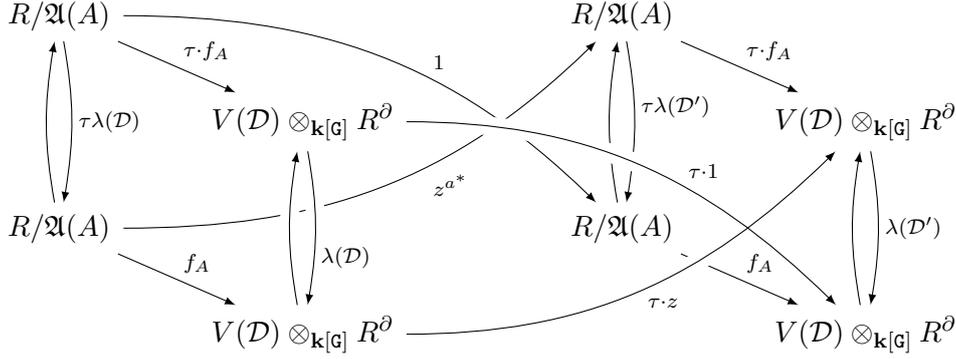
\begin{figure}[t]
		\(
			\begin{tikzcd}[column sep=1cm]
			R/\mathfrak{A}(A)
			\arrow[bend left=10]{dd}{\tau\lambda(\Diag)}
			\arrow{rd}{\tau\cdot f_A}
			\arrow[in=135,out=0]{rrrdd}[pos=0.6]{1}
			&&&
			R/\mathfrak{A}(A)
			\arrow[bend left=10]{dd}[pos=0.4]{\tau\lambda(\Diag')}
			\arrow[dd,bend right=10,leftarrow]
			\arrow{rd}{\tau\cdot f_A}
			\arrow[out=-135,in=0,leftarrow]{llldd}[pos=0.4]{z^{a^*}}
			\\
			&
			V(\Diag) \otimes_{\field[\varG]} \Rd
			\arrow[in=135,out=0,phantom,crossing over]{rrrdd}{}
			&&&
			V(\Diag) \otimes_{\field[\varG]} \Rd
			\arrow[dd,bend left=10, "\lambda(\Diag')"right]
			\\
			R/\mathfrak{A}(A)
			\arrow[uu,bend left=10]
			\arrow{rd}{f_A}
			&&&
			R/\mathfrak{A}(A)
			\arrow{rd}{f_A}
			\\
			&
			V(\Diag) \otimes_{\field[\varG]} \Rd
			\arrow[uu,bend left=10,crossing over]
			\arrow[uu,bend right=10, "\lambda(\Diag)"right,crossing over,leftarrow,pos=0.3]
			\arrow[in=-135,out=0,crossing over]{rrruu}[pos=0.5,swap]{\tau\cdot z}
			&&&
			V(\Diag) \otimes_{\field[\varG]} \Rd
			\arrow[uu,bend left=10]
			\arrow[in=0,out=135,leftarrow]{llluu}[pos=0.4,swap]{\tau\cdot 1}
		\end{tikzcd}
		\)
		\caption{%
			The edge map in the third step, case 1, 
			in the proof of \cref{lem:delooping:mf} 
		}\label{fig:delooping:mf:edge_maps:step1}
	\end{figure}
	
\subsection{Identification of edge maps up to signs}
	
	We now show that the constructed deformation retracts \(\varphi_v\) 
	satisfy the compatibility condition with edge maps \(D^{v}_{v'}\) 
	as stated in \cref{lem:delooping:mf}. 
	By \cref{prop:retracts-independent-of-arc-system}, 
	it suffices to verify this condition 
	for any arc systems on \(\Diag_T(v)\) and \(\Diag_T(v')\) 
	that we find convenient. 
	
	For notation, let \(\hat{a}\) be the node \(a^v_{v'}\)
	at which the labels of 
	\(\Diag=\Diag_T(v)=(D,\lab)\) and 
	\(\Diag'=\Diag_T(v')=(D,\lab')\) differ. 
	Let \(\lab_0=\lab(\hat{a})\) 
	and \(\lab_1=\lab'(\hat{a})\).
	Observe that \(\hat{a}\) is either a bonding arc 
	in both \(\Diag\) and \(\Diag'\) 
	or in neither labelled diagrams. 
	We treat these two cases separately.  
	
	\myitheading{Case 1}
	Suppose \(\hat{a}\) is a bonding arc in both \(\Diag\) and \(\Diag'\). 
	We can choose a common arc system \(A\) 
	for \(\Diag\) and \(\Diag'\)
	containing the bonding arc \(a^*=\hat{a}\). 
	Note that 
	\(\tau^{a^*}_{\Diag}=\tau^{a^*}_{\Diag'}=\tau\), 
	since the labelling of \(\Diag\) and \(\Diag'\) 
	agree away from \(a^*\). 
	By \cref{prop:edge-label-independent},
	\[
	[z^{a^*}]
	=
	[X^{a^*}_1+X^{a^*}_3+\varG]
	=
	[X^{a^*}_0+X^{a^*}_2+\varG]
	=
	[Y^{a^*}_x+Y^{a^*}_y+\varG]
	=
	x+y+\varG=z\in R/\mathfrak{a}(A).
	\]
	Therefore 
	the diagram in \cref{fig:delooping:mf:edge_maps:step1}
	commutes. 
	Moreover, its front face agrees with the map \(\mathcal{S}^v_{v'}\) 
	up to the sign \(\tau\).
	Hence
	\(
	\tau\cdot\mathcal{S}^v_{v'}\circ \varphi_v
	=
	\varphi_{v'}\circ \boldnowarningf
	\).
	So if \(\psi_v\) is the right inverse of the special deformation retract \(\varphi_v\), 
	\(
	\tau\cdot\mathcal{S}^v_{v'}
	=
	\varphi_{v'}\circ \boldnowarningf \circ \psi_v
	\).	

	\myitheading{Case 2}
	Suppose \(\hat{a}\) is a not bonding arc 
	in \(\Diag\) nor in \(\Diag'\). 
	In this case, \(\Diag'\) is obtained from \(\Diag\) 
	by either merging a closed component with another component or 
	splitting a closed component off another component.
	Thus, one diagram has one more closed component than the other diagram. 
	Let us pick a bonding arc \(a^*\) 
	in the diagram with one more closed component. 
	The arc \(a^*\) is then also a bonding arc in the other diagram.
	We treat merge and split maps separately:

	\myitheading{Merge maps} 
	Suppose \(\Diag'\) is obtained from \(\Diag\) by merging two components. 
	Let \(A=A_1\) be an arc system for \(\Diag'\) 
	containing the bonding arc \(a^*\). 
	Since the ends of the arc \(\hat{a}\) 
	lie on the same component of \(\Diag'\), 
	\(\hat{a}\not\in A\), and 
	\(A_0=A\cup\{\hat{a}\}\) is an arc system for~\(\Diag\). 	
	Let \(\mathfrak{A}\) be the ideal generated by 
	\(x^{a}_{\lab(a)}=x^{a}_{\lab'(a)}\) 
	for arcs \(a\not\in A_0\) of \(\Diag\) and
	\(y^{a}_{\lab(a)}z^{a}=y^{a}_{\lab'(a)}z^{a}\) 
	for all \(a\in A\smallsetminus \{a^*\}\).
	Observe that
	\[
	\mathfrak{A}_0
	=
	\mathfrak{A}(A_0)
	=
	\mathfrak{A}+(y^{\hat{a}}_{\ell_0}z^{\hat{a}})
	\subseteq
	\mathfrak{A}+(x^{\hat{a}}_{\ell_1})
	=
	\mathfrak{A}(A_1)
	=
	\mathfrak{A}_1.
	\]
	This inclusion comes from the fact that 
	\(x^{\hat{a}}_{\ell_1}=y^{\hat{a}}_{\ell_0}\). 
	
	The left hand side of \cref{fig:delooping:mf:edge_maps:step2} 
	shows the special deformation retract for \(M(\Diag)\) 
	obtained by applying the same procedure 
	as in the construction of the map \(\varphi_{A_0}\), 
	except that we do not simplify the matrix factorization 
	for the arc \(\hat{a}\). 
	The right hand side is a special deformation retract of \(M(\Diag')\). 
	It can be obtained from the first 
	by replacing the matrix factorization \(M^{\hat{a}}_{\lab_0}\) 
	by \(M^{\hat{a}}_{\lab_1}\).
	The map induced by \(\mathbf{f}\) 
	is indicated by the two diagonal arrows. 
	Only one of them matters to us, 
	namely the diagonal arrow 
	from the top left 
	to the bottom right
	labelled by the identity. 
	Indeed, after eliminating 
	\(y^{\hat{a}}_{\lab_0}z^{\hat{a}}\) and 
	\(x^{\hat{a}}_{\lab_1}\), respectively, 
	using \cref{lem:lin_sp_def_retr}, 
	this map induces the horizontal map 
	at the top of the following diagram:
	\[
	\begin{tikzcd}
		R/\mathfrak{A}_0\otimes M^{a^*}
		\arrow{d}{\varphi^3_{A_0}}
		\arrow{r}{1\otimes 1}
		&
		R/\mathfrak{A}_1\otimes M^{a^*}
		\arrow{d}{\varphi^3_{A_1}}
		\\
		V(\Diag)\otimes M(o(\Diag))
		\arrow{r}{}
		&
		V(\Diag')\otimes M(o(\Diag'))
	\end{tikzcd}
	\] 
	Noting that 
	\(o(\Diag')=o(\Diag)\) and that 
	\(\tau^{a^*}_{\Diag}=\tau^{a^*}_{\Diag'}\) 
	in the definition of 
	\(\varphi^3_{A_0}\) and \(\varphi^3_{A_1}\), 
	we restrict our attention to the first tensor factors:
	\[
	\begin{tikzcd}
		R/\mathfrak{A}_0
		\arrow{d}{f_{A_0}}
		\arrow{r}{1}
		&
		R/\mathfrak{A}_1
		\arrow{d}{f_{A_1}}
		\\
		V(\Diag)
		\arrow{r}{}
		&
		V(\Diag')
	\end{tikzcd}
	\] 
	We now do a case analysis to express the horizontal map 
	at the bottom of this diagram in terms of our basis 
	given by products of \(\mu_j\). 
	
	Suppose the saddle map merges two closed components 
	\(c_i\) and \(c_j\) of \(\Diag\) 
	to some closed component \(c_k\) of \(\Diag'\). 
	Let \(C\) be the index set of closed components 
	not involved in the merge operation.  
	Observe that as elements of \(R/\mathfrak{a}(A_1)\), 
	\([Y_i]=[Y_k]=[Y_j]\) by 
	\cref{prop:edge-label-independent}.
	Moreover, as elements of \(R/\mathfrak{A}_1\), \([Y_i][Y_j]=[Y_k]^2=\Delta_\star-\varG\cdot[Y_k]\) 
	by \cref{cor:yz}, 
	where \(c_\star\) is the open diagram segment 
	that \(c_k\) is connected to in \(A_1\). 
	Thus, the induced map is given by
	\[
	\mu_i^{\varepsilon_i}\mu_j^{\varepsilon_j}
	\textstyle\prod_{\ell\in C} \mu_\ell^{\varepsilon_\ell}
	\mapsto
	\begin{cases}
	\textstyle\prod_{\ell\in C} \mu_\ell^{\varepsilon_\ell}
	&
	\varepsilon_i=0=\varepsilon_j
	\\
	\mu_k\textstyle\prod_{\ell\in C} \mu_\ell^{\varepsilon_\ell}
	&
	\{\varepsilon_i,\varepsilon_j\}=\{0,1\}
	\\
	(\Delta_\star-\varG\cdot \mu_k)\textstyle\prod_{\ell\in C} \mu_\ell^{\varepsilon_\ell}
	&
	\varepsilon_i=1=\varepsilon_j
	\end{cases}
	\]
	After tensoring with the identity map on \(M(o(\Diag))=M(o(\Diag'))\), 
	we can eliminate the term \(\Delta_\star\) in the last case by a homotopy, 
	using \cref{lem:null-homotopies:Delta}.
	So the induced map is homotopic to \(\mathcal{S}^v_{v'}\). 
	
	Suppose the saddle map merges a closed component \(c_i\) 
	with an open segment \(c_\star\). 
	Let \(C\) be as above. 
	Then \([Y_i]=[Y_\star]\in \Rd\subset R/\mathfrak{a}(A_1)\) 
	and the induced map is given by 
	\begin{equation}\label{eq:delooping:mf:merge_with_open}
	\mu_i^{\varepsilon_i}\textstyle\prod_{\ell\in C} \mu_\ell^{\varepsilon_\ell}
	\mapsto
	\begin{cases}
	\prod_{\ell\in C} \mu_\ell^{\varepsilon_\ell}
	&
	\varepsilon_i=0
	\\
	[Y_\star]\prod_{\ell\in C} \mu_\ell^{\varepsilon_\ell}
	&
	\varepsilon_i=1
	\end{cases}
	\end{equation}
	The desired and actual values of \([Y_\star]\) 
	are shown in the second and third column of 
	\cref{tab:delooping:mf:edge_maps}. 
	Observe that either they agree or they differ by
	the summand \(x\) for \(\conn{\Diag}=\arcD\) and 
	the summand \(y\) for \(\conn{\Diag}=\arcT\). 
	In the first case, we are done. 
	In the second case, 
	we apply \cref{lem:null-homotopies:basic} to see that, 
	after tensoring with the identity map on \(M(o(\Diag))=M(o(\Diag'))\), 
	the induced map becomes homotopic to \(\mathcal{S}^v_{v'}\).

	\begin{figure}[t]
		\(
		\begin{tikzcd}[column sep=3cm,row sep=1cm]
			R/\mathfrak{A}\otimes M^{a^*}
			\arrow[d,bend left=10, "x^{\hat{a}}_{\lab_0}\otimes 1"right]
			\arrow[in=180,out=0]{rd}[description,label,pos=0.25]{1\otimes 1}
			&
			R/\mathfrak{A}\otimes M^{a^*}
			\arrow[d,bend left=10, "x^{\hat{a}}_{\lab_1}\otimes 1"right]
			\\
			R/\mathfrak{A}\otimes M^{a^*}
			\arrow[u,bend left=10, "y^{\hat{a}}_{\lab_0}z^{\hat{a}}\otimes 1"left]
			\arrow[in=180,out=0]{ru}[description,label,pos=0.25]{z^{\hat{a}}\otimes 1}
			&
			R/\mathfrak{A}\otimes M^{a^*}
			\arrow[u,bend left=10, "y^{\hat{a}}_{\lab_1}z^{\hat{a}}\otimes 1"left]
		\end{tikzcd}
		\)
		\caption{%
			The edge map in the third step, case 2, 
			in the proof of \cref{lem:delooping:mf}.
		}\label{fig:delooping:mf:edge_maps:step2}
	\end{figure}
	
	\begin{table}[t]
		\centering
		\begin{tabular}{c|c|c|c|c}
			\toprule
			&
			\multicolumn{2}{c|}{merge maps}
			&
			\multicolumn{2}{c}{split maps}
			\\
			\(\star\)
			&
			\(\operatorname{x}=\arcD\)
			&
			\(\operatorname{x}=\arcT\)
			&
			\(\operatorname{x}=\arcD\)
			&
			\(\operatorname{x}=\arcT\)
			\\
			\midrule
			\(0\)
			&
			\(0\)
			&
			\(0\)
			&
			\(\mu_i+\varG\)
			&
			\(\mu_i+\varG\)
			\\
			\(x\)
			&
			\DW{}{x}
			&
			\(x\)
			&
			\DW{\mu_i+\varG}{+x}
			&
			\DW{\mu_i+z}{-y}
			\\
			\(x+y\)
			&
			\DW{y}{+x}
			&
			\DW{x}{+y}
			&
			\(\mu_i+z\)
			&
			\(\mu_i+z\)
			\\
			\(y\)
			&
			\(y\)
			&
			\DW{}{y}
			&
			\DW{\mu_i+z}{-x}
			&
			\DW{\mu_i+\varG}{+y}
			\\
			\bottomrule
		\end{tabular}
		\medskip
		\caption{%
			Comparison of the edge maps induced by 
			merging/splitting with an open component. 
			\(\operatorname{x}= \conn{\Diag}=\conn{\Diag'}\).
			The first column indicates the open diagram segment \(c_\star\)
			subordinate to \(a^*\) 
			that is involved in the merge/split operation at \(\hat{a}\). 
			The non-highlighted entries in the second and third column 
			show the desired values of \([Y_\star]\in\Rd\) 
			for the map in \eqref{eq:delooping:mf:merge_with_open} 
			to agree with \(\mathcal{S}^v_{v'}\);
			the complete entries show the actual values of \([Y_\star]\in\Rd\). 
			Similarly, the last two columns show 
			the desired values of \((\mu_i+[Y_\star]+\varG)\) 
			for the map in \eqref{eq:delooping:mf:split_from_open} 
			to agree with \(\mathcal{S}^v_{v'}\); 
			the full entries show their actual values.
		}
		\label{tab:delooping:mf:edge_maps}
	\end{table}
	
	\myitheading{Split maps}
	Suppose \(\Diag'\) is obtained by splitting 
	a closed component off a component in \(\Diag\).
	Let \(A=A_0\) be an arc system for \(\Diag\) containing \(a^*\). 
	As above, $\hat{a}\not\in A$.
	Then \(A_1=A\cup\{\hat{a}\}\) is an arc system for \(\Diag'\). 
	Let \(\mathfrak{A}\) be the ideal 
	generated by \(x^{a}_{\lab(a)}=x^{a}_{\lab'(a)}\) 
	for arcs \(a\not\in A_1\) of \(\Diag'\) and 
	\(y^{a}_{\lab(a)}z^{a}=y^{a}_{\lab'(a)}z^{a}\) 
	for all \(a\in A\smallsetminus\{a^*\}\).
	Observe that
	\[
	\mathfrak{A}_0
	=
	\mathfrak{A}(A_0)
	=
	\mathfrak{A}+(x^{\hat{a}}_{\ell_0})
	\quad
	\text{and}
	\quad
	\mathfrak{A}_1
	=
	\mathfrak{A}(A_1)
	=
	\mathfrak{A}+(y^{\hat{a}}_{\ell_1}z^{\hat{a}}), 
	\]
	and since \(x^{\hat{a}}_{\ell_0}=y^{\hat{a}}_{\ell_1}\), 
	we get 
	\(\mathfrak{A}_0z^{\hat{a}}\subseteq \mathfrak{A}_1\). 
	
	As in the case of the merge maps, 
	we apply special deformation retractions 
	to all factors but the ones for \(a^*\) and \(\hat{a}\). 
	We obtain the same diagram as before, 
	namely the one in \cref{fig:delooping:mf:edge_maps:step2}.
	However, when passing to 
	\(V(\Diag)\otimes M(o(\Diag))\) and 
	\(V(\Diag')\otimes M(o(\Diag'))\), 
	we now need to eliminate 
	\(x^{\hat{a}}_{\lab_0}\) and 
	\(y^{\hat{a}}_{\lab_1}z^{\hat{a}}\), respectively. 
	So the edge map is induced by the map 
	indicated by the diagonal arrow 
	from the bottom left 
	to the top right. 
	We obtain the following diagram:
	\[
	\begin{tikzcd}
		R/\mathfrak{A}_0\otimes M^{a^*}
		\arrow{d}{\varphi^3_{A_0}}
		\arrow{r}{z^{\hat{a}}\otimes 1}
		&
		R/\mathfrak{A}_1\otimes M^{a^*}
		\arrow{d}{\varphi^3_{A_1}}
		\\
		V(\Diag)\otimes M(o(\Diag))
		\arrow{r}{}
		&
		V(\Diag')\otimes M(o(\Diag'))
	\end{tikzcd}
	\] 
	As in the previous case, 
	\(o(\Diag')=o(\Diag)\) and 
	\(\tau^{a^*}_{\Diag}=\tau^{a^*}_{\Diag'}\), 
	so we restrict our attention to the first tensor factors:
	\[
	\begin{tikzcd}
		R/\mathfrak{A}_0
		\arrow{d}{f_{A_0}}
		\arrow{r}{z^{\hat{a}}}
		&
		R/\mathfrak{A}_1
		\arrow{d}{f_{A_1}}
		\\
		V(\Diag)
		\arrow{r}{}
		&
		V(\Diag')
	\end{tikzcd}
	\] 
	By \cref{prop:edge-label-independent}, 
	\(
	z^{\hat{a}}
	=
	[Y_\diamond]+[Y_{\vartriangle}]+\varG
	\in R/\mathfrak{a}(A_1)
	\), 
	where \(c_\diamond\) and \(c_\vartriangle\) 
	are the two closed components/open diagram segments 
	connected by the arc \({\hat{a}}\) in \(\Diag'\). 
	Again, we distinguish two subcases:
	
	Suppose the saddle map splits a closed component \(c_k\) of~\(\Diag\) 
	into two closed components \(c_i\) and \(c_j\) of~\(\Diag'\). 
	Let \(C\) be the index set of closed components not involved in the merge operation. 
	Let \(c_\star\) be the open diagram segment 
	that \(c_k\) is connected to via arcs in \(A\). 
	Since $[Y_k]=[Y_i]=[Y_j]\in  R/\mathfrak{a}(A_0)$, 
	we have 
	\[
	[Y_k]
	\mapsto 
	[Y_j]([Y_i]+[Y_j]+\varG)
	=
	[Y_i][Y_j]+[Y_j](\varG+[Y_j])
	=
	[Y_i][Y_j]+\Delta_\star
	\]
	by \cref{cor:yz},  
	and the induced map is given by 
	\[
	\mu_k^{\varepsilon_k}
	\textstyle\prod_{\ell\in C} \mu_\ell^{\varepsilon_\ell}
	\mapsto
	\begin{cases}
	(\mu_i+\mu_j+\varG)\prod_{\ell\in C} \mu_\ell^{\varepsilon_\ell}
	&
	\varepsilon_k=0
	\\
	(\mu_i\mu_j+\Delta_\star)\textstyle\prod_{\ell\in C} \mu_\ell^{\varepsilon_\ell}
	&
	\varepsilon_k=1
	\end{cases}
	\]
	As in the case of the merge maps, 
	we now use \cref{lem:null-homotopies:Delta} 
	to conclude that the induced map 
	is homotopic to \(\mathcal{S}^v_{v'}\). 
	
	Suppose the saddle map splits a closed component \(c_i\) 
	off an open diagram segment \(c_\star\). 
	Let \(C\) be as above. 
	Then \([Y_i]=[Y_\star]\in R/\mathfrak{a}(A_0)\) and 
	the induced map is given by 
	\begin{equation}\label{eq:delooping:mf:split_from_open}
	\textstyle\prod_{\ell\in C} \mu_\ell^{\varepsilon_\ell}
	\mapsto
	(\mu_i+[Y_\star]+\varG)
	\textstyle\prod_{\ell\in C} \mu_\ell^{\varepsilon_\ell}
	\end{equation}
	The desired and actual values of \((\mu_i+[Y_\star]+\varG)\) 
	are shown in the last two columns of \cref{tab:delooping:mf:edge_maps}.
	As in the case of the merge maps, 
	we now use \cref{lem:null-homotopies:basic} 
	to conclude that the induced map is homotopic to \(\mathcal{S}^v_{v'}\). 
\qed

\bigskip
\begin{small}
	\pdfbookmark[section]{Acknowledgements}{Acknowledgements}
	\noindent\textbf{Acknowledgements.}
	The authors thank John Baldwin, William Ballinger, Adeel Khan, Yank\i\ Lekili, and Lukas Lewark for helpful conversations. 
	In the initial phase of this project, AK was supported by an AMS-Simons travel grant; LW was supported by an NSERC discovery accelerator supplement and the Simons Foundation and the Centre de Recherches Math\'ematiques through the Simons-CRM scholar-in-residence program; and CZ was supported by Lukas Lewark's Emmy Noether grant (project no.~412851057) from the DFG, as well as the SFB 1085 Higher Invariants in Regensburg. LW is currently supported by an NSERC discovery grant and a Killam accelerator research fellowship; and CZ is currently supported by an individual research grant (project no.~505125645) from the DFG and a Starting Grant (project no.~101115822) from the European Research Council (ERC).
\end{small}

\newcommand*{\arxiv}[1]{(\href{http://arxiv.org/abs/#1}{ArXiv:\ #1})}
\newcommand*{\arxivPreprint}[1]{\href{http://arxiv.org/abs/#1}{ArXiv preprint #1}}
\nocite{khtpp}
\bibliographystyle{myamsalpha}
\bibliography{main}

\end{document}